%% file: maybeThisTime.tex
\documentclass{article}
\usepackage{amsmath}
\usepackage{amsthm}
\usepackage{amssymb}
\usepackage{hyperref}
\usepackage{verbatim}
\usepackage{graphicx}
	
\usepackage[hmargin=3cm,vmargin=3.5cm]{geometry}
\usepackage[all,cmtip]{xy}
\def\a{\alpha}
\def\b{\beta}
\def\ga{\gamma}

\def\de{\delta}
\def\De{\Delta}
\def\ep{\epsilon}

\def\la{\lambda}
\def\La{\Lambda}
\def\si{\sigma}

\def\om{\omega}
\def\Om{\Omega}

\def\th{\theta}

\def\nab{\nabla}
\def\varep{\varepsilon}

\def\NN{{\cal N}}
\def\II{{\cal I}}

\def\HH{{\cal H}}

\def\LL{{\cal L}}

\def\OO{{\cal O}}
\def\PP{{\cal P}}

\def\RR{{\cal R}}
\def\SS{{\cal S}}

\newcommand{\F}[0]{\mathbb{F}}
\newcommand{\R}[0]{\mathbb{R}}
\newcommand{\Z}[0]{\mathbb{Z}}

\newcommand{\Q}[0]{\mathbb{Q}}
\newcommand{\C}[0]{\mathbb{C}}

\newcommand{\T}[0]{\mathbb{T}}

\newcommand{\pd}[2]{\frac{\partial #1}{\partial #2}}
\newcommand{\fr}[2]{\frac{#1}{#2}}

\newcommand{\vect}[1]{\left[ \begin{array}{c} #1 \end{array} \right]}
\newcommand{\mat}[2]{\left[ \begin{array}{ #1} #2 \end{array} \right]}

\newcommand{\tx}[1]{\mbox{#1}}
\newcommand{\leqc}[0]{\stackrel{<}{\sim}}

\newcommand{\wed}[0]{\wedge}
\newcommand{\pr}[0]{\partial}

\newcommand{\co}[1]{\| #1 \|_{C^0}}

\newcommand{\Ddt}[0]{\fr{\bar D}{\partial t}}
\newcommand{\Ddtof}[1]{\fr{\bar D #1}{\partial t}}
\newcommand{\DDdt}[0]{\fr{\bar D^2}{\partial t^2}}
\newcommand{\DDdtof}[1]{\fr{\bar D^2 #1}{\partial t^2}}
\newcommand{\badFactor}[0]{ B_\la^{1/2} \left( \fr{e_v^{1/2}}{e_R^{1/2} N} \right)^{1/2} }
\newcommand{\badFact}[0]{ \left( \fr{e_v^{1/2}}{e_R^{1/2} N} \right)^{-1/2} }

\newcommand{\ali}[1]{ \begin{align} #1 \end{align} }

\def\XXint#1#2#3{{\setbox0=\hbox{$#1{#2#3}{\int}$}
     \vcenter{\hbox{$#2#3$}}\kern-.5\wd0}}

\newtheorem{thm}{Theorem}[section]
\newtheorem{lem}{Lemma}[section]

\newtheorem{prop}{Proposition}[section]
\newtheorem{cor}{Corollary}[section]
\newtheorem{conject}{Conjecture}[section]

\theoremstyle{definition}
\newtheorem{defn}{Definition}[section]
\newtheorem*{conj}{Conjecture}

\newtheorem{req}{Requirement}[section]
\newtheorem{samLem}{Sample Lemma}[section]

\theoremstyle{remark}

\title{ H\"{o}lder Continuous Euler Flows in Three Dimensions with Compact Support in Time }
\author{ Philip Isett\thanks{Department of Mathematics, Princeton University, Princeton, NJ 08544, USA. (\href{mailto:pisett@math.princeton.edu}{pisett@math.princeton.edu}). \newline Supported by NSF Graduate Research Fellowship Grant DGE-1148900. } }

\begin{document}
\maketitle
\begin{abstract}
Building on the recent work of C. De Lellis and L. Sz\'{e}kelyhidi, we construct global weak solutions to the three-dimensional incompressible Euler equations which are zero outside of a finite time interval and have velocity in the H\"{o}lder class $C_{t,x}^{1/5 - \ep}$.  By slightly modifying the proof, we show that every smooth solution to incompressible Euler on $(-2, 2) \times \T^3$ coincides on $(-1, 1) \times \T^3$ with some H\"{o}lder continuous solution that is constant outside $(-3/2, 3/2) \times \T^3$.  We also propose a conjecture related to our main result that would imply Onsager's conjecture that there exist energy dissipating solutions to Euler whose velocity fields have H\"{o}lder exponent $1/3 - \ep$.
\end{abstract}

\tableofcontents

\part{Introduction}
\input{introductionArxiv}
\section{The Euler Reynolds System} \label{motivation}
\input{motivation}
\part{General Considerations of the Scheme}
\input{generalConsiderations}
\section{Structure of the Paper}
\input{structureOfPaper}

\section{Basic Technical Outline}
\input{mainLemTechOutline}

\part{Basic Construction of the Correction} \label{basicConstruct}
\section{Notation} \label{notation}
We employ the Einstein summation convention, according to which there is an implied summation when a pair of indices is repeated (e.g. $\pr_j v^j$ is the divergence of a vector field, $(v \cdot \nab) f = v^j \pr_j f$, etc.).  We employ the conventions of abstract index notation, so that upper-indices and lower-indices distinguish contravariant and covariant tensors.  

We define the space $\SS = \tx{Sym}^2(\R^3)$ to be the $6$ dimensional space of symmetric, $(2,0)$ tensors on $\R^3$.  That is, the vectors in $\SS$ are symmetric bilinear maps $G^{jl} : (\R^3)^* \times (\R^3)^* \to \R$ on the dual of $\R^3$, whose action on a pair of covectors $u, v \in (\R^3)^*$ can be written as $G(u,v) = G^{jl} u_j v_l = G^{lj} u_j v_l = G(v,u)$.  We also define $\SS_+$ to be the cone of positive definite, symmetric, $(2,0)$ tensors.

We will use the notation
\[ (vu)^{jl} = \fr{1}{2}(v^j u^l + v^l u^j) \]
to refer to the symmetric product of two vectors $v$ and $u$.

At many points, we will write inequalities of the form
\[ X \unlhd Y \]
The symbol $\unlhd$ expresses that the above inequality is a goal, and that at some point later on in the proof we will have to show that the inequality
$X \leq Y$ does in fact hold.

The following notation concerning multi-indices will later be helpful for expressing higher order derivatives of a composition (c.f. Section (\ref{sec:estForCoarseScaleFlow}) below).
\begin{defn} \label{defn:orderedPartition} We say that a $K$-tuple of multi-indices $(a^1, a^2, \ldots, a^K)$ forms an ordered $K$-partition of a multi-index $a = (a_1, a_2, \ldots, a_N)$ if there is a partition $\{1, \ldots, N \} = \pi_1 \cup \ldots \cup \pi_K$, such that the subsets $\pi_j \subseteq \{ 1, \ldots, N \}$, $j = 1, \ldots, K$ are pairwise disjoint, and each $a^j$ has the form 
\[ (a^j_{\pi_j(1)}, a^j_{\pi_j(2)}, \ldots, a^j_{\pi_j(N_j)}) \] 
where $\pi_j = \{ \pi_j(1), \ldots, \pi_j(N_j) \}$ is written in increasing order, and the subsets are ``ordered'' in the sense that 
\[ \pi_1(N_1) < \pi_2(N_2) < \ldots < \pi_K(N_K) \]
\end{defn}



\section{A Main Lemma for Continuous Solutions}
\input{mainLemCts}
\section{The Divergence Equation} \label{sec:oscEstimate}
\input{oscillatoryEstimate}
\section{Constructing the Correction}
\subsection{Transportation of the Phase Functions} \label{transport}
\input{transport1}
\subsection{The High-High Interference Problem and Beltrami Flows} \label{highHighSection}
\input{highHigh1}

\subsection{Eliminating the Stress} \label{stress}
\input{stress1}

\part{Obtaining Solutions from the Construction} \label{obtainSolutions}
\section{Constructing Continuous Solutions}
\input{constructCts}

\section{Frequency and Energy Levels}
\input{freqEnLevels}
\section{The Main Iteration Lemma} \label{sec:mainLem}
\input{mainLem}
\section{The Main Lemma implies the Main Theorem} \label{sec:lemImpThm}
\input{newLemImpliesThm}
\section{Gluing Solutions}
\input{gluing}
\section{On Onsager's Conjecture} \label{sec:onOnsag}
\input{onOnsager}

\part{Construction of Weak Solutions in the H{\" o}lder Class : Preliminaries} \label{hardPart}
\section{Preparatory Lemmas}\label{prepareToMollify}
\input{mollifyPrep}
\section{The Coarse Scale Velocity}\label{coarseScaleVelocity}
\input{mollifyingV}
\section{Commutator Estimates and the Coarse Scale Flow}\label{coarseScaleFlow}
\input{coarseScaleFlow}
\section{Transport Estimates}\label{transportEstimates}
\input{lowFreqTransport}
\section{Mollification along the Coarse Scale Flow}
\input{fixMollifyAlongFlow}
\section{Accounting for the parameters and the problem with the High-High term} \label{sec:accountForParams}
\input{accountForParams}
\part{ Construction of Weak Solutions in the H{\" o}lder Class : Estimating the Correction }

\input{correctionAmplitudes}
\part{ Construction of Weak Solutions in the H{\" o}lder Class : Estimating the New Stress } \label{part:thirdHardPart}
\input{introToStressEstimates}

\section{Estimates for the Stress Terms not involving the Divergence Equation}
\input{mollTermEstimates}
\section{Estimates for the Stress Terms involving the Divergence Equation} \label{sec:oscStressTerms}
\input{oscillatoryStress}
\section{Transport-Elliptic Estimates} \label{sec:transEllipt}
\input{transportElliptic}

\part{Acknowledgements}

The author is endebted to his advisor Sergiu Klainerman for his guidance, for introducing him to the work of DeLellis and Sz{\' e}kelyhidi and for introducing him to many great opportunities during his graduate studies.  The author is thankful to professors C. Fefferman, A. Ionescu, V. Vicol and P. Constantin, and to his colleagues S-J. Oh and J. Luk for all of their help.  The author also thanks S-J. Oh for suggestions regarding the proofs in Section (\ref{sec:transEllipt}), and thanks Vlad Vicol and Ting Zhang for corrections to earlier drafts of the paper.  The author is deeply grateful to C. DeLellis and L. Sz\'{e}kelyhidi for discussions in 2011 about basic principles of convex integration and difficulties related to producing continuous Euler flows which helped to start the author's investigations in the subject.  The author also thanks M. Gromov and C. Villani for insightful conversations about convex integration. 

This work was supported by the NSF Graduate Research Fellowship Grant DGE-1148900.

\newpage

\part{References}
\bibliographystyle{plain}
\bibliography{maybeThisTime}

\newpage 
\part{Appendix} \label{part:Appendix}
\input{uniqueIco}

\end{document}

%% file: introductionArxiv.tex
In the paper \cite{deLSzeCts}, De Lellis and Sz\'{e}kelyhidi introduce a method by which one can construct weak solutions to the incompressible Euler equations
\begin{align} \label{eq:euSystem}
\left\{
 \begin{aligned} \pr_t v + \tx{div } v \otimes v + \nab p = 0  \\
 \tx{div } v = 0
 \end{aligned}
\right.
\end{align}
in three spatial dimensions which are continuous but do not conserve energy.  The motivation for constructing such solutions comes from a conjecture of Lars Onsager \cite{onsag} on the theory of turbulence in an ideal fluid.  In the modern language of PDE, Onsager's conjecture can be translated as follows

\begin{conj}[Onsager (1949)]

$ $

\begin{enumerate}
\item \label{itm:first} Weak solutions to the incompressible Euler equations obeying a H\"{o}lder estimate \[ |v(t,x+y) - v(t,x)| \leq C |y|^\a \] for some $\a > 1/3$ must conserve energy.
\item \label{itm:second} Furthermore, for any $\a < 1/3$, there exist weak solutions to the Euler equations which belong to $C^\a$ and which do not conserve energy
\end{enumerate}
\end{conj}

Onsager's conjecture can be appreciated in the context of the theory of turbulence famously proposed by Kolmogorov \cite{K41} in 1941.  One key postulate of Kolmogorov's theory is an anomalous dissipation of energy for solutions of the three-dimensional Navier Stokes equations
\begin{align}\label{eqn:navStokes}
\left\{
 \begin{aligned} \pr_t v + \tx{div } v \otimes v + \nab p = \nu \De v  \\
 \tx{div } v = 0
 \end{aligned}
\right.
\end{align}
in the low viscosity regime $\nu \to 0$ (or, more precisely, at high Reynolds number).  One formulation of anomalous dissipation is that a sequence of solutions $v_\nu$ to the three-dimensional Navier Stokes equations with the same initial data $v_\nu(0,x) = v(0,x)$ may have energy functions $e_\nu(t) = \fr{1}{2}\int |v_\nu|^2 dx$ which do not converge to a constant function of time as $\nu \to 0$, but rather may possess some energy dissipation independent of the viscosity parameter.  Kolmogorov's theory proposes that the phenomenon of anomalous dissipation is generic in a statistical sense for ensembles of solutions to the Navier Stokes equations at low viscosity.  The limiting energy dissipation rate is one of the main quantities that are proposed to govern the statistical properties of turbulent flows in Kolmogorov's theory.


Onsager proposed that one might be able to study turbulence even in the absence of viscosity through the Euler equations (the case $\nu = 0$), and that only low regularity solutions to the Euler equations can exhibit turbulent behavior since smooth solutions must conserve energy.  By studying the interactions between different frequency components of the solution which arise from the nonlinearity, Onsager proposed that anomalous dissipation could be explained in terms of a transfer of energy from coarser to smaller scales, and deduced that the exponent $1/3$ should be critical for energy conservation.  His notion of solution was based on a Fourier series representation, but it can be shown to be equivalent to the modern notion of a weak solution.  A review of his computations can be found in the note \cite{deLSzeCtsSurv}.

It is known that solutions to the incompressible Euler equations with H\"{o}lder regularity greater than $1/3$ conserve energy in any dimension, so part (\ref{itm:second}) of Onsager's conjecture has been settled.  A short proof of this statement was presented in \cite{CET} after a slightly weaker result was established in a series of papers by Eyink \cite{eyink}.  More precise results, as well as a discussion of what ``Onsager critical'' function space could best be used to model ideal turbulence, can be found in \cite{shvOns}.  In recent years, substantial progress has also been made toward constructing dissipative solutions with H{\" o}lder regularity less than $1/3$.


The first proof that weak solutions to the Euler equation need not conserve energy came in a groundbreaking paper of Scheffer \cite{scheff}, in which he produced weak solutions to the Euler equations with compact support in space and time belonging to the class $L^2(\R^2 \times \R)$.  Following Sheffer's discovery, in \cite{shnNonUnq}, Shnirelman found a simpler construction of weak solutions with compact support in time on $L^2(\T^2 \times \R)$.  Shnirelman later in \cite{shnDiss} produced weak solutions in the class $L_t^\infty L_x^2(\R \times \T^3)$ which dissipate energy using the concept of a generalized flow introduced by Y. Brenier.

In the breakthrough paper \cite{deLSzeIncl}, De Lellis and Sz\'{e}kelyhidi were able to construct weak solutions in the class $L^\infty(\R^n\times \R)$ for any $n \geq 2$.  In a subsequent paper \cite{deLSzeAdmiss}, they were also able to produce solutions belonging to the energy space $C_tL^2_x$, and their main theorem demonstrates that the energy density $\fr{1}{2} |v|^2(t,x)$ of these weak solutions can be essentially any prescribed non-negative, continuous function $e(t,x)$.  These breakthroughs led to new results concerning weak solutions to several equations of fluid dynamics which are surveyed in \cite{deLSzeHFluid}, and also demonstrated that many entropy criteria one might propose are unable to recover uniqueness of solutions in the energy class.  The constructions are performed through a technique known as convex integration, which originated in the work of Nash on $C^1$ isometric embeddings \cite{nashC1} and was generalized by Gromov to establish the $h$-principle in many other applications to topology and geometry (see \cite{gromPdr}).

When solving differential equations, the essence of the convex integration procedure, beginning with the work of Nash, is to first formulate a notion of ``subsolution'' to the equation one is trying to solve, and then to show that any given subsolution can be perturbed by adding a sequence of highly oscillatory corrections in such a way that a solution is achieved in the limit.  For isometric embeddings of, say, $u : S^2 \to \R^3$, Nash's notion of a subsolution is that of a short map, that is, a smooth map $u_0 : S^2 \to \R^3$ such the pullback of the Euclidean metric $Du^T Du$ is pointwise less than or equal to the metric on the sphere as a quadratic form; equivalently, a short map is one for which the arclength of $u(\gamma)$ is less than the arclength of $\gamma$ for any curve $\gamma$ on $S^2$.  For example, rescaling the standard sphere into a smaller ball is a short map.  Clearly, any map which can be uniformly approximated by isometric embeddings will be a short map.  To obtain an isometric embedding approximating an initial short map $u_0$, one adds a sequence of oscillatory corrections to $u_0$ which can be chosen arbitrarily smaller and smaller in the $C^0$ norm, but which still make significant changes in the derivative $Du$.  Iteratively adding such oscillations, one obtains a sequence of short maps converging to an isometric embedding in the limit\footnote{Actually the ``spiral'' corrections used in Nash's original argument would require that the $u$ map into a higher dimensional space such as $\R^4$, but shortly after Nash's paper, Kuiper was able to design corrections which can achieve the same goal in the codimension $1$ case. }.  On the other hand, the $C^2$ norms of these corrections grow without bound, and the embedding obtained by this procedure cannot be $C^2$, as $C^2$ isometric embeddings of $S^2 \hookrightarrow \R^3$ are unique up to a rigid motion.  We refer to \cite{deLSzeC1iso} for more regarding the $h$-principle for low regularity isometric embeddings.

In fact, there is a very useful analogy between the isometric embedding problem and the Euler equations where the velocity field $v$ plays a role analogous to the derivative $Du$ of the embedding.  Most of the analogies extend from the fact that the nonlinearities $v \otimes v$ and $Du Du^T$ for the two equations are both symmetric, quadratic, and non-negative in the appropriate sense.  This analogy was first discussed by De Lellis, Sz{\' e}kelyhidi and Conti in their work \cite{deLSzeC1iso} on proving the $h$-principle for $C^{1,\a}$ isometric embeddings.  See \cite{deLSzeHFluid} for further discussion.

The version of convex integration employed in \cite{deLSzeIncl, deLSzeAdmiss} for the Euler equations is very different from the original version of convex integration applied to the isometric embedding problem by Nash and extended in \cite{deLSzeC1iso}.  It is based on an extension of convex integration used by M\"{u}ller and \v{S}ver\'{a}k \cite{muSve, kMuSve} to construct solutions to differential inclusions $\nab u \in \RR$ which are only Lipschitz (i.e. $\nab u \in L^\infty$), rather than $C^1$ (i.e. $\nab u \in C^0$).  As originally explained by Kirchheim \cite{kirch}, this ``weak'' version of convex integration for Lipschitz maps can be implemented in an elegant and simple manner through Baire category arguments, or by an ``explicit'' iteration which is basically equivalent to the proof of the Baire category theorem.

The weak version of convex integration is unable to produce continuous solutions to the Euler equations.  Rather, the solutions produced by De Lellis and Sz{\' e}kelyhidi in \cite{deLSzeAdmiss} generically have no better regularity than the generic vector field of prescribed energy density $\fr{|v|^2}{2}(t,x) = e(t,x)$ when the set of such vector fields is equipped with a weak topology (for instance, the $L^\infty$ weak-$*$ topology will work, and using the topology $C_t H^{-1}_x$ can ensure that solutions belong to the energy space $C_tL_x^2$).  The fundamental obstruction to achieving continuous weak solutions by this variant of convex integration is that even though one can choose the frequencies of the oscillatory corrections so large that the corrections may be arbitrarily small in a weak topology, these corrections are still required to have a certain size in a strong topology ($C^0$ for Euler) in order for any noticeable progress towards achieving a solution to be measured.  The many solutions obtained by this construction are connected to Gromov's $h$-principle in that one actually shows that the subsolutions used to perform the construction can be approximated (in a weak topology) by solutions, and that the solutions generated by the process exhibit a huge amount of flexibility despite solving the equation.

Recently in \cite{deLSzeCts}, De Lellis and Sz\'{e}kelyhidi have made another outstanding breakthrough by constructing continuous weak solutions to the Euler equation on a periodic domain $\R \times (\R / \Z)^n$ ($n = 3$)  whose energy $\int_{\T^n} \fr{|v|^2}{2}(t,x) dx$ can be any smooth function $e : \R \to \R_{> 0}$ that is bounded below by a strictly positive constant.  In particular, these solutions may dissipate energy.  They also achieved the same result in $n = 2$ spatial dimensions in the preprint \cite{deLSzeCts2d} with A. Choffrut.  

Following their construction of continuous solutions and building on the methods in \cite{deLSzeC1iso}, they extended their method in the paper \cite{deLSzeHoldCts} to construct weak solutions to the Euler equations on $\R \times (\R / \Z)^3$ with velocity in the H\"{o}lder space $C_{t,x}^{1/10 - \ep}$ and having any prescribed energy obeying the same restrictions as in the continuous case.  This result was generalized to the case of two spatial dimensions with the same H{\" o}lder exponent of $1/10 - \ep$ by A. Choffrut in \cite{choff}.  The paper \cite{choff} also contains more detailed results describing the flexibility of the family of solutions produced by the method.

  The convex integration scheme used in these recent results more directly resembles the original scheme used by Nash to construct $C^1$ isometric embeddings, and it also bears more resemblance to the argument of Shnirelman in \cite{shnNonUnq} than does the argument in \cite{deLSzeIncl}.  To achieve their results, De Lellis and Sz{\' e}kelyhidi have introduced several important, new ideas which represent dramatic changes in the point of view of the convex integration scheme.  Although these ideas cannot be summarized at this stage of the introduction, we will refer to them as the analogous aspects arise in the present paper, and we urge the reader to study their papers.   For now we mention two new aspects:
\begin{itemize}
\item They introduce an underdetermined system of PDEs called the ``Euler-Reynolds equations'' which form the correct space of ``subsolutions'' in which to perform the convex integration procedure.
\item Surprisingly, they use linear spaces of high frequency, stationary solutions of the Euler equations called ``Beltrami flows'' in order to construct the basic building blocks.
\end{itemize} 
The idea that turbulent Euler flows may be constructed from Beltrami solutions has appeared in the turbulence literature, and was suggested to De Lellis and Sz\'{e}kelyhidi by Peter Constantin \cite{deLSzeCts}.

In this paper, we build upon and rework the convex integration scheme of De Lellis and Sz\'{e}kelyhidi in order to achieve the following theorem.
\begin{thm} \label{mainThm}  
For every $\de > 0$, there exists a nontrivial weak solution 
\begin{align}
v(t,x) : \R \times \T^3 &\to \R^3 \\
p(t,x) : \R \times \T^3 &\to \R 
\end{align} 
to the Euler equations which belong to the H{\" o}lder class 
\begin{align}
 v &\in C_{t,x}^{1/5 - \de} \\
 p &\in C_{t,x}^{2(1/5 - \de)} 
\end{align}
such that the support of $(v,p)$ is contained in a compact time interval.

\end{thm}

The framework we develop appears robust enough to obtain the regularity of $1/3 - \de$ conjectured by Onsager except for one term where stationary flows are used in a crucial way.  We discuss this difficulty as it arises in the argument.  We propose as a conjecture an ``Ideal Case'' scenario which summarizes what the method would yield if the $C^0$ norm of this term were suitably well-controlled.  This conjecture, if true, could be used to construct energy-dissipative solutions in the H{\" o}lder class $v \in C_{t,x}^{1/3 - \de}$, and in particular would imply Onsager's conjecture.

The proof of Theorem (\ref{mainThm}), which builds heavily upon the ideas of De Lellis and Sz\'{e}kelyhidi, also implements the method of convex integration.  The argument to be presented here is based on their approach in \cite{deLSzeCts}, but includes several novel features
\begin{itemize}
\item  We use nonlinear phase functions to form the basic building blocks of the construction, and adapt the method of De Lellis and Sz\'{e}kelyhidi for obtaining small solutions to the relevant elliptic equation.  We also develop a new method for constructing the amplitudes of these building blocks which is suitable for the use of nonlinear phase functions.
\item  Our main lemma, which summarizes the result of a single iteration of the scheme, is organized so that one easily controls the time interval supporting the solutions, allowing us to obtain solutions with support in a finite time interval, and to prove Theorem (\ref{thm:mainThm2}) below regarding the gluing of solutions.
\item  We define a notion of frequency and energy levels for measuring the size and derivatives of the error and approximate solutions during the iteration process.  This notion has the important feature that it distinguishes the bounds for the derivatives of the velocity, the pressure and the error.  The frequency energy levels also keep track of second derivative bounds which play an important role in some of the estimates and appear to be necessary for estimating one of the error terms in the conjectural ideal case scenario.
\item  We present a sharp, general framework for calculating the regularity achieved by the construction which is based on the above notion of frequency and energy levels.  This framework reduces regularity computations and bounds for other physical quantities to simple, linear algebra calculations.
\item  We isolate the material derivative $\pr_t + v \cdot \nab$ as a special derivative in the construction.  It appears that unless improved bounds for the material derivative are taken into account, the highest regularity one can achieve through this construction is $1/(3 + \sqrt{8}) - \de$.
\end{itemize}
To take advantage of the special role of the material derivative, we introduce several additional ideas into the scheme, for example:
\begin{itemize}
\item We incorporate improved bounds for $\pr_t + v \cdot \nab$ into the notion of frequency - energy levels.  In particular, the material derivative obeys better bounds than do the spatial derivatives or the time derivative. 
\item We use time averaging along the coarse scale flow of the fluid as a special form of mollification.
\item We introduce a ``transport-elliptic'' equation in order to eliminate the error in the parametrix.  
\item To bound material derivatives, we use estimates coming from the Euler-Reynolds equation itself and related commutator estimates to close the argument.
\end{itemize}
Considerations regarding the symmetries of the Euler equations, including scaling and Galilean transformations, also play an important role underlying the analysis.  In particular, thanks to the ideas listed above, the bounds for the iteration depend only on relative velocities (i.e. derivatives of the velocity) but not on absolute velocities (i.e. the $C^0$ norm of the velocity).

As an interesting observation, it turns out that the energy $\int \fr{|v|^2}{2}(t,x) dx$ for the solution obtained by the construction enjoys better regularity in time than what is proven for the solution itself (it is almost $C^{1/2}$ in $t$ for the $C_{t,x}^{1/5 - \de}$ solutions we construct).  In fact, in the conjectural ``ideal case'' scenario, the construction yields solutions whose energy functions are ``almost automatically'' in $C^1$ even though the velocity is only guaranteed to belong to $C_{t,x}^{1/3 - \de}$.

By slightly modifying the proof of Theorem (\ref{mainThm}), we also prove the following theorem regarding the gluing of solutions.
\begin{thm} \label{thm:mainThm2}
For every smooth solution $(v, p)$ to incompressible Euler on $(-2, 2) \times \T^3$, there exists a H\"{o}lder continuous solution to Euler $({\bar v}, {\bar p})$ which coincides with $(v, p)$ on $(-1,1) \times \T^3$ but is equal to a constant outside of $(-3/2, 3/2) \times \T^3$.
\end{thm}
The Theorem (\ref{thm:mainThm2}) above was motivated by related discussions with P. Constantin, Y. Sinai, and T. Buckmaster.

%% file: motivation.tex
We start by giving a motivation for the argument.  A similar discussion can be found in \cite{deLSzeCtsSurv}.

Consider any solution $(v_0, p_0)$ to the Euler equations on, say, $\R^n\times\R$, which we write in the form
\begin{equation} \label{eq:euEqns}
\left\{
 \begin{aligned} \pr_t v_0^l + \pr_j(v_0^j v_0^l) + \pr^l p_0 = 0  \\
 \pr_j v_0^j = 0
 \end{aligned}
\right.
\end{equation}
Here and in the rest of the paper, we use the Einstein summation convention, according to which we understand that there is a summation over the $j$ index because it is repeated.  

We imagine that $v_0$ could be very singular, such as one of the wild solutions that will be constructed in the argument.  One way to describe the ``coarse scale'' or ``low-frequency'' behavior of the velocity field $v_0$ is to consider a mollification $v_\ep = \eta_\ep \ast v_0$.  The mollifier $\eta_\ep$ in this argument could be a standard mollifier, so that the value of $v_\ep$ at each point is the average velocity in an $\ep$-neighborhood of the point, or the mollifier $\eta_\ep \ast$ could operate by projecting to frequencies less than $\fr{1}{\ep}$.  By mollifying the equations (\ref{eq:euEqns}), we see that the coarse scale velocity field $v = v_\ep$ and pressure $p = \eta_\ep \ast p_0$ satisfy the system of equations
\begin{equation} \label{eq:euReynMot}
\left\{
 \begin{aligned} \pr_t v^l + \pr_j(v^j v^l) + \pr^l p = \pr_j R^{jl}  \\
 \pr_j v^j = 0
 \end{aligned}
\right.
\end{equation}
with a symmetric, non-positive $(2,0)$ tensor $R^{jl} = v_\ep^j v_\ep^l - (v_0^j v_0^l)_\ep$ arising from the failure of the nonlinearity to commute with the averaging.  Since $v_0 \in L^2$, $R = R_\ep$ converges strongly to $0$ in $L^1$ as $\ep \to 0$, and if $v$ is continuous, $R_\ep$ converges to $0$ in $C^0$; in general $R_\ep$ converges to $0$ at a rate depending on the regularity of $v_0$.  The tensor $R^{jl}$ giving rise to the forcing term on the right hand side of the equation is an example of what is known in the theory of turbulence as a ``Reynolds stress'', and for this reason De Lellis and Sz{\' e}kelyhidi have named this system the ``Euler-Reynolds equations''.  Since the trace part of $R^{jl}$ can be regarded as part of the pressure, the trace free part of $R^{jl}$ measures the error for $v$ to be a solution to the Euler equations.  This system implicitly appears in Constantin, E, and Titi's proof \cite{CET} that $C^{0,\a}$ solutions to the Euler equations conserve energy if $\a > 1/3$.  In their proof, after mollifying the equation as above, one then integrates against the test function $v_\ep^l$, and proceeds to integrate by parts as in the usual proof of conservation of energy.  The key to their proof is a commutator estimate which establishes a quadratic bound $\| R_\ep \| \leq C \ep^{2 \a} \| v \|_{C_tC_x^\a}^2$ on the rate at which $R_\ep$ tends to $0$.

Actually, one can see that the average of {\bf any} family of solutions to Euler will be a solution of the Euler-Reynolds equations.  The most relevant type of averaging to convex integration arises during the operation of taking weak limits, which can be regarded as an averaging process when viewed in terms of Young measures.  For example, suppose that $v_n$ is a sequence of solutions to the Euler equations uniformly bounded by $|v_n| \leq M$.  Along some subsequence, which we also denote by $v_n$, the sequence of measures $\mu_n(t,x, \tilde{v}) = dt dx \de_{v_n}(\tilde{v})$ on $\R^n \times \R \times \R^n$ obtained by pushing forward the Lebesgue measure to the graph of $v_n$ will obtain a weak limit of the form $dt dx d\mu_{t,x}(v)$.  The parameterized family of measures $\mu_{t,x}(v)$ are probability measures which record the local in space-time oscillations of the subsequence $v_n$; measures which arise in this way are called ``Young measures'' since they were first introduced by Young to describe the behavior of minimizing sequences in the calculus of variations.  The center of mass $\bar{v}(t,x) = \int \tilde{v} d\mu_{t,x}(\tilde{v})$ of the limiting measure is the weak limit of the subsequence $v_n$, and obeys the Euler-Reynolds system with Reynolds stress $R^{jl}(t,x) = \bar{v}^j \bar{v}^l - \int (\tilde{v}^j \tilde{v}^l) d\mu(\tilde{v})$.  This statement follows from the general fact that for any continuous function $g(v)$ defined on $\{ |v| \leq M \}$ the weak limit of $g(v_n)$ along the subsequence $v_n$ will exist and be given by the expectation $\int g(\tilde{v}) d\mu_{t,x}(v)$.  While we do not use the theory of Young measures in this work, we find that their consideration is useful for visualizing and understanding the intuition behind the proofs in some of the previous literature and the one given here; see \cite{mulMic} for more on Young measures.  

Heuristically, the convex integration procedure begins with a solution $v_0$ to the Euler-Reynolds system, and obtains a solution to the Euler equations by ``reintroducing'' the oscillations responsible for the forces that are exerted on $v_0$ ``during'' the weak limiting process.  Namely, the method actually proves that essentially every solution to the Euler-Reynolds equations can be approximated in a weak topology by solutions to Euler.  Calculating what sort of velocity fields can be weakly approximated by solutions to Euler provides a candidate space of subsolutions (or approximate solutions) in which one can work while performing convex integration; however, it is only after the convex integration procedure is proven successful that we know the correct space has been found.  A priori, the class of solutions to the Euler-Reynolds system may seem too general, since any incompressible flow $v^l$ which conserves momentum can be regarded as a solution to Euler-Reynolds after solving $\pr_j R^{jl} = \pr_t v^l + \pr_j(v^jv^l)$.  Some more remarks about momentum conservation for Euler-Reynolds flows are included in Section (\ref{sec:momentumConservation}).

We now describe the convex integration scheme for Euler in some more technical detail.  Beginning with a solution $(v, p, R)$ of the Euler Reynolds system (\ref{eq:euReynMot}), we show that we can correct the velocity and pressure to obtain another velocity field $v_1 = v + V$, and pressure $p_1 = p + P$, which obey the Euler-Reynolds system with some new stress tensor $R_1^{jl}$ that is much smaller than the previous stress $R$ in an appropriate topology (for the present paper, $R$ will be measured in $C^0$).  We then iterate this procedure infinitely many times so that the stress tends to $0$; by ensuring that the sequence of velocity fields converges strongly in $L^2$ to some limit $\bar{v}$, we obtain a weak solution to the Euler equations in the limit by passing to the limit in the weak formulation of the Euler-Reynolds system.  At each stage of the iteration, the correction $V$ is chosen so that it oscillates very rapidly compared to the velocity field $v$ which results from the previous stage of the construction.  Therefore one can interpret the procedure as recovering some weak solution $\bar{v}$ to the Euler equations in a sequence, starting from the coarse scales and passing to fine scale oscillations, in a fashion similar to taking $\ep \to 0$ in the mollification procedure described above.  Thus, we think of $v$ as being basically a low frequency projection of the solution we will ultimately construct, and the correction $V$ is a higher frequency component of the solution to be constructed.  With this intuition in mind, we expect $R$ to behave analogously to the sequence $R^{jl}_\ep = (v_\ep^j v_\ep^l) - (v_0^j v_0^l)_\ep$ studied above.

In contrast to methods of solving differential equations which produce unique solutions, the choice of the correction $V$ at each stage is quite arbitrary, allowing one to obtain a huge family of solutions depending on these choices.  In particular, the frequency of the very first oscillation can be chosen arbitrarily large (without affecting at all the regularity of the solution obtained in the limit).  Therefore, since every correction to follow will also be of highly oscillatory nature, we can reason that if it is indeed possible to begin with an Euler Reynolds solution $v_0$ and then, by carrying out convex integration, arrive at a solution $\bar{v}$, it is necessary that $v_0$ can be approximated in a weak topology by solutions to the Euler equations.  By similar considerations, every velocity field $v_1, v_2, \ldots$ which arises after each iteration of the process can also be realized as such a weak limit.  These considerations explain how we are forced to calculate what weak limits can be reached by solutions to the Euler equations in order to construct a space of approximate solutions.  They also explain why the results obtained are connected to the $h$-principle; the method applied here always gives an approximation in a weak topology and relies on an abundance of solutions.  This feature of the method is demonstrated in \cite{choff}, which contains a characterization of the $H^{-1}$ closure of the space of $C_{t,x}^\a$ Euler flows for $\a < 1/10$ in $2$ and $3$ spatial dimensions.

%% file: generalConsiderations.tex
We now highlight the main issues and the general philosophy underlying the proof of the main lemma.  A part of this philosophy can also be found in \cite{deLSzeCtsSurv}, since it also underlies the proof in \cite{deLSzeCts}.  We use this section as an opportunity to introduce some heuristics which will be usefully formalized in the proof.

Suppose that $(v, p, R)$ are a given solution of the Euler-Reynolds system.
\begin{equation} \label{eq:euReyn}
\left\{
 \begin{aligned} \pr_t v^l + \pr_j(v^j v^l) + \pr^l p = \pr_j R^{jl}  \\
 \pr_j v^j = 0
 \end{aligned}
\right.
\end{equation}

For the purpose of the present discussion, let us imagine that $v$ and $p$ are smooth functions of size $5$, and that $R$ is a smooth, symmetric $(2,0)$ tensor field with absolute value smaller than one.

We introduce highly oscillatory corrections $V$ and $P$ to the velocity field and pressure such that $\tx{ div } V = \pr_j V^j = 0$.  The corrected velocity field $v_1 = v + V$ and pressure $p_1 = p + P$ satisfy the system
\begin{align*}
 \pr_t v_1^l + \pr_j(v_1^j v_1^l) + \pr^l p_1 &= \pr_t V^l + \pr_j(v^j V^l) + \pr_j(V^j v^l) + \pr_j(V^j V^l) + \pr^l P + \pr_j R^{jl} \\
\pr_j v_1^j &= 0
\end{align*}

Our goal is to choose high frequency corrections $V$ and $P$ so that the forcing term in the equation can be represented as $\pr_j R_1^{jl}$ for a new Reynolds stress $R_1^{jl}$ much smaller than $R^{jl}$.\footnote{ In general the word ``smaller'' here should refer to some norm which controls the $L^1_{t,x}$ norm since we expect $R$ to behave like the stress $R_\ep = (v_\ep^j v_\ep^l) - (v^j v^l)_\ep $ arising from mollifying a solution $v \in L^2_{t,x}$.  For the present paper, we will measure $R$ in the norm $C^0_{t,x}$.}  Let us express the gradient $\pr^l P$ as a divergence $\pr_j(P \de^{jl})$, where $\de^{jl}$ is the (inverse) Euclidean inner product (or, if one prefers, the ``Kronecker delta'' or the ``identity matrix'').  We then collect terms as follows:
\begin{align*}
 \pr_t v_1^l + \pr_j(v_1^j v_1^l) + \pr^l p_1 &= \left[ \pr_t V^l + \pr_j(v^j V^l) \right] + \left [\pr_j(V^j v^l) \right] + \pr_j\left[ (V^j V^l) + P\de^{jl} + R^{jl} \right]\\
\pr_j v_1^j &= 0
\end{align*}

We wish to express each of these terms in the form $\pr_j Q^{jl}$ with $Q$ much smaller than $|R|$.  Let us first consider the term
\begin{align} \label{eqn:totalStress}
 Q^{jl} = (V^j V^l) + P\de^{jl} + R^{jl},
 \end{align}
We cannot immediately make this term small because $V^j V^l$ is a rank one tensor, whereas $P\de^{jl} + R^{jl}$ may be an arbitrary symmetric (2,0) tensor.  We can, however, ensure that the low frequency part of $Q^{jl}$ is small by choosing $V$ so that the low frequency part of $V^j V^l$ cancels with the low frequency part of $P\de^{jl} + R^{jl}$; to make sure this cancellation is possible, it is necessary to first choose $P$ large enough so that $P\de^{jl} + R^{jl}$ is negative definite, since the coarse scale part of $V^j V^l$ must be essentially positive definite.  The fact that $V^j V^l$ can have a nontrivial low frequency part even though $V$ itself has very high frequency demonstrates a lack of cancellation in the nonlinearity, which can ultimately be blamed for the fact that weak limits of Euler flows may fail to solve the Euler equation.\footnote{Regarding the analogous aspect of the construction in \cite{shnNonUnq}, Shnirelman suggested that the above idea could be thought of as emulating the frequency cascades predicted in the theory of turbulence.}  Note that the above choices of $V$ and $P$ imply that $V$ (or, rather, the absolute value of $V$) has the size of a ``square root'' of $R$, and that $P$ has the size of $R$.%
%
\[ |V| \approx |R|^{1/2}, \quad \quad |P| \approx |R| \]
Thus, after the construction is iterated $k$ times with $R_{(k)} \to 0$ quickly enough, we can ensure that the velocities and pressures $v_{(k+1)} = v_{(k)} + V_{(k)}$, and $p_{(k+1)} = p_{(k)} + P_{(k)}$ constructed through the iteration will converge to solutions $(v,p)$ of the Euler equations.

Although the low frequency part of the $Q^{jl}$ term in (\ref{eqn:totalStress}) will be small, there will be many high frequency interaction terms in $Q^{jl}$ which will not be small.  So far, the only method available to handle these terms is to construct $V$ out of stationary solutions to $3$D-Euler called Beltrami flows, which also requires adding additional components to the pressure.\footnote{For the isometric embedding problem, there is no known method to handle interference terms, and this difficulty actually limits the regularity of the solutions which can currently be obtained through convex integration.  See \cite{deLSzeC1iso}.}

Following \cite{deLSzeCts}, we refer to the error term $\pr_t V^l + \pr_j(v^j V^l)$ as the ``transport term'', since it can also be written as $(\pr_t + v^j \pr_j)V^l $.  We want to find a small, symmetric tensor $Q_T^{jl}$ solving the divergence equation \[\pr_j Q_T^{jl} = (\pr_t + v^j \pr_j)V^l.\]  We also need to do the same for the ``high-low frequency'' interaction term $\pr_j(V^j v^l)$, because $V^j v^l$ is of size ``$|R|^{1/2}$'', whereas we need $Q$ to be much smaller than $R$.  Achieving either of these goals requires us to find a small solution to the first-order elliptic equation $\pr_j Q^{jl} = U^l$, with high frequency data $U^l$.  

In general, we expect that if $U^l$ is a smooth vector field with amplitude of ``size $1$'' and with ``frequency $\la$'', solving the elliptic equation should allow us to achieve a solution of ``size $1/\la$''.  One way to give rigorous evidence to this heuristic expectation is to use the Fourier transform to solve the equation, and more evidence can be drawn by analogy with the ODE $\fr{dQ}{dx} = e^{i \la x} u(x)$, where repeated integration by parts yields the estimate,
\begin{align}
\fr{dQ}{dx} &= e^{i \la x} u(x) \\
\Rightarrow ||Q|| &\leqc \fr{||u||}{\la} + \fr{||\nab u||}{\la^2} + \ldots + \fr{||\nab^D u||}{\la^D} + \fr{||\nab^{D + 1} u||}{\la^D}
\end{align}  
Thus, with the appropriate ``integration by parts'' type argument, we expect to gain a smallness factor of $||Q|| \leqc \fr{|| u ||}{\la}$ by solving the divergence equation $\pr_j Q^{jl} = e^{i \la \xi \cdot x} u^l$, as long as the oscillations of $u$ are slower than $\la$ (i.e. $u$ cannot be of the form $e^{- i \la \xi \cdot x } v^l$ for some slowly varying $v$).

The transport term presents a more serious problem than does the high-low term $\pr_j(V^j v^l) = V^j \pr_j v^l$, because the transport term necessarily involves differentiating the oscillatory correction $V$ and then solving the equation
\begin{align}
\pr_j Q_T^{jl} = (\pr_t + v^j \pr_j)V^l
\end{align}
Now, $v$ has size $1$, and if $V$ has the expected size $|R|^{1/2}$ and frequency $\la$, its derivative $v^j \pr_j V$ on the right hand side will have size at least $\la |R|^{1/2}$ and will also have frequency $\la$.  On the other hand, solving the elliptic equation only gains one power of $\la^{-1}$, leaving no hope to find a solution with $|Q| < |R|$ when $|R|$ is small.  The equations therefore impose a requirement that $V$ be essentially transported by the coarse scale velocity field $v$.  

The philosophy described above was executed in \cite{deLSzeCts} to construct continuous, weak solutions of Euler which do not conserve energy (with the superficial difference that the stress $R^{jl}$ was required to have trace $R^{jl} \de_{jl} = 0$).  While the paper \cite{deLSzeCts} was the first to execute this philosophy by representing the forcing term as a divergence $\pr_j R^{jl}$, many of the above considerations implicitly underlie preceding constructions of badly behaved weak solutions to Euler (see \cite{shnNonUnq} and \cite{deLSzeIncl} for instructive examples).




%% file: structureOfPaper.tex
While the general considerations above pertain to both the present paper and to the contruction in \cite{deLSzeCts}, the manner in which the philosophy is executed in the present paper is different.  Our strategy for the exposition is as follows.

\begin{itemize}
\item In the following section we identify the error terms which need to be controlled.
\item In Part (\ref{basicConstruct}) we explain some notation of the paper and write down a basic construction of the correction.  In this part, we do not explain how the various parameters involved should be chosen to optimize the construction.  Rather, our goal is to give enough detail so that it is clear how the scheme can be used to construct weak solutions which are continuous in space and time.
\item In Part (\ref{obtainSolutions}) we explain how to iterate the construction of Part (\ref{basicConstruct}) to obtain continuous solutions to the Euler equations.  We then explain the concept of frequency and energy levels, and the Main Lemma which is iterated to give the full proof of the Theorems (\ref{mainThm}) and (\ref{thm:mainThm2}).  Through this discussion, we explain some additional difficulties which present themselves as one approaches the optimal regularity, and how these difficulties can be overcome by isolating the material derivative $\pr_t + v \cdot \nab$ as a special derivative.
\item In Parts (\ref{hardPart})-(\ref{part:thirdHardPart}), we verify all the estimates needed for the proof of the Main Lemma.
\end{itemize}

%% file: mainLemTechOutline.tex
We now give a more technical outline of the construction.

Let $(v^l, p, R^{jl})$ be a solution to the Euler-Reynolds system, and consider a correction $v_1 = v + V$, $p_1 = p + P$.

The correction $V$ is a divergence free vector field which oscillates rapidly compared to $v$, and can be written as a sum of divergence free vector fields $V = \sum_I V_I$, which oscillate in various different directions, and which are supported in several different regions of the space-time $\T^3 \times \R$.

In our construction, there will always be a bounded number of waves $V_I$ (at most $192$) which are nonzero at any given time $t$.

Each individual wave $V_I$ composing $V$ is a complex-valued, divergence free vector field that oscillates rapidly in only one direction.  We introduce several ways in which we will represent each $V_I$
\begin{align*}
 V_I&= e^{i \la \xi_I}\tilde{v}_I \\
&= \fr{1}{\la}\nab \times ( e^{i \la \xi_I} w_I ) \\
&= e^{i \la \xi_I}(v_I + \fr{\nab \times w_I}{\la}) \\
v_I &= (i \nab \xi_I) \times w_I
\end{align*}

Here, $\la \in \R$ is a large number, independent of $(t,x)$, which will be chosen in the argument.  The phase functions $\xi_I(t,x)$ will be real-valued functions of $(t,x)$ whose gradients indicate the direction of oscillation of the components $V_I$.  Unlike the argument of \cite{deLSzeCts}, we will not require the functions $\xi_I(t,x)$ to be linear (that is, to have constant gradients).  

Because we use nonlinear phase functions, we must modify the ``nonstationary phase'' type arguments used in \cite{deLSzeCts} to gain cancellation when solving $\pr_j Q^{jl} = e^{i \la \xi(x)} u^l $.  We therefore must ensure, for one thing, that the gradients $\pr^l \xi_I $ of the phase functions $\xi_I$ remain bounded away from $0$.  If one is willing to restrict to $\la \in \Z$ (as is done in \cite{deLSzeCts}), one could obtain globally defined functions $\xi_I$ taking values in $\R/(2 \pi \Z)$ which satisfy this lower bound on their gradients.  Instead, we choose to define $\xi_I$ on only part of the space $\T^3$ at any given time, and the coefficient $\tilde{v}_I^l$ will be a vector field that is compactly supported in that region, and also compactly supported in time.  In particular, we will not be using the globally defined frequencies on the torus.

The $\tilde{v}_I = \tilde{v}_I^l$, and $w_I = w_I^l$ here are complex-valued vector fields.  Then $v_I = (i \nab \xi_I) \times w_I$ is a vector field which will be chosen later, but which necessarily takes values which point into the level sets of $\xi_I$.  Since $\la$ will be a large parameter, $\tilde{v}_I \cdot \nab \xi_I = \tilde{v}_I^l \pr_l \xi_I = \fr{\nab \times w_I}{ \la} \cdot ( \nab \xi_I )$ will be small.  The character of these building blocks reflects the fundamental principle that any high frequency, divergence free plane wave must point perpendicular to its direction of oscillation, causing the vector field to generate a shear type of flow.  This principle, which can be formalized on $\R^n$ by considering the Fourier transform, still holds approximately in our setting.  In our argument, we will first construct the vector field $v_I$, and then solve the linear equation $v_I = (i \nab \xi_I) \times w_I$ in order to obtain $w_I$.

We remark that, in order to ensure that $V$ is real-valued, each index $I$ will have a conjugate index $\bar{I}$ such that $V_{\bar{I}} = \bar{V_I}$.  Moreover, $\xi_{\bar{I}} = - \xi_I$, and $w_{\bar{I}} = \bar{w_I}$ so that $v_{\bar{I}} = \bar{v}_I$ as well.

Let us introduce the notation $( v w )^{jl}$ to represent the symmetric product of two vectors $v^l$ and $w^l$.  That is, $(v v)^{jl} = v^j v^l$ is the square $v \otimes v$ of $v$, and from the polarization identity $( v w )^{jl} = \fr{1}{2} ( v^j w^l + w^j v^l )$ in general. The symmetric product so defined is commutative.

Using the representation of $V$ as a sum of individual waves with distinct direction and location, we see that the corrected velocity field and pressure $v_1, p_1$ satisfy the system.
\begin{align} \label{euReynExp}
  \begin{split} \pr_t v_1^l + \pr_j(v_1^j v_1^l) + \pr^l p_1 &= \left[ \pr_t V^l + \pr_j(v^j V^l) \right] + \left [\pr_j(V^j v^l) \right] \\
   &+ \sum_{J \neq \bar{I}} \pr_j ( V_I V_J )^{jl} + \pr_j\left[ \sum_I (V_I^j \bar{V}_I^l) + P\de^{jl} + R^{jl} \right]
   \end{split} \\
\pr_j v_1^j &= 0
\end{align}

We have separated out the interaction terms between non-colliding frequencies in the hope that these interference terms can be shown to be negligible.  Indeed, if one measures errors in a weak topology, these interaction terms will be small because the products $(V_I V_J)^{jl}$ can easily be made high frequency for $J \neq \bar{I}$, and the term $\sum_I (V_I^j \bar{V}_I^l) + P\de^{jl} + R^{jl}$ will become small in a strong topology as long as the amplitudes $v_I$ are chosen appropriately\footnote{One can interpret this construction as requiring that $V^l = V^l_{(\la)}$ generates an appropriate Young measure as $\la \to \infty$.  The fact that this equation can be solved for arbitrary $R$, implying in a sense that high frequencies may emulate arbitrary forces on the lower frequency part of the solution, demonstrates a lack of cancellation characteristic of the nonlinearity in the Euler equations.}.  However, we cannot use a weak topology to measure the size of the remaining stress \footnote{Recall from Section \ref{motivation} that we expect $R_\ep^{jl} \approx (v^j v^l)_\ep - v_\ep^jv_\ep^l$ to vanish uniformly if $v$ is continuous, and at least to converge to $0$ in $L^1$ if $v$ is in $L^2$.}.

Ultimately, we will only be able to handle the interaction terms after imposing a ``Beltrami flow'' condition on the structure of the $V_I$.  This condition will allow us to show that after adding appropriate terms to the pressure, the interference terms $( V_I V_J )^{jl} - P_{IJ} \de^{jl}$ are small in $C^0$ modulo solutions of $\pr_j Q^{jl} = 0$.  

The pressure will therefore be of the form
\[ P = P_0 + \sum_{J \neq {\bar I}} P_{I,J} \]
where $P_0$ appears in the equation
\begin{align}
 \sum_{I} V_I^j {\bar V}_I^l + P_0 \de^{jl} + R^{jl} &\approx 0 \label{eq:approxStressPre}
\end{align}
We remark here that, by choosing $P_0$ appropriately, we will also be able to prescribe the energy increment 
\[  \int(| v + V|^2 - |v|^2) dx \approx e(t) \]
of the correction with great accuracy.

There are two reasons for the $\approx$ symbol in (\ref{eq:approxStressPre}).  One reason is that we cannot control the pointwise values of
\[ V_I^j {\bar V}_I^l = {\tilde v}_I^j {\overline {\tilde v}}_I^l \]
exactly; rather we are only able to determine these products approximately in the sense that
\[ V_I^j {\bar V}_I^l \approx v_I^j {\bar v}_I^l \]
and we have freedom to prescribe $v_I$ up to some constraints such as the fact that $v_I \in \langle \nab \xi_I \rangle^\perp$.

The other reason for the approximation symbol $\approx$ in (\ref{eq:approxStressPre}) is that we will also define a suitable mollification $R_\ep$ of $R$ before solving the equation
\begin{align}
\sum_I v_I^j {\bar v}_I^l + P_0 \de^{jl} + R_\ep^{jl} &= 0 
\end{align}
pointwise.  In this way, the amplitudes $v_I$ will only carry coarse scale information regarding $R$.  

Making this mollification introduces another error term of the form
\[ R - R_\ep \]
which will be carried as part of the new stress $R_1$.  For similar reasons, we will also introduce an appropriate mollification $v_\ep$ of $v$ at the beginning of the argument, giving rise to another term
\[ (v^j - v_\ep^j) V^l + V^j (v^l - v_\ep^l) \]
in the new stress.  While these mollifications are very important for the construction of H{\" o}lder continuous solutions, for the present discussion we will ignore them, and return to them later.

To summarize, we have seen five main error terms, each of which must be expressed in the form $\pr_j Q^{jl}$ for some $Q^{jl}(t,x)$ taking values in the space of symmetric, $(2,0)$ tensors $\SS$, and each $Q^{jl}$ is required to be smaller than $R^{jl}$ to ensure we have made visible progress towards a solution of the Euler equations.  The terms we must deal with are named:
\begin{itemize}
 \item {\bf The Transport term }
\[ \pr_j Q_T^{jl} =  \pr_t V^l + \pr_j(v_\ep^j V^l) \]
 \item {\bf The High-Low interaction term}
\[ \pr_j Q_L^{jl} = \pr_j(V^j v^l) = V^j \pr_j v_\ep^l \]
 \item {\bf The High-High interference terms}
\[ \pr_j Q_H^{jl} = \sum_{J \neq \bar{I}} \pr_j\left[ ( V_I V_J )^{jl} + P_{I,J} \de^{jl} \right] \]
 \item {\bf The Stress term}
\[ Q_S^{jl} = \left[ \sum_I (V_I^j \bar{V}_I^l) + P_0 \de^{jl} + R_\ep^{jl} \right] \]
\item {\bf The Mollification Terms}
\[ Q_M^{jl} = (v^j - v_\ep^j) V^l + V^j (v^l - v_\ep^l) + (R^{jl} - R_\ep^{jl}) \]
\end{itemize}
In \cite{deLSzeCts}, the High-High interference terms correspond to part of what is called the ``oscillation part'' of the error.  

Each of these terms has its own set of difficulties, and these difficulties are coupled together as the transport term and stress term both impose constraints on the correction $V$.  The one difficulty that is common to the first three of the terms is that they require one to find small solutions to the underdetermined, first order, elliptic equation
\[ \pr_j Q^{jl} = e^{i \la \xi(t,x)} u^l(t,x) \]
For example, the High-Low interaction term (which is the simplest of these three error terms), can be expanded in the form
\[ V^j \pr_j v^l = \sum_I e^{i \la \xi_I} \tilde{v}_I^j \pr_j v_\ep^l \]

As we previously remarked, either by considering the problem in frequency space or by drawing an analogy with the ODE $\fr{dQ}{dx} = e^{i \la \xi(x)} u(x)$, one expects to gain a smallness factor of $\fr{1}{\la}$ for the solution $Q$.  This gain is achieved by an ``oscillatory estimate'', whose proof we will describe shortly.

This gain of $\la^{-1}$ alone is not sufficient to handle every term in the argument; both the Transport term and the High-High interference terms require one to differentiate the oscillatory correction $V$, and thereby produce oscillatory data whose main terms are of the form $e^{i \la \xi} \la u^l$, meaning their amplitudes are too large for the oscillatory estimate to obtain a small solution.  For example, the right hand side of the transport term can be written
\begin{align*}
 \pr_t V^l + \pr_j(v_\ep^j V^l) &= (\pr_t + v_\ep^j \pr_j)V^l \\
&= \sum_I e^{i \la \xi_I} \left( i \la (\pr_t + v_\ep^j \pr_j) \xi_I \tilde{v}_I^l + (\pr_t + v_\ep^j \pr_j) \tilde{v}_I^l \right)
\end{align*}
For the transport term, we deal with this difficulty by allowing the phase functions $\xi_I$ to obey a transport equation; in particular, the phase functions will not be linear functions, in contrast to the argument in \cite{deLSzeCts}.  

We will also make sure the lifespan of each $V_I$ is a sufficiently small time interval, so that phase functions $\xi_I$ remain nonstationary and suitably regular for applying the main, oscillatory estimate.  The length of this time interval will depend on bounds for the derivatives of $v$, including the $L^\infty$ norm of $\nab v$.



%% file: mainLemCts.tex
In order to have a concrete goal for the construction, we state a Lemma which implies the existence of continuous solutions.  

\begin{lem}[Main Lemma for Continuous Solutions]\label{lem:mainLemCts}
There exist constants $K$ and $C$ such that the following holds.

Let $\ep > 0$, and suppose that $(v, p, R)$ are uniformly continuous solutions to the Euler-Reynolds equations on $\R \times \T^3$, with $v$ uniformly bounded\footnote{The assumption of uniform boundedness here may be somewhat unnatural and actually does not enter into the construction of H\"{o}lder continuous solutions, where all the bounds in the construction depend only on relative velocities. } and
\[ \mbox{ supp } R \subseteq I \times \T^3 \]
for some time interval, and
\[ \| R \|_{C_{t,x}^0} \leq e_R \]
Let 
\[e(t) : \R \to \R_{\geq 0} \]
be any function satisfying the lower bound
\ali{
 e(t) &\geq K e_R \label{eq:lowBdEtCts}
}
on a neighborhood of $I$ such that 
\ali{
 \fr{d}{dt} e^{1/2}(t) &\in C^0(\R) \label{eq:RegEtCts}
}
is continuous and uniformly bounded, and such that $e(t)$ is bounded by
\begin{align}
 e(t) &\leq 1000 K e_R  \label{eq:upBdet}
\end{align}

Then there exists a uniformly continuous solution $(v_1, p_1, R_1)$ to the Euler-Reynolds equations of the form
\begin{align}
v_1 &= v + V \\
p_1 &= p + P
\end{align}
such that the supports of the corrections and the new stress 
\begin{align}
\mbox{ supp } V \cup \mbox{ supp } P \cup \mbox{ supp } R_1 &\subseteq \mbox{ supp } e \times \T^3 
\end{align}
and so that $V$ and $P$ obey the bounds
\begin{align}
\co{ V } \leq C e_R^{1/2} \label{eq:ctsBoundForV} \\
\co{ P } \leq C e_R \label{eq:ctsBoundForP}
\end{align}
and 
\begin{align}
\| \int |v_1|^2(t,x) dx - \int (|v|^2 + e(t)) dx \|_{C_t^0} &\leq \ep \label{eq:wantEnergyIncreaseCts}\\
\co{ R_1 } &\leq \ep
\end{align}
\end{lem}

It is not difficult to check that the main lemma implies the following theorem:
\begin{thm} \label{thm:mainThmCts}
There exist continuous solutions $(v, p)$ to the Euler equations which are nontrivial and have compact support in time.
\end{thm}
\begin{proof}[Proof that Lemma (\ref{lem:mainLemCts}) implies Theorem (\ref{lem:mainLemCts}) ]
To prove this implication, one iteratively applies the Lemma to produce a set of solutions $(v_{(k)}, p_{(k)}, R_{(k)})$ to the Euler-Reynolds equations for which the stress is bounded by
\ali{
 \co{R_{(k)}} &\leq e_{R,(k)}
}
with $e_{R,(k)} > 0$ a decreasing sequence of positive numbers chosen to tend to $0$ rapidly.  The corrections to the velocity $v_{(k+1)} = v_{(k)} + V_{(k)}$ and to the pressure $p_{(k+1)} = p_{(k)} + P_{(k)}$ can be summed in $C^0$
\ali{
\sum_k \co{V_{(k)} } &\leq C \sum_k e_{R,(k)}^{1/2} < \infty \\
\sum_k \co{P_{(k)} } &\leq C \sum_k e_{R,(k)} < \infty
}
by (\ref{eq:ctsBoundForV}) and (\ref{eq:ctsBoundForP}), which implies the uniform convergence of 
\[ v = \lim_{k} v_{(k)} \]
\[ p = \lim_k p_{(k)} \]
to uniformly continuous solutions $(v,p)$ of the Euler equations.

To make sure the solutions constructed in this way are nontrivial and compactly supported, we apply the lemma with any non-negative functions $e(t) = e_{(k)}(t)$ which satisfy (\ref{eq:lowBdEtCts}), (\ref{eq:RegEtCts}) and (\ref{eq:upBdet}) for $e_R = e_{R,(k)}$ and whose supports are all contained in some finite time interval.  We can also arrange that these functions are strictly positive at time $0$ so that
\[ e_{(k)}(0) > 0 \]
for all $k$.

If we choose $\ep_{(k)}$ small enough at each stage, the inequality (\ref{eq:wantEnergyIncreaseCts}) implies that the energy at $t = 0$
\ali{
\int |v_{(k+1)}|^2(0,x) dx & \geq \int (|v_{(k)}|^2(0,x) + e_{(k)}(0)) dx - \ep_{(k)}  \\
&>  \int ~|v_{(k)}|^2(0,x) ~ dx
}
increases with each stage.  Then, by the dominated convergence theorem,
\ali{
\int |v|^2(0,x) dx &= \lim_k \int ~ |v_{(k)}|^2(0,x) ~ dx \\
&\geq \int |v_{(1)}|^2(0,x) dx > 0
}
which ensures that the solution $v = \lim_k v_{(k)}$ is nonzero at $t = 0$.
\end{proof}

Let us now proceed with the construction.  During the construction, we will not be completely specific at certain points (such as how to mollify $v$ and $R$) because these aspects must be changed or handled more delicately in order to construct the H{\" o}lder continuous solutions of Theorem (\ref{mainThm}), and we will want to be able to refer to the same construction for both continuous and H{\" o}lder continuous solutions.

%% file: oscillatoryEstimate.tex
A key ingredient in the Proof of Lemma (\ref{lem:mainLemCts}) is to find special solutions to the divergence equation
\begin{align} \label{eq:theDivergence}
\pr_j Q^{jl} &= e^{i \la \xi(x)} u^l
\end{align}
where $\xi$ and $u^l$ are respectively a smooth function and a smooth vector field on $\T^3$, and $Q^{jl}$ is an unknown, symmetric $(2,0)$ tensor.  In our applications, $u^l$ will always be supported in a single coordinate chart, and $\xi$ will only be defined in a neighborhood of the support of $u^l$.  We wish to take advantage of the oscillatory nature of the data in order to gain a smallness factor of $\fr{1}{\la}$ for the solution $Q$.  This estimate is analogous to how one can integrate by parts to prove a bound
\[ \| Q \|_{C^0} \leq C ( \fr{\| u \|_{C^0}}{\la} + \fr{\| \nab u \|_{L^1}}{\la} ) \]
for certain periodic solutions to the ODE $\fr{dQ}{dx} = e^{i \la \xi(x)} u(x)$, provided bounds on $\| |\xi'(x)|^{-1} \|_{C^0}$ and $\| \xi''(x) \|_{L^1}$.

\subsection{A remark about momentum conservation} \label{sec:momentumConservation}

A first remark regards why one cannot solve the equation 
\[ \pr_j Q^{jl} = U^l \]
for arbitrary data on the right hand side.  Namely, the right hand side must have integral $0$ in order to be a divergence.  This condition is also sufficient, but let us briefly discuss how this condition relates to the conservation of momentum.

From a physical point of view, the right hand side of the Euler equation is a force, and the condition
\[ \int U^l dx = 0 \]
reflects Newton's law that every action must have an equal and opposite reaction.  This axiom, in turn, implies the conservation of momentum in classical mechanics.

When translated into the general Hilbert space framework for solving elliptic equations, the condition required on the data $U^l = e^{i\xi(x)} u^l$ is that $U^l$ be orthogonal to the kernel of the operator adjoint to the divergence operator $\PP(Q) = \pr_j Q^{jl}$.  Here the operator adjoint to $\PP$ is the operator $\PP^*(v) = -\fr{1}{2}( \pr^j v^l + \pr^l v^j )$, which is essentially the symmetric part of the derivative.  The tensor $\pr^j v^l + \pr^l v^j$ has the familiar geometric interpretation as a deformation tensor of the vector field $v$ -- that is, it expresses the Lie derivative $\LL_v \de^{jl}$ of the inner product $\de^{jl}$ when pulled back along the flow of $v$.  Those vector fields whose deformation tensor $\pr^j v^l + \pr^l v^j$ vanishes are called Killing vector fields, and they consist of those vector fields whose flows generate isometries of the space.  On the torus, these Killing vector fields consist of the translation vector fields; that is, vector fields whose component functions $v^l$ are constant on $\T^3$ -- the integral $0$ condition is the condition that $v$ be orthogonal to translations.

In view of Noether's theorem, the constant vector fields which act as Galilean symmetries of the Euler equation are responsible the conservation of momentum.  In this case, conservation of momentum for solutions to Euler is proven directly by integrating any component of the equation
\[ \pr_t v^l + \pr_j(v^j v^l) + \pr^l p = 0 \]
in space.  

As the proof reveals, {\bf all} solutions to the Euler equations, even those which only belong to $L^2$, conserve momentum.  In fact, by the same proof for the equations
\[ \pr_t v^l + \pr_j(v^j v^l) + \pr^l p = \pr_j R^{jl} \]
all solutions to the Euler-Reynolds equations also conserve momentum.  Conversely, any incompressible flow which conserves momentum can be realized as a solution to the Euler-Reynolds equations by solving the divergence equation $\pr_j R^{jl} = \pr_t v^l + \pr_j(v^j v^l)$.  The orthogonality condition implies that only those forces $U^l$ which do not inject any momentum into the system can be realized as the divergence of a stress $\pr_j Q^{jl}$, thereby guaranteeing the conservation of momentum for the solutions we construct.

This condition will not present any difficulty for the present argument, since all the terms $U^l$ which arise when we introduce the correction $v_1 = v + V$ satisfy the orthogonality condition.  For example, the data for High-Low interaction term
\[ \pr_j Q_L^{jl} = \pr_j(V^j v_\ep^l) \]
has integral $0$ since it is a divergence, and the data for the Transport term
\[ \pr_j Q_T^{jl} = \pr_t V + \pr_j(v_\ep^j V^l) \]
has integral $0$ because
\[ \pr_t V = \nab \times \pr_t W \]
and the curl of any vector field on $\T^3$ has integral $0$.



\subsection{The Parametrix} \label{sec:theParametrix}

Having discussed the compatibility condition for solving the equation, let us now prove the main oscillatory estimate on the solution.  To convey the idea, we first prove a simple ($C^0$) version.  To be consistent with the notation of the paper, we keep a scalar $\la$ in the statement of the theorem, although $\la$ can without loss of generality be absorbed into the phase function if one prefers.

\begin{prop} \label{prop:canSolveDivWithLa}
Let $U^l$ be a smooth vector field on $\T^3$ such that $\int_{\T^3} U^l dx = 0$, and suppose that $U^l$ can be represented as $e^{i \la \xi(x)} u^l$ for some smooth vector field $u^l$ and some smooth, real-valued phase function $\xi$ defined in a neighborhood of the support of $u^l$ whose gradient does not vanish at any point.  Then there exists a symmetric, $(2,0)$ tensor field $Q^{jl}$ on $\T^3$ solving the equation
\[ \pr_j Q^{jl} = U^l = e^{i \la \xi(x)}u^l \]
which depends linearly on $U^l$ and for any $p > 3$, $0 \leq s \leq 1$ satisfies the estimate
\[ || Q^{jl} ||_{C^0(\T^3)} \leq C \left( \fr{||~|\nab \xi|^{-1}~||_{L^\infty} ( 1 + ||~|\nab \xi|^{-1}~||_{L^\infty} || D^2 \xi ||_{L^p} )}{\la} \right)^s || u ||_{W^{s,p}} \]  
The implied constant depends only on $s$ and $p$.
\end{prop}

The $s = 0$ case derives from the following, standard elliptic estimate, which is used as a lemma in the proof.
\begin{lem} \label{lem:divL2}
If $f^l$ be a smooth vector field on $\T^3$ such that $\int f^l dx = 0$, there exists a symmetric, $(2,0)$ tensor field $Q^{jl}$ depending linearly on $f^l$ which solves the equation
\[ \pr_j Q^{jl} = f^l \]
and satisfies the bound
\[ ||Q^{jl}||_{W^{1, p}(\T^3)} \leq C || f ||_{L^p(\T^3)} \]
so that, in particular when $p > 3$
\[ ||Q^{jl}||_{C^0(\T^3)} \leq C || f ||_{L^p(\T^3)} \]
\end{lem}
The theorem itself provides an interpolation between the $C^0$ estimate in the lemma, and the $s = 1$ bound.  In fact, the $s = 1$ case of the following lemma is sufficient for the construction of continuous weak solutions, but we prove the proposition for fractional regularity in order to demonstrate the robustness of the method, and because the proof illustrates a main theme of the paper.

The idea of the proof, which generalizes the approach taken in \cite{deLSzeCts} to nonlinear phase functions, is as follows.  When using the Fourier transform to solve this equation, the key observation to be made is that whenever $u^l$ is constant, and $\xi(x)$ is linear (so that the derivatives $\pr_j\xi$ are constant), an exact solution can be obtained by the formula
\[ Q^{jl} = \fr{1}{\la} \left(e^{i \la \xi(x)}q^{jl} \right) \]
for any constant tensor $q^{jl}$ solving the linear equation 
\[ i \pr_j \xi q^{jl} = u^l \]
pointwise.  In fact, one can arrange that the solution $q^{jl}$ is given by a map
\begin{align}
q^{jl} &= q^{jl}(\nab \xi)[ u ] \label{eq:theqMap}
\end{align}
which is homogeneous of degree $-1$ in $\nab \xi$, and which is linear in $u$.

  If $\pr_j \xi$ and $u^l$ are not constant, but are still smooth, one can still obtain an approximate solution of this form, and eliminating the error requires one to solve $\pr_j Q^{jl} = f^l$, where the data $f^l$ is small in $L^{p}$, allowing us to apply the lemma (\ref{lem:divL2}).

We now proceed with the formal proof:
\begin{proof}
The preceding discussion suggests that we begin with an approximate solution $e^{i \la \xi(x)} q^{jl}$, where $q^{jl}$ solves the linear equation $i \pr_j \xi q^{jl} = u^l$ pointwise, but doing so is only a good idea when $u^l$ is smooth.  Therefore, we find an approximate solution to $\pr_j Q^{jl} = e^{i \la \xi} u_\ep^l$, where $u_\ep^l = \eta_\ep \ast u^l$ is a standard mollification of $u^l$. As is typical of stationary phase arguments, the parameter $\ep$ will be optimized later to ensure that the oscillations of the phase function are rapid compared to the coarser scale variations of $u_\ep^l$; a mollification in a similar spirit will be used at many other instances in the main body of the paper.  In any case, we do not expect to observe cancellation here unless the ``frequency of $u$'' is less than that of $\la \xi$; otherwise $u^l$ could be of the form $e^{- i \la \xi(x)} v^l$.

To solve the linear equation $i \pr_j \xi q_\ep^{jl} = u_\ep^l$, we first decompose $u_\ep^l = u_\perp^l + \fr{(u_\ep \cdot \nab \xi)}{|\nab \xi|^2} \pr^l \xi = u_\perp^l + u^l_\parallel$, where $(u_\perp) \cdot \nab \xi = 0$, and $u_\parallel$ points in the direction of $\nab \xi$.  We then set $q_\ep^{jl}$ to be the sum
\begin{align}
q_\ep^{jl} &= i^{-1}( q_\perp^{jl} + q_\parallel^{jl} ) \\
&= q^{jl}(\nab \xi)[u_\ep] \label{eq:theLinMapq}
\end{align}
where
\begin{align*}
q_\perp^{jl} &= \fr{1}{|\nab \xi|^2}( \pr^j \xi u_\perp^l + \pr^l \xi u_\perp^j)
\end{align*}
so that $\pr_j\xi q_\perp^{jl} = u_\perp^l$, and
\begin{align*}
 q_\parallel^{jl} &=  \fr{(u_\ep \cdot \nab \xi)}{|\nab \xi|^2} \de^{jl}
\end{align*}
so that $\pr_j\xi q_\parallel^{jl} = u_\parallel^l$.

We now seek a solution of the form $Q^{jl} = \fr{1}{\la} e^{i \la \xi} q_\ep^{jl} + Q_1^{jl}$ where $Q_1^{jl}$ must satisfy the equation:
\begin{align*}
\pr_j Q_1^{jl} &= f^l
\end{align*}
with the data $f^l = -e^{i \la \xi(x)}( (u^l - u_\ep^l) + \fr{1}{\la} \pr_j q_\ep^{jl} )$.

Since the error term on the right hand side is still of the form $e^{i \la \xi} \bar{u}^l$, it is possible to iterate the preceding method in order to produce a higher order approximate solution at the price of taking more derivatives of $\xi$ and $u^l$.  Later in the paper, we will use a higher-order parametrix.

For now, we will simply observe that the remaining data can already be regarded as bounded in $L^{p}$.  The right hand side also has integral zero because it is of the form $U^l + \pr_j {\tilde Q}^{jl}$, and the original data $U^l$ was assumed to have integral $0$.  Therefore, by the lemma (\ref{lem:divL2}), there exists a smooth solution $Q_1^{jl}$ on $\T^3$ obeying the estimate
\begin{align*}
 ||Q_1^{jl}||_{C^0(\T^3)} &\leqc || e^{i \la \xi(x)}( (u^l - u_\ep^l) + \fr{1}{\la} \pr_j q_\ep^{jl} ) ||_{L^p(\T^3)} \\
&\leqc || (u - u_\ep) ||_{L^p} + \fr{1}{\la} ||\pr_j q_\ep^{jl}||_{L^p} \\
&\leqc \ep^s || u ||_{\dot{W}^{s,p}} + \fr{1}{\la} || \pr_j q_\ep^{jl} ||_{L^p}
\end{align*}

We will not be using any special structure involving the divergence, so one might as well take the entire norm $||\nab q_\ep||_{L^{p}}$.  Let us estimate the size of $\pr_j q_\ep^{jl}$ as follows.  By differentiating the system of equations
\begin{align*}
 |\nab \xi|^2 q_\ep^{jl} &= ( \pr^j \xi u_\perp^l + \pr^l \xi u_\perp^j)  + (u_\ep \cdot \nab \xi) \de^{jl} \\
|\nab \xi|^2 u_\ep^l &= |\nab \xi|^2 u_\perp^l + (u_\ep \cdot \nab \xi) \pr^l \xi
\end{align*}
one computes that $\pr_j q_\ep^{jl}$ as a linear combination of many terms of only two types.
\begin{itemize}
 \item {\bf Type 1}:  Terms whose $L^p$ norm can be controlled by $||~|\nab \xi|^{-1}~||_{L^\infty}^2 ||\nab^2 \xi||_{L^p} || u_\ep ||_{C^0}$
 \item {\bf Type 2}:  Terms whose $L^p$ norm can be controlled by $||~|\nab \xi|^{-1}~||_{L^\infty} || \nab u_\ep ||_{L^p}$
\end{itemize}

One can check this estimate by enumerating the terms, but we point out that these estimates are exactly what one expects from a solution to $i \pr_j \xi q^{jl} = u^l$ if one thinks of $q^{jl}$ as schematically given by ``$q \approx \fr{u}{\nab \xi}$'' and applies the quotient rule formally.  Note that, by the Sobolev imbedding theorem, we can also control $||u_\ep||_{C^0}$ with $||u_\ep||_{W^{1,p}}$, so that only the $W^{1,p}$ norm of $u_\ep$ needs to appear in the estimate.  

The upshot of these estimates is that we can now bound
\begin{align*}
 ||Q_1^{jl}||_{C^0(\T^3)} &\leqc \ep^s || u ||_{W^{s,p}} + \fr{1}{\la} || \pr_j q^{jl} ||_{L^{p}} \\
&\leqc  \ep^s || u ||_{W^{s,p}} + \fr{||~|\nab \xi|^{-1}~||_{L^\infty}}{\la} \left( 1 + ||~|\nab \xi|^{-1}~||_{L^\infty} ||\nab^2 \xi||_{L^p} \right) || u_\ep||_{W^{1,p}} \\
&\leqc \ep^s || u ||_{W^{s,p}} + \fr{||~|\nab \xi|^{-1}~||_{L^\infty}}{\la} \left( 1 + ||~|\nab \xi|^{-1}~||_{L^\infty} ||\nab^2 \xi||_{L^p} \right) \ep^{s-1} ||u||_{W^{s,p}}
\end{align*}
It is now clear that, regardless of $s$, the optimal choice of $\ep$, which balances the two terms, is
\[ \ep = \fr{||~|\nab \xi|^{-1}~||_{L^\infty}}{\la} \left( 1 + ||~|\nab \xi|^{-1}~||_{L^\infty} ||\nab^2 \xi||_{L^p} \right)\]
As expected, this choice of $\ep$ ensures that $u_\ep$ can only oscillate at a scale coarser than the scale at which the phase function $\la \xi$ oscillates.\footnote{We remark that when $s = 1$, the result can be obtained without any mollification ($\ep = 0$).}  With the above choice of $\ep$, the bounds stated in the theorem apply to $Q_1^{jl}$.  By the same considerations, the parametrix $\fr{1}{i \la} e^{i \la \xi} q_\ep^{jl}$ can also be controlled in $C^0$ by $\fr{||~|\nab \xi|^{-1}~||_{L^\infty}}{\la} || u_\ep ||_{W^{1,p}}$, and therefore obeys the estimates stated in the theorem, which concludes the proof.
\end{proof}

\subsection{Higher order parametrix expansion}\label{sec:paramExpand}

Later on in the proof, we will have to modify the argument used in the proof of Proposition (\ref{prop:canSolveDivWithLa}) for solving the equation (\ref{eq:theDivergence}).  Namely, we will construct a solution to 
\begin{align} \label{eq:theDivergence2}
\pr_j Q^{jl} &= e^{i \la \xi(x)} u^l
\end{align}
by taking a higher order expansion of the parametrix
\begin{align}
Q^{jl} &= {\tilde Q}_{(D)}^{jl} + Q_{(D)}^{jl} \\
{\tilde Q}_{(D)}^{jl} &= \sum_{(k) = 1}^D e^{i \la \xi} \fr{q_{(k)}^{jl}}{\la^{k}}
\end{align}
where each $q_{(k)}$ solves a linear equation
\begin{align}
i \pr_j \xi q_{(1)}^{jl} &= u^l \\
i \pr_j \xi q_{(k)}^{jl} &= u_{(k)}^l \\
u_{(k)}^l &= -\pr_j q_{(k - 1)}^{jl} \quad \quad 1 < k \leq D +1
\end{align}
using the linear map 
\begin{align}
q^{jl} &= q^{jl}(\nab \xi)[u]
\end{align}
defined in (\ref{eq:theLinMapq}).  Then $Q_D$ will be used to eliminate the error
\begin{align}
\pr_j Q_D^{jl} &= e^{i \la \xi} \fr{u_{(D+1)}^l}{\la^D}
\end{align}

The reason we will need a higher order expansion for constructing H{\" o}lder continuous solutions is that the parameter $\la$ will be large but limited in size.  In fact $u$ will have ``frequency'' $\Xi$ where $\Xi$ is basically the frequency $\la$ chosen in the previous stage of the iteration, while we will have $\la \sim \Xi^{1 + \eta}$ for some small $\eta > 0$.  In this case, one can only obtain an accurate approximate solution from the parametrix after a high order expansion in $\la$.

\subsection{An Inverse for Divergence}

One way to prove Lemma (\ref{lem:divL2}) is to define the following operator, whose symbol in frequency space is exactly the map $q^{jl}(\nab \xi)[u]$ defined in Line (\ref{eq:theLinMapq}) of the proof of Proposition (\ref{prop:canSolveDivWithLa}).

\begin{prop}[Elliptic Estimates]\label{eq:smoothingOp}
There exists a linear operator $\RR^{jl}$ such that for all vector fields $U \in C^\infty(\T^3)$ we have
\[ \pr_j \RR^{jl}[U] = \PP[U]^l \]
where $\PP$ denotes the projection to integral $0$ vector fields.  

Furthermore, $\RR$ satisfies the bounds
\begin{align}
 ||\nab^{k + 1} \RR[U]||_{L^4(\T^3)} &\leq C_k || \nab^k U ||_{L^4(\T^3)} \quad \quad k \geq 0 \label{bound:stdElliptic}
\end{align}
and gives a solution with integral $0$
\begin{align}
\int \RR^{jl}[U] dx &= 0
\end{align}

\end{prop}

{\bf Remark:}  Here we only state estimates for $L^4$ norms rather than general $L^p$ norms, since the $L^4$ estimate suffices for our proof.

\begin{proof}

We define $\RR$ as a sum
\begin{align}
\RR[U] &= \RR_1[U] + \RR_2[U]  
\end{align}
using the Helmholtz decomposition of $U$ into its divergence free and curl free parts.  
\begin{align}
U^l &= \pr^l \De^{-1} \pr_i U^i +  (U^l - \pr^l \De^{-1} \pr_i U^i) \\
&= \pr_j[ \De^{-1} \pr_i U^i \de^{jl} ] + (U^l - \pr^l \De^{-1} \pr_i U^i) \\
&= \pr_j[ \pr_i(\De^{-1} \PP U^i)\de^{jl} ] + (U^l - \pr^l \De^{-1} \pr_i U^i) \\
&= \pr_j[ \RR_1^{jl}[U] ] + \HH U^l
\end{align}
Observe that $\HH U^l$ is divergence free, so to define $\RR_2$, it suffices to choose
\begin{align}
\RR_2^{jl}[U] &= \pr^j [ \De^{-1} \PP \HH U^l ] + \pr^l [ \De^{-1} \PP \HH U^j ]  
\end{align}
so that
\begin{align}
\pr_j \RR_2^{jl}[U] &= \PP \HH U^l
\end{align}
and
\begin{align}
\pr_j [ \RR_1^{jl}[U] + \RR_2^{jl}[U] ] &= \pr_j[ \De^{-1} \pr_i U^i \de^{jl} ] + \PP (U^l - \pr^l \De^{-1} \pr_i U^i) \\
&= \PP U^l
\end{align}

For $k = 0$, the estimate (\ref{bound:stdElliptic}) reads
\begin{align}
 ||\nab \RR[U]||_{L^4(\T^3)} &\leq C || U ||_{L^4(\T^3)}
\end{align}
This inequality is a consequence of the local Calder{\' o}n Zygmund estimate
\begin{align}
||D^2 f||_{L^4(B)} &\leq C(||f||_{L^4(2B)} + ||\De f||_{L^4(2B)})
\end{align}
applied to $f^l = \De^{-1} \PP U^l$ and taking any ball $B \subseteq \R^3$ containing an entire periodic box $\T^3 \subseteq (B / \Z^3)$.  In this case, because $f^l = \De^{-1} \PP U^l$ has integral $0$, we are able to replace the term
\[ ||f^l||_{L^4} = ||\De^{-1} \PP U||_{L^4} \]
by
\begin{align}
|| \De^{-1} \PP U ||_{L^4(2 B)} &\leq C || \De^{-1} \PP U ||_{L^4(\T^3)} \\
&\leq C || U ||_{L^4(\T^3)}
\end{align}
which can be quickly verified using, for example, Littlewood-Paley theory.

This estimate finishes the proof of (\ref{bound:stdElliptic}) for $k = 0$.  For $k > 0$, the bound follows from the fact that $\RR$ commutes with spatial derivatives.
\end{proof}


%% file: transport1.tex
We begin the main construction by considering the Transport term.  Let us fix a solution $(v, p, R)$ to the Euler Reynolds equations and consider a correction $v_1 = v + V$, $p_1 = p + P$.  We will think of $v$ as an approximation to the ``coarse scale velocity'' since the solution ultimately achieved by the process will resemble $v$ at a sufficiently coarse scale.

We must eliminate the Transport term by solving the equation
\[ \pr_j Q_T^{jl} =  \pr_t V^l + \pr_j(v^j V^l) \]
with a stress $Q_T^{jl}$ of size $|Q| < |R|$.  Ideally, the stress $Q$ can be made arbitrarily small by choosing $\la$ sufficiently large.

Recall that the correction $V^l$ takes the form of a superposition of waves $V^l = \sum_I V_I^l$, where each individual wave $V_I^l$ is an oscillatory, divergence free vector field of the form
\[ V_I^l = e^{i \la \xi_I} \tilde{v}_I^l \]
Before we discuss how to handle the stress term, we must leave the amplitude $\tilde{v}_I^l$ unspecified, however we will insist that its space-time support be such that, among other things, $|\nab \xi_I|$ remains bounded from $0$ so that the oscillatory estimate can be applied.  Substituting for $V$ and differentiating naively gives
\begin{align*}
 \pr_t V^l + \pr_j(v^j V^l) &= (\pr_t+ v^j\pr_j)V^l) \\
&= \sum_I e^{i \la \xi_I} [ (i \la) (\pr_t+ v^j\pr_j) \xi_I \tilde{v}_I^l + (\pr_t+ v^j\pr_j) \tilde{v}_I^l ]
\end{align*}

As we are able to gain at most a power of $\la$ from solving the divergence equation, it seems that unless $\xi_I$ obeys the transport equation $(\pr_t+ v^j\pr_j) \xi_I = 0$, we will not be able to adequately control the size of the solution $Q_T^{jl}$ -- formally, the data is at least of size $\la |\tilde{v}_I| \sim \la |R|^{1/2}$ giving a solution $Q$ which is at best size $|R|^{1/2}$.  This impression is almost correct, however, one cannot afford to allow $\xi_I$ to inherit the fine scale oscillations of the coarse scale velocity $v^j$.  Instead, it is natural to limit the frequency of the transport term, by using a mollifier $v_\ep^j = \eta_{\ep_v} \ast v^j$, and letting $\ep_v$ be a parameter that must be chosen sufficiently small.\footnote{To produce continuous solutions the mollifier $\eta_{\ep_v}$ may be in both space and time variables.  For H{\" o}lder continuous solutions, we will only mollify $v$ in the spatial variables.}

We will refer to $v_\ep$ as the {\bf coarse scale velocity}.

So in fact we treat the transport term by first decomposing
\begin{align*}
 \pr_t V^l + \pr_j(v^j V^l) &=  \pr_t V^l + v_\ep^j \pr_j V^l + \pr_j( (v^j - v_\ep^j) V^l )
\end{align*}

This mollification introduces a term 
\[ (v^j - v_\ep^j) V^l + V^j ( v^l - v_\ep^l) \]
which will constitute part of the new stress $R_1^{jl}$, and the parameters $\ep_v$ and $\la$ will later be chosen sufficiently small and large respectively so that this term is acceptably small.

The transport term can now be written
\begin{align*}
\pr_t V^l + v_\ep^j \pr_j V^l &= \sum_I e^{i \la \xi_I} [ (i \la) (\pr_t+ v_\ep^j\pr_j) \xi_I \tilde{v}_I^l + (\pr_t+ v_\ep^j\pr_j) \tilde{v}_I^l ]
\end{align*}
We will require that all the phase functions obey the same transport equation
\begin{align}
 (\pr_t+ v_\ep^j\pr_j) \xi_I &= 0 \\
 \xi_I(t(I), x) &= {\hat \xi}_I \label{eq:theDataWeNeed}
\end{align}

Therefore, although we will choose linear initial data ${\hat \xi}_I$ for the phase functions $\xi_I$, they will immediately cease to be linear, and will only remain useful in the oscillatory estimates for a short period of time $| t - t(I) | \leq \tau$ depending on $v$.  We are therefore forced to embed appropriate time cutoff functions into the definition of the amplitudes $\tilde{v}_I^l$.  These cutoff functions must restrict the lifespan of $V_I$ so that the following requirements remain fulfilled:
\begin{itemize}
\item {\bf Nondegeneracy}: We require  $| \nab \xi_I |$ to remain bounded from $0$
\item {\bf Higher derivative bounds}: We will also require that $D^2 \xi_I$ remains bounded (uniformly in $I$), because the oscillatory estimate requires us to differentiate the amplitude of the data.
\end{itemize}
The price we pay for introducing a time cutoff function is that the time cutoff itself will be differentiated in the transport term, and the shorter the lifespan, the larger the induced stress.  Therefore, the time parameter $\tau$ governing the lifespan of $V_I$ must be chosen to achieve the above objectives before the parameter $\la$ is chosen to ensure that the resulting stress is small.

It is interesting to reflect on how some approximation to this phase transport requirement can be seen in previous constructions of solutions to the Euler equation which fail to conserve energy, including Schnirelman's use of ``modulated Kolmogorov flow'' in \cite{shnNonUnq}, the late stages of the solutions to Tartar's wave cone used in the argument \cite{deLSzeIncl}, and the phase modulations in the argument \cite{deLSzeCts}.

We have now explained the constraint that the transport term imposes on the phase functions.  We cannot begin to make estimates on the solution until we have specified the amplitudes $v_I$ of the correction, and these cannot be specified until we describe how to eliminate the Stress term $Q_S^{jl} = \sum_I (V_I^j \bar{V}_I^l) + P_0 \de^{jl} + \pr_j R^{jl} $.  However, before we can resolve the stress term, we must first explain how the High-High interference term is dealt with, since eliminating these terms will involve imposing an additional constraint on the amplitudes $v_I$.

%% file: highHigh1.tex
The High-High interference term is given by
\begin{align} \label{eq:highHighOrigin}
 \pr_j Q_H^{jl} = \sum_{J \neq \bar{I}} \pr_j ( V_I V_J )^{jl}
\end{align}
where each $V_I$ is an oscillatory wave of the form $V_I = e^{i \la \xi_I} \tilde{v}_I$, and the condition that $I \neq \bar{J}$ will ensure that $\nab \xi_J$ is separated from $- \nab \xi_I$, so that the symmetric product $( V_I V_J )^{jl} = \fr{1}{2} e^{i \la (\xi_I + \xi_J) } ( \tilde{v}_I^j \tilde{v}_J^l + \tilde{v}_I^j \tilde{v}_J^l )$ genuinely has high frequency after we have proven an upper bound on $|\nab ( \xi_I + \xi_J)|^{-1}$.

Interference terms such as these are typical in even the simplest examples of convex integration and other related constructions such as the construction of nowhere differentiable functions.  There must be a way to ensure that the correction chosen to achieve a certain objective in a region $\Om$ does not interfere with the correction chosen in an overlapping region $\Om'$, otherwise there is a danger of ``stepping on your own foot'' during the construction.  
The device introduced in \cite{deLSzeCts} which overcomes this difficulty for the Euler equations is to look at Beltrami flows -- that is, eigenfunctions $\nab \times v = \mu v$ of the curl operator, which are automatically stationary solutions to the Euler equations.\footnote{More generally, a Beltrami flow can be of the form $\nab \times v = \mu(x) v$ for an arbitrary function $\mu(x)$.}  

Here we will impose a microlocal, Beltrami flow condition.  To be specific, recall that the individual waves take the form
\[ V_I = \fr{1}{\la} \nab \times ( e^{i \xi_I} w_I ) = e^{i \la \xi_I}( v_I + \fr{\nab \times w_I}{ \la} ) \]
where $w_I$ is a pointwise solution to the linear equation $(i\nab \xi_I) \times w_I = v_I$, so that the amplitudes $v_I$ are required to point in a direction perpendicular to the direction of oscillation $\nab \xi_I$, like any high-frequency, divergence free plane wave should.  The high frequency analogue of the Beltrami flow condition $\nab \times v = \mu v$, is the pointwise condition $(i \nab \xi_I) \times v_I = \mu v_I$, where the eigenvalue $\mu$ must be one of $ \mu = \pm |\nab \xi_I|$ since those are the eigenvalues of $(i \nab \xi_I) \times$ when viewed as an operator on $\langle \nab \xi_I \rangle^\perp$.

We will therefore require that
\begin{align*}
  \nab \xi_I \cdot v_I &= 0 \\
 (i \nab \xi_I) \times v_I &= |\nab \xi_I| v_I
\end{align*}
holds pointwise, and choose
\[  w_I = \fr{(i \nab \xi_I) }{|\nab \xi_I|^2} \times v_I = \fr{1}{|\nab \xi_I|} v_I \]

If we untangle this requirement in terms of the real and imaginary parts $a_I^l, b_I^l \in \R^3$ of $v_I^l = a_I^l + i b_I^l$, then the Beltrami flow condition can be phrased as:
\begin{align*}
 \nab \xi_I \cdot a_I &= 0 \\
 b_I &= \fr{( \nab \xi_I )}{|\nab \xi_I|} \times a_I
\end{align*}
or equivalently
\begin{align*}
 \nab \xi_I \cdot b_I &= 0 \\
 a_I &= -\fr{( \nab \xi_I )}{|\nab \xi_I|} \times b_I
\end{align*}

Therefore, our final remaining degree of freedom in the choice of $V_I$ is, without loss of generality, the imaginary part of its amplitude, $b_I$, which itself is required to be perpendicular to $\nab \xi_I$.  We will specify this amplitude, its support in space-time, along with the index set $I$ in the following section.  For now, we explain why the choice of approximate Beltrami flows helps to estimate the High-high interference term (essentially, we will go through the proof that Beltrami flows are stationary solutions to Euler).

We rewrite (\ref{eq:highHighOrigin}) as
\begin{align}
\pr_j Q_H^{jl} = \fr{1}{2} \sum_{J \neq \bar{I}} \pr_j (V_I^j V_J^l + V_J^j V_I^l) 
\end{align}
and use the divergence free condition to write
\begin{align}
\pr_j Q_H^{jl} = \fr{1}{2} \sum_{J \neq \bar{I}} V_I^j \pr_j V_J^l + V_J^j \pr_j V_I^l
\end{align}
Let us fix a pair $I, J$, and permute the indices
\begin{align}
V_I^j \pr_j V_J^l + V_J^j \pr_j V_I^l &= (V_I)_j [ \pr^j V_J^l - \pr^l V_J^j] + (V_J)_j [ \pr^j V_I^l - \pr^l V_I^j] \\
&+ (V_I)_j \pr^l V_J^j + (V_J)_j\pr^l V_I^j
\end{align}
The second of these terms is actually a gradient,
\begin{align}
(V_I)_j \pr^l V_J^j + (V_J)_j\pr^l V_I^j &= \pr^l ( V_I \cdot V_J )
\end{align}
and it will be absorbed into the pressure by a term
\begin{align}
P &= P_0 + \sum_{J \neq {\bar I} } P_{I,J} \\
P_{I, J} &= - \fr{ ( V_I \cdot V_J ) }{2}
\end{align}

We now are reduced to solving
\begin{align}
\pr_j Q_{H, (IJ)}^{jl} &= (V_I)_j [ \pr^j V_J^l - \pr^l V_J^j] + (V_J)_j [ \pr^j V_I^l - \pr^l V_I^j] \label{eq:almostHighTime}
\end{align}

In three dimensions, a special observation can be made using the alternating structure of the right hand side.  One way to see it is to recall some notation.  Let us recall the volume element $\ep^{abc}$ which is the fully antisymmetric tensor defined so that the trilinear form $\ep^{abc}u_a v_b w_c$ gives the signed volume of the parallelogram $u \wed v \wed w$.  We recall the basic definitions $(u \times v)^a = \ep^{abc} u_b v_c$ and $(\nab \times v)^a = \ep^{abc} \pr_b v_c$.  Using these definitions and the elementary identity\footnote{This identity gives two different expressions for the inner product on the second exterior power of $\La^2(\R^3)$.  It can be proven by testing against components of an orthonormal frame, and the proof can be accelerated by observing that these tensors are both antisymmetric in $(ab)$ and $(fg)$.}
\[ \ep^{abc}\ep_{cfg} = \de^a_f \de^b_g - \de^a_g \de^b_f \]
we establish the calculus identity
\begin{align}
(V_I)_j [ \pr^j V_J^l - \pr^l V_J^j] + (V_J)_j [ \pr^j V_I^l - \pr^l V_I^j] &= - V_I \times (\nab \times V_J) - V_J \times (\nab \times V_I) \label{eq:theSum}
\end{align}

At this stage, it is clear why curl eigenfunctions are helpful.  Namely, if we were using eigenfunctions of curl
\begin{align}
\nab \times V_I &= \la_I V_I \\
\nab \times V_J &= \la_J V_J \\
\la_I &= \la_J
\end{align}
with the same eigenvalue, then their sum would be a curl eigenfunction and (\ref{eq:theSum}) would be identically $0$.  The ``microlocal Beltrami-flow'' condition says that $v_I$ is an eigenfunction for the symbol of curl, namely
\begin{align}
 (i \nab \xi_I) \times v_I &= |\nab \xi_I| v_I \label{eq:microlocalBeltrami}
\end{align}
In this way, each $V_I$ looks like a curl eigenfunction to leading order in $\la$.

We now express (\ref{eq:almostHighTime}) more completely as
\begin{align}
\pr_j Q_{H, (IJ)}^{jl} &= - \la e^{i \la (\xi_I + \xi_J) } \left({\tilde v}_I \times [ (i \nab \xi_J) \times {\tilde v}_J ] + {\tilde v}_J \times [ (i \nab \xi_I) \times {\tilde v}_I   ]\right) \\
&- e^{i \la(\xi_I + \xi_J)} \left( \tilde{v}_I \times(\nab \times \tilde{v}_J) + \tilde{v}_J \times(\nab \times \tilde{v}_I) \right) \\
\begin{split}
&= - \la e^{i \la (\xi_I + \xi_J) } \left( v_I \times [ (i \nab \xi_J) \times  v_J ] + v_J \times [ (i \nab \xi_I) \times v_I   ]\right) \\
&- e^{i \la (\xi_I + \xi_J) }\left( (\nab \times w_I) \times [ (i \nab \xi_J) \times {\tilde v}_J ] + (\nab \times w_J) \times [ (i \nab \xi_I) \times {\tilde v}_I ] \right) \\
&- e^{i \la (\xi_I + \xi_J) } \left( v_I \times [ (i \nab \xi_J) \times  (\nab \times w_J) ] + v_J \times [ (i \nab \xi_I) \times (\nab \times w_I) ]\right) \\
&- e^{i \la(\xi_I + \xi_J)} \left( \tilde{v}_I \times(\nab \times \tilde{v}_J) + \tilde{v}_J \times(\nab \times \tilde{v}_I) \right)
\end{split} \\
&= - \la e^{i \la (\xi_I + \xi_J) } \left( v_I \times [ (i \nab \xi_J) \times  v_J ] + v_J \times [ (i \nab \xi_I) \times v_I ]\right) + \mbox{ Lower Order Terms} \label{eq:almostHighHigh}
\end{align}

Using the condition (\ref{eq:microlocalBeltrami}), we add
\[\la e^{i \la (\xi_I + \xi_J) } ( v_I \times v_J + v_J \times v_I ) = 0 \]
to rewrite the High-High interference term (\ref{eq:almostHighHigh}) as
\begin{align}
\pr_j Q_{H, (IJ)}^{jl} &= - \la e^{i \la (\xi_I + \xi_J) } \left( v_I \times [ (|\nab \xi_J| - 1)v_J ] + v_J \times [ (|\nab \xi_I| - 1) \times v_I ]\right) + \mbox{Lower Order Terms}
\end{align}

Our strategy for ensuring this term can be controlled involves imposing the following additional conditions on the phase functions:
\begin{itemize}
 \item {\bf Synchronized Periods}: We require that $|\nab \xi_I| \approx 1 \approx |\nab \xi_J|$ for all $I, J$.  
 \item {\bf Angle separation}:  We require  $| \nab \xi_I + \nab \xi_J|$ to remain bounded from $0$ for $J \neq \bar{I}$ so that the parametrix can be used. Since the $\nab \xi_I$ and $\nab \xi_J$ both have absolute value essentially one, this lower bound is equivalent to the angles between the gradients being well-separated.
\end{itemize}
These properties will be guaranteed by taking initial values 
\begin{align}
\xi_I(t(I), x) &= {\hat \xi}_I
\end{align}
for the phase functions which have angle-separated gradients, and such that $|\nab {\hat \xi}_I| = 1$ everywhere.  We then use the time cutoff function in the amplitude to make sure that the lifespan of $V_I$ are sufficiently short so that the equal periodicity condition holds is maintained.  This lifespan will depend on the given $v$, as the gradient of $v_\ep$ enters into the transport equations for the phase gradients.

From our preceding discussion, we have established that the final remaining degree of freedom in constructing the correction goes into choosing $b_I^l$, the imaginary part of the amplitude of $v_I^l$, as well as precisely specifying the phase functions.  We know that the amplitude $b_I$ will involve cutoff functions in space and time chosen so that the phase function $\xi_I$ can exist on the support of $b_I$ without any critical points and its derivatives can remain bounded, and also so that the the building blocks will still approximate Beltrami flows.  Up to the choice of some parameters, the construction will be completely specified in the following section, where the amplitudes $b_I$ are chosen to make the contribution of the stress term small.

%% file: stress1.tex
Here we discuss how to construct the corrections $V_I$, $P_0$ in such a way that the Stress term
\begin{align} \label{eq:stressTermAgain}
Q_S^{jl} = \sum_I (V_I^j \bar{V}_I^l) + P_0\de^{jl} + R^{jl} 
\end{align}
can be reduced to a new stress $Q_S^{jl}$ much smaller than $R^{jl}$ in $C^0$ after appropriate choices of $V_I$ and $P_0$.  In the process, we will describe the initial data (\ref{eq:theDataWeNeed}) for the phase functions and also the set $\II$ by which the corrections $V_I$ are indexed.  A brief summary of the results of this section are summarized by the Proposition in Section (\ref{sec:Summary}).  

In this part of the argument, we will take a different approach than the one taken in \cite{deLSzeCts}.  Their approach is based on Carath{\' e}odory's theorem for convex hulls, and does not seem to generalize readily to the setting of nonlinear phase functions.

Unlike the Transport term, the High-Low interaction term and the High-High interference terms, the Stress term (\ref{eq:stressTermAgain}) does not require us to solve a divergence equation.  Rather, eliminating this term is essentially an algebraic problem which can be thought of as calculating an elaborate ``square root''.

We will choose $V = \sum_I V_I$ and $P_0$ to ensure that
\begin{align} \label{eq:stressError}
 \sum_I (V_I^j \bar{V}_I^l) + P_0 \de^{jl} + R^{jl} \approx 0
\end{align}
is sufficiently small pointwise.  

Roughly speaking, making the error (\ref{eq:stressError}) small pointwise will give the amplitudes of $V_I$ size $|V_I| \leqc |R|^{1/2}$.  We will need to measure the size of the derivatives of $V_I$, so we will attempt to ensure that the terms $V_I$ are balanced in size $|V_I| \sim |R|^{1/2}$ where they are supported, since a derivative $\pr R^{1/2}$ can be quite large when $|R|$ is small.

\subsubsection{The Approximate Stress Equation}

Before starting the construction, we observe that one only has a hope of solving this equation if $P_0 \de^{jl} +R^{jl}$ is negative definite.  Although we cannot fully eliminate $R^{jl}$ by adding a multiple of the metric $\de^{jl}$, we can at least ensure that the sum $P_0 \de^{jl} +R^{jl}$ is negative definite, and is also not significantly larger than $R^{jl}$ itself by choosing $P_0$ of size $|P_0| \sim |R|$.

The other issue we must take into account is that if we solve the equation $\sum_I (V_I^j \bar{V}_I^l) + P_0 \de^{jl} + R^{jl} = 0$ pointwise, then the individual waves $V_I$ inherit the fine scale behavior of $R^{jl}$ (which is inconsistent with the basic philosophy of the construction), and their higher derivatives will therefore not obey good estimates.  A similar issue arose while considering the Transport term.

To deal with the latter issue, we will first replace $R$ by a version $R_\ep = \eta_{\ep_R} \ast_{x,t} R$ which is mollified in both space and time variables.  The scale $\ep_R$ of this mollification will be chosen later in the argument so that the error $R - R_\ep$ is acceptably small.  Actually, in order to achieve sharp regularity in the main theorem, it will be necessary to do something more complicated than a simple-minded space-time mollification.

Finally, observe that, without loss of generality, we can assume that $R_\ep$ is trace free by absorbing the trace part $(R_\ep^{jl} \de_{jl})\fr{\de^{jl}}{n}$ into the pressure.  More specifically, we choose a correction of the form
\begin{align} \label{eq:p0def}
P_0 &= -  \fr{e(t)}{n} - \fr{(R_\ep^{jl} \de_{jl})}{n}, \quad \quad n = 3
\end{align}
and with this choice the equation
\begin{align} \label{eq:sqrt1}
 \sum_I (V_I^j \bar{V}_I^l) &= - P_0\de^{jl} - R_\ep^{jl}
\end{align}
takes the form
\begin{align} \label{eq:stressEq1}
 \sum_I (V_I^j \bar{V}_I^l) &= e(t) \fr{\de^{jl}}{n} - {\mathring R}_\ep^{jl}
\end{align}
where the tensor
\begin{align*}
{\mathring R}_\ep^{jl} &= R_\ep^{jl} - (R_\ep^{jl} \de_{jl})\fr{\de^{jl}}{n}
\end{align*}
is the trace-free part of $R_\ep$

Here $e(t)$ is a non-negative function which must satisfy certain restrictions such as those of Lemma (\ref{lem:mainLemCts}).  For example, we impose a lower bound (\ref{eq:lowBdEtCts}) on $e(t)$ to make sure that the right hand side of (\ref{eq:stressEq1}) is, among other things, non-negative definite.

Prescribing the trace $e(t)$ in equation (\ref{eq:stressEq1}) gives us precise control over the energy
\[ \int |V|^2 dx = \int V^j V^l \de_{jl} dx \]
added into the system by the correction $V$.  Namely, since the components $V_I$ and $V_J$ oscillate rapidly in different directions for $J \neq \bar{I}$, they are almost orthogonal
\[ \int V_I \cdot V_J dx \approx 0 \quad \quad J \neq \bar{I} \]
leading to an approximate energy increase of
\begin{align}
\int |V|^2 dx &= \sum_{I, J} \int  V_I \cdot V_J dx \\
&\approx \int \sum_I V_I \cdot {\bar V}_I dx \\
&\approx \int_{\T^3} e dx
\end{align}
if we can solve the stress equation equation (\ref{eq:stressEq1}).  Therefore, in terms of dimensional analysis, the function $e$, much like $p$ and $R$, should be regarded as having the dimensions of an energy density (or $\fr{length^2}{time^2}$).

Of course, we cannot solve equation (\ref{eq:stressEq1}) exactly, because we have also required that each $V_I$ be divergence free.  However, as we explain in the following section, equation (\ref{eq:stressEq1}) can be solved approximately when the oscillation parameter $\la$ is sufficiently large.

\subsubsection{The Stress Equation and the Initial Phase Directions} \label{subsec:stressEqn}

Let us recall that each $V_I$ can be represented as
\begin{align*}
 V_I&= e^{i \la \xi_I}\tilde{v}_I \\
&= \nab \times ( \fr{e^{i \la \xi_I}}{\la} w_I) \\
&= e^{i \la \xi_I}(v_I + \fr{\nab \times w_I}{\la})
\end{align*}
When $\la$ is large, the $v_I$ term dominates the $\fr{\nab \times w_I}{\la}$ term, and we have an approximation
\begin{align*}
V_I &\approx e^{i \la \xi_I} v_I
\end{align*}

We still have freedom to choose the amplitude $v_I$, so instead of solving equation (\ref{eq:stressEq1}) exactly, we will solve the equation
\begin{center}
{\bf The Stress Equation}
\end{center}
\begin{align} \label{eq:theStressEq}
\sum_I v_I^j {\bar v}_I^l &= e \fr{\de^{jl}}{n} - {\mathring R}_\ep^{jl}
\end{align}
and solving (\ref{eq:theStressEq}) will leave an error term
\begin{align}
\begin{split}
Q_S^{jl} &= \sum_I \fr{(\nab \times w_I)^j{\bar v}_I^l }{\la} + \sum_I \fr{v_I^j(\nab \times {\bar w}_I)^l }{\la} \\
&+  \sum_I \fr{(\nab \times w_I)^j(\nab \times {\bar w}_I)^l}{\la^2}
\end{split}
\end{align}
which composes part of the new stress $R_1^{jl}$.

In Section (\ref{highHighSection}), we introduced the requirement that the complex amplitude $v_I$ must satisfy the microlocal Beltrami flow condition $(i \nab \xi_I) \times v_I = |\nab \xi_I| v_I$, and assuming this requirement was satisfied, we defined $w_I = \fr{(i \nab \xi_I)}{|\nab \xi_I|^2}\times v_I = \fr{1}{|\nab \xi_I|} v_I$.  If $v_I = a_I + i b_I$ is expanded into its real and imaginary parts, then the Beltrami flow condition can be written $b_I = \fr{(i \nab \xi_I)}{|\nab \xi_I|} \times a_I$.

The imaginary part $b_I$ is our last remaining degree of freedom, but it is still constrained by the requirement that
\[
 \nab \xi_I \cdot b_I = \pr_l \xi_I b_I^l = 0
\]
expressing the condition that the waves be divergence free.  The real part $a_I$ defined by 
\ali{
 a_I = - \fr{(i \nab \xi_I)}{|\nab \xi_I|} \times b_I \label{eq:aIsbRotate}
 }
satisfies the same requirement, since it is a rotation of $b_I$ within the plane orthogonal to $\nab \xi_I$ by an angle of $\fr{\pi}{2}$.

By expanding each $v_I$ in terms of real and imaginary parts, the equation (\ref{eq:theStressEq}) becomes
\begin{align*}
 \sum_I (a_I^ja_I^l + b_I^j b_I^l) = e \fr{\de^{jl}}{n} - {\mathring R}_\ep^{jl}
\end{align*}
No matter how $b_I$ is chosen, as long as $b_I$ is indeed orthogonal to $\nab \xi_I$, the bilinear form which arises on the left hand side is given by
\begin{align*}
a_I^ja_I^l + b_I^j b_I^l = |b_I|^2 ( \de^{jl} - \fr{\pr^j \xi_I \pr^l \xi_I }{|\nab \xi_I|^2} )
\end{align*}
where $\de^{jl} - \fr{\pr^j \xi_I \pr^l \xi_I}{|\nab \xi_I|^2}$ is the orthogonal projection of the inner product $\de^{jl}$ to the plane $\langle \nab \xi_I \rangle^\perp$.  This fact follows from equation (\ref{eq:aIsbRotate}), which implies that $a_I$ is orthogonal to $b_I$ and has the same absolute value.  One can also check this identity by showing that the form $v_I^j {\bar v}_I^l$ restricted to the plane $\langle \nab \xi_I \rangle^\perp$ is invariant under any rotation within the plane, and is determined up to a constant by this property.

Using these observations we can write the equation (\ref{eq:theStressEq}) in the form
\begin{align} \label{eq:sqrt3}
 \sum_I |b_I|^2 ( \de^{jl} - \fr{\pr^j \xi_I \pr^l \xi_I }{|\nab \xi_I|^2} ) &= G^{jl},
\end{align}
for a tensor
\begin{align}
G^{jl} &= e \fr{\de^{jl}}{n} - {\mathring R}_\ep^{jl}
\end{align}
which is positive definite as long as.
\begin{align} \label{ineq:eAtLeast}
e &\geq K |R_\ep|
\end{align}
for a sufficiently large constant $K$.  The constant $K$ in (\ref{ineq:eAtLeast}) is an absolute constant which we will wait until later to specify.  Regarding the choice of $e$, the reader should keep in mind that we expect an estimate of the type $|V_I| \approx |b_I| \leqc |R|^{1/2}$.  Furthermore, obtaining bounds on the derivatives of $V_I$ will basically require us to differentiate the equation (\ref{eq:sqrt3}), which will require lower bounds on $|b_I|$ as well.

With the stress equation in the form (\ref{eq:sqrt3}), we can now explain how the phase functions $\xi_I$ and the amplitudes $b_I$ are constructed.

\subsubsection{The Index Set, the cutoffs and the phase functions} \label{subsec:indexPhase}
We have established already that the phase functions $\xi_I$ must satisfy a transport equation
\[ (\pr_t + v^j_\ep \pr_j) \xi_I = 0. \]
What remains is to specify their data.  Each phase function will be given an initial condition at time $t(I)$, and to assure that equation (\ref{eq:sqrt3}), one requirement we impose on the phase functions is that the tensors
\[ \de^{jl} - \fr{\pr^j \xi_I \pr^l \xi_I }{|\nab \xi_I|^2} \]
must span the space $\SS$ of symmetric, $(2,0)$ tensors at any point.  Therefore, at each point in $\R \times \T^3$ we will require at least $6$ phase functions to solve (\ref{eq:sqrt3}) -- or $12$, rather, since the negative of each phase function will also be included.  The phase functions will only be defined on some open set $\Om_I(t) \subseteq \T^3$, where $\Om_I$ is the image of a parameterization $(\Phi_I(x, t), t) : \Om_0 \times \R \to \T^3 \times \R$, which moves under the flow of $v_\ep$.

The amplitude function will be equipped with two cutoff functions:
\[ b_I^l = \eta((t - t(I))/\tau) \psi_I(x,t) b_{I, 1}^l = \eta_I \psi_I b_{I,1}^l \]
The cutoff $\psi_I$ localizes to $\Om_I$, so that we will not have to worry that the phase function fails to be defined globally.  In order to be compatible with the transport of $\xi_I$, we will also have
\[ (\pr_t + v^j_\ep \pr_j) \psi_I = 0, \]
so we only have to specify the initial values for both $\xi_I$ and $\psi_I$.  The cutoff $\eta((t - t(I))/\tau)$ limits the lifespan of $V_I$ to a time period less than or equal to $\tau$, where $\tau$ is a small parameter that remains to be chosen.

To continue, we will have to explain the set $\II$ to which the indices $I$ belong.  The purpose of the index $I$ is to record both the location of $V_I$ in space-time, and the direction $\nab \xi_I$ in which $V_I$ oscillates.  The strategy is the following: at time $0$, we cover the torus $\T^3$ with $8$ charts, and on each of these charts we place $12$ phase functions $\xi_I$ so that the six directions determined by the $\xi_I$ allow us to solve (\ref{eq:sqrt3}) for an arbitrary right hand side within that chart.  We must cut off these amplitudes after a very short time interval, at which point the phase functions must be replaced by another group of phase functions whose amplitudes will also be supported on $8$ charts.  To avoid interference in the High-high interactions, we also must make sure that any pair $V_I$ and $V_J$ who share support in space time oscillate in well-separated directions.  The cutoff functions are chosen as part of a partition of unity which allows us to patch together local solutions of the quadratic equation (\ref{eq:sqrt3}), and such a choice is consistent with the transport equation.

\paragraph{The index set as a graph and some notation}

With this rough description in place, we choose the index set $\II = ( \Z/(2\Z) )^3 \times \Z \times F$.  The $F$ coordinate is meant to specify the direction $d \xi_I$ in which $V_I$ oscillates, so $F$ will have $12$ elements; to be concrete, let $\varphi = (1 + \sqrt{5})/2$ be the golden ratio and set
\[ F = \left\{ \pm\fr{(0, 1, \pm \varphi)}{\sqrt{1 + \varphi^2}}, \pm\fr{(1, \pm \varphi, 0)}{\sqrt{1 + \varphi^2}}, \pm\fr{(\pm \varphi, 0, 1)}{\sqrt{1 + \varphi^2}} \right\} \]
to be the set of normal vectors to the faces of a dodecahedron in $\R^3$, which themselves form the vertices of an icosahedron.  It will also be useful to define the projective dodacehedron
\[ \F = F/\langle \pm 1 \rangle = \left\{ \left[\fr{(0, 1, \pm \varphi)}{\sqrt{1 + \varphi^2}}\right], \left[\fr{(1, \pm \varphi, 0)}{\sqrt{1 + \varphi^2}}\right], \left[\fr{(\pm \varphi, 0, 1)}{\sqrt{1 + \varphi^2}}\right] \right\}. \]
The advantage of doing so is that $\F$ inherits the action of the icosahedral group, but accounts for each of the $|\F| = 6$ directions in $F$ exactly once.

The discrete coordinate $k(I) = (\kappa, k_4) = ( \kappa_1, \kappa_2, \kappa_3, k_4 ) \in (\Z/(2\Z))^3 \times \Z$ determines which chart $\Om_I = \Om_{k}$ on $\T^3 \times \R$ contains the space-time support of $b_I$, and there is a function \[ (x(I), t(I)) : (\Z/(2\Z))^3 \times \Z \to \T^3 \times \R \] which will describe the ``center'' of $\Om_I$.  If we let $e^i$, $i = 1, 2, 3$ be the standard generators of the lattice $\Z^3$, then we set
\[ (x(I), t(I)) = ( \fr{1}{2} \sum_{i = 1}^3 \kappa_i e^i, \tau k_4 ) \]

Rather than view $\II$ as simply an index set, it is useful to think of $\II$ as the vertices of a graph, where we connect any two indices $I$ if the supports of $V_I$ overlap.  So let us also define the set of neighbors of $I$, including $I$ itself, to be the set
\[ \NN(I) = \NN(\kappa, k_4, f) = \bigcup_{k' \in \NN(k)} \{k'\} \times F \]
where the $2^3 \times 3$ neighboring discrete positions $\NN(k)$ are defined by
\[ \NN(\kappa, k_4) = \{ (\kappa, k_4) + u_1 e^1 + u_2 e^2 + u_3 e^3 + u_4 e^4  ~|~ u_i = -1, 0, 1 \} \]
Each vertex $I$ in the above graph has $|\NN(I)| = 2^3 \times 3 \times 12$ neighbors.

Whenever some object with an index, say $\Phi_I$ or $\Om_I$, only depends on only part of $I$, say $k = (\kappa, k_4)$, we will abuse notation and write both $\Phi_I = \Phi_{k}$ or $\Om_I = \Om_{k}$ to refer to the same object.  We will also introduce the notation
\[ \II(k) = k \times F \]
and
\[ \II(k_4) = (\Z /(2 \Z))^3 \times \{k_4\} \times F \]
to abbreviate important families of the indices.

\paragraph{The Cutoff Functions and Partitions of Unity}

The amplitude is equipped with two cutoff functions:
\[ b_I^l = \eta((t - t(I))/\tau) \psi_I(x,t) b_{I, 1}^l = \eta((t - \tau k_4)/\tau) \psi_I(x,t) b_{I, 1}^l.\]
The function $\eta(t)$ is a smooth, non-negative bump function with compact support in $C_c^\infty( |t| < 5/6 )$, such that 
\[ \sum_{k \in \Z} \eta^2( t - k ) = 1,\] and it is rescaled and reoriented to construct many of the other cutoff functions we use.  To construct such an $\eta$, first take a smooth bump function $\tilde{\eta}$ equal to $1$ for $|t| < 2/3$, and supported in $|t| < 5/6$, and set
\[ \eta(t) = \fr{\tilde{\eta}(t)}{\sqrt{\sum_{k \in \Z} \tilde{\eta}^2(t - k) } }, \]
which is smooth because the denominator is bounded away from $0$.  We use a partition of unity based on squares because this will allow us to patch together local solutions of
\begin{align*}
 \sum_I |b_{I, 1}|^2 ( \de^{jl} - \fr{\pr^j \xi_I \pr^l \xi_I }{|\nab \xi_I|^2} ) &= G^{jl}.
\end{align*}
which is homogeneous of degree two in the unknown absolute values $|b_{I,1}|$.

For each index $k_4$, the cutoff functions $\psi_I = \psi_{(\kappa, k_4)}$ will also be part of a similar partition of unity.  We start by periodizing $\eta$ to construct a partition of unity ${\bar \psi}_{\kappa}(x)$, $\kappa \in (\Z/2\Z)^3$ of the torus, where each ${\bar \psi}_{\kappa}(x)$ is supported on a single chart of the torus
\begin{align}
{\bar \psi}_{\kappa}(x_1 e^1 + x_2 e^2 + x_3 e^3) &= \eta_{2 \Z}(2 x_1 - \kappa_1 )\eta_{2\Z}(2x_2 - \kappa_2)\eta_{2\Z}(2 x_3 - \kappa_3 ) \\
x_i \in \R \\
\eta_{2 \Z}(y) &= \sum_{k \in \Z} \eta( y - 2 k) 
\end{align}
Observe that ${\bar \psi}_{\kappa}(x)$ is integer-periodic in each $x_i$ and is well-defined as a function of $\kappa_i \in \Z/2\Z$.  Also note that ${\bar \psi}_{\kappa}$ is supported in a single $\fr{5}{6} \times \fr{5}{6} \times \fr{5}{6}$ box on which our phase functions can be safely defined without critical points. 

Now define $\psi_I = \psi_{\kappa, k_4}$ to be the unique solution to the initial value problem
\begin{align}
 (\pr_t + v^j_\ep \pr_j) \psi_{k} &= 0 \\
 \psi_{\kappa, k_4}(x, t(I)) &= {\bar \psi}_{\kappa(I)}(x)
\end{align}

This way, at the initial time $t = t(I) = \tau k_4$, we have
\[ \sum_{ \kappa \in (\Z/(2\Z))^3} \psi_{(\kappa, k_4)}^2(x,t) =  \sum_{\kappa \in (\Z/(2\Z))^3} {\bar \psi}_\kappa^2(x) = 1 \]
and this identity will propagate for time $t \in \R$ by the transport equation.

\paragraph{The Phase Functions} \label{sec:thePhaseFunctions}

For an index $I = (\kappa, k_4, f)$, the phase function $\xi_I$ will only be defined around the support of $\psi_{{\vec k}}(x,t)$, allowing it to exist without critical points.  It is convenient to specify the domain of $\Om_I(t)$, which is transported by the flow of $v_\ep$.  Set $\Om_0 = (-7/8, 7/8) \times (-7/8, 7/8) \times (-7/8, 7/8)$, and define $\Phi_I : \Om_0 \times \R \to \T^3 \times \R$ to be the unique solution to the ODE
\begin{align}
 \fr{d\Phi_I^j}{dt} &= v_\ep^j(\Phi_I(x,t)) \\
 \Phi_I(x,t(I)) &= x + x(I)  \tx{ mod } \Z^3
\end{align}
From the well known theorems on transport equations and ODE, the solution $\Phi_I$ to this equation is unique, it exists for all time because $\T^3$ is compact, and it remains a bijection onto its image $\Om_I(t)$ because it can be inverted by reversing the flow.  The fact that $\Phi_I$ is defined for all time is unimportant for us since we will be restricting to a short time interval, but observe that $\Om_I$ covers the support of $\psi_{k}$, moreover $ \psi_{k}(\Phi_I(x,t)) = {\bar \psi}_0(x) $.

We can now define the phase functions as solutions to the transport equation
\begin{align}
 (\pr_t + v^j_\ep \pr_j) \xi_I &= 0 \quad \mbox{ on } \Om_I(t)\\
 \xi_I(t(I), x) &= {\hat \xi}_I(x) \quad \mbox{ on } \Om_I(t(I))
\end{align}
Our study of the High-high frequency term motivates us to require that the initial phases ${\hat \xi}_I$ have the following properties:
\begin{itemize}
 \item ${\hat \xi}_I$ is a linear function
 \item $|\nab {\hat \xi}_I| = 1$ for all $I$
 \item There exists a $c > 0$ such that $|\nab {\hat \xi}_I +  \nab{\hat \xi}_J| \geq c > 0$ for all $J \neq \bar{I}$
 \item For each $k = (\kappa, k_4)$, the tensors
\[ \left\{ \de^{jl} - \fr{\pr^l {\hat \xi}_{(k,f)} \pr^j {\hat \xi}_{(k,f)}}{|\nab {\hat \xi}_{(k,f)}|^2} ~|~ f \in \F \right\} \]
form a basis for the space $\SS$ of symmetric, $(2,0)$ tensors at every point on $\Om_{k}$.
\end{itemize}
Once these requirements are satisfied initially, the latter two will remain true for a short time, and the first two will only hold approximately.

For the chart $k = (\kappa, k_4) = 0$, it suffices to pick phases $\{ {\hat \xi}_{(0, f)} ~|~ f \in F \}$ with icosahedral symmetry
\begin{align}
 {\hat \xi}_{(0, f)}(\Phi_0(x,0)) &= <f,x>
\end{align}
All the above properties can be checked for this choice, the first two being immediate and the latter two we will prove shortly.

For the other coordinate charts, we need to make sure that every phase function is separated in angle from all of its neighboring phase functions so that we still have the bound $|\nab \xi_I +  \nab \xi_J| \geq c > 0$ for all $I \neq \bar{J}$, with a possible smaller constant $c$.  A simple way to achieve these requirements while ensuring uniformity in the construction is to construct the initial data by rotating the dodecahedron.  For this purpose,  the following fact, which is proven in the Appendix, is helpful.
\begin{lem}\label{lem:someRotations}
There exists a collection of $2^4$ rotations $O_m$, indexed by $m \in (\Z / (2\Z))^4$, and a positive number $c > 0$, with the property that
\[ |f \circ O_m + f' \circ O_{m'}| \geq c \quad f, f' \in F, m, m' \in (\Z / (2\Z))^4 \]
holds unless $f' = -f$ and $m' = m$.  Without loss of generality, one can choose $O_0 = \tx{Id}$.
\end{lem}

With this fact in hand, we can then use the above rotations to assign initial values to all the other phase functions
\begin{align}
 {\hat \xi}_{(k, f)}(\Phi_I(x, t(I))) &= <f\circ O_{k},x>  \quad k = k(I)
\end{align}

\subsubsection{ Localizing the Stress Equation } \label{sec:localizeStress}

Recall that we intend to solve the equation
\begin{align*}
 \sum_I |b_I|^2 \left( \de^{jl} - \fr{\pr^j \xi_I \pr^l \xi_I }{|\nab \xi_I|^2} \right) &= G^{jl},
\end{align*}
where the right hand side
\begin{align}
G^{jl} = e \fr{\de^{jl}}{n} - {\mathring R}_\ep^{jl}
\end{align}
could be a fairly arbitrary, symmetric, positive $(2,0)$ tensor in $\SS_+$.

By representing the amplitude as \[ b_I^l = \eta((t - t(I))/\tau) \psi_I(x,t) b_{I, 1}^l = \eta_{k_4} \psi_{k} b_{I, 1}^l, \] we can use the partition of unity property of the cutoff functions to form a global solution $V$ by gluing together the locally defined solutions to the linear systems
\begin{align} \label{eq:sqrtAlmost}
 \sum_{I \in k \times F} |b_{I,1}|^2 ( \de^{jl} - \fr{\pr^j \xi_ I\pr^l \xi_I}{|\nab \xi_I|^2} ) &= G^{jl} \quad \mbox{ on } \Om_{k}
\end{align}

Indeed, once we have solved the local equation (\ref{eq:sqrtAlmost}) for all $k = (\kappa, k_4)$, then we can check that the stress equation (\ref{eq:sqrt3}) is satisfied as follows:
\begin{align*}
\sum_I |b_I|^2 ( \de^{jl} - \fr{\pr^j \xi_ I\pr^l \xi_I}{|\nab \xi_I|^2} ) &= \sum_I \eta_{k_4}^2 \psi_I^2 |b_{I,1}|^2 ( \de^{jl} - \fr{\pr^j \xi_ I\pr^l \xi_I}{|\nab \xi_I|^2} ) \\
&= \sum_{k_4} \eta_{k_4}^2 \left( \sum_{I \in \II(k_4)} \psi_{(\kappa, k_4)}^2 |b_{I,1}|^2 ( \de^{jl} - \fr{\pr^j \xi_ I\pr^l \xi_I}{|\nab \xi_I|^2} ) \right) \\
&= \sum_{k_4} \eta_{k_4}^2 \left( \sum_{\kappa \in (\Z/(2 \Z))^3} \psi_{(\kappa, k_4)}^2 \left( \sum_{I \in (\kappa, k_4) \times F} |b_{I,1}|^2 ( \de^{jl} - \fr{\pr^j \xi_ I\pr^l \xi_I}{|\nab \xi_I|^2} ) \right) \right) \\
&= \sum_{k_4} \eta_{k_4}^2 \left( \sum_{\kappa \in (\Z/(2 \Z))^3} \psi_{(\kappa, k_4)}^2  \right) G^{jl} \\
&= \sum_{k_4} \eta_{k_4}^2 G^{jl} = G^{jl}
\end{align*}
Thus it suffices to solve the local equation (\ref{eq:sqrtAlmost}) on each chart $\Om_k$ of $\R \times \T^3$.

\subsubsection{Solving the quadratic equation}\label{sec:solveQuadEqn}

Now observe that each term on the left hand side of (\ref{eq:sqrtAlmost}) is repeated twice because $b_{\bar{I}} = - b_I$ and $\nab \xi_{\bar{I}} = - \nab \xi_I$.  To take this repetition into account, let us divide by $2$, and rewrite the system (\ref{eq:sqrtAlmost}) as
\begin{align} \label{eq:sqrtAlmost2}
 \sum_{I \in k \times \F} |b_{I,1}|^2 ( \de^{jl} - \fr{\pr^j \xi_I \pr^l \xi_I }{|\nab \xi_I|^2} ) &= \fr{1}{2} G^{jl} \quad \mbox{ on } \Om_{k}
\end{align}
where $\F$ denotes the set of faces of the projective dodecahedron introduced in Section (\ref{subsec:indexPhase}).  In this form, the six unknown coefficients $|b_{I,1}|^2$, $I \in k \times \F$, of the linear system will be uniquely determined.

If the tensor $G^{jl}$ were completely arbitrary, we could not guarantee that the coefficients obtained by solving the linear equation would all be positive, but because of our choice of $P_0$, the trace part of the right hand side dominates the trace free part, and we can rewrite the right hand side in the form
\begin{align}
 \fr{G^{jl}}{2} &= \fr{1}{2} ( e \fr{\de^{jl}}{n} - {\mathring R}_\ep^{jl} ) \\
 &= e ( \fr{\de^{jl}}{2n} + \varepsilon^{jl} )  \label{eq:gIsSizeE}
\end{align}
and $\varepsilon^{jl}$ satisfies the bound
\begin{align}
\varepsilon(R_\ep)^{jl} &= - \fr{{\mathring R}_\ep^{jl}}{2 e} \\
|\varepsilon| &\leq \fr{100}{K}
\end{align}
where $K$ is the constant in the inequality (\ref{ineq:eAtLeast}).

With the expression (\ref{eq:gIsSizeE}) in mind, we write the amplitude in the form
\[ b_{I,1}^l = e^{1/2} b_{I,2}^l. \]
With this normalization, we can expect the absolute value of $b_{I,2}$ to be of size $\approx 1$ in order of magnitude.

Although we would like the factor $b_{I,2}$ to possess as much smoothness as possible, the pointwise constraint that $< \nab \xi_I, b_{I,2} > = 0$ forces $b_{I,2}$ to be develop some degree of irregularity.  That is, even though $b_{I,2}$ will be qualitatively smooth, its derivatives will grow as its corresponding phase function is transported.  Based on this observation, we first choose a vector field ${\mathring b}_I$ with size $\approx 1$ which satisfies $< \nab \xi_I, {\mathring b}_I > = 0$, and then write
\begin{align} \label{eq:ampExpand}
 b_{I,1}^l &= e^{1/2} b_{I,2}^l = e^{1/2} \ga_I {\mathring b}_I^l = \rho_I {\mathring b}_I^l
\end{align}
where the coefficients $\ga_I$ and $\rho_I = e^{1/2} \ga_I$ are determined by solving the linear system (\ref{eq:sqrtAlmost2}).

\paragraph{Constructing a section of $\langle \nab \xi \rangle^\perp$ and the gauge freedom}

We now wish to choose a vector field ${\mathring b}_I$ which satisfies $< \nab \xi_I, {\mathring b}_I > = 0$ pointwise.  We first remark that there is a huge amount of freedom in making such a choice.  One can think of $\langle \nab \xi_I \rangle^\perp$ as a bundle of oriented planes on which the linear map ${\bf J}_I = \fr{\nab \xi_I}{|\nab \xi_I|} \times$ acts like a complex structure, as it satisfies the equation ${\bf J}_I^2 = -1$ on $\langle \nab \xi \rangle^\perp$.  Then given any section ${\mathring b}_I$ of $\langle \nab \xi_I \rangle^\perp$ (that is, any vector field satisfying $< \nab \xi_I, {\mathring b}_I > = 0$ pointwise), and any function $\th(t,x)$, one can obtain another section ${\mathring b}_I'$ by rotating at each point ${\mathring b}_I' = e^{{\bf J}_I \th} {\mathring b}_I$ within $\langle \nab \xi_I \rangle^\perp$ by an oriented angle $\th$.  Furthermore, because the transormation $e^{{\bf J}_I \th}$ preserves absolute values, any solution to the pointwise quadratic equation (\ref{eq:sqrtAlmost2}) remains a solution to (\ref{eq:sqrtAlmost2}).  This type of freedom can be regarded as a ``gauge freedom'', and the group of operators
\[ \{ e^{{\bf J}_I \th} \} \]
would then be the considered the ``gauge group''.  In fact, the oscillations we produce while choosing $\la$ to be very large are just specific examples of this gauge freedom with the function $\th(t,x) = \la \xi_I(t,x)$.

Here, we do not want to use the gauge group to produce oscillations, but rather we want a specific section ${\mathring b}_I$ of $\langle \nab \xi_I \rangle^\perp$ which is as smooth as possible.  To be compatible with the transport of phases, we find it convenient to construct ${\mathring b}_I$ with an orthogonal projection
\begin{align}
 {\mathring b}_I^l &= \pr^l \xi_{\si I} - \fr{(\nab \xi_{\si I} \cdot \nab \xi_I)}{|\nab \xi_I|^2} \pr^l \xi_I \\
 &= \PP_I^\perp(\nab \xi_{\si I})^l
\end{align}

The phase function of index $\si I = (k, \si f)$ inhabits the same coordinate patch as that of $I = (k, f)$, but it points in a different direction, so that we can guarantee that, at times sufficiently close to $t(I)$, we have
\begin{align} \label{eq:projLowBound}
|{\mathring b}_I| &\geq c
\end{align}
satisfied for some constant $c > 0$, uniformly on $\T^3$.  In future formulas, we will express the absolute value of ${\mathring b}_I$ as
\begin{align} \label{eq:wedSq}
|{\mathring b}_I|^2 = \fr{|\nab \xi_I \wed \nab \xi_{\si I}|^2}{|\nab \xi_I|^2}
\end{align}
because $|\nab \xi_I|^2 |{\mathring b}_I|^2 = |\nab \xi_I|^2 |\nab \xi_{\si I}|^2 - ( \nab \xi_I \cdot \nab \xi_{\si I})^2 $ is actually the square norm of $\nab \xi_I \wed \nab \xi_{\si I}$ in $\La^2(\R^3)$.

One way to choose the index $\si I$ is to let $\si$ act as an element of order $3$ in the icosahedral group.  That is, let $\si$ act on a dodecahedron as a rotation by an angle $2 \pi / 3$ around an axis which goes through two opposite vertices of the dodecahedron.  It is easy to see that a rotation defined in this way satisfies $\si^3 = 1$.\footnote{The icosahedral group can be identified with $A_5 \times (\Z/(2 \Z))$ by considering its action on the compound of $5$ cubes (or the compound of 5 tetrahedra) inscribed in the dodecahedron.  The group generated by $\si$ is therefore a $3$ Sylow subgroup of the icosahedral group, which is unique up to conjugacy by Sylow's theorem, so in fact all elements of order three act as rotations in this way.}   With this choice of $\si$, we can then be assured that $\si f \neq f$ and $\si f \neq -f$, since the eigenvalue $1$ is only obtained on the axis of rotation $\langle f + \si f + \si^2 f \rangle$, and on its orthogonal complement $\si$ satisfies $\si^2 + \si + 1 = 0$, which has neither $1$ nor $-1$ as a root.  With this choice, we can guarantee the property (\ref{eq:projLowBound}) holds initially at time $t = t(I)$, and this property will remain true for a short time with a possibly smaller $c$.

Making the above choice ensures that $v_{\bar{I}} = {\overline v_I}$ and is ultimately why we chose to prescribe the imaginary part $b_I$ in the first place rather than the real part $a_I$.  Namely, since $\nab \xi_{\bar{I}} = - \nab \xi_I$, we have that $b_{\bar{I}} = - b_I$, and $a_{\bar{I}} = -\fr{( \nab \xi_{\bar{I}} )}{|\nab \xi_{\bar{I}}|} \times b_{\bar{I}} = a_I$.

\paragraph{The Stress Equation relative to a Frame}

The correction $V$ has now been written down completely except for the choice of several parameters, but we are not ready to proceed with the construction until it is clear how to estimate the derivatives of the correction.  In particular, the equation (\ref{eq:sqrtAlmost2}) involves taking a square root to solve for the unknown coefficients, so we must establish positive lower bounds on the solutions of the linear equation in order to proceed.  Using the expansion (\ref{eq:ampExpand}), we can rewrite (\ref{eq:sqrtAlmost2}) as
\begin{align*}
 \sum_{I \in k \times \F} \rho_I^2 |{\mathring b}_I|^2 ( \de^{jl} - \fr{\pr^j \xi_I \pr^l \xi_I }{|\nab \xi_I|^2} ) &= \fr{G^{jl}}{2}
\end{align*}
which can also be written
\begin{align} \label{eq:sqrtAlmost3}
 \sum_{I \in k \times \F} \rho_I^2 \fr{|\nab \xi_I \wed \nab \xi_{\si I}|^2}{|\nab \xi_I|^2} ( \de^{jl} - \fr{\pr^j \xi_I \pr^l \xi_I}{|\nab \xi_I|^2} ) &= \fr{G^{jl}}{2}
\end{align}
using the identity (\ref{eq:wedSq}).

To get more explicit control of $\rho_I$, it helps to express this tensorial equation relative to a basis.  The most natural basis available in this regard is the set of squares $\nab \xi_J \otimes \nab \xi_J$, because they are compatible the underlying transport.  In this spirit, we apply both sides of the equation (\ref{eq:sqrtAlmost3}) to the square $\pr_j \xi_J \pr_l \xi_J$ with $J \in k \times \F$, to obtain a $6 \times 6$ system of linear equations in matrix form
\begin{align} \label{eq:sqrtReal1}
 \sum_{I \in k \times \F} A(\nab \xi)_J^I \rho_I^2 &= \fr{G^{jl}\pr_j \xi_J \pr_l \xi_J}{2} \quad J \in k \times \F
\end{align}
Here the matrix $A(\nab \xi)_J^I = A(\nab \xi, \La^2(\nab \xi))_J^I$ is
\begin{align} \label{eq:theMatrx}
A(\nab \xi)_J^I &= \fr{|\nab \xi_I \wed \nab \xi_{\si I}|^2}{|\nab \xi_I|^2} \cdot \fr{|\nab \xi_I \wed \nab \xi_J|^2}{|\nab \xi_I|^2} \quad I, J \in k \times \F
\end{align}
We need to prove that this linear equation can be solved, and also that the coefficients which arise are manifestly positive.  At the initial time $t = t(I)$, we have the following lemma.
\begin{lem} \label{lem:canInvert}
 The matrix $A_J^I = A(\nab {\hat \xi})_J^I : \R^{\F} \to \R^{\F}$ at the initial time $t(I)$ does not depend on $k$ and is invertible.
\end{lem}

It is clear by the construction of the data $\hat{\xi}_I$ by rotations and the formula (\ref{eq:theMatrx}) that the initial matrix $A_J^I = A(\nab {\hat \xi})_J^I$ is independent of $k$.  One can check invertibility by hand using the coefficients of the matrix $A(\nab {\hat \xi})_J^I$; in fact, one can show that all the entries in this matrix have the same value, except for the diagonal entries, which are all $0$.  Here we outline a slightly slower proof based on an analysis of the symmetries of the dodecahedron.  We include the details in the appendix Part (\ref{part:Appendix}).

Given that the numbers $\fr{|\nab \xi_I \wed \nab \xi_{\si I}|^2}{|\nab \xi_I|^2}$ are not zero, the matrix $A_J^I$ is invertible if and only if both of the following statements are true.
\begin{lem} \label{lem:basisLem}

$ $

\begin{enumerate}
 \item The tensors $\{ \pr^j \xi_I \pr^l \xi_I = f(I)^j f(I)^l ~|~ I \in \F \}$ form a basis for $\SS$ \label{en:basis1}
 \item The tensors $\{ \de^{jl} - \fr{\pr^j \xi_I \pr^l \xi_I}{|\nab \xi_I|^2} = \de^{jl} - f^j f^l~|~ I \in \F \}$ form a basis for $\SS$ \label{en:basis2}
\end{enumerate}
\end{lem}
These conditions are equivalent to the invertibility of $A_J^I$ because the matrix coefficients of $A_J^I$ are computed by taking inner products of the latter set of vectors with the former.

It is possible to reduce Claim (\ref{en:basis2}) of Lemma (\ref{lem:basisLem}) to Claim (\ref{en:basis1}) of Lemma (\ref{lem:basisLem}) by using the following formula
\begin{lem}  
\begin{align} \label{eq:icoMetric}
\de^{jl} = \fr{1}{2} \sum_{f \in \F} f^j f^l
\end{align}
\end{lem}
\noindent to write down an invertible transformation taking one basis to the other.  In fact, this form is, up to a constant, the unique bilinear form in $\SS$ invariant under the action of the icosahedral group.  This uniqueness property and Formula (\ref{eq:icoMetric}) both follow easily if we have already proven Claim (\ref{en:basis1}) of Lemma (\ref{lem:basisLem}).  Indeed, if $G^{jl}$ is any form invariant under the icosahedral group, $G^{jl}f_j f_l = C$ for just one $f \in \F$, then by rotating we have \[ G^{jl}f_j f_l = C \de^{jl}f_jf_l \] for all $f \in \F$.  Then we can determine $C = \fr{G^{jl}\de_{jl}}{3}$, which gives the identity (\ref{eq:icoMetric}) when we take $G^{jl}$ to be the form on the right hand side.

In our approach, we will use identity (\ref{eq:icoMetric}), which encodes the basic geometric properties of the dodecahedron, as a first step to prove the linear independence (\ref{en:basis1}).  This identity is certainly very well known, but for the benefit of the reader and for completeness of the argument, we include a proof in the appendix, where the Lemma (\ref{lem:basisLem}) is also proven.

\paragraph{Remark about exterior powers} 

It is interesting that the second exterior power appears in the formulas here for several reasons.  First, note that the second exterior power can be implicated in explaining why the approximation $V^j V^l + P_0 \de^{jl} + R^{jl} \approx 0$ can only be achieved in a weak topology.  Namely, it is not possible to approximate an arbitrary tensor field $G^{jl}$ by rank one tensors $V^j V^l$ in a strong topology because the form $\La^2 G : \La^2(\R^3)^* \times \La^2(\R^3)^*$ which acts according to the formula
\[ \La^2 G(u_1 \wed u_2, w_1\wed w_2) = \det \mat{cc}{ G(u_1, w_1) & G(u_1, w_2) \\ G(u_2, w_1) & G(u_2, w_2) } \]
cannot be approximated by $\La^2 V \otimes V = 0$, implying that almost everywhere convergence of $V^j V^l \to G^{jl}$ is not possible.  On the other hand, because the function $G \mapsto \La^2 G$ is nonlinear, it fails to be continuous with respect to weak convergence, and a weak approximation $V^j V^l \approx G^{jl}$ is possible when $G^{jl}$ is non-negative definite.    

The second exterior power also plays a role in our ability to control the High-High term with Beltrami flows.  There, the key identity is 
\[ \ep^{abc}\ep_{cfg} = \de^a_f \de^b_g - \de^a_g \de^b_f \]
which provides two different formulas for the inner product on $\La^2(\R^3)$.

Finally, the fact that a weak limit of solutions to Euler may fail to be an Euler flow is made possible by the fact that a sequence of rank one tensors $v^j v^l$ (which are characterized among non-negative, symmetric tensors by the equation $\La^2 G = 0$) may have a weak limit which fails to be rank one.  Thus, the success of our convex integration scheme, which implies that weak limits of Euler flows fail to be solutions, is tied in several ways to the second exterior power.

\subsubsection{The Renormalized Stress Equation in Scalar Form} \label{sec:lowBounds}

From the above discussion, we have established that it is possible to solve the linear system
\begin{align}
A^I_J x_I &= y_J \quad \quad I, J \in \F
\end{align}
when the matrix $A^I_J = A(\nab \xi)^I_J = A(\nab {\hat \xi})^I_J : \R^\F \to \R^\F$ is given by its initial value according to (\ref{eq:theMatrx}).

In general, we have to solve the nonlinear equation
\begin{align} \label{eq:solveForRho}
 \sum_{I \in k \times \F} A(\nab \xi)_J^I \rho_I^2 &= \fr{1}{2} (e \fr{\de^{jl}}{n} - {\mathring R}_\ep^{jl} ) \pr_j \xi_J \pr_l \xi_J \quad \quad J \in k \times \F
\end{align}

Moreover, in order to control the derivatives of $\rho_I$, we will also have to differentiate the equation (\ref{eq:solveForRho}), which implies that we must obtain lower bounds on the coefficients $\rho_I$ solving (\ref{eq:solveForRho}).

To see that (\ref{eq:solveForRho}) admits positive solutions $\rho_I^2$, we first rescale the unknowns $\rho_I = e^{1/2} \ga_I$ and the right hand side $G^{jl} = e ( \fr{\de^{jl}}{2n} + \varepsilon(R_\ep)^{jl} )$, to renormalize equation (\ref{eq:solveForRho}) into
\begin{align} \label{eq:solveForGa1}
 \sum_{I \in k \times \F} A(\nab \xi)_J^I \ga_I^2 &= (\fr{\de^{jl}}{2n} + \varepsilon(R_\ep)^{jl}) \pr_j \xi_J \pr_l \xi_J 
\end{align}
This equation is equivalent to the tensorial form
\begin{align} \label{eq:solveForGaTensor}
 \sum_{I \in k \times \F} \ga_I^2 \fr{|\nab \xi_I \wed \nab \xi_{\si I}|^2}{|\nab \xi_I|^2} ( \de^{jl} - \fr{\pr^j \xi_I \pr^l \xi_I}{|\nab \xi_I|^2} ) &= \fr{\de^{jl}}{2n} + \varepsilon(R_\ep)^{jl}
\end{align}

In order to control solutions to (\ref{eq:solveForGaTensor}), we view the equation as a perturbation of the system
\begin{align} \label{eq:approxGaTensor}
 \sum_{I \in k \times \F} {\tilde \ga}_I^2 \fr{|\nab {\hat \xi}_I \wed \nab {\hat \xi}_{\si I}|^2}{|\nab {\hat \xi}_I|^2} ( \de^{jl} - \fr{\pr^j {\hat \xi}_I \pr^l {\hat \xi}_I}{|\nab {\hat \xi}_I|^2} ) &= \fr{\de^{jl}}{2n}
\end{align}

The system (\ref{eq:approxGaTensor}) admits positive solutions ${\tilde \ga}_I^2$.  In fact, as we demonstrate in the appendix, the formula (\ref{eq:icoMetric}) together with a symmetry argument implies the identities
\begin{align}
|\nab {\hat \xi}_{(k,f)}|^2 &= |f|^2 = 1 \\
|\nab {\hat \xi}_{(k,f)} \wed \nab {\hat \xi}_{(k,\si f)} |^2 &= |f|^2 |\si f|^2 - (f \cdot \si f)^2 = \fr{4}{5}
\end{align}

Using these identities, we see that the equation (\ref{eq:approxGaTensor}), along with its equivalent scalar form
\begin{align} \label{eq:approxGaScalar}
\begin{split}
 \sum_{I \in k \times \F} A(\nab {\hat \xi})_J^I {\tilde \ga}_I^2 &= (\fr{\de^{jl}}{2n})\pr_j {\hat \xi}_J \pr_l {\hat \xi}_J \\
 &=  \fr{|\nab {\hat \xi}_J|^2}{6} = 1/6
\end{split}
\end{align}
is solved by any set of ${\tilde \ga}_I$ satisfying 
\begin{align}
{\tilde \ga}_I^2 &= \fr{5}{2^5 \cdot 3}
\end{align}
In particular, each component of the solution ${\tilde \ga}_I = \sqrt{\fr{5}{2^5 \cdot 3}}$ is nonzero.

We hope to treat the equation (\ref{eq:solveForGa1}) as a perturbation of the equation (\ref{eq:approxGaScalar}).  The first step in this direction is to show that equation (\ref{eq:approxGaScalar}) can be perturbed to one which still can be solved, and whose solutions remain bounded from $0$.  This preparatory step is accomplished by the following lemma:

\begin{lem}[The Implicit Function Lemma] \label{lem:impFuncLem}
There is a positive $c > 0$ and $a < 1$ such that the whenever
\begin{align} \label{eq:perturbDomain}
|| A^I_J - A(\nab {\hat \xi})^I_J || &\leq c \\
|| y_J - \fr{|\nab {\hat \xi}_J|^2}{6} || &\leq c
\end{align}
the equation
\begin{align} \label{eq:perturbEq}
\sum_I A^I_J \ga_I^2 &= y_J
\end{align}
has a unique solution in the ball $| \ga_I - {\tilde \ga}_I | \leq a |\tilde{\ga}_I|$.

The solution $\ga_I$ is represented by a function 
\begin{align}
\ga_I &= \ga_f(A, y) 
\end{align}
depends smoothly on $A^I_J$ and $y_J$ in the domain (\ref{eq:perturbDomain}) where the function $\ga_f$ is determined by the direction coordinate $f(I)$.  

Furthermore, the matrix $B^I_J = 2 A^I_J \ga_I$ obtained by linearizing the equation is invertible in this domain with uniform bounds
\begin{align}
|| A^{-1} || &\leq C \\
|| (\ga_I)^{-1} || &\leq C
\end{align}
\end{lem}
{\bf Remark.}  It is important to observe that the constant $c > 0$ is completely independent of the location $k$ of the phase directions.  This independence follows from the fact that the matrix $A(\nab {\hat \xi})_J^I$ and the right hand side $\fr{|{\hat \xi}_J|^2}{6}$ of the equation (\ref{eq:approxGaScalar}) which we perturb are independent of the location. 
\begin{proof}
By writing the equation (\ref{eq:perturbEq}) in the form
\begin{align}
F(\ga, A, y) &= 0 \\
F(\ga, A, y)_J &\equiv \sum_I A^I_J \ga_I^2 - y_J
\end{align}
we can deduce the lemma as an immediate consequence of the implicit function theorem once we verify that the matrix
\begin{align}
\pd{F_J}{\ga_I}(\ga, A, y) &= (A^I_J)(2 \ga_I) \\
\pd{F_J}{\ga_I} &: \R^\F \to \R^\F
\end{align}
is invertible at the point $(\ga_I, A^I_J, y_J) = ({\tilde \ga}_I, A(\nab {\hat \xi})^I_J, |\nab {\hat \xi}_J|^2/6)$.

To prove that the above matrix is one-to-one, it suffices to show that its null space is zero.  If $h_I$ is any element in the null space at this point, then using the fact that $A(\nab {\hat \xi})^I_J$ is invertible and that all the ${\tilde \ga}_I$ are nonzero, we deduce that
\begin{align}
\pd{F_J}{\ga_I}h_I = 2 A(\nab {\hat \xi})^I_J {\tilde \ga}_I h_I &= 0  \quad \quad \forall J \in \F\\
\Rightarrow {\tilde \ga}_I h_I &= 0 \quad \quad \forall I \in \F \\
\Rightarrow h_I &= 0 \quad \quad \forall I \in \F
\end{align}

The implicit function theorem then guarantees the existence of $c$ and $a$ as in the lemma, because it also guarantees that the matrix $\pd{F_J}{\ga_I}$ remains invertible in the domain (\ref{eq:perturbDomain}).
\end{proof}

The lemma (\ref{lem:impFuncLem}) applies to equation (\ref{eq:solveForGa1}) with
\begin{align}
 A_J^I &= A(\nab \xi)_J^I \\
 y_J &= (\fr{\de^{jl}}{2n} + \varepsilon(R_\ep)^{jl}) \pr_j \xi_J \pr_l \xi_J  \\
 &= \fr{|\nab {\hat \xi}|^2}{6} + \fr{1}{6} (|\nab \xi|^2 - |\nab {\hat \xi}|^2 ) + \varepsilon^{jl} \pr_j \xi_J \pr_l \xi_J 
\end{align}
provided that $K$ is sufficiently large and that $\tau$ is sufficiently small.

At this point, we can choose $K$ once and for all.  Let $c$ be the positive constant guaranteed by the Lemma (\ref{lem:impFuncLem}), and we pick \begin{align}
 K &= 7000 c^{-1} \label{eq:wePickedK}
 \end{align}
 so that
\begin{align}\label{eq:varepPerturb}
|\varepsilon| &\leq \fr{100}{K} \leq \fr{c}{70}
\end{align}

Having chosen $K$, we can prove the following theorem which restricts how small $\tau$ must be chosen
\begin{prop}\label{prop:needTauSmall}

  There exist constants $c > 0$ and $a > 0$ with $a < 1$ such that the equation (\ref{eq:solveForGa1}) admits a unique solution $\ga_I \in \R^\F$ in a neighborhood $\{ | \ga_I -  {\tilde \ga}_I | \leq a |{\tilde \ga}_I| \}$ provided that
\begin{align} \label{eq:tPerturbConditions}
|~|\nab \xi_I|^2 - |\nab {\hat \xi}_I|^2~| &\leq c \quad \quad \forall I \in F \\
|~|\nab(\xi_I + \xi_J)|^2 - |\nab ( {\hat \xi}_I + {\hat \xi}_J )|^2~| &\leq c \quad \quad \forall I, J \in F
\end{align}

The solution $\ga_I( \nab \xi_I, \varepsilon)$ depends smoothly on $\nab \xi_I \in \F$ and $\varepsilon$ for $\nab \xi_I$ in the domain (\ref{eq:tPerturbConditions}) and $\varepsilon$ in the domain (\ref{eq:varepPerturb}).
\end{prop}
\begin{proof}
In the statement of the proposition, we have chosen to assume control over the deviation of the gradient energies 
\[ |~|\nab(\xi_I + \xi_J)|^2 - |\nab ( {\hat \xi}_I + {\hat \xi}_J )|^2~| \]
rather than the matrices $A(\nab \xi) - A(\nab {\hat \xi})$ because the entries of the matrix $A(\nab \xi)^I_J$ depend only on the inner products $\nab \xi_I \cdot \nab \xi_J$, which themselves can be calculated from the polarization identity 
\begin{align}
 |\nab (\xi_I + \xi_J)|^2 - |\nab (\xi_I - \xi_J)|^2~ &= |\nab (\xi_I + \xi_J)|^2 - |\nab (\xi_I + \xi_{\bar{J}})|^2 \\
 &= 2 \nab \xi_I \cdot \nab \xi_J 
\end{align}

The proposition then immediately follows from Lemma (\ref{lem:impFuncLem}) once we observe the bound
\begin{align*}
| y_J - \fr{| \nab {\hat \xi}_J|^2}{6} | &\leq \fr{|~ |\nab \xi_J|^2 - |\nab {\hat \xi}_J|^2 ~|}{6} + |\varepsilon| |\nab \xi_J|^2 \\
&\leq \fr{|~ |\nab \xi_J|^2 - |\nab {\hat \xi}_J|^2 ~|}{6} + \fr{c}{70} |\nab \xi_J|^2 \\
&\leq \fr{c}{70} |\nab {\hat \xi}_J|^2 + (\fr{1}{6} + \fr{c}{70}) |~ |\nab \xi_J|^2 - |\nab {\hat \xi}_J|^2 ~| \\
&= \fr{c}{70} + A |~ |\nab \xi_J|^2 - |\nab {\hat \xi}_J|^2 ~|
\end{align*}

\end{proof}
Thanks to the above proposition, we are allowed to differentiate the equation (\ref{eq:sqrtReal1}) in order to obtain upper bounds on the derivatives of $\rho_I$.

In carrying out the construction, we always require that $\tau$ is chosen small enough so that the Proposition (\ref{prop:needTauSmall}) applies.  Besides the requirement of Proposition (\ref{prop:needTauSmall}), we also require that the gradients $| \nab (\xi_I + \xi_J) | \geq c > 0$ remain bounded from $0$ by a constant.  

\subsubsection{Summary} \label{sec:Summary}

To summarize the results of the preceding sections, we have established the following:

\begin{prop}
There exist absolute constants $c > 0$ and $K > 0$ such that as long as 
\begin{align}
\co{ \nab \xi_I - \nab {\hat \xi}_I } &\leq c \quad \quad \forall I \in k \times F
\end{align}
and
\begin{align}
e(t) &\geq K |R_\ep(t,x)|
\end{align}
pointwise, then the local form (\ref{eq:sqrtAlmost}) of the Stress Equation (\ref{eq:theStressEq}) admits a unique solution of the form
\begin{align}
v_I &= a_I + i b_I \\
I &= (k,f) = (k_1, k_2, k_3, k_4, f) \\
b_I^l &= \eta\left(\fr{t - t(I)}{\tau} \right) \psi_k \ga_I \PP_I^\perp(\nab \xi_{\si I})^l \\
a_I &= - \fr{(\nab \xi_I)}{|\nab \xi_I|} \times b_I
\end{align}
where 
\begin{align}
\PP_I^\perp(\nab \xi_{\si I})^l &= \pr^l \xi_{\si I} - \fr{(\nab \xi_{\si I} \cdot \nab \xi_I)}{|\nab \xi_I|^2} \pr^l \xi_I
\end{align}
is the orthogonal projection of the phase gradient $\nab \xi_{\si I}$ into $\langle \nab \xi_I \rangle^\perp$ and the coefficients
\begin{align}
\ga_I &= \ga_I(\nab \xi_k, \varepsilon) \\
&= \ga_f(\nab \xi_k, \varepsilon ) \label{eq:gaIisgaFof}
\end{align}
depend on the phase gradients in region $k$ and the tensor
\begin{align}
\varepsilon^{jl} &= - \fr{{\mathring R}_\ep}{e(t)}
\end{align}
The smooth function $\ga_f$ in (\ref{eq:gaIisgaFof}) is one of $6$ smooth functions indexed by the direction coordinate $f(I) \in \F$ of $I = (k, f)$, which is independent of the location coordinate $k \in (\Z / 2 \Z)^3 \times \Z$.
\end{prop}

For reference, the functions $\eta$ and $\psi_k$ are elements of partitions of unity which are explained in Section (\ref{sec:localizeStress}).  The phase functions $\xi_I$ are only defined on only a portion $\Om_I$ of $\R \times \T^3$ defined in (\ref{subsec:indexPhase}) on which they can live without critical points, and they satisfy the transport equation 
\begin{align*}
(\pr_t + v_\ep^j \pr_j) \xi_I &= 0 \\
\xi_I(t(I), x) &= {\hat \xi}_I(x)
\end{align*}
for initial data ${\hat \xi}_I$ at time $t(I) = \tau k_4$ specified in Section (\ref{sec:thePhaseFunctions}).


%% file: constructCts.tex
We will now show how the preceding construction, combined with a few estimates from Part (\ref{hardPart}), can be used to prove Lemma (\ref{lem:mainLemCts}).

\begin{proof}[Proof of Lemma (\ref{lem:mainLemCts})]
Let $\ep > 0$ and $(v, p, R)$ be given, and let $e(t)$ be the function specified by the lemma.

We define $K$ to be the constant chosen in Line (\ref{eq:wePickedK}).

We describe the argument as a sequence of $6$ steps.

\subsection{Step 1: Mollifying the velocity} \label{sec:ctsMollV}

The error term generated by mollifying $v$ has the form
\begin{align}
Q_v^{jl} &= (v^j - v_\ep^j) V^l + V^j (v^l - v_\ep^l) \\
&= 2 [( v - v_\ep) V]^{jl}
\end{align}
where $V$ is a {\bf locally finite} sum
\begin{align}
V &= \sum_I V_I \\
V_I &= e^{i \la \xi_I} (v_I + \fr{\nab\times w_I}{\la} )
\end{align}
whose cardinality at each point $(t,x)$ is bounded by an absolute constant.  Namely, at each time $t$ there are at most two generations $k_4$ for which $V_I$ are nonzero, and for each generation $k_4$, there are at most $8$ families $(k_1, k_2, k_3, k_4)$, each having at most $|F| = 2 \cdot 6$ phase functions, which are nonzero.

We choose $\la$ at the end of the argument in order to shrink the term
\begin{align}
Q_{v, (ii)}^{jl} &= \sum_I 2 e^{i \la \xi_I} \fr{[(v - v_\ep)\nab\times w_I]^{jl}}{\la} \label{eq:waitingForLa}
\end{align}
since the bounds for $\nab \times w_I$ depend on $v_\ep$.

For $v_\ep$, we choose a space-time mollification $v_\ep = \eta_{\ep_{t,x}} \ast v$ close enough so that we can guarantee that for the main term
\begin{align}
Q_{v,(i)} &= \sum_I 2[(v - v_\ep)(e^{i \la \xi_I} v_I )]
\end{align}
we have
\begin{align}
\co{ Q_{v,(i)} } &\leq \co{ (v - v_\ep)} \co{\sum_I v_I } \\
&\unlhd \fr{\ep}{40} \label{eq:weCanDoWithVMoll}
\end{align}

Logically speaking, the term $v_I$ is not defined until $v_\ep$ has been constructed, however, we can give a priori bounds for its size since we know that it will have the form 
\[ v_I = a_I + i b_I, \] where
\begin{align}
b_I^l &= \eta\left(\fr{t - t(I)}{\tau}\right) \psi_k(t,x) e^{1/2}(t) \ga_I(\nab \xi_k, \fr{R_\ep}{e}) P_I^\perp(\nab \xi_{\si I})^l \label{eq:bIlooksLikeCts}
\end{align}
and $a_I = - \fr{(\nab \xi_I)\times}{|\nab \xi_I|} b_I $ is obtained by a $\pi/2$ rotation of $b_I$ in $\langle \nab \xi_I \rangle^\perp$.

Until $v_\ep$ is constructed, no bounds for derivatives of $v_I$ can be stated, but it is clear that each $v_I$ {\bf will} be bounded in absolute value by 
\begin{align}
\co{v_I } &\unlhd 2000 K^{1/2} e_R^{1/2} \label{eq:willHaveForvI}
\end{align}
after it is constructed, as we have assumed an upper bound (\ref{eq:upBdet}) for $e^{1/2}(t)$.

Therefore the locally finite sum
\begin{align}
\co{\sum_ I e^{i \la \xi_I} v_I } &\unlhd A e_R^{1/2}
\end{align}
will also be bounded, a priori, since the cardinality of the terms which are nonzero at any given $(t,x)$ is also bounded.

This estimate can only be confirmed when the lifespan parameter $\tau$ is chosen, but anticipating the above estimate, we can choose the mollifier $\eta_{\ep_{t,x}}$ for $v$ so that (\ref{eq:weCanDoWithVMoll}) will be a consequence of (\ref{eq:willHaveForvI}) once (\ref{eq:willHaveForvI}) has been proven.

\subsection{Step 2: Mollifying the Stress}

Let $R_\ep = \eta_{\ep_{t,x}} \ast R$ be a mollification of $R$ in the variables $(t, x)$ so that
\begin{itemize}
\item $\mbox{ supp } R_\ep \subseteq I' \times \T^3$ where $I'$ is an interval on which $e(t) \geq K e_R$.  
\item \[ \co{ R - R_\ep } \leq \fr{\ep}{100} \]
\end{itemize} 

\subsection{Step 3: Choosing the Lifespan}

Since we have assumed that the velocity $v$ is uniform bounded, we can set
\[\th^{-1} \geq \ep_v^{-1} \co{v} \] 
be an upper bound\footnote{If we remove the assumption of uniform boundedness, it is still possible to estimate the derivatives of $v_\ep$ in terms of only the modulus of continuity of $v$.  This task can be accomplished by using the expression
\begin{align*}
 \pr_i \eta_\ep \ast v^l &= \int v^l(x + h) \pr_i \eta_\ep(h) dh \\
&= \int ( v^l(x+h) - v^l(x)) \pr_i \eta_\ep(h) dh
\end{align*}
Later on during the construction of H{\" o}lder continuous solutions, the choice of $\tau$ will turn out to be completely independent of $\co{v}$.  }
 for $\co{ \nab v_\ep }$.  Then set
\begin{align}
\tau &= b \th
\end{align}
for some parameter $b$.  According to the bounds for transport equations proven in Section (\ref{transportEstimates}), we have that
\begin{align}
\co{ \nab \xi_I - \nab {\hat \xi}_I } &\leq A b
\end{align}
for some constant $A$.  We will therefore choose $b$ small enough, so that, first of all, the conditions in Line (\ref{eq:tPerturbConditions}) are guaranteed.  Our second goal for choosing a small parameter $\tau$ is to ensure that the first term in the parametrix for the High-High term is controlled.

Namely, we should be able to make the $C^0$ norm of
\begin{align}
{\tilde Q}_{H,(1)}^{jl} &= \sum_{J \neq {\bar I} } e^{i \la (\xi_I + \xi_J)} i^{-1} q_{H, IJ, (1)}^{jl}
\end{align}
smaller than
\begin{align}
\co{ {\tilde Q}_{H,(1)}^{jl} } &\unlhd \fr{\ep}{500}. \label{eq:getRidOfFirstHighTerm}
\end{align}
  Here, $q_{H, IJ, (1)}^{jl}$ is the solution to the linear equation
\begin{align}
\pr_j \xi_I q_{H, IJ, (1)}^{jl} &= \left( v_J \times ([ |\nab \xi_I| - 1 ] v_I )^l + v_I \times (  [ |\nab \xi_J| - 1] v_J )^l \right) \label{eq:firstPmtrxHigh}
\end{align}
described in Section (\ref{sec:theParametrix}).  Although the $v_I$ and $v_J$ are not defined until $\tau$ is chosen, as long as $\tau$ satisfies the conditions (\ref{eq:tPerturbConditions}), which we have already guaranteed, we know that 
\begin{align}
 \co{ v_J \times ([ |\nab \xi_I| - 1 ] v_I ) } &\leq A e_R \co{  |\nab \xi_I| - 1 }
\end{align}
which will be bounded by another constant
\begin{align}
\co{ v_J \times ([ |\nab \xi_I| - 1 ] v_I ) } &\leq A' e_R b 
\end{align}
according to the bounds in Section (\ref{transportEstimates}).

Therefore, it is possible to choose $b$ such that (\ref{eq:getRidOfFirstHighTerm}) will be guaranteed by the construction.  Note that $\tau$ only depends on the given $v$ and $e_R$.


\subsection{Step 4: Bounds for the new Stress}

All the remaining terms in the new stress $R_1$ are $O(1/\la)$ in $C^0$ once the above parameters have been chosen.

For example, the term (\ref{eq:waitingForLa}) and the Stress term
\begin{align}
Q_S^{jl} &= \sum_I V_I^j{\bar V}_I^l + P_0 \de^{jl} + R_\ep^{jl} \\
&= \sum_I \fr{(\nab \times w_I)^j{\bar v}_I^l }{\la} + \sum_I \fr{v_I^j(\nab \times {\bar w}_I)^l }{\la} \\
&+  \sum_I \fr{(\nab \times w_I)^j(\nab \times {\bar w}_I)^l }{\la^2}
\end{align}
is $O(\la^{-1})$ in $C^0$.

The remaining terms in the stress require us to solve the divergence equation.  They include, for example the High-Low term
\begin{align}
\pr_j Q_L^{jl} &= \sum_I e^{i \la \xi_I} {\tilde v}_I^j \pr_j v_\ep^l \\
&= e^{i \la \xi_I} u_L^l
\end{align}
Using the parametrix in Section (\ref{sec:theParametrix}), which requires taking spatial derivatives of the phase gradients.  In Section (\ref{transportEstimates}) we show that for $\tau$ as chosen above each derivative of $\nab \xi$ costs, essentially, a factor of $C \ep_v^{-1}$.  Therefore the vector field $u_L$, its derivatives, and the derivatives of the phase functions can be bounded uniformly in terms of $\ep_v$ and $v$.  We have also guaranteed bounds for $\co{ |\nab \xi_I|^{-1} }$ by the choice of $\tau$.

Applying the Lemma (\ref{prop:canSolveDivWithLa}) with sufficiently large $\la$, there exists a solution of size 
\begin{align}
\co{ Q_L } &\leq C_{v, R} \la^{-1}
\end{align}

The transport term 
\begin{align}
\pr_j Q_T^{jl} &= \sum_I e^{i \la \xi_I}[ (\pr_t + v_\ep^j \pr_j) {\tilde v}_I^l ]
\end{align}
likewise has a solution of size
\begin{align}
\co{ Q_T^{jl} } &\leq C_{v_\ep, R_\ep} \la^{-1} 
\end{align}
by Lemma (\ref{prop:canSolveDivWithLa}).

Finally, what remains of the High-High term after the first iteration of the parametrix must solve the equation
\begin{align}
\pr_j Q_{H, (2)}^{jl} &= \sum_{J \neq {\bar I} } -e^{i \la(\xi_I + \xi_J)} i^{-1} \pr_j q_{H, IJ, (1)}^{jl}
\end{align}
where $q_{H, IJ, (1)}^{jl}$ is defined as in (\ref{eq:firstPmtrxHigh}).  This equation also has a solution of size  
\begin{align}
\co{ Q_H } &\leq C_{v_\ep, R_\ep} \la^{-1}
\end{align}
by Lemma (\ref{prop:canSolveDivWithLa}), where here we apply our bounds on third derivatives $\nab^3 \xi$ from Section (\ref{transportEstimates}).

We demand that $\la$ be at least large enough so that the sum of the bounds for these solutions is smaller than $\fr{\ep}{3}$ in $C^0$.  The last task that remains is to show that the energy increment can be controlled precisely.

\subsection{Step 5: Bounds for the Corrections} 

Recall that 
\begin{align} 
 V_I &= e^{i \la \xi_I}( v_I + \fr{\nab \times w_I}{\la}) \label{eq:whatVIis} \\
 &= e^{i \la \xi_I}{\tilde v}_I
\end{align}
We have already observed that 
\begin{align}
\co{ v_I } &\leq A e_R^{1/2}
\end{align}
so to prove that $\co{V_I}$ also satisfies this bound (with a slightly larger constant), it suffices to choose $\la$ large enough.

Similarly, the correction to the pressure takes the form
\begin{align}
P &= P_0 + \sum_{J \neq {\bar I} } P_{I, J} \\
P_0 &= - \fr{e(t)}{3} + \fr{R_\ep^{jl}\de_{jl}}{3} \\
P_{I,J} &= - \fr{V_I \cdot V_J}{2} \label{eq:formForPIJcts} \\
&= - \fr{1}{2} e^{i \la (\xi_I + \xi_J) } {\tilde v}_I \cdot {\tilde v}_J \\
\end{align}
so that the when $\la$ is at least as large as above, we also have
\begin{align}
\co{ P } &\leq C e_R
\end{align}

Thus we have established the bounds for the corrections.

\subsection{Step 6: Control of the Energy Increment}

Consider the energy increment
\begin{align}
\int[ |v_1|^2(t,x) - |v|^2(t,x)]  dx &= \int |V|^2(t,x) dx + 2 \int V \cdot v dx 
\end{align}

For the second term, we expect that because $V$ is highly oscillatory, it can be made orthogonal to $v$.  To prove this fact, let $v_{S}$ by any smooth vector field such that
\begin{align}
C e_R^{1/2} \left\| \int | v - v_S|(t,x) dx \right\|_{C_t^0} &\leq \fr{\ep}{3} 
\end{align}
where $C e_R^{1/2}$ is the upper bound for $\| V \|_{C^0}$ established after equation (\ref{eq:whatVIis}).

Now we estimate the remaining term
\begin{align}
\int V \cdot v_S dx &= \int \nab \times W \cdot v_S dx \\
&= \int W \cdot \nab \times v_S dx \\
&= O(\la^{-1})
\end{align} 

It remains to bound
\begin{align}
\left\| \int |V|^2(t,x) dx - \int e(t) dx \right\|_{C^0_t} &\unlhd \fr{\ep}{3}
\end{align}

In order to achieve this approximation, we compute
\begin{align}
\int |V|^2(t,x) dx &= \int V^j V^l \de_{jl} dx \\
&= \int \sum_{I, J} V_I^j V_J^l \de_{jl} dx \\
&= \sum_I \int V_I^j {\bar V}_I^l \de_{jl} dx + \sum_{J \neq {\bar I} } \int V_I \cdot V_J dx \\
&=  \sum_I \int {\tilde v}_I^j {\overline {\tilde v}}_I^l \de_{jl} dx + \sum_{J \neq {\bar I} } \int e^{i \la (\xi_I + \xi_J)} {\tilde v}_I \cdot {\tilde v}_J dx \\
&= \int \left( \sum_I v_I^j {\bar v}_I^l \right) \de_{jl} dx + \sum_{J \neq {\bar I} } \int e^{i \la (\xi_I + \xi_J)} {\tilde v}_I \cdot {\tilde v}_J dx \\
\begin{split}
&+ \sum_I \int \fr{(\nab \times w_I)}{\la} \cdot{\bar v}_I dx  + \sum_I \int \fr{(\nab \times {\bar w}_I)}{\la} \cdot v_I dx \\
&+ \sum_I \int \fr{(\nab \times w_I)}{\la} \cdot \fr{(\nab \times w_I)}{\la} dx \end{split}\\
&= \int e(t) dx + \sum_{J \neq {\bar I} } \int e^{i \la (\xi_I + \xi_J)} {\tilde v}_I \cdot {\tilde v}_J dx + O(\la^{-1})
\end{align}
where the $O(\la^{-1})$ bound on the error term is uniform in time.

To bound the error term, we integrate by parts
\begin{align}
\int e^{i \la (\xi_I + \xi_J)} {\tilde v}_I \cdot {\tilde v}_J dx
&= \int \left[ \fr{\pr^a(\xi_I + \xi_J)}{\la |\nab(\xi_I + \xi_J)|^2} \pr_a[e^{i \la (\xi_I + \xi_J)}] \right] {\tilde v}_I \cdot {\tilde v}_J dx \\
&= \fr{-1}{\la} \int e^{i \la (\xi_I + \xi_J)} \pr_a \left[ \fr{\pr^a(\xi_I + \xi_J)}{|\nab(\xi_I + \xi_J)|^2} {\tilde v}_I \cdot {\tilde v}_J \right] dx 
\end{align}

Taking $\la$ large concludes the proof of Lemma (\ref{lem:mainLemCts}).
\end{proof}

%% file: freqEnLevels.tex
Now that we have constructed continuous solutions, let us discuss how to measure the H{\" o}lder regularity of the solutions we construct when the scheme is executed more carefully.  For this aspect of the convex integration scheme, our biggest contribution is a notion of {\bf frequency and energy levels}.  This notion is meant to accurately record the bounds which apply to the $(v, p, R)$ coming from the previous stage of the construction.  This notion is based on the paper \cite{deLSzeC1iso}, but differs from that paper in that here there is a gap between the estimates for $\nab v$ and $\nab R$ which must be recorded, and the definition below also records higher derivatives.

Here is an example of a candidate definition for frequency and energy levels, although it is not the one that will be used in our paper.

\begin{defn}[Practice Frequency Energy Levels] \label{def:subSolDefPr}
Let $L \geq 1$ be a fixed integer.  Let $\Xi \geq 2$, and let $e_v$ and $e_R$ be positive numbers with $e_R \leq e_v$.  Let $(v, p, R)$ be a solution to the Euler-Reynolds system.  We say that the practice frequency and energy levels of $(v,p,R)$ are below $(\Xi, e_v, e_R)$ (to order $L$ in $C^0 = C^0_{t,x}(\T^3 \times \R)$) if the following estimates hold.
\begin{align}
|| \nab^k v ||_{C^0} &\leq \Xi^k e_v^{1/2} &k = 1, \ldots, L \label{bound:nabkvPr} \\
|| \nab^k p ||_{C^0} &\leq \Xi^k e_v &k = 1, \ldots, L \label{bound:nabkpPr} \\
|| \nab^k R ||_{C^0} &\leq \Xi^k e_R &k = 0, \ldots, L \label{bound:nabkRPr}
\end{align}
Here $\nab$ refers only to derivatives in the spatial variables.
\end{defn}

The idea of these bounds is that at stage $k$, $e_{v,(k)}$ is the size ($e_{R,(k-1)}$) of the stress $R_{(k-1)}$ coming from the previous stage, and that $\Xi_{(k)}$ is essentially the parameter $\la_{(k - 1)}$ chosen in the previous stage.  Then the bounds (\ref{bound:nabkvPr}) are consistent with the heuristic representation
\[V_{(k-1)} \approx e^{i \Xi_{(k)} x } |R_{(k-1)}|^{1/2} \]
since $V_{(k-1)}$ gives the dominant contribution to the derivatives of $v_{(k)} = v_{(k-1)} + V_{(k-1)}$.

Based on a definition of frequency and energy levels such as Definition (\ref{def:subSolDefPr}) we can summarize the effect of one iteration of the convex integration procedure in a single lemma, which has roughly the following form
\begin{samLem} \label{samLem:practice}
For every solution $(v, p, R)$ with frequency and energy levels below
\[ (\Xi, e_v, e_R) \]
and every 
\begin{align}
 N &\geq \left( \fr{e_v}{e_R} \right) \label{eq:practNCond}
\end{align}
there exists a solution $(v_1, p_1, R_1)$ to the Euler-Reynolds equations with new frequency and energy levels below 
\begin{align}
(\Xi', e_v', e_R') &= ( N \Xi, e_R, e_R' ) 
\end{align}
such that $v_1 = v + V$, $p_1 = p + P$, and
\begin{align}
\co{ V } &\leq C e_R^{1/2} \label{boundForVNeqPr} \\
\co{\nab  V } &\leq C N \Xi e_R^{1/2} \label{boundForVNeqPr2} \\
\co{ P } &\leq C e_R \\
\co{ \nab P } &\leq C N \Xi e_R
\end{align}

\end{samLem}

The idea of the above outline is that 
\[ N \Xi \approx \la \] 
will be essentially the chosen value for $\la$.  The bounds (\ref{boundForVNeqPr}), (\ref{boundForVNeqPr2}) for the new energy and frequency levels then follow from the heuristic representation
\[ V \approx e^{i (N \Xi) x } |R|^{1/2} \]

  The efficiency of the convex integration process is measured by how small one can make $e_R'$ if one is only allowed to grow the frequency by a factor $N$.

The condition (\ref{eq:practNCond}) ensures, for example, that the largest term in 
\begin{align*}
\co{ \nab v_1 } &= \co{ \nab v }  + \co{ \nab V } \\
&\leq C ( \Xi e_v^{1/2} + N \Xi e^{1/2}_R ) \\
&\leq C N \Xi e_R^{1/2} 
\end{align*}
is the one coming from the high-frequency $V$, and similarly for $P$
\begin{align*}
\co{ \nab p_1 } &= \co{ \nab p }  + \co{ \nab P } \\
&\leq C ( \Xi e_v + N \Xi e_R ) \\
&\leq C N \Xi e_R 
\end{align*}
Of course, to be useful, the definition (\ref{def:subSolDefPr}) should be modified to contain information about time derivatives as well, and at the moment it is not clear that bounds for the pressure would be relevant given the passive role played by the pressure so far.

In terms of this rough notion of frequency and energy levels, let us see what kind of rate $e_R'$ we can expect for the reduced stress in the present problem.

First consider the High-Low interaction term
\begin{align}
\pr_j Q_L^{jl} &=\sum_I e^{i \la \xi_I} {\tilde v}_I^j \pr_j v_\ep^l
\end{align}
According to Definition (\ref{def:subSolDefPr}) and the expected bound $\co{ {\tilde v}_I } \leq A e_R^{1/2}$, the right hand side has size
\begin{align}
\co{ {\tilde v}_I^j \pr_j v_\ep^l } &\leq e_R^{1/2} ( \Xi e_v^{1/2} )
\end{align}
On the other hand, the parametrix gains a factor $\la^{-1} \approx N^{-1} \Xi^{-1}$, so we expect a solution
\begin{align}
\co{ Q_L } &\leq \fr{e_v^{1/2} e_R^{1/2}}{N} 
\end{align}
If we could actually prove a theorem such as Sample Lemma (\ref{samLem:practice}) with 
\begin{align}
 e_R' &= \fr{e_v^{1/2} e_R^{1/2} }{N} \label{eq:idealEr} 
\end{align}
then our construction could achieve solutions in the class
\begin{align} 
v \in C^{1/3 - \ep} \\
p \in C^{2(1/3 - \ep)} 
\end{align} 
which have the regularity conjectured by Onsager.  To apply the lemma, we choose $N_{(k)}$ at each stage so that the energy levels decay just slightly super-exponentially with a self-similarity Ansatz
\begin{align}
e_{R,(k+1)} &= \fr{e_{R,(k)}^{1 + \de} }{Z} \label{eq:theAnsatz}\\
Z &= \mbox{ constant } 
\end{align}
Letting $\de \to 0$ causes the gaps in the frequency spectrum between consecutive stages of the construction to be small, and allows the regularity to increase up to $1/3$ if the rate (\ref{eq:idealEr}) were to hold.  This ideal case is explained further in Section (\ref{sec:onOnsag}).  The above considerations mirror the analysis in the introduction to \cite{deLSzeHoldCts}.

However, before we can come close to the ideal rate of (\ref{eq:idealEr}), we must significantly change the Definition (\ref{def:subSolDefPr}).  To see why a drastic change is necessary, consider the contribution to the new stress which arises from the transport term
\begin{align}
\pr_j Q_T^{jl} &\approx \sum_I e^{i \la \xi_I} (\pr_t + v_\ep^j \pr_j)v_I^l \label{eq:aTransportExample}
\end{align}

The $v_I$ itself has amplitude $e_R^{1/2}$ since it is constructed from solving the Stress equation.  Formally, we can think of $v_I \approx R^{1/2}$.  The derivative in $\pr_j v_I^l$ costs a factor of $\Xi$, and the factor
\[ v_\ep^j \approx 1 \]
is bounded by some constant, giving the second term in (\ref{eq:aTransportExample}) size $\Xi e_R^{1/2}$.  Then, solving the divergence equation gains a factor of $(N \Xi)^{-1}$, leading to an expected bound for the new stress
\begin{align}
\co{ Q_T } &\leq \fr{e_R^{1/2}}{N} \\
\Rightarrow e_R' \geq \fr{e_R^{1/2}}{N} \label{eq:uhOh}
\end{align}
which is a factor $e_v^{-1/2}$ worse than the ideal case (\ref{eq:idealEr}), and which is also inconsistent with dimensional analysis.

In order to calculate the regularity achieved from applying the Lemma (\ref{samLem:practice}) with a rate of (\ref{eq:uhOh}), one still imposes the Ansatz (\ref{eq:theAnsatz}), but in this case it is {\bf not optimal} to let $\de$ tend to $0$.  Doing so would result in a new frequency 
\begin{align}
\Xi_{(k+1)} &= e_{R,(k)}^{-1/2 - \de} \Xi_{(k)}
\end{align}
which is much larger than the previous frequency, but would result in only a very small reduction of stress.  On the other hand, letting $\de$ be very large would be wasteful since the frequencies $\Xi_{(k)}$ would grow very quickly.

In this case, iteration of (\ref{eq:theAnsatz}) at a rate of (\ref{eq:uhOh}) leads to a regularity up to
\begin{align}
v &\in C^{\a^* - \ep} \\
p &\in C^{2(\a^* - \ep)} 
\end{align}
with
\begin{align}
\a^* &= \fr{\de}{(1+\de)(1 + 2 \de)} 
\end{align}
The optimal $\de$ to choose is $\de^* = \fr{1}{\sqrt{2}}$, which would lead to solutions with regularity \footnote{Of course, to actually obtain this regularity would at least require a modification of the Definition (\ref{def:subSolDefPr}) to incorporate time derivatives, and the Lemma (\ref{samLem:practice}) must be stated precisely and proven.} up to
\begin{align}
\a^* &= \fr{1}{3 + \sqrt{8}}  \label{eq:giveUp}
\end{align}

These regularity exponents can be computed quickly using the machinery introduced in Section (\ref{sec:lemImpThm}).

In order to achieve a regularity better than (\ref{eq:giveUp}) we examine the size of the transport term more carefully, without separating the terms in the coarse scale material derivative.
\[\Ddt =  (\pr_t + v_\ep \cdot \nab) \]

In fact, when a coarse scale material derivative hits a phase gradient, which obeys the equation
\begin{align}
(\pr_t + v_\ep^a \pr_a) \pr_l \xi_I &= - \pr_l v_\ep^a \pr_a \xi_I 
\end{align}
the cost is 
\[ |\Ddt| \leq \Xi e_v^{1/2} \]
which is exactly what we desire for (\ref{eq:idealEr}).

The trouble is that the amplitude $v_I \approx R^{1/2}$ is also influenced by the stress $R$.  To deal with this difficulty, we must assume that the bound for the material derivative of $R$ is {\bf better} than the bound for the spatial derivative of $R$ in our definition of frequency and energy levels.  In order to make this assumption on the material derivative of $R$, we must also verify it in order to continue the iteration.   To obtain these improved bounds, we introduce the following changes to the scheme:
\begin{enumerate}
\item  We change the way in which $R_\ep$ is mollified so that the material derivative of the transport term 
\begin{align}
\pr_j Q_T^{jl} &\approx \sum_I e^{i \la \xi_I} (\pr_t + v_\ep^j \pr_j)v_I^l
\end{align}
can be estimated.  Namely, checking the bound for the material derivative of $Q_T$ requires taking a {\it second} material derivative of $R_\ep$, whereas we only assume a bound on the first material derivative of $R$.
\item  We also change the way we eliminate the error for the parametrix by solving the divergence equation
\begin{align}
\pr_j Q^{jl} &= U^l \label{eq:undetElliptEqn}
\end{align}
since it is not true that every solution to the underdetermined equation (\ref{eq:undetElliptEqn}) will have a good bound on its material derivative.  We find that the simplest way to obtain bounds on the material derivative of a solution to (\ref{eq:undetElliptEqn}) is to solve (\ref{eq:undetElliptEqn}) by constructing an appropriate transport equation.
\item  Furthermore, we cannot perform a standard mollification of $v$ in the time variable because bounds on time derivatives are not as good as bounds on material derivatives.  In fact, we only mollify $v \to v_\ep$ in space, and the regularity of $v_\ep$ in time follows from the Euler-Reynolds equations themselves.
\end{enumerate}

These modifications are presented in the remainder of the paper, which we begin by properly formalizing the template Lemma (\ref{samLem:practice}).

%% file: mainLem.tex
\subsection{Frequency and Energy Levels for the Euler Reynolds Equations and Symmetries}

The main lemma in the paper which is responsible for the proof of Theorem (\ref{mainThm}) relies on the following definition.

\begin{defn} \label{def:subSolDef}
Let $L \geq 1$ be a fixed integer.  Let $\Xi \geq 2$, and let $e_v$ and $e_R$ be positive numbers with $e_R \leq e_v$.  Let $(v, p, R)$ be a solution to the Euler-Reynolds system.  We say that the frequency and energy levels of $(v,p,R)$ are below $(\Xi, e_v, e_R)$ (to order $L$ in $C^0 = C^0_{t,x}(\T^3 \times \R)$) if the following estimates hold.
\begin{align}
|| \nab^k v ||_{C^0} &\leq \Xi^k e_v^{1/2} &k = 1, \ldots, L \label{bound:nabkv} \\
|| \nab^k p ||_{C^0} &\leq \Xi^k e_v &k = 1, \ldots, L \label{bound:nabkp} \\
|| \nab^k R ||_{C^0} &\leq \Xi^k e_R &k = 0, \ldots, L \label{bound:nabkR} \\
|| \nab^k (\pr_t + v \cdot \nab) R ||_{C^0} &\leq \Xi^{k+1} e_v^{1/2} e_R  &k = 0, \ldots, L - 1 \label{bound:dtnabkR}
\end{align}
Here $\nab$ refers only to derivatives in the spatial variables.
\end{defn}

We first remark that these bounds are all consistent with the symmetries of the Euler equations.  The scaling symmetry is reflected by dimensional analysis.  Informally, if $X$ denotes a spatial unit (e.g. ``meters'') and $T$ denotes a unit of time (e.g. ``seconds''), then the dimensions of each term involved in the estimates are $e_R, e_V \sim p, R \sim X^2 / T^2$, $\Xi \sim \nab \sim X^{-1}$, $\pr_t \sim T^{-1}$ and $v \sim X/T$.  

We have also assumed {\bf no} bound on $|| v||_{C^0}$ or on $|| \pr_t R ||_{C^0}$.  These assumptions are consistent with the Galilean invariance of the Euler equations and the Euler Reynolds equations, which plays a special role underlying what follows.  Namely, if $(v, p, R)$ solve the Euler-Reynolds equations, then a new solution to Euler-Reynolds with the same frequency energy levels can be obtained by taking 
\ali{
 {\hat v}^j(t,x) &= C^j + v^j(t, x - t C) \\
  {\hat p}(t,x) &= p(t, x - t C) \\
{\hat R}^{jl}(t,x) &= R^{jl}(t, x - t C)
}
for any $C \in \R^3$.  This transformation corresponds to observing the fluid from a new frame of reference with relative velocity $C$.  Furthermore, the aspects of the construction we introduce to take the special role of the material derivative into account (namely, the mollifications of $v$ and $R$ and the new method used to solve the divergence equation) are also well-behaved under this group of transformations.  

Also note that we have not assumed a bound on $||p||_{C^0}$ and that only the pressure gradients are assumed to have bounds, which is consistent with the fact that the pressure is only determined from the velocity up to an arbitrary function of time.  It is interesting to remark that one can include bounds on the material derivative of the pressure gradient such as
 \[ \co{ (\pr_t + v \cdot \nab) \nab p } \leq \Xi^2 e_v \]
in the definition of frequency energy levels.  However, these bounds are not necessary for the results in the present paper.

\subsection{ Statement of the Main Lemma }

With this definition in hand, we can state the main lemma.  Note that we require the order $L \geq 2$, as we would not be able to prove a lemma as clean as the one below if we assume control over only $1$ derivative; this requirement is discussed in Section (\ref{coarseScaleVelocity}).  The requirement of at least two derivatives appears to be natural in view of the calculations of Section (\ref{sec:higherEnergyReg}), and may be necessary for controlling one of the error terms if the method could be pushed to regularity $1/3$, as is noted in Section (\ref{sec:mollVOK}).

\begin{lem}[The Main Lemma] \label{lem:iterateLem}
Suppose $L \geq 2$ and that $\eta > 0$.  Let $K$ be the constant in line (\ref{eq:wePickedK}).  There is a constant $C$ depending only on $\eta$ and $L$ such that the following holds:

Let $(v,p,R)$ be any solution of the Euler-Reynolds system whose frequency and energy levels are below $(\Xi, e_v, e_R)$
 to order $L$ in $C^0$ and let $I$ be a union of nonempty time intervals (possibly unbounded) such that
\begin{align}
\mbox{ supp } R \subseteq I \times \T^3
\end{align}
Define the time-scale $\th = \Xi^{-1} e_v^{-1/2}$, and let
\[ e(t) : \R \to \R_{\geq 0} \]
be any non-negative function satisfying the lower bound
\begin{align}
 e(t) \geq K e_R \quad \quad \mbox{ for all } t \in I \pm \th
\end{align}
and whose square root satisfies the estimates
\begin{align}
|| \fr{d^r e^{1/2}}{dt^r} ||_{C^0} &\leq 1000 (\Xi e_v^{1/2})^r (K e_R)^{1/2} &r = 0, 1, 2 \label{ineq:goodEnergy}
\end{align}

Now let $N$ be any positive number obeying the bounds
\begin{align} \label{eq:conditionsOnN}
 N &\geq \Xi^\eta \\
 N &\geq \left(\fr{e_v}{e_R} \right)^{3/2}\label{eq:conditionsOnN2}
\end{align}
and define the dimensionless parameter ${\bf b} = \left(\fr{e_v^{1/2}}{e_R^{1/2}N} \right)^{1/2}$.

Then there exists a solution $(v_1, p_1, R_1)$ of the Euler Reynolds system of the form $v_1 = v + V$, $p_1 = p + P$ whose frequency and energy levels are below
\begin{align}
 (\Xi', e_{v}', e_{R}') &= (C N \Xi, e_R, \fr{e_v^{1/4} e_R^{3/4}}{N^{1/2}}   ) \\
 &= (C N \Xi, e_R, \left(\fr{e_v^{1/2}}{e_R^{1/2}N} \right)^{1/2} e_R   ) \\
&= (C N \Xi, e_R, {\bf b}^{-1} \fr{e_v^{1/2} e_R^{1/2}}{N})
\end{align}
to order $L$ in $C^0$, and whose stress $R_1$ is supported in
\begin{align}
 \mbox{ supp } R_1 \subseteq ( \mbox{ supp } e \pm \th ) \times \T^3.
\end{align}

The correction $V = v_1 - v$ is of the form $V = \nab \times W$ and can be guaranteed to obey the bounds
\begin{align}
||V||_{C^0} &\leq C e_R^{1/2} \label{eq:Vco} \\
||\nab V ||_{C^0} &\leq C N \Xi e_R^{1/2} \label{eq:nabVco} \\
||(\pr_t + v^j \pr_j) V ||_{C^0} &\leq C {\bf b}^{-1} \Xi e_v^{1/2} e_R^{1/2} \label{eq:matVco} \\
||W||_{C^0} &\leq C \Xi^{-1} N^{-1} e_R^{1/2} \label{eq:Wco}\\
\co{ \nab W } &\leq C e_R^{1/2} \label{eq:nabWco} \\
||(\pr_t + v^j \pr_j) W ||_{C^0} &\leq C {\bf b}^{-1} N^{-1} e_v^{1/2} e_R^{1/2} \label{eq:matWco}
\end{align}
and its energy can be prescribed up to errors bounded by
\begin{align}
|| \int_{\T^3} |V|^2 dx - \int_{\T^3} e(t) dx ||_{C^0} &\leq C \fr{e_R}{N} \label{eq:energyPrescribed}\\
|| \fr{d}{dt} ( \int_{\T^3} |V|^2 dx - \int_{\T^3} e(t) dx ) ||_{C^0} &\leq C {\bf b}^{-1} \fr{\Xi e_v^{1/2} e_R}{N} \label{eq:DtenergyPrescribed}
\end{align}

The correction to the pressure $P = p_1 - p_0$ satisfies the estimates
\begin{align}
|| P ||_{C^0} &\leq C e_R \label{eq:Pco} \\
|| \nab P ||_{C^0} &\leq C N \Xi e_R \label{eq:nabPco} \\
|| (\pr_t + v\cdot \nab) P ||_{C^0} &\leq C {\bf b}^{-1} \Xi e_v^{1/2} e_R \label{eq:matPco}
\end{align}

Finally, the supports of both $V$ and $P$ are contained in $ \mbox{ supp } e \times \T^3$.
\end{lem}

{\bf Remark about dimensional analysis and (\ref{eq:conditionsOnN}):} Although it is not consistent with inequality  (\ref{eq:conditionsOnN}), the parameter $N$ should be regarded as a dimensionless parameter.  When the lemma is applied, the parameter $\eta$ will tend to $0$, and the constant $C$ in the lemma will diverge to infinity.  The assumption (\ref{eq:conditionsOnN}) will be used to bound the number of terms taken in the parametrix before the approximate solution is sufficiently accurate.

\paragraph{A conjecture and a remark about energy regularity}
Having stated the Main Lemma, we propose a ``conjecture'' which is probably false as stated, but which summarizes what the proof of Lemma (\ref{lem:iterateLem}) would prove if the High-High interference terms could be sufficiently well-controlled.
\begin{conject}[The Ideal Case]\label{conj:idealCase}
The conclusion of the Main Lemma (\ref{lem:iterateLem}) holds with the factor ${\bf b}^{-1}$ replaced by $1$ in all time derivative and material derivative estimates, and with
\[ (\Xi', e_v', e_R') = (C N \Xi, e_R, \fr{e_v^{1/2} e_R^{1/2}}{N} ) \]
\end{conject}
As we discuss in Section (\ref{sec:onOnsag}), the above Conjecture could be used to construct solutions of regularity up to $1/3$, and would therefore imply Onsager's conjecture.  In Section, (\ref{sec:accountForParams}), we discuss the term which prevents us from proving the above conjecture. 

We also remark that although the solutions produced by Lemma (\ref{lem:iterateLem}) go up to the H{\" o}lder exponent $1/5$, the energy function 
\[ E(t) = \int \fr{|v|^2}{2}(t,x) dx\]
turns out to possess better regularity in time, as it necessarily belongs to the class $C^{1/2 - \ep}(\R)$ from the way it is constructed.  The reason for this higher regularity will be discussed in Section (\ref{sec:onOnsag}).

We now show that the Lemma (\ref{lem:iterateLem}) above contains enough quantitative information regarding the construction to prove the main theorem.  For convenience, we state the form of the main theorem which is proven directly by iterating Lemma (\ref{lem:iterateLem}).

\begin{thm}\label{thm:theMainThm}
Let $\de > 0$ and set 
\begin{align}
\a^* &=\fr{1+ \de}{5 + 9 \de + 4 \de^2} 
\end{align}

Then there exists a nonzero solution to the Euler equations with compact support in time such that
\begin{align}
 ||v||_{C_{t,x}^{0,\a}} < \infty
\end{align}
for all $\a < \a^*$ and
\begin{align}
 ||p||_{C_{t,x}^{0,\a}} < \infty
\end{align}
for all $\a < 2 \a^*$.

Furthermore, the energy of $v$ can be made arbitrarily large, and the time interval containing the support of $(v, p)$ can be made arbitrarily small.
\end{thm}

%% file: newLemImpliesThm.tex
The proof of Theorem (\ref{thm:theMainThm}) proceeds by inductively applying the Main Lemma (\ref{lem:iterateLem}) in order to construct a sequence of solutions $(v_{(k)}, p_{(k)}, R_{(k)})$ of the Euler Reynolds system such that $R_{(k)}$ converges to $0$ uniformly and such that $v_{(k)}$ converges in $C_{t,x}^{0,\a}$ for all $\a < \a^*$.  The parameters $\eta > 0$ and $L$ are fixed at the beginning of the argument -- in fact we can simply take $L = 2$.  At each stage $(k)$ of the induction, we choose an energy function $e_{(k)}(t)$ and a parameter $N_{(k)}$ whose choice determines the growth of the frequency parameter $\Xi_{(k)}$ and the decay of the energy level $e_{R,(k)}$.

In fact, we will iterate the Main Lemma  in such a way that
\begin{align}
 e_{v,(k+1)} &= e_{R, (k)} \\
 e_{R, (k+1)} &= \fr{e_{R,(k)}^{(1 + \de)}}{Z} \label{eq:selfSimAns}
\end{align}
Where $Z$ is a constant chosen at the beginning of the iteration in order to make sure that $e_{R,(1)} < e_{R,(0)}$, and so that other inequalities are satisfied.  The self-similarity Ansatz (\ref{eq:selfSimAns}) leads to super-exponential decay of the energy levels $e_{v,(k)}, e_{R,(k)}$ and super-exponential growth of the frequencies $\Xi_{(k)}$.

In order to be able to achieve this Ansatz, we must apply the lemma with
\begin{align}
N_{(k)} &= Z^2 \left( \fr{e_v}{e_R} \right)_{(k)}^{1/2} e_{R,(k)}^{- 2 \de}
\end{align}
which will imply a frequency growth of
\begin{align}
\Xi_{(k + 1)} &= C_{\eta, L} Z^2 \left( \fr{e_v}{e_R} \right)_{(k)}^{1/2} e_{R,(k)}^{- 2 \de} \Xi_{(k)} \label{eq:XiEvolves}
\end{align}

The exponent $\eta$ (which determines how many terms must be taken in the parametrix expansion for the elliptic equation in the main argument, and therefore comes into the constant in the above equality) is also fixed at the beginning of the iteration.

\subsection{The Base Case} \label{sec:baseCase}
\input{theBaseCase}

\subsection{The Main Lemma implies the Main Theorem}

We now begin the proof that Lemma (\ref{lem:iterateLem}) implies Theorem (\ref{thm:theMainThm}).

\begin{proof}[Proof of the Main Theorem (\ref{thm:theMainThm})]

Let $\de > 0,$ $e_{R,(0)} > 0$ and $\Xi_{(0)} \geq 2$ be given numbers.  We will construct a sequence of solutions $(v_{(k)}, p_{(k)}, R_{(k)})$ of the Euler Reynolds system such that $R_{(k)}$ converges to $0$ uniformly and such that $v_{(k)}$ converges in $C^{0, \a}_{t,x}$ for all \[ \a < \a^* = \fr{1+ \de}{5 + 9 \de + 4 \de^2} \] We will also see that $p_{(k)}$ converges in $C^{0, \a}_{t,x}$ for all $\a < 2 \a^*$.

$ $

\noindent{\bf Step 0:}

Choose $\eta > 0$ small enough such that the total exponent 
\[ \de^2(2 + \fr{\ga}{2}) - \eta \de (2 + \fr{\ga}{2}) \]
to which $e_{R,(k)}$ is raised in equation (\ref{eq:doesXiEtaContinue}) is positive.  For example, $\eta = \fr{\de}{2}$ will suffice.  For this choice of $\eta$ and $L = 2$, let $C = C_\eta$ be the constant given by the Main Lemma (\ref{lem:iterateLem}).

To initialize the construction, apply the base case Lemma (\ref{lem:baseCase}) with the parameters \[ (\de, C_\eta, e_{R,(0)}, \Xi_{(0)}) \] to obtain a constant $Z > 0$ and a solution $(v_{(0)}, p_{(0)}, R_{(0)})$ to the Euler-Reynolds equations with frequency and energy levels below $(\Xi_{(0)}, e_{v,(0)}, e_{R,(0)})$ where
\begin{align}
e_{R,(0)} &= \fr{e_{v,(0)}^{1 + \de}}{Z}
\end{align}
We can require this solution to have support in the interval $I_{(0)} = \{ 0 \}$, since the solution guaranteed by the Lemma is identically $0$.

\subsubsection{Choosing the Parameters}

\noindent {\bf Step k:}

To continue the construction, we will iterate Lemma (\ref{lem:iterateLem}) to produce a sequence $(v_{(k)}, p_{(k)}, R_{(k)})$ with corresponding energy and frequency levels below $(\Xi_{(k)}, e_{v, (k)}, e_{R, (k)} )$, and we impose the ansatz
\begin{align} \label{eq:energyAnsatz}
e_{R, (k+1)} &= Z^{-1} e_{R,(k)}^{(1 + \de)}
\end{align}
The inequality (\ref{eq:vAndRinit}) implies that $e_{R,(1)} = Z^{-1} e_{R,(0)}^{(1 + \de)} < e_{R,(0)}$, so by induction the sequence $e_{R,(k)}$ defined recursively by (\ref{eq:energyAnsatz}) is a decreasing sequence.

In fact, the ansatz (\ref{eq:energyAnsatz}) together with the requirement $e_{R,(1)} < e_{R,(0)}$ guarantees that the sequence $e_{R,(k)}$ decreases to $0$ super-exponentially, as one can see by writing the relation as
\begin{align}
e_{R, (k+1)} &= (Z^{-1} e_{R,(k)}^{\de}) e_{R,(k)}
\end{align}
Then $Z^{-1} e_{R,(k)}^{\de} < 1$ for $k = 0$, so by induction even the ratio between consecutive terms 
\[ \fr{e_{R,(k+1)}}{e_{R,(k)}} = Z^{-1} e_{R,(k)}^{\de} \]
decreases to $0$.

Because $e_{v,(k+1)} = e_{R, (k)}$, the ansatz (\ref{eq:energyAnsatz}) also implies by induction a gap between the energy levels $e_v$ and $e_R$
\begin{align}
e_{v,(k)} &= Z^{\ga} e_{R, (k)}^\ga \quad k \geq 1 \\
\ga ( 1 + \de ) &= 1
\end{align}
implying that the ratio between the energy levels
\begin{align}
\left( \fr{e_v}{e_R} \right)_{(k)} &= Z^\ga e_{R,(k)}^{-\de \ga}
\end{align}
increases super-exponentially in $(k)$.

In order to achieve this Ansatz, we must verify that the choice of 
\begin{align}
N_{(k)} &= Z^2 \left( \fr{e_v}{e_R} \right)_{(k)}^{1/2} e_{R,(k)}^{-2\de}
\end{align}
always satisfies the admissibility conditions (\ref{eq:conditionsOnN}) and (\ref{eq:conditionsOnN2}).

In line (\ref{goal:threeHalvesOk}) in the proof of Lemma (\ref{lem:baseCase}), we showed that the condition (\ref{eq:conditionsOnN}) is equivalent to the inequality 
\begin{align} \label{eq:simplerConditionsOnN}
Z^{\ga - 2} e_{R,(k)}^{- \de \ga + 2 \de} &\leq 1
\end{align}

For $(k) = 0$, Lemma (\ref{lem:baseCase}) guarantees that this inequality is satisfied.  Since the power $\de( 2 - \ga)$ to which $e_{R,(k)}$ is raised is positive, and $e_{R,(k)}$ decreases, we see that (\ref{eq:simplerConditionsOnN}) and therefore (\ref{eq:conditionsOnN}) are always satisfied by induction.

Similarly, inequality (\ref{eq:conditionsOnN2}) is satisfied for $(k) = 0$ by Lemma (\ref{lem:baseCase}), and the proof showed that (\ref{eq:conditionsOnN2}) continues to be satisfied if we have that
\begin{align}
\left[ Z^{-1} e_{R,(k)}^{ \de } \right]^{\de ( 2 + \fr{\ga}{2} )} \left[  C_\eta Z^{2 +\fr{\ga}{2}} e_{R,(k)}^{ - \de (2 + \fr{\ga}{2}) } \right]^\eta \leq 1 \label{eq:easyConditionsOnN2}
\end{align}

According to the Lemma (\ref{lem:baseCase}), this inequality is satisfied for $(k) = 0$.  By the choice of $\eta < \de$, the power to which $e_{R,(k)}$ is raised on the left hand side of inequality (\ref{eq:easyConditionsOnN2}) is positive, so the inequality (\ref{eq:easyConditionsOnN2}) and therefore (\ref{eq:conditionsOnN2}) continues to hold for all $(k) \geq 0$.  This completes the proof that the choice of $N_{(k)}$ is always admissible.

\subsubsection{Choosing the Energies} \label{sec:energyChoice}

Along with these choices of $(\Xi_{(k)}, e_{v,(k)}, e_{R,(k)})$ and $N_{(k)}$,  the main lemma also grants us the ability to approximately prescribe the energy increment $e_{(k)}(t)$ of the correction.

The energy must satisfy the estimates:
\begin{align} \label{eq:energyRequired2}
|| \fr{d^l e_{(k)}^{1/2}}{dt^l} ||_{C^0} &\leq 1000 K^{1/2} (\Xi e_{v,(k)}^{-1/2})^l e_{R,(k)}^{1/2} &l = 0, 1, 2
\end{align}
and must satisfy the lower bound
\begin{align}
e_{(k)}(t) \geq K e_{R,(k)}
\end{align}
for $t$ belonging to the active region $t \in I_{(k)} \pm \th_{(k)}$, where the stress is supported $\mbox{ supp } R_{(k)} \subseteq I_{(k)} \times \T^3$.


For our construction, let $I_{(k)} = [a_{(k)}, b_{(k)}]$ be an interval such that
\[ \mbox{ supp } R_{(k)} \subseteq I_{(k)} \]
and let $\tau_{(k)}$ be a small number which we will specify.  Denote by  \[ I_{(k)} \pm (\th_{(k)} + \tau_{(k)}) = [a_{(k)} - (\th_{(k)} + \tau_{(k)}), b_{(k)} + (\th_{(k)} + \tau_{(k)})], \] and let 
\[ \chi_{(k)}(t) = \chi_{I_{(k)} \pm (\th_{(k)} + \tau_{(k)})}(t) \] be the characteristic function of this time interval.

Set
\begin{align}
 e_{(k)}^{1/2} (t) &= (20 K e_{R,(k)})^{1/2} \eta_{\tau_{(k)}} \ast \chi_{(k)}(t)
\end{align}
where 
\begin{align}
\eta_{\tau_{(k)}}(t) &= \tau_{(k)}^{-1} \eta(\fr{t}{\tau_{(k)}}) \label{eq:theTimeMoll}
\end{align} 
is a standard, smooth mollifier of integral $\int \eta_{\tau_{(k)}}(t) dt = 1$, supported in the ball $\{ |t| \leq |\tau_{(k)}| \}$.

Then, on the interval $I_{(k)} \pm \th_{(k)}$ containing the support of $(v_{(k)}, p_{(k)}, R_{(k)})$, we clearly have $e_{(k)}(t) = 20 K e_{R, (k)} \geq K e_{R,(k)}$, and we also have bounds of the form
\begin{align}
 || e_{(k)}^{1/2} (t) ||_{C^0} &\leq (20 K)^{1/2} e_{R,(k)}^{1/2} \\
\label{eq:energyVarBound} || \fr{d^l e_{(k)}^{1/2} }{dt^l} ||_{C^0} &\leq B \tau_{(k)}^{-l} K^{1/2} e_{R,(k)}^{1/2} &l = 1, 2
\end{align}

Therefore $e_{(k)}(t)$ is admissible as an approximate energy function if we choose 
\begin{align}
\tau_{(k)} = B_0 \Xi_{(k)}^{-1} e_{v,(k)}^{-1/2} = B_0 \th_{(k)} \label{eq:weChoseTauk}
\end{align}
where $B_0$ is a constant.  

\input{regOfSoln}

\end{proof}

%% file: theBaseCase.tex
We begin the proof by establishing a base case lemma, which will ensure that the choice of $N_{(k)}$ above will always satisfy the admissibility conditions (\ref{eq:conditionsOnN}),(\ref{eq:conditionsOnN2}).

\begin{lem}[Base Case] \label{lem:baseCase} Suppose that $\de > 0$ and $L \geq 2$.  Then there exists an $\eta_0 > 0$ such that for all positive numbers $e_{R,(0)} > 0, C > 0, \Xi_{(0)} \geq 2$ and $\eta > 0$ with $\eta \leq \eta_0$, and for every nonempty interval $I_{(0)}$ there exists an arbitrarily large number $Z$ and a solution $(v_0, p_0, R_0)$ to the Euler Reynolds system
\begin{align}
 \pr_t v^l + \pr_j(v^j v^l) + \pr^l p &= \pr_j R^{jl} \\
\pr_i v^i &= 0
\end{align}
with support in $I_{(0)} \times \T^3$ having frequency and energy levels below $(\Xi_{(0)}, e_{v,(0)}, e_{R,(0)})$ to order $L$ in $C^0$, where $(\Xi_{(0)}, e_{v,(0)}, e_{R,(0)})$ satisfy all of the following:

\begin{enumerate}
\item For
\[ \ga = 1/(1 + \de) \]
we have the relationship
\begin{align} \label{eq:vAndRinit}
 \begin{aligned}
 e_{v, (0)} &= Z^\ga e_{R,(0)}^\ga \\
 e_{R,(0)} &< e_{v,(0)}
 \end{aligned}
\end{align}
which can also be written
\begin{align}
 e_{R,(0)} &= \fr{e_{v,(0)}^{1 + \de}}{Z}
\end{align}
\item Condition (\ref{eq:conditionsOnN2}) is satisfied for
\begin{align}
 N_{(0)} &= Z^2  \left( \fr{e_v}{e_R} \right)_{(0)}^{1/2} e_{R,(0)}^{-2\de}
\end{align}
Namely,
\begin{align}
 \left( \fr{e_v}{e_R} \right)_{(0)}^{3/2} \leq Z^2  \left( \fr{e_v}{e_R} \right)_{(0)}^{1/2} e_{R,(0)}^{-2\de} \label{ineq:3halves}
\end{align}
\item Condition (\ref{eq:conditionsOnN}) is also satisfied for $N_{(0)}$, namely
\begin{align}
 \Xi_{(0)}^\eta &\leq Z^2 \left( \fr{e_v}{e_R} \right)_{(0)}^{1/2} e_{R,(0)}^{-2 \de} \label{ineq:atLeastXiEta}
\end{align}
\item We can also ensure that
\begin{align} \label{ineq:continueAtLeastXi}
\left[ Z^{-1} e_{R,(k)}^{ \de } \right]^{\de ( 2 + \fr{\ga}{2} )} \left[  C Z^{2 +\fr{\ga}{2}} e_{R,(k)}^{  - \de (2 + \fr{\ga}{2})  } \right]^\eta &< 1 
\end{align}
which will be used later in the proof in order to make sure that (\ref{eq:conditionsOnN}) continues to be verified for the choices of $N_{(k)}$ at every later stage despite the increase in $\Xi_{(k)}$.

\end{enumerate}
\end{lem}
\begin{proof}
We can take
\[ (v_0, p_0, R_0) = (0, 0, 0), \]
which can be regarded as a solution to Euler-Reynolds with frequency and energy bounded below any trio $(\Xi_{(0)}, e_{v,(0)}, e_{R,(0)})$ which are admissible according to the Definition (\ref{def:subSolDef}).  The lemma reduces to verifying that for every $e_{R,(0)}$ and $C > 0$, there exist numbers $\Xi_{(0)}$ and $Z$ so that for $e_{v,(0)} = Z^\ga e_{R,(0)}^\ga$, the trio $(\Xi_{(0)}, e_{v,(0)}, e_{R,(0)})$ satisfies all of the inequalities in the lemma simultaneously.  Thus, let $e_{R,(0)}$ and $C > 0$ be given.

\paragraph{Inequality (\ref{ineq:3halves})}

To see that inequality (\ref{ineq:3halves}) can be satisfied, we rewrite (\ref{ineq:3halves}) as a goal
\begin{align}\label{goal:threeHalves}
Z^{-2} \left( \fr{e_v}{e_R} \right)_{(0)} e_{R,(0)}^{2\de} &\unlhd 1
\end{align}

From the identity (\ref{eq:vAndRinit}), we have that
\begin{align}
\left(\fr{e_v}{e_R} \right)_{(0)} &= Z^\ga e_{R,(0)}^{ \ga - 1 } \\
&= Z^\ga e_{R,(0)}^{ - \de \ga } \label{eq:energyRatioIs}
\end{align}
since $\ga(1 + \de) = 1$.

Substituting into the left hand side of (\ref{goal:threeHalves}) gives
\begin{align} \label{goal:threeHalvesOk}
Z^{-2} \left( \fr{e_v}{e_R} \right)_{(0)} e_{R,(0)}^{2\de} &= Z^{\ga - 2} e_{R,(0)}^{- \de (\ga - 2)}
\end{align}
which will be less than $1$ for any $Z$ sufficiently large, depending on $e_{R,(0)}$.

\paragraph{Inequality (\ref{ineq:atLeastXiEta})}

In order to verify inequality (\ref{ineq:atLeastXiEta}), we rewrite it as a goal
\begin{align}
Z^{-2} \left( \fr{e_v}{e_R} \right)_{(0)}^{-1/2} e_{R,(0)}^{2 \de} \Xi_{(0)}^\eta &\unlhd 1 \\
Z^{-2} \left[ Z^\ga e_{R,(0)}^{ - \de \ga } \right]^{-1/2} e_{R,(0)}^{2 \de} \Xi_{(0)}^\eta &\unlhd 1\\
Z^{-( 2 + \fr{\ga}{2})} e_{R,(0)}^{ \fr{(\de \ga)}{2} + 2 \de } \Xi_{(0)}^\eta &\leq 1 \label{ineq:atLeastXiEta2}
\end{align}

For any given $\Xi_{(0)}$, $\eta$ and $e_{R,(0)}$, inequality (\ref{ineq:atLeastXiEta2}) will be satisfied if $Z$ is chosen large enough.

\paragraph{Motivation for the inequality (\ref{ineq:continueAtLeastXi})}

Let us discuss the importance of the inequality (\ref{ineq:continueAtLeastXi}).

In order to ensure that inequality (\ref{ineq:atLeastXiEta}), or the equivalent inequality (\ref{ineq:atLeastXiEta2}), are satisfied for choices of the parameter $N_{(k)}$ in later stages, we need to show that
\begin{align}
Z^{-( 2 + \fr{\ga}{2})} e_{R,(k)}^{ \fr{\de \ga}{2} + 2 \de } \Xi_{(k)}^\eta &\unlhd 1 \label{ineq:atLeastXiEtak}
\end{align}
for all $(k) \geq 0$.  This inequality has been established for $(k) = 0$, so in order to accomplish this goal for later $(k)$, it suffices to show that the sequence
\begin{align} 
x_{(k)} &\equiv  e_{R,(k)}^{ \fr{\de \ga}{2} + 2 \de } \Xi_{(k)}^\eta \\
&= e_{R,(k)}^{ \de ( 2 + \fr{\ga}{2} ) } \Xi_{(k)}^\eta \label{eq:xSubKtemp}
\end{align}
is decreasing in $(k)$.  (Here we have applied the identity $\ga(1 + \de) = 1$.)  

From (\ref{eq:selfSimAns}), (\ref{eq:XiEvolves}) and (\ref{eq:energyRatioIs}), we can deduce the following evolution rules for the parameters
\begin{align}
e_{R,(k+1)} &= \fr{e_{R,(k)}^\de}{Z} e_{R,(k)} \\
\Xi_{(k + 1)} &= C Z^2 \left( \fr{e_v}{e_R} \right)_{(k)}^{1/2} e_{R,(k)}^{- 2 \de} \Xi_{(k)} \\
&= C Z^2 \left[Z^\ga e_{R,(k)}^{ - \de \ga } \right]^{1/2} e_{R,(k)}^{- 2 \de} \Xi_{(k)} \\
&= C Z^{2 +\fr{\ga}{2}} e_{R,(k)}^{ \fr{- \de \ga}{2} - 2 \de } \Xi_{(k)}
\end{align}
implying that
\begin{align}
x_{(k+1)} &= \left[ Z^{-1} e_{R,(k)}^{ \de } \right]^{\de ( 2 + \fr{\ga}{2} )} \left[  C Z^{2 +\fr{\ga}{2}} e_{R,(k)}^{ - \de (2 + \fr{\ga}{2}) } \right]^\eta x_{(k)} \label{eq:doesXiEtaContinue}
\end{align}

\paragraph{Inequality (\ref{ineq:continueAtLeastXi})}

Observe that for $\eta < \eta_0$ sufficiently small (say, $\eta_0 = \fr{\de}{2}$), the exponent of $Z$ in inequality (\ref{ineq:continueAtLeastXi}) is strictly negative.  Fix such an $\eta$.  Then for any given $C$ and $e_{R,(0)}$, it is possible to choose $Z$ large enough so that (\ref{ineq:continueAtLeastXi}) is satisfied.
\end{proof}

%% file: regOfSoln.tex
\subsubsection{Regularity of the Velocity Field} \label{sec:regVelocity}

We now check that using the parameters chosen above, the sequence of solutions $(v_{(k)}, p_{(k)}, R_{(k)})$ in the above argument does in fact converge to a solution with the properties stated in the Theorem (\ref{thm:theMainThm}).

\paragraph{Continuity}

The $C^0$ convergence of \[ v_{(K)} = \sum_{k = 0}^K V_{(k)} \] and follows easily from the super-exponential decay of $e_{R,(k)}$, the estimate \[ ||V_{(k)}||_{C^0} \leq C e_{R,(k)}^{1/2} \] and the triangle inequality
\begin{align}
\sum_{k = 0}^\infty || V_{(k)} ||_{C^0} &\leq C \sum_{k = 0}^\infty e_{R,(k)}^{1/2}
\end{align}

Similarly, the $C^0$ convergence of \[ p_{(K)} = \sum_{k = 0}^K P_{(k)} \] follows from the estimate
\[ \co{ P_{(k)} } \leq C e_{R,(k)} \]

\paragraph{H{\" o}lder Norms of the Correction}

To measure what kind of higher regularity is achieved by this choice, observe that the perturbation $V_{(k)} = v_{(k+1)} - v_{(k)}$ obeys the estimates
\begin{align*}
||V_{(k)}||_{C^0} &\leq C e_{R,(k)}^{1/2} \\
||\nab V_{(k)} ||_{C^0} &\leq C N_{(k)} \Xi_{(k)} e_{R,(k)}^{1/2} \\
||\pr_t V_{(k)} ||_{C^0} &\leq \co{ (\pr_t + v_{(k)} \cdot \nab) V_{(k)} } + \co{v_{(k)} \cdot \nab V_{(k)} }
\end{align*}
The estimate for the time derivative is no worse than the bound on the spatial derivative for the following reasons.  We have already established that $\co{v_{(k)}}$ are uniformly bounded, so the term 
\[ \co{v_{(k)} \cdot \nab V_{(k)} } \leq C \co{ \nab V_{(k)} } \]
with a constant depending only on $e_{R,(0)}$ and other absolute constants from the lemmas.

Despite the loss of a factor 
\begin{align}
{\bf b}_{(k)}^{-1} &= \left( \fr{N_{(k)} e_{R,(k)}^{1/2}}{e_{v,(k)}^{1/2}} \right)^{1/2} \\
 &= \left( Z^2 e_{R,(k)}^{-2 \de} \right)^{1/2} = Z e_{R,(k)}^{- \de}
\end{align}
in the estimates for material derivatives in Lemma (\ref{lem:iterateLem}), the estimate for the material derivative is still better than the bounds on the spatial derivatives.  Namely
\begin{align}
\co{ (\pr_t + v_{(k)} \cdot \nab) V_{(k)} } &\leq C {\bf b}_{(k)}^{-1} \Xi_{(k)} e_{v,(k)}^{1/2} e_{R,(k)}^{1/2} \\
&\leq C Z e_{R,(k)}^{- \de} \Xi_{(k)} e_{v,(k)}^{1/2} e_{R,(k)}^{1/2} \\
&= C Z e_{v,(k)}^{1/2} \Xi_{(k)} e_{R,(k)}^{1/2 - \de}
\end{align}
whereas the spatial derivatives are bounded by
\begin{align}
||\nab V_{(k)} ||_{C^0} &\leq C N_{(k)} \Xi_{(k)} e_{R,(k)}^{1/2} \\
&\leq C Z^2 \left( \fr{e_v}{e_R} \right)_{(k)}^{1/2} e_{R,(k)}^{ - 2 \de} \Xi_{(k)} e_{R,(k)}^{1/2} \\
&= C \left( \fr{e_v}{e_R} \right)_{(k)}^{1/2} \Xi_{(k)} e_{R,(k)}^{1/2 - 2 \de}
\end{align}
which is worse than the estimate for the material derivative by a super-exponentially large factor of $e_{R,(k)}^{-1/2 - \de}$.

These observations lead to the bounds 
\begin{align}
||V_{(k)}||_{C^0} &\leq C e_{R,(k)}^{1/2} \\
\co{ \nab_{t,x} V_{(k)} } &\leq C \left( \fr{e_v}{e_R} \right)_{(k)}^{1/2} e_{R,(k)}^{- 2 \de} \Xi_{(k)} e_{R,(k)}^{1/2}
\end{align}

Which, by interpolation, imply a bound
\begin{align}
|| V_{(k)} ||_{C^\a_{t,x}} &\leq C \left( \fr{e_v}{e_R} \right)_{(k)}^{\a/2} \Xi_{(k)}^\a e_{R,(k)}^{1/2 - \de \a}
\end{align}
for $0 \leq \a \leq 1$.

Therefore, to determine whether or not the series
\begin{align}
 \sum_{(k)} || V_{(k)} ||_{C^\a_{t,x}} &< \infty  \label{eq:holdSeries}
\end{align}
converges, it suffices to study the growth or decay of the quantity
\begin{align}
E_{\a, (k)} &\equiv  \left( \fr{e_v}{e_R} \right)_{(k)}^{\a/2} \Xi_{(k)}^\a e_{R,(k)}^{1/2 - \de \a} 
\end{align}

At this point, one way we can proceed would be to follow the method of Section (\ref{sec:baseCase}) and establish an evolution law of the form
\begin{align}
E_{\a, (k+1)} &\leq C_\a e_{R,(k)}^{\de f(\a)} E_{\a,(k)} \label{eq:evLawErForm}
\end{align}
using parameter evolution laws (\ref{eq:selfSimAns}), (\ref{eq:XiEvolves}) and (\ref{eq:energyRatioIs}).  Once such a law has been established, we have a trichotomy:
\begin{itemize}
\item If $f(\a)$ is positive, then $E_{\a,(k)}$ decays super-exponentially 
\item If $f(\a)$ is negative, then $E_{\a,(k)}$ grows super-exponentially 
\item If $f(\a)$ is zero, then $E_{\a,(k)}$ grows exponentially
\end{itemize}

The exponent $f(\a)$ depends linearly on $\a$, and it turns out that there is a critical $\a^* \in (0,1)$ depending on $\de$ at which $f(\a^*) = 0$.  For all $\a < \a^*$, the series (\ref{eq:holdSeries}) converges, and the solution $v \in C_{t,x}^\a$.  In our case, $\a^* = \fr{1 + \de}{5 + 9 \de + 4 \de^2}$.

Rather than proceeding to derive a law such as (\ref{eq:evLawErForm}), we illustrate an equivalent method for determining $\a^*$ which gives conceptual insight into the iteration, and which tends to be simpler for computation.  In fact, the method we present here can be used to accelerate the computations of the preceding Section (\ref{sec:baseCase}).  Namely, the evolution laws take the form 
\begin{align}
E_{\a, (k+1)} &\leq C_\a Z^{-f(\a)} e_{R,(k)}^{\de f(\a)} E_{\a,(k)} \label{eq:evLawErFormWithZ}
\end{align}
where we observe that $Z^{-1}$ and $e_{R,(k)}^\de$ are raised to the same power $f(\a)$.  This dimensional observation is a direct consequence of the Ansatz (\ref{eq:selfSimAns}).  The method we present in the following Section allows us to determine whether these exponents are positive or negative, which allows us to determine whether choosing $Z$ large can make such an evolution factor small at stage $(k) = 0$.  Informally, the quantities whose evolution remains under control in the late stages of the iteration are also the quantities whose evolution we can control at the beginning of the iteration.

\subsubsection{Asymptotics for the Parameters } \label{sec:asympParams}

In order to determine whether or not quantities such as the bounds 
\begin{align}
E_{\a, (k)} &\equiv  \left( \fr{e_v}{e_R} \right)_{(k)}^{\a/2} \Xi_{(k)}^\a e_{R,(k)}^{1/2 - \de \a} \label{eq:Ealpha}
\end{align}
for the H{\" o}lder norms remain under control in the late stages of the iteration, it suffices to understand the asymptotic relationships
\begin{align}
e_{v,(k)} &\approx \Xi_{(k)}^{\b_v} \label{eq:evandXi}\\
e_{R,(k)} &\approx \Xi_{(k)}^{\b_R} \label{eq:eRandXi}
\end{align}
between the energy and frequency parameters which emerge during the late stages of the iteration.  The following method can be used to determine these relationships, and reduces the control of quantities such as $E_{\a,(k)}$ to simple linear algebra calculations.

\paragraph{ A remark about dimensional analysis}

The asymptotic relationships (\ref{eq:evandXi}) and (\ref{eq:eRandXi}) cannot be deduced from dimensional analysis, since the dimensions of time and space for the Euler equations are independent of each other.  Rather, these relationships depend on the efficiency of the convex integration procedure that we have constructed.

\paragraph{Heuristics for Parameter Evolution}

Let us recall the parameter evolution laws
\begin{align}
e_{v,(k + 1)} &= e_{R,(k)} \\
e_{R,(k + 1)} &= Z^{-1} e_{R,(k)}^{1 + \de} \label{eq:erLaw}\\
\Xi_{(k+1)} &= C Z^2 \left( \fr{e_v}{e_R} \right)_{(k)}^{1/2} e_{R,(k)}^{- 2 \de}  \Xi_{(k)} \label{eq:XiLaw}
\end{align}
In view of the formulas (\ref{eq:Ealpha}) and (\ref{eq:XiLaw}), it is more convenient to study the evolution of the energy ratio $\left( \fr{e_v}{e_R} \right)_{(k)}$, rather than the energy level $e_{v,(k)}$.  For the energy ratio, we have the law
\begin{align}
\left( \fr{e_v}{e_R} \right)_{(k+1)} &= Z e_{R,(k)}^{- \de}
\end{align}

We can summarize these laws compactly by taking the logarithms of the equations, and writing.
\begin{align}
\vect{ \log e_R \\ \log(e_v/e_R) \\ \log \Xi}_{(k+1)} &\approx \mat{ccc}{ 1 + \de & 0 & 0 \\ - \de & 0 & 0 \\ - 2 \de & 1/2 & 1} \vect{ \log e_R \\ \log(e_v/e_R) \\ \log \Xi}_{(k)} 
\end{align}
This evolution law is not completely precise because it forgets about the constants in the exact parameter evolution laws; however, the {\bf parameter evolution matrix}
\begin{align}
T_\de &= \mat{ccc}{ 1 + \de & 0 & 0 \\ - \de & 0 & 0 \\ - 2 \de & 1/2 & 1}
\end{align}
contains all the important information regarding the asymptotic relationships between the parameters.  Informally, we have an asymptotic
\begin{align}
\vect{ \log e_R \\ \log(e_v/e_R) \\ \log \Xi}_{(k)} &\approx \left[ T_\de^k \right] \vect{ \log e_R \\ \log(e_v/e_R) \\ \log \Xi}_{(0)} \label{eq:lateParams1}
\end{align}
The reason being that, regardless of the initial choice of parameters, the right hand side will point in the direction of the eigenvector of $T_\de$ corresponding to its largest eigenvalue $\la_+$.  In this case, $T_\de$ is a lower triangular matrix whose eigenvalues $0$, $1$ and $\la_+ = 1 + \de$ can be immediately read from the diagonal or by inspecting the form of the matrix.  For example, the top row shows that $[1,0,0]$ is a row eigenvector for $\la_+ = (1 + \de)$, the third column gives $1$ as an eigenvalue, and the second and third columns are linearly dependent, giving $0$ as another eigenvalue.  Because $\la_+ > 1$, the heuristic (\ref{eq:lateParams1}) suggests that
\begin{align}
\vect{ \log e_R \\ \log(e_v/e_R) \\ \log \Xi}_{(k)} &\approx c_{(0)} (1 + \de)^k \psi_+ \label{eq:lateParams}
\end{align}
where $\psi_+$ is an eigenvector corresponding to the dominant eigenvalue 
\[ T_\de \psi_+ = (1 + \de) \psi_+ \] 
which belongs to the {\bf stable octant} $\OO(-, +, +)$ defined by
\begin{align} 
\psi_+ &\in \OO(-, +, +) = \left\{ \vect{a \\ b \\ c} ~|~ a < 0, b > 0, c > 0 \right\} 
\end{align}
of $\R^3$ in which the parameter logarithms
\begin{align}
\psi_{(k)} &\equiv \vect{ \log e_R \\ \log(e_v/e_R) \\ \log \Xi}_{(k)}
\end{align}
eventually reside.   The coefficient $c_{(0)}$ is obtained by expressing $\psi_{(0)}$ relative to a basis which diagonalizes $T_\de$, and can be assumed to be positive.

In this example, the minimal polynomial satisfied by the parameter evolution matrix is 
\begin{align}
(T_\de - (1 + \de) ) T_\de (T_\de - 1) &= 0
\end{align}
Thus, one can take for $\psi_+$ any vector in $\tx{Im } T_\de(T_\de - 1)$ which also belongs to the stable octant $\OO(-, +, +)$, such as
\begin{align}
\psi_+ &= - \de^{-1} T_\de (T_\de - 1)\vect{1 \\ 0 \\ 0} \\
&= - T_\de \vect{1 \\ -1 \\ -2} \\
\psi_+ &= \vect{-1 \\ 0 \\ \fr{5}{2} } + \de \vect{-1 \\ 1 \\ 2}
\end{align}

Applying the heuristic (\ref{eq:lateParams}) to the approximation
\begin{align}
\log E_{\a, (k)} &\approx (1/2 - 2 \de \a) \log e_{R,(k)} + \fr{\a}{2} \log (e_v/e_R)_{(k)} + \a \log \Xi_{(k)} \\
\log E_{\a,(k)} &\approx [ 1/2 - 2 \de \a, \a/2, \a] \psi_{(k)} \\
&\approx c_{(0)} (1 + \de)^k [ 1/2 - 2 \de \a, \a/2, \a] \psi_+ \\
&= c_{(0)} (1 + \de)^k [ (1/2 - 2 \de \a)\cdot(-1 - \de) + (\a/2)\cdot(\de) + (\a)\cdot(5/2 + 2 \de) \\
&= c_{(0)} (1 + \de)^k \left[ -\fr{(1 + \de)}{2} + (\fr{5}{2} + \fr{9 \de}{2} + 2 \de^2)\a \right]
\end{align}

According to this approximation, the bounds on the H{\" o}lder norms decay at a rate of
\begin{align}
E_{\a,(k)} &\approx e^{- c_\a (1 + \de)^k }
\end{align}
with $c_\a > 0$ for all 
\[\a < \a^* = \fr{(1 + \de)}{5 + 9 \de + 4 \de^2}\]
so that the resulting solution can have any H{\" o}lder regularity less than $1/5$.

Let us discuss how to make this heuristic argument more precise.

\paragraph{Rigorous proof of H{\" o}lder regularity}

We can assume without loss of generality that $e_{R,(0)} < 1$ after reindexing the stages, and furthermore that 
\[ \psi_{(0)} = \vect{ \log e_R \\ \log(e_v/e_R) \\ \log \Xi }_{(0)} \]
belongs to the stable octant $\OO(-, +, +)$.

Define
\begin{align}
E_{\a, (k)} &\equiv  \left( \fr{e_v}{e_R} \right)_{(k)}^{\a/2} \Xi_{(k)}^\a e_{R,(k)}^{1/2 - \de \a} \label{eq:EalphaAgain}
\end{align}
  
Then 
\begin{align}
\log E_{\a, (0)} &= [ 1/2 - 2 \de \a, \a/2, \a] \psi_{(0)}
\end{align}
and upon the first iteration
\begin{align}
\log E_{\a,(1)} &= \log C_\a + [ 1/2 - 2 \de \a, \a/2, \a] \left[ T_\de \right] \psi_{(0)} 
\end{align}
for some constant $C_\a > 1$ that can be expressed in terms of $Z$ and the constant given by the Lemma (\ref{lem:iterateLem}).  In general, by induction, we can write
\begin{align}
\log (E_{\a,(k)} / C_\a^k ) &= [ 1/2 - 2 \de \a, \a/2, \a] \left[ T_\de^k \right] \psi_{(0)} 
\end{align}
where we have absorbed the exponential growth factor into the left hand side.

By the Lagrange interpolation formula, we can write $T_\de^k$ as a linear combination of the operators which project to the eigenspaces of $T_\de$
\begin{align}
\begin{split}
T_\de^k = (1 + \de)^k \left[ \fr{T_\de(T_\de - 1)}{(1 + \de)\cdot\de} \right] &+ 1^k \left[ \fr{(T_\de - (1 + \de))T_\de}{(-\de)\cdot 1} \right] \\
&+ 0^k \left[ \fr{(T_\de - (1 + \de))(T_\de - 1)}{-(1 + \de)\cdot (-1)} \right] 
\end{split} \\
= (1 + \de)^k \left[ \fr{T_\de(T_\de - 1)}{(1 + \de)\cdot\de} \right] &+ \left[ \fr{(T_\de - (1 + \de))T_\de}{(-\de)\cdot 1} \right]
\end{align}

This formula is close to proving our desired asymptotic, since it implies that
\begin{align}
\log (E_{\a,(k)} / C_\a^k ) &= (1 + \de)^k [ 1/2 - 2 \de \a, \a/2, \a] \left[ \fr{T_\de(T_\de - 1)}{(1 + \de)\cdot\de} \right] \psi_{(0)} + O(1) 
\end{align}

One way to conclude the proof at this point is to first observe that 
\begin{align}
\left[ \fr{T_\de(T_\de - 1)}{(1 + \de)\cdot\de} \right] \psi_{(0)} &= c_{(0)} \psi_+ \label{eq:findCo}
\end{align}
and to prove that $c_{(0)}$ is positive.  To conclude that the coefficient $c_{(0)}$ is positive, we can use the fact that $[1, 0, 0]$ is a row eigenvector for the dominant eigenvalue $\la_+$.  Namely, we can determine $c_{(0)}$ by applying $[1, 0, 0]$ to both sides of equation (\ref{eq:findCo}).  Then, using the fact that $[1, 0, 0]$ is invariant under projection to the $\la_+$ eigenspace, it suffices to observe that the first entry
\begin{align}
[1, 0, 0 ] \left[ \fr{T_\de(T_\de - 1)}{(1 + \de)\cdot\de} \right] \psi_{(0)} &= [1, 0, 0] \psi_{(0)} \\
&= \log e_{R,(0)}
\end{align} 
is negative, because
\begin{align}
[1, 0, 0] \psi_+ &= - (1 + \de)
\end{align}
is negative as well.

This calculation establishes a super-exponential decay estimate on the bounds of the H{\" o}lder norms
\[ E_{\a,(k)} = C_\a^k e^{- c_\a (1 + \de)^k + O(1)} \quad \quad \a < \a^* \]
for some constants $C_\a, c_\a > 0$, and therefore concludes the proof of the H{\" o}lder regularity.

\subsubsection{ Regularity of the Pressure }

Building on the calculations and methods of Sections (\ref{sec:regVelocity}) and (\ref{sec:asympParams}), we can quickly determine what other quantities remain under control during the iteration.

For example, the pressure increment satisfies the estimates
\begin{align}
||P_{(k)}||_{C^0} &\leq C e_{R,(k)} \\
\co{ \nab P_{(k)} } &\leq C N_{(k)} \Xi_{(k)} e_{R,(k)} \\
&\leq C \left( \fr{e_v}{e_R} \right)_{(k)}^{1/2} e_{R,(k)}^{- 2 \de} \Xi_{(k)} e_{R,(k)}
\end{align}
and just as in the computations at the beginning of Section (\ref{sec:regVelocity}), writing the time derivative as 
\[ \pr_t = (\pr_t + v \cdot \nab) - v \cdot \nab \]
reveals that the time derivative of $P$ obeys the same quality of estimates as the spatial derivatives, leading to the bounds
\begin{align}
||P_{(k)}||_{C^0} &\leq C e_{R,(k)} \\
\co{ \nab_{t,x} P_{(k)} } &\leq C \left( \fr{e_v}{e_R} \right)_{(k)}^{1/2} e_{R,(k)}^{- 2 \de} \Xi_{(k)} e_{R,(k)}
\end{align}
which by interpolation give H{\" o}lder estimates of the form
\begin{align}
|| P_{(k)} ||_{C^\a_{t,x}} &\leq C \left( \fr{e_v}{e_R} \right)_{(k)}^{\a/2} \Xi_{(k)}^\a e_{R,(k)}^{1 - 2 \de \a}
\end{align}
Following the methods of Section (\ref{sec:asympParams}), define
\begin{align}
E_{\a, (k)} &\equiv \left( \fr{e_v}{e_R} \right)_{(k)}^{\a/2} \Xi_{(k)}^\a e_{R,(k)}^{1 - 2 \de \a}
\end{align}
so that $|| P_{(k)} ||_{C^\a_{t,x}} \leq C E_{\a,(k)}$ and
\begin{align}
\log (E_{\a, (k)} / C_\a^k ) &= [1 - 2 \de \a, \fr{\a}{2}, \a ] T_\de^k \psi_{(0)}
\end{align}

As in Section (\ref{sec:asympParams}), we have that $E_{\a, (k)}$ decays super exponentially when
\begin{align}
[1 - 2 \de \a, \fr{\a}{2}, \a ] \psi_+ &= [1 - 2 \de \a, \fr{\a}{2}, \a ] \left(\vect{-1 \\ 0 \\ \fr{5}{2} } + \de \vect{-1 \\ 1 \\ 2} \right)
\end{align}
is negative.  In this case,
\begin{align}
[1 - 2 \de \a, \fr{\a}{2}, \a ] \psi_+ &= -(1 + \de) + \a ( \fr{5}{2} + (1 + \de) 2 \de + \fr{\de}{2} + 2\de )
\end{align}
which is negative when
\[ \a < 2 \a^* = \fr{1 + \de}{5 + 9 \de + 4 \de^2 } \]
Thus, for the solution $p$, $p \in C_{t,x}^\a$ when $\a < 2 \a^*$.

\subsubsection{ Compact Support in Time } \label{sec:cpctSuppInTime}

According to the choices made in section (\ref{sec:energyChoice}), the time interval $T_{(k)}$ containing the support of the solutions grows by at most a multiple of the time scale $\th_{(k)} = \Xi_{(k)}^{-1} e_{v,(k)}^{-1/2}$ after each iteration.  Therefore, the resulting solution has compact support in time if the series
\begin{align} \sum_{(k)} \th_{(k)} < \infty \label{eq:seriesTimes} \end{align} converges.  Note that the time scale $\th_{(k)}^{-1}$ is essentially the Lipschitz norm of the approximate solution $v_{(k)}$.  The fact that the natural time scale tends to $0$ as we add in the finer scales $\Xi_{(k)}^{-1}$ is ultimately tied to the fact that the solution we construct is far from being a Lipschitz vector field.

To determine whether the series (\ref{eq:seriesTimes}) is finite, it suffices to show that the logarithm of the time scale
\begin{align}
\log \th_{(k)} &= -\fr{1}{2} \log(e_{v,(k)}) - \log \Xi_{(k)} \\
&= - \fr{1}{2} \log e_{R,(k)} - \fr{1}{2} \log( e_v/e_R )_{(k)} - \log \Xi_{(k)} 
\end{align}
decreases super-exponentially, and for this purpose it is enough to check that
\begin{align}
[-1/2, -1/2, -1 ] \psi_+ &= [-1/2, -1/2, -1 ] \left(\vect{-1 \\ 0 \\ \fr{5}{2} } + \de \vect{-1 \\ 1 \\ 2} \right) \\
&= \fr{(1 + \de)}{2} - \fr{\de}{2} - (\fr{5}{2} + 2 \de) < - 2 
\end{align}
is negative, which is clearly the case.

Actually, one can even show that the time interval containing the solution can be made arbitrarily small by first observing that that
\begin{align}
\th_{(0)} &= \Xi_{(0)}^{-1} e_{v,(0)}^{-1/2} \\
&= \Xi_{(0)}^{-1}[ Z^{-\fr{1}{2(1 + \de)}} e_{R,(0)}^{-\fr{1}{2(1 + \de)}} ] 
\end{align}
can be made arbitrarily small by choosing $Z$ large, and then using the evolution rules for the above formula to observe that one can easily arrange
\begin{align}
\th_{(k+1)} &\leq \fr{\th_{(k)}}{2}
\end{align}
for all $(k) = 0, \ldots, \infty$ once $Z$ is chosen large enough.

\subsubsection{ Nontriviality of the Solution } \label{sec:nontriviality}

In order to show that the solutions produced by this process are nontrivial, we prove that the energy 
\[ \fr{1}{2} \int |v_{(k)}|^2 dx \]
of $v_{(k)} = \sum_{(j) = 0}^k V_{(j)}$ eventually increases in $k$ on the nonempty interval $I_{(0)}$.

Suppose that $t^* \in I_{(0)}$. To verify that $\fr{1}{2}\int |v_{(k)}|^2(t^*, x) dx$ eventually increases with $(k)$, we calculate
\begin{align}
\int \fr{|v_{(k+1)}|^2}{2}(t^*, x) dx - \int \fr{|v_{(k)}|^2}{2}(t^*, x) dx &= \int [ \fr{|v_{(k)} + V_{(k)}|^2}{2} - \fr{|v_{(k)}|^2}{2} ] dx \\
&= \int \fr{|V_{(k)}|^2}{2} dx +  \int v_{(k)} \cdot V_{(k)} dx \label{eq:energyInc}
\end{align}

Our choice of $e_{(k)}$ assures us that, we have a lower bound of
\begin{align}
\int \fr{|V_{(k)}|^2}{2}(t^*, x) dx &\geq A e_{R,(k)}
\end{align}
for some constant $A$, since Lemma (\ref{lem:iterateLem}) implies that the difference
\[
\int \fr{|V_{(k)}|^2}{2} dx - \int e_{(k)}(t) dx
\]
is bounded by
\begin{align}
| \int \fr{|V_{(k)}|^2}{2}(t,x) dx - \int e_{(k)}(t) dx | &\leq C \fr{e_{R,(k)}}{N_{(k)}}
\end{align}
and $N_{(k)}$ is an increasing sequence whose first term
\begin{align}
 N_{(0)} &= Z^2 \left( \fr{e_v}{e_R} \right)_{(0)}^{1/2} e_{R,(0)}^{- 2 \de}  \\
 &= Z^{2 + \fr{\de}{2(1 + \de)}} e_{R,(0)}^{-\de( 2 + \fr{1}{2(1 + \de)})}
\end{align}
can be made arbitrarily large by choosing $Z$ large.

Thus, in order to check that the energy increases for each stage $(k)$, it suffices to bound the inner product
\begin{align}
 {<} v_{(k)}, V_{(k)} {>}_{L^2(\T^3)} &= \int v_{(k)} \cdot V_{(k)} dx \label{eq:vVcorrel}
 \end{align}
which measures the correlation between the correction $V_{(k)}$ and the velocity field $v_{(k)}$.  

We expect the inner product (\ref{eq:vVcorrel}) to be fairly small because $V_{(k)}$ oscillates at a much higher frequency.  Thus, the first step to estimating the inner product (\ref{eq:vVcorrel}) is to write
\[ V_{(k)} = \nab \times W_{(k)} \]
and integrate by parts to see that
\begin{align}
 {<} v_{(k)}, V_{(k)} {>} &= \int v_{(k)} \cdot \nab \times W_{(k)} dx \\
 &= \int \nab \times v_{(k)} \cdot W_{(k)} dx
\end{align}
which, according to Lemma (\ref{lem:iterateLem}) and the bounds $\co{ \nab v_{(k)} } \leq \Xi_{(k)} e_{v,(k)}^{1/2}$ coming from the definition of frequency and energy levels, implies a bound
\begin{align}
|{<} v_{(k)}, V_{(k)} {>}| &\leq C [ \Xi_{(k)} e_{v,(k)}^{1/2} ] \cdot [\fr{e_{R,(k)}^{1/2}}{\Xi_{(k)} N_{(k)}} ] \\
&\leq C \fr{ e_{v,(k)}^{1/2} e_{R,(k)}^{1/2}}{N_{(k)}} \\
&= C \left(\fr{ e_v}{e_R} \right)_{(k)}^{1/2} N_{(k)}^{-1} e_{R,(k)} \\
&= C ( Z^{-2} e_{R,(k)}^{2 \de} )e_{R,(k)}
\end{align}

Applying this estimate to the formula (\ref{eq:energyInc}) for the energy increment, we can even ensure that the energy increment
\begin{align}
\int \fr{|v_{(k+1)}|^2}{2}(t^*, x) dx - \int \fr{|v_{(k)}|^2}{2}(t^*, x) dx &\geq A e_{R,(k)} - C ( Z^{-2} e_{R,(k)}^{2 \de} )e_{R,(k)}
\end{align}
is positive at the initial stage $(k) = 0$ by choosing $Z$ large.  Since $e_{R,(k)}^{2 \de}$ decreases, the energy increment will be positive at all future stages as well.

Thus, the solution obtained at the end of the process is nontrivial because the energy increases at $t^*$ after each iteration, and because $e_{R,(0)}$ was arbitrary, we can make the energy of the solutions arbitrarily large.

This completes the proof of Theorem (\ref{thm:theMainThm}).

%% file: gluing.tex
Here we note that the method of the previous section can be extended to prove Theorem (\ref{thm:mainThm2}) on the gluing of solutions, which states the following.

\begin{thm} \label{thm:mainThm2repeat}
For every smooth solution $(v, p)$ to incompressible Euler on $(-2, 2) \times \T^3$, there exists a H\"{o}lder continuous solution $({\bar v}, {\bar p})$ which coincides with $(v, p)$ on $(-1,1) \times \T^3$ but is equal to a constant outside of $(-3/2, 3/2) \times \T^3$.
\end{thm}
\begin{proof}
By taking a Galilean transformation if necessary, we can assume without loss of generality that the solution has $0$ total momentum $\int v^l(t,x) dx = 0$.  Let ${\bar \eta}(t)$ be a smooth cutoff function equal to $1$ in a neighborhood of $[-1, 1]$, and equal to $0$, outside $(-6/5, 6/5)$.  Then the cut-off velocity and pressure $v_{(0)}(t,x) = {\bar \eta}(t) v(t,x)$, $p_{(0)}(t,x) = {\bar \eta}(t) p$ form a smooth solution to the Euler-Reynolds equations with compact support when coupled to a smooth stress tensor $R_{(0)}^{jl}$ which solves
\ali{
\pr_j R_{(0)}^{jl} &= {\bar \eta}'(t) v^l + \pr_j[ ({\bar \eta}^2(t) - {\bar \eta}(t))v^j v^l  ] \label{eq:needMyFirstR}
}
The existence of a smooth $R^{jl}$ solving (\ref{eq:needMyFirstR}) follows from the fact that the right hand side has integral $0$.

The proof of Theorem (\ref{thm:mainThm2repeat}) then proceeds by iterating Lemma (\ref{lem:iterateLem}) just as in the proof of Theorem (\ref{thm:theMainThm}).  Here the only difference is that the sets $I_{(k)}$ containing the support of each $R_{(k)}$ are unions of two intervals $I = I^1 \cup I^2$, with $I^1_{(k)} \subseteq (-\infty, 1)$ and $I^2_{(k)} \subseteq (1, \infty)$.  

We still choose energy functions of the form
\begin{align}
 e_{(k)}^{1/2} (t) &= (20 K e_{R,(k)})^{1/2} \eta_{\tau_{(k)}} \ast \chi_{(k)}(t)
\end{align}
with
\[ \chi_{(k)}(t) = \chi_{I_{(k)} \pm (\th_{(k)} + \tau_{(k)})}(t) \] 
being the characteristic function of a slight expansion of $I_{(k)}$.  According to the computations in Section (\ref{sec:cpctSuppInTime}), we can ensure that the corrections, which live on small neighborhoods of the supports of $R_{(k)}$, never touch the intervals $(-\infty, -3/2] \cup [-1,1] \cup [-3/2, \infty)$ by taking the initial frequency level $\Xi_{(0)}$ for $(v_{(0)}, p_{(0)}, R_{(0)} )$ to be sufficiently large.  Applying another Galilean transformation to return to the original frame of reference, we have the Theorem.
\end{proof}

%% file: onOnsager.tex
As mentioned at the end of Section (\ref{sec:mainLem}), a stronger form of Lemma (\ref{lem:iterateLem}) would imply Onsager's conjecture.  Here we investigate what could be proven, given Conjecture (\ref{conj:idealCase}).  

\begin{thm} \label{thm:conditionalOnsag}
Assume conjecture (\ref{conj:idealCase}).  Then for every $\de > 0$, there exist nontrivial weak solutions $(v, p)$ to the Euler equations on $\R \times \T^3$ such that for
\[ \a^* = \fr{1}{3 + 2 \de} \]
we have
\begin{align}
v &\in C_{t,x}^\a(\R \times \R^3) \quad \quad \mbox{ for all } \a < \a^* \\
 p &\in C_{t,x}^{2\a}(\R \times \R^3)
\end{align}
whose support is contained in a bounded time interval.

Furthermore, one can arrange that the energy
\begin{align}
E(t) &= \fr{1}{2} \int |v|^2(t,x) dx 
\end{align}
is $C^1$ as a function of $t$.  In particular, the energy will increase or decrease in certain time intervals.
\end{thm}
{\bf Remark}.  In fact, the proof shows that, at least for the solutions constructed by this method, the energy $E(t)$ will necessarily belong to $C^{0, \a}$ for all $\a < 1$.
\begin{proof}
We iterate the conjectured Lemma (\ref{sec:lemImpThm}) so that the energy levels
\begin{align}
\fr{e_{v,(k)}^{1/2} e_{R,(k)}^{1/2}}{N_{(k)}} = e_{R,(k+1)} &= \fr{e_R^{1 + \de}}{Z}
\end{align}
decay super-exponentially.  This Ansatz leads to a choice of
\begin{align}
N_{(k)} &= Z \left( \fr{e_v}{e_R} \right)^{1/2}_{(k)} e_{R,(k)}^{- \de} \\
\Xi_{(k + 1)} &= C Z \left( \fr{e_v}{e_R} \right)^{1/2}_{(k)} e_{R,(k)}^{- \de} \Xi_{(k)} \\
\left( \fr{e_v}{e_R} \right)_{(k+1)} &= Z e_{R,(k)}^{-\de}
\end{align}
resulting in a parameter evolution matrix
\begin{align}
T_\de &= \mat{ccc}{1 + \de & 0 & 0 \\ - \de & 0 & 0 \\ - \de & 1/2 & 1}
\end{align}
for the parameter logarithms
\begin{align}
\psi_{(k)} &= \vect{ \log e_R \\ \log(e_v/e_R) \\ \log \Xi}_{(k)} 
\end{align}
Since $T_\de$ satisfies
\begin{align}
( T_\de - (1 + \de) ) T_\de (T_\de - 1) &= 0 
\end{align}
the stable direction for iterating $T_\de$ corresponds to portion of the $1 + \de$ eigenspace of $T_\de$ in the stable octant
\begin{align} 
\OO(-, +, +) &= \left\{ \vect{a \\ b \\ c} ~|~ a < 0, b > 0, c > 0 \right\} 
\end{align}
in which the parameter logarithms $\psi_{(k)}$ eventually reside.

To find a vector in $\OO$ belonging to the $1 + \de$ eigenspace, we compute
\begin{align}
\psi_+ &= - \de^{-1} T_\de (T_\de - 1)\vect{1 \\ 0 \\ 0} \\
&= - T_\de \vect{1 \\ -1 \\ -1} \\
\psi_+ &= \vect{-1 \\ 0 \\ \fr{3}{2} } + \de \vect{-1 \\ 1 \\ 1}
\end{align}

In order to determine which H{\" o}lder norms stay under control during the iteration, we again observe that the bound for the spatial derivative of the corrections $V$ and $P$ also controls their full space-time derivative since
\[ \pr_t = (\pr_t + v \cdot \nab) - v \cdot \nab \]
and the material derivatives obeys even better estimates.  Therefore we have an estimate
\begin{align}
\| V_{(k)} \|_{C^\a_{t,x}} &\leq C (N_{(k)} \Xi_{(k)})^\a e_{R,(k)}^{1/2} \\
&\leq C' \Xi_{(k)}^\a\left( \fr{e_v}{e_R} \right)^{\a/2}_{(k)} e_{R,(k)}^{1/2 - \de\a} \label{eq:calphaIdealV}\\
\| P_{(k)} \|_{C^\a_{t,x}} &\leq C (N_{(k)} \Xi_{(k)})^\a e_{R,(k)} \\
&\leq C' \Xi_{(k)}^\a\left( \fr{e_v}{e_R} \right)^{\a/2}_{(k)} e_{R,(k)}^{1 - \de\a} \label{eq:calphaIdealP}
\end{align}

As in Section (\ref{sec:lemImpThm}), we take the logarithm of (\ref{eq:calphaIdealV}) to see that the solution 
\[ v \in C_{t,x}^\a \] 
if
\begin{align}
[(1/2 - \de \a), \a/2, \a]\left(\vect{-1 \\ 0 \\ \fr{3}{2} } + \de \vect{-1 \\ 1 \\ 1} \right) &= - \fr{(1 + \de)}{2} + \de (1 + \de) \a + \fr{\de \a}{2} + \left( \fr{3 \a}{2} + \de \a \right) \\
&= - \fr{(1 + \de)}{2} + \fr{\a}{2} (3 + 2 \de) (1 + \de)
\end{align}
is negative.  Therefore, the norm whose bounds grow exponentially corresponds to the regularity $\a^* = \fr{1}{3 + 2 \de}$.

Similarly, we see that $p \in C_{t,x}^\a$ for $\a < 2 \a^*$ by taking a logarithm of  computing
\begin{align}
[(1 - \de \a), \a/2, \a]\left(\vect{-1 \\ 0 \\ \fr{3}{2} } + \de \vect{-1 \\ 1 \\ 1} \right) &=  - (1 + \de) + \fr{\a}{2} (3 + 2 \de) (1 + \de)
\end{align}

\subsection{Higher Regularity for the Energy} \label{sec:higherEnergyReg}

Here we observe that the energy 
\begin{align}
E(t) &= \int |v|^2(t,x) dx
\end{align}
of the solution $v$ that we ultimately construct actually enjoys better regularity than the solution $v$ itself has in the time variable.  Namely, the construction necessarily results in a $v$ for which $E(t) \in C_t^\a$ for all $\a < 1$, and by choosing $e_{(k)}$ carefully at each stage, we can even arrange that $E(t) \in C^1(\R)$.  

Let us begin by writing the energy as a sum of energy increments
\begin{align}
E(t) &= \sum_{k = 0}^\infty E_{(k)} \\
E_{(k)}(t) &= \int ( \fr{|v_{(k+1)}|^2}{2}(t,x) - \fr{|v_{(k)}|^2}{2}(t,x) ) dx \\
&= \int \fr{|V_{(k)}|^2}{2} dx + \int v_{(k)} \cdot V_{(k)} dx \label{eq:energyIncrement1}
\end{align}

For the ideal case with ${\bf b}^{-1}$, the Main Lemma (\ref{lem:iterateLem}) implies that
\begin{align}
|| \int_{\T^3} |V_{(k)}|^2 dx - \int_{\T^3} e_{(k)}(t) dx ||_{C^0} &\leq C \fr{e_{R,(k)}}{N_{(k)}} \label{eq:energyPrescribedIdeal}\\
|| \fr{d}{dt} ( \int_{\T^3} |V_{(k)}|^2 dx - \int_{\T^3} e_{(k)}(t) dx ) ||_{C^0} &\leq C \fr{\Xi_{(k)} e_{v,(k)}^{1/2} e_{R,(k)} }{N_{(k)}} \label{eq:DtenergyPrescribedIdeal}
\end{align}

In fact, for the ideal case, the quantity
\[ \Xi_{(k)} e_{v,(k)}^{1/2} e_{R,(k)} \]
is exactly the critical quantity which grows exponentially, in the sense that
\begin{align}
\Xi_{(k + 1)} e_{v,(k + 1)}^{1/2} e_{R,(k + 1)} &= C (N_{(k)} \Xi_{(k)} e_{R,(k)}^{1/2} \fr{e_{v,(k)}^{1/2} e_{R,(k)}^{1/2}}{N_{(k)}} \\
&= C \Xi_{(k)} e_{v,(k)}^{1/2} e_{R,(k)}
\end{align}

Since $N_{(k)}$ grows super-exponentially, we see from (\ref{eq:DtenergyPrescribedIdeal}) that the error for prescribing the energy of the correction can actually be summed in $C^1(\R)$ for any admissible choices of energy functions.  Namely, from (\ref{eq:energyIncrement1}), we have
\begin{align}
\fr{d}{dt} E_{(k)}(t) &= \fr{d}{dt} \int e(t) dx + \fr{d}{dt} \int v_{(k)} \cdot V_{(k)} dx + o(1) \label{eq:energyIncrement2}
\end{align}
where $o(1)$ goes to $0$ in $C_t^0$ uniformly in $(k)$.

We can also sum the correlation term
\begin{align}
\| \fr{d}{dt} \int v_{(k)} \cdot V_{(k)} dx \|_{C_t^0} &\leq C \fr{ \Xi_{(k)} e_{v,(k)} e_{R,(k)}^{1/2}}{N_{(k)}} \label{eq:wantBoundEnergyVariation} \\
&\leq C \Xi_{(k)} e_{v,(k)}^{1/2} e_{R,(k)}^{1 + \de} \\
&= o(1)
\end{align}
which implies that 
\begin{align}
\| \fr{d}{dt} E_{(k)}(t) - \fr{d}{dt} \int e_{(k)}(t) dx \|_{C_t^0} &= o(1)
\end{align}
where $e_{(k)}(t)$ is the energy function chosen at stage $(k)$.

To establish the estimate (\ref{eq:wantBoundEnergyVariation}), we write
\begin{align}
\fr{d}{dt} \int v_{(k)} \cdot V_{(k)} dx &= \fr{d}{dt} \int {<} v_{(k)}, \nab \times W_{(k)} {>} dx \\
&= \fr{d}{dt} \int {<} \nab \times v_{(k)} , W_{(k)} {>} dx \\
&= \int (\pr_t  + v_{(k)}^j \pr_j) {<} \nab \times v_{(k)} , W_{(k)} {>} dx \\
\begin{split}
&= \int {<} (\pr_t  + v_{(k)}^j \pr_j) \nab \times v_{(k)}, W_{(k)} {>} dx \\
&+ \int {<} \nab \times v_{(k)}, (\pr_t  + v_{(k)}^j \pr_j) W_{(k)} {>} dx
\end{split}
\end{align}

The second of these terms can be estimated using the bounds for $W_{(k)}$ in Lemma (\ref{lem:iterateLem}).
\begin{align}
\| \int_{\T^3} {<} \nab \times v_{(k)}, (\pr_t  + v_{(k)}^j \pr_j) W_{(k)} {>} dx \|_{C^0_t} &\leq C || \nab v_{(k)} ||_{C^0} || (\pr_t  + v_{(k)}^j \pr_j) W_{(k)} ||_{C^0} \\
&\leq C \left(\Xi_{(k)} e_{v,(k)}^{1/2}\right) \left(\fr{e_{v,(k)}^{1/2} e_{R,(k)}^{1/2}}{N_{(k)}} \right) \\
&\leq C (\Xi_{(k)} e_{v,(k)}^{1/2}) e_{R,(k)}^{1 + \de} = o(1)
\end{align}

The first term can be estimated by commuting the curl $\nab \times$ and the material derivative $\pr_t + v \cdot \nab$.
\begin{align}
\int {<} (\pr_t  + v_{(k)}^j \pr_j) \nab \times v_{(k)}, W_{(k)} {>} dx &= \int {<} \nab \times (\pr_t  + v_{(k)} \cdot \nab )v_{(k)}, W_{(k)} {>} dx + \int {<} u_{(k)}, W_{(k)} {>} dx \notag \\
&= E_{1,(k)} + E_{2,(k)} \\
u_{(k)} &= (\pr_t  + v_{(k)} \cdot \nab) \nab \times v_{(k)} - \nab \times (\pr_t  + v_{(k)} \cdot \nab )v_{(k)}
\end{align}

Using the definition $(\nab \times w)^l = \ep^l{}_{ab} \pr^a w^b$, we can compute the commutator $u_{(k)}$
\begin{align}
( \nab \times (\pr_t + v_{(k)} \cdot \nab) v_{(k)} )^l &= \ep^l{}_{ab} \pr^a[(\pr_t + v_{(k)}^j \pr_j ) v_{(k)}^b ] \\
&= (\pr_t + v_{(k)}^j \pr_j) \ep^l{}_{ab} \pr^a v_{(k)}^b + \ep^l{}_{ab} \pr^a v_{(k)}^j \pr_j v_{(k)}^b \\
u_{(k)}^l &= - \ep^l{}_{ab} \pr^a v_{(k)}^j \pr_j v_{(k)}^b
\end{align}
allowing us to bound the term $E_{2,(k)}$ by comparing it to the exponential growth of $\Xi_{(k)} e_{v,(k)}^{1/2} e_{R,(k)}$
\begin{align*}
\| \int {<} u_{(k)}, W_{(k)} {>} dx \|_{C^0_t} &\leq || u_{(k)} ||_{C_t^0 L_x^{3/2}} || W_{(k)} ||_{C_t^0 L_x^3} \\
&\leq C || \nab v_{(k)} ||_{C_t^0 L_x^3}^2 || W_{(k)} ||_{C_t^0 L_x^3} \\
&\leq C \left(\Xi_{(k)} e_{v,(k)}^{1/2} \right)^2 \left(\Xi_{(k)}^{-1} N_{(k)}^{-1} e_{R, (k)}^{1/2} \right) \\
&\leq C \Xi_{(k)} e_{v,(k)}^{1/2} \fr{ e_{v,(k)}^{1/2} e_{R, (k)}^{1/2} }{N_{(k)}} \\
&\leq C \Xi_{(k)} e_{v,(k)}^{1/2} e_{R,(k)}^{1 + \de}
\end{align*}
Here we chose the $L^3$ norm in space to simply to draw an analogy to the trilinear terms appearing in the proof of energy conservation in \cite{CET}.

To bound the other term $E_{1,(k)}$, we can use the Euler-Reynolds equation
\begin{align}
(\pr_t + v_{(k)}\cdot \nab)v_{(k)}^l &= \pr^l p_{(k)} + \pr_j R_{(k)}^{jl} 
\end{align}
to estimate
\begin{align}
\| \int {<} \nab \times (\pr_t  + v_{(k)} \cdot \nab )v_{(k)}, W_{(k)} {>} dx \|_{C^0_t} &\leq ( || \nab^2 p ||_{C^0} + || \nab^2 R ||_{C^0} ) || W_{(k)} ||_{C^0} \\
&\leq C \Xi_{(k)}^2 e_{v,(k)} \fr{e_{R,(k)}^{1/2}}{\Xi_{(k)} N_{(k)} } \\
&\leq C \Xi_{(k)} e_{v,(k)}^{1/2} e_{R,(k)}^{1 + \de} = o(1)
\end{align}

We have therefore established that
\begin{align}
\| \fr{d}{dt} E_{(k)}(t) - \fr{d}{dt} \int e_{(k)}(t) dx \|_{C_t^0} &= o(1)
\end{align}

If we desire the energy of the resulting solution to be $C^1$ in time, we must be more careful than using the bounds
\begin{align}
\| \fr{d}{dt} e_{(k)}(t) \|_{C^0} &= 2  \| e^{1/2}_{(k)}(t) \fr{d}{dt} e^{1/2}_{(k)}(t) \|_{C^0} \\
&\leq C \Xi_{(k)} e_{v,(k)}^{1/2} e_{R,(k)} \\
&\leq C^k  \label{bound:badExpEnergyBound}
\end{align}
required by Lemma (\ref{lem:iterateLem}), since these bounds grow exponentially.

However, we can easily choose $e_{(k)}$ so that the bounds (\ref{bound:badExpEnergyBound}) are not sharp.

For example, averaging over longer time intervals
\begin{align}
\tau_{(k)} &= B_0^k  \Xi_{(k)}^{-1} e_{v,(k)}^{-1/2}
\end{align}
for the time mollifier in (\ref{eq:theTimeMoll}), (\ref{eq:weChoseTauk}) will lead to bounds which decay exponentially for
\begin{align}
\| \fr{d}{dt} e_{(k)}(t) \|_{C^0} &\leq C A^{-k}
\end{align}
if $B_0$ is chosen sufficiently large.
\end{proof}

%% file: mollifyPrep.tex
To prepare for the proof, we present the following, well-known method concerning the general rate of convergence of mollifiers.

\begin{lem} \label{lem:mollifierRate}
Suppose that $\eta \in L^1(\R^3) \cap L^N(\R^3)$ and that for any multi-index $\a$ with $1 \leq |\a| \leq N - 1$ the moment vanishing conditions
\begin{align} \label{eq:mollifierConditions}
\int \eta(y) dy &= 1 \\
\int y^\a \eta(y) dy &= 0 \quad \quad ~\forall \a, 1 \leq |\a| \leq N - 1
\end{align}
hold.  

Let $\ep > 0$ and define $\eta_\ep(y) = \ep^{-3} \eta(\fr{y}{\ep})$.  Then there exists a constant $C$ such that
\begin{align} \label{ineq:convergenceOfMollifiers}
|| v - \eta_\ep \ast v ||_{C^0(\T^3)} &\leq C \ep^N || \eta ||_{L^N(\R^3)}  || \nab^N v ||_{C^0(\T^3)} 
\end{align}
for all smooth functions $v$ on $\T^3$
\end{lem}
\begin{proof}
To reduce the number of minus signs appearing in the proof, let us assume $\eta_{\ep}$ is even, and let us restrict to the case $N = 3$ to convey the idea.  To verify (\ref{ineq:convergenceOfMollifiers}),  we can then Taylor expand by repeated integration by parts in the formula
\begin{align*}
\eta_{\ep} \ast v(x) - v(x) &= \int_{\R^3} ( v(x-y) - v(x) ) \eta_{\ep}(y) dy \\
&= \int ( v(x+y) - v(x) ) \eta_{\ep}(y) dy \\
&= \int (v(x + \ep y) - v(x) ) \eta(y) dy \\  
&= \int \left[ \int_0^1 \fr{d}{ds}v(x + s \ep y) ds \right] \eta(y) dy \\
&= \ep \int_0^1 \left[ \int \pr_i v(x+ s\ep y) y^i \eta(y) dy \right] ds \\
&= \ep \int \pr_i v(x) y^i \eta(y) dy + \int \left[ \int_0^1 \fr{d^2}{ds^2}v(x + s \ep y) (1 - s)ds \right] \eta(y) dy \\
&= \ep^2 \int_0^1 \left[ \int \pr_{i_2} \pr_{i_1} v(x + s \ep y) y^{i_1} y^{i_2} \eta(y) dy \right] (1-s) ds \\
&= \fr{\ep^2}{2!} \int \pr_{i_2} \pr_{i_1} v(x) y^{i_1} y^{i_2} \eta(y) dy  + \int \left[ \int_0^1 \fr{d^3}{ds^3}v(x + s \ep y) \fr{(1 - s)^2}{2!}ds \right] \eta(y) dy \\
&= \fr{\ep^3}{2!} \int_0^1 \left[ \int \pr_{i_1} \pr_{i_2} \pr_{i_3}v(x + s \ep y) y^{i_1} y^{i_2} y^{i_3} \eta(y) dy \right] (1 - s)^2 ds
\end{align*}
\end{proof}

{\bf Remark}. \quad The case $N = 2$ of the above lemma is already contained in \cite{deLSzeC1iso}.  In their argument, they take advantage of how the case $N = 2$ permits $\eta_\ep$ to be non-negative.  Actually, the case $N = 2$ also suffices for the proof of the main theorem in the present paper, but we include the more general statement (\ref{ineq:convergenceOfMollifiers}) to convey the flexibility available in this aspect of the argument.

Another way to work out the above details, which will be repeatedly used in the remainder of the proof, is to write
\begin{align}
\eta_\ep \ast v(x) - v(x) &= \int [v(x - y) - v(x)] \eta_\ep(y) dy \\
&= - \int_0^1 \left[ \int \pr_i v(x - sy) y^i \eta_\ep(y) dy \right] ds\\
&= \ep \int_0^1 \int  \pr_i v(x - sy) {\tilde \eta}_\ep^i(y) dy ds
\end{align}
where
\begin{align}
{\tilde \eta}_\ep^i(y) &= - \fr{1}{\ep^3} \fr{y^i}{\ep} \eta(\fr{y}{\ep})
\end{align}
are other functions whose integrals are not normalized to $1$, but which satisfy the same type of estimates as $\eta_\ep$.  

Namely, every function $\eta_\ep$ and $\tilde{\eta}_\ep$ which appears in the argument will satisfy bounds of the form
\begin{align}
\| \nab^k \eta_\ep \|_{L^1(\R^3)} &\leq C_\eta \ep^{-k}
\end{align}
and altogether through the course of the argument, only finitely many different $\eta$ and ${\tilde \eta}$ will be used.

%% file: mollifyingV.tex
To begin the proof of Lemma (\ref{lem:iterateLem}), we choose a mollification $v_\ep^j$ for the velocity field.  For our purposes, it is convenient to take a double mollification
\begin{align}
\eta_{\ep + \ep} &= \eta_{\ep_v} \ast \eta_{\ep_v} \\
v_\ep^j &= \eta_{\ep + \ep} \ast v^j
\end{align}
for reasons which will be explained in the footnote of Section (\ref{coarseScaleFlow}) after we derive a transport equation for $v_\ep$.


We would like to take $\ep_v$ to be as large as possible so that higher derivatives of $v_\ep$ are less costly.  More specifically, the basic building blocks of the construction are of the form
\[ e^{i \la \xi_I} v_I^l \]
and we must ensure that each $v_I$ has frequency smaller than $\la$, so $\ep_v^{-1}$ must be smaller than $\la$ in order of magnitude.

To achieve a frequency of $\Xi' = C \Xi N$ as in the statement of the main lemma (\ref{lem:iterateLem}), we will take $\la$ to have the form 
\begin{align}
\la &= B_\la \Xi N
\end{align}
where $B_\la \geq 1$ is a very large constant which will be chosen at the very end of the proof.

The parameter $\ep_v$ will then have the form.
\begin{align}
\ep_v &= a_v \Xi^{-1} N^{-\ga}\\
0 < \ga &< 1
\end{align}
where $a_v \leq 1$ is a small constant that will also be chosen in the proof.  At this point we can discuss the exponent $\ga$.

Analogously to Section (\ref{sec:ctsMollV}) in the continuous case, we choose $\ep_v$ so that the main term in the error
\begin{align}
(v^j - v_\ep^j) V^l + V^j(v^l - v_\ep^l) &= \sum_I 2[(v - v_\ep)(e^{i \la \xi_I} v_I )]^{jl} + O(\la^{-1}) \\ 
&= Q_{v}^{jl} + O(\la^{-1})
\end{align}
will be small.  Namely, we know a priori, as in (\ref{sec:ctsMollV}), that each $v_I$ will be no larger than
\begin{align}
\co{ v_I } &\leq A e_R^{1/2}
\end{align}
for some absolute constant provided that we require 
\begin{align}
 \tau &\leq b_0 \Xi^{-1} e_v^{-1/2} \label{ineq:tauMustBe}
\end{align}
for a certain constant $b_0$ discussed in Section (\ref{sec:ctsMollV}).  Furthermore, there is a uniform bound on the number of $v_I$ which are nonzero at any given $(t,x)$, so we also have an a priori bound
\begin{align}
\co{ \sum_I |v_I| } &\leq A e_R^{1/2}
\end{align}

Therefore, we have an a-priori bound on 
\begin{align}
\co{  Q_{v} } &\leq A e_R^{1/2} \co{ v - v_\ep} \label{ineq:theVMollC0ineq}
\end{align}

In the ideal case for which there are no losses of ${\bf b}^{-1}$ in the Main Lemma (\ref{lem:iterateLem}), the target size for the new stress would be $\co{ R_1 } \unlhd \fr{e_v^{1/2} e_R^{1/2}}{N}$.  We will therefore choose $\ep_v$ to achieve the target
\begin{align}\label{target:R1v}
\co{  Q_{v} } &\unlhd \fr{1}{50} \fr{e_v^{1/2} e_R^{1/2}}{N}
\end{align}

To approach the target (\ref{target:R1v}), let $\eta_{\ep_v}$ be any smooth mollifier constructed from an even function of compact support which satisfies the vanishing moment conditions of Lemma (\ref{eq:mollifierConditions}) for $N = L$.  We then have an estimate
\begin{align}
|| \eta_{\ep + \ep} \ast v - v ||_{C^0} &= || (\eta_{\ep_v} \ast v - v) + \eta_{\ep_v} \ast ( \eta_{\ep_v} \ast v - v) ||_{C^0} \\
&\leq C \ep_v^L || \nab^L v ||_{C^0} \\
&\leq C \ep_v^L (\Xi^L e_v^{1/2}) \label{eq:boundMollErv}
\end{align}

In anticipation of proving (\ref{target:R1v}), we set $\ga = 1/L$, so that
\begin{align} \label{eq:epvdef}
\ep_v &= a_v \Xi^{-1} N^{-1/L}
\end{align}
and (\ref{eq:boundMollErv}) becomes
\begin{align} \label{ineq:mollifyVerror}
|| \eta_{\ep + \ep} \ast v - v ||_{C^0} &\leq C a_v^L \fr{e_v^{1/2}}{N}
\end{align}
Applying the above inequality to (\ref{ineq:theVMollC0ineq}) gives a bound of
\begin{align}
\co{  Q_{v} } &\leq A a_V^L e_R^{1/2} \fr{e_v^{1/2}}{N} \label{eq:whereWeChooseav}
\end{align}
for the leading term in the stress created by mollifying $v$.  

By comparing (\ref{eq:whereWeChooseav}) with (\ref{target:R1v}), we can choose the constant $a_v$ so that the goal (\ref{target:R1v}) is satisfied.

With the above choice, $\nab v_\ep$ has amplitude $\Xi e_v^{1/2}$ and frequency $\Xi$ to high order 
\begin{align}
|| \pr^\a v_\ep ||_{C^0} &\leq \Xi^{|\a|} e_v^{1/2} \quad \quad \a = 1 , \ldots, L \\
|| \pr^\a v_\ep ||_{C^0} &\leq C_{\a} N^{ (|\a| - L)/L } \Xi^{|\a|} e_v^{1/2} \quad \quad  |\a| > L
\end{align}
which we choose to abbreviate by
\begin{align} \label{bound:nabkvep}
|| \pr^\a v_\ep ||_{C^0} &\leq C_{\a} N^{ (|\a| - L)_+/L }\Xi^{|\a|} e_v^{1/2} \quad \quad \mbox{for all } |\a| \geq 1
\end{align}
We can summarize (\ref{bound:nabkvep}) informally by saying that each derivative of $v_\ep$ up to order $L$ costs a factor of $\Xi$, and then derivatives beyond order $L$ cost $N^{1/L} \Xi$.  The factor $N^{1/L}$ is like an ``interest payment'' which must be paid on top of the usual cost of $\Xi$ for the ``borrowed'' derivatives beyond order $L$.

Note that we do not assume any control over $|| v_\ep ||_{C^0}$; in fact the low frequency part of $v$, though bounded, can be very large when the lemma is used in the proof of Theorem (\ref{thm:theMainThm}).  There, the largest part of the $C^0$ norm comes from the very first step of the iteration.

At this point, we can already check to make sure that we have not fallen short of the goal
\begin{align} \label{target:Qvd1}
|| \nab Q_v||_{C^0} &\unlhd \Xi' e_{R'} \\
&= C (N \Xi) \fr{e_v^{1/2} e_R^{1/2}}{N} \\
&= C \Xi e_v^{1/2} e_R^{1/2}
\end{align}
by estimating
\begin{align} 
|| \nab ( v - v_\ep ) ||_{C^0} &\leq \co{ \nab v } + \co{ \nab v_\ep } \\
\label{crudeDerivV1} &\leq 2 \Xi e_v^{1/2}
\end{align}

Since we expect that $|V| \leqc e_R^{1/2}$, the crude estimate (\ref{crudeDerivV1}) will suffice to prove (\ref{target:Qvd1}).

\paragraph{The choice of $\ep_v$ and the parametrix Expansion} 
\input{whyChooseEpv}

In the following section, we derive a transport equation for $v_\ep$, and derive some estimates that will be necessary for the proof.  At this point, we abbreviate $\eta_\ep$ for $\eta_{\ep_v}$.

%% file: whyChooseEpv.tex
For the parametrix expansion at the end to be effective, it is important to observe that 
\[ \ep_v^{-1} \leqc \Xi N^{1/L} \]
is significantly smaller than
\[ \la \approx \Xi N \]
because $L \geq 2$.  The point here is that the parametrix expansion for the divergence equation proceeds by solving
\begin{align}
\pr_j Q^{jl} &= e^{i \la \xi_I} u_1^l \\
Q^{jl} &= {\tilde Q}_1^{jl} + Q_{(1)}^{jl} \\
{\tilde Q}_1^{jl} &= e^{i \la \xi_I} \fr{q^{jl}(\nab \xi)[u^l] }{\la} \\
\pr_j Q_{(1)}^{jl} &=-  e^{i \la \xi_I} \fr{1}{\la} \pr_j [ q^{jl}(\nab \xi)[u^l] ] \label{eq:theNextEllipticTerm}
\end{align}
The derivative $\pr_j$ in (\ref{eq:theNextEllipticTerm}) will turn out to cost a factor at most $\ep_v^{-1}$ (cf. Section (\ref{sec:oscStressTerms}) below.).  On the other hand, the parametrix gains a factor of $\la^{-1}$.  The way we have chosen $\ep_v$ ensures us that each iteration of the parametrix gains a factor
\begin{align}
\fr{\ep_v^{-1}}{\la} &\leq C \fr{\Xi N^{1/L}}{B_\la N \Xi} \\
&\leq \fr{C}{B_\la N^{(1- 1/L)}} \\
&\leq \fr{ C}{B_\la N^{1/2}}
\end{align}
since $L \geq 2$.

%% file: coarseScaleFlow.tex
In contrast to the argument in \cite{deLSzeHoldCts}, we do not mollify the velocity field in the time variable.  Instead, by using the fact that $v$ obeys the Euler-Reynolds equation
\begin{align} \label{eq:euReynPrep}
(\pr_t + v^a \pr_a) v^j &= \pr^j p + \pr_i R^{ij} 
\end{align}
we derive a transport equation for $v_\ep^j$ as follows.

By convolving (\ref{eq:euReynPrep}) with $\eta_{\ep + \ep}$, we can see that $v_\ep^j = \eta_\ep \ast \eta_\ep \ast v^j$ also obeys its own transport type equation
\begin{align} \label{eq:dtvep}
(\pr_t + v_\ep^a \pr_a) v_\ep^j &= \eta_{\ep + \ep} \ast (\pr^j p + \pr_i R^{ij} ) + Q^j(v,v) \\
&= f_\ep^j \\
Q^j(v,v) &= v_\ep^a \pr_a v_\ep^j - \eta_{\ep + \ep} \ast ( v^a \pr_a v^j )
\end{align}
The quadratic term arises from the failure of the nonlinearity to commute with the averaging, just as the Reynolds stress arises in Section (\ref{motivation}).  In what follows we will derive estimates for the commutator $Q^j(v,v)$ which should be compared to the commutator estimates used in the proof of energy conservation in \cite{CET} and the similar commutator estimates used in \cite{deLSzeC1iso}; we give some remarks about the comparison at the end of the section.  The method of proof involves integral formulas which are very standard, for example in Littlewood-Paley theory.

According to the Definition (\ref{def:subSolDef}), we expect that the cost of the derivative $(\pr_t + v_\ep^a \pr_a)$ should be a factor of $\Xi e_v^{1/2}$, so we expect an estimate of the type $| (\pr_t + v_\ep^a \pr_a) v_\ep^j| \leqc \Xi e_v$.  However, without observing the cancellation between the two parts of the quadratic term, the only estimate we could expect for $Q^j(v,v)$ would take the form\footnote{For example, in $v_\ep^a \pr_a v_\ep^j$, $v_\ep^a$ is bounded in size by a constant and $\pr_a v_\ep^j$ is bounded by $\Xi e_v^{1/2}$. } $|Q^j(v,v)| \leqc \Xi e_v^{1/2}$, and even this estimate would require control over $\co{v}$ which we have not assumed.

The key to observing cancellation in the quadratic term is to use the control over the higher-frequency part of $v$, and obtain cancellation from the lower-frequency parts.  We accomplish this task through the following commutator calculation:
\begin{align}
Q^j(v,v) &= v_\ep^a \pr_a v_\ep^j - \eta_{\ep + \ep} \ast ( v^a \pr_a v^j ) \\
&= v_\ep^a \pr_a ( \eta_\ep \ast \eta_\ep \ast v^j ) - \eta_\ep \ast \eta_\ep \ast ( v^a \pr_a v^j) \\
\begin{split}
&= \eta_\ep \ast ( v_\ep^a \pr_a(\eta_\ep \ast v^j) ) + [v_\ep^a \pr_a, \eta_\ep \ast](\eta_\ep \ast v^j) \\
&- \eta_\ep \ast ( v^a \pr_a(\eta_\ep \ast v^j) ) + \eta_\ep \ast([v^a \pr_a, \eta_\ep \ast]v^j)
\end{split}
\end{align}
where we can express the commutator terms as
\begin{align}
[v_\ep^a \pr_a, \eta_\ep \ast](\eta_\ep \ast v^j) &= \int_{\R^3} (v_\ep^a(x - y) - v_\ep^a(x) )\pr_a(\eta_\ep \ast v^j)(x - y) \eta_\ep(y) dy \\
&= \int_0^1 \int_{\R^3} \pr_i v_\ep^a(x - s y) \pr_a (\eta_\ep \ast v^j)(x - y) y^i \eta_\ep(y) dy ds \\
&= \ep_v \int_0^1 \int_{\R^3} \pr_i v_\ep^a(x - s y) \pr_a(\eta_\ep \ast v^j)(x - y) {\tilde \eta}_\ep^i(y) dy ds \label{eq:velCommut1}
\end{align}
and similarly 
\begin{align}
\eta_\ep \ast([v^a \pr_a, \eta_\ep \ast]v^j)(x)
&= \ep_v \int_0^1 \eta_\ep \ast \left[ \int_{\R^3} \pr_i v^a(x - s y_1) \pr_a v^j(x - y_1) {\tilde \eta}_\ep^i(y_1) dy_1 \right] ds \\
&= \ep_v \int_0^1 \left[ \int_{\R^3 \times \R^3} \pr_iv^a(x - y_2 - s y_1) \pr_a v^j(x - y_2 - y_1) {\tilde \eta}_\ep^i(y_1) dy_1 \eta_\ep(y_2) dy_2 \right] ds \label{eq:velCommut2}
\end{align}

It is now clear that the commutator terms can be estimated using our control of only the derivatives of $v$.  A similar expression for the remaining terms allows us to finally observe the cancellation between the low frequency contributions to $Q$
\begin{align}
\eta_\ep \ast ( v_\ep^a \pr_a(\eta_\ep \ast v^j) ) - \eta_\ep \ast ( v^a \pr_a(\eta_\ep \ast v^j) ) &= \eta_\ep \ast[ (v_\ep^a - v^a) \pr_a(\eta_\ep \ast v^j) ] \\
&= \ep_v \int_0^1 \eta_\ep \ast \left[ \int_{\R^3} \pr_iv^j(x - sy) \pr_a (\eta_\ep \ast v^j)(x - y) {\tilde \eta}_\ep^i(y) dy \right] ds \label{eq:velCommut3}
\end{align}

Each of the terms (\ref{eq:velCommut1}), (\ref{eq:velCommut2}) and (\ref{eq:velCommut3}) leads to a bound
\begin{align}
\co{ Q } &\leq C \ep_v \Xi^2 e_v \\
&\leq C N^{-1/L} \Xi e_v
\end{align}
Differentiating (\ref{eq:velCommut1}), (\ref{eq:velCommut2}) and (\ref{eq:velCommut3}) up to order $L - 1$ costs a factor of $\Xi$ according to the bounds (\ref{bound:nabkv}) and (\ref{bound:nabkvep}), and beyond order $L - 1$, each derivative costs $N^{1/L} \Xi$ as it hits the mollifiers or as we apply the bounds in (\ref{bound:nabkvep}).  These observations lead to the estimates
\begin{align} \label{eq:boundForQderivsHere}
\co{ \nab^k Q } &\leq N^{[(k + 1 - L)_+ - 1]/L} \Xi^{k + 1} e_v
\end{align}
for the commutator.

These bounds are better than the bounds for
\begin{align}
\co{ \nab^k \eta_{\ep + \ep} \ast \nab p } &\leq N^{[(k + 1 - L)_+]/L} \Xi^{k + 1} e_v
\end{align}
coming from (\ref{bound:nabkp}) by a factor of $N^{-1/L}$.  The bounds on 
\begin{align}
\co{ \nab^k \eta_{\ep + \ep} \ast \nab R } &\leq N^{[(k + 1 - L)_+]/L} \Xi^{k + 1} e_R
\end{align}
coming from (\ref{bound:nabkR}) are also better than those for the pressure because $e_R \leq e_v$.  

For $v_\ep^j$ defined as above, we therefore have the bounds
\begin{thm}[Coarse Scale Force Estimates]
\begin{align} \label{bound:lowFreqTransport}
\co{ \nab^\a (\pr_t + v_\ep^j \pr_j) v_\ep^j } &\leq C_\a N^{( 1 + |\a| - L)_+/L} \Xi^{1 + |\a|} e_v
\end{align}
for all $|\a| \geq 0$.\footnote{As the above analysis reveals, the double mollification of Section (\ref{coarseScaleVelocity}) is useful because it allows us to easily prove an estimate for derivatives of order higher than $L$.  Another way to make sure these estimates are available is to choose the mollifier to be a projection $P_{\leq k}$ to frequencies of wavelength $2^{-k}$ less than $\ep$.  With such a choice of mollifier, it is then easy to see that the quadratic term $Q(v,v)$ has bounded support in frequency space, and therefore obeys estimates for its higher derivatives as well.  Thus, the use of the double mollification plays the same role in the analysis as the formal identity ``$P_{\leq k}^2 = P_{\leq k}$'' in the standard Littlewood Paley theory.}
\end{thm}

In particular, for $|\a| = 0$,
\begin{align} \label{bound:fepc0}
\co{ f_\ep^j } = \co{ (\pr_t + v_\ep^a \pr_a) v_\ep^j } &\leq C \Xi e_v
\end{align}

The equation (\ref{eq:dtvep}) has a physical interpretation in terms of the coarse scale flow, which we define here for future reference
\begin{defn} \label{def:csflow} Let $\Phi_s(t,x) = (\Phi^0, \Phi^i) : \R \times \R \times \T^3 \to \R \times \T^3$ be the unique solution to the ODE
\begin{align}
\fr{d}{ds}\Phi_s^0(t,x) &= 1 \\
\fr{d}{ds}\Phi_s^j(t,x) &=  v_\ep^j(\Phi_s(t,x)) \quad \quad j = 1, 2, 3 \\
\Phi_0(t,x) &= (t,x)
\end{align}
We call $\Phi_s(t,x) = (t + s, \Phi_s^j(t,x))$ the coarse scale flow.
\end{defn}

According to equation (\ref{eq:dtvep}), any particle which travels along the coarse scale flow experiences an acceleration
\begin{align}
\fr{d^2}{ds^2} \Phi_s^j(t,x) &= (\pr_t + v_\ep^a \pr_a) v^j(\Phi_s(t,x)) \\
&= f_\ep^j(\Phi_s(t,x)) 
\end{align}
where the acceleration $f_\ep^j$ is given by the right hand side of (\ref{eq:dtvep}) and is bounded by $|f_\ep^j| \leq C \Xi e_v$.


The method used here to obtain commutator estimates differs from the computations in \cite{CET} and \cite{deLSzeC1iso} in that our proof only involves commuting mollifiers with differential operators $v^a \pr_a$ and involves a double mollification trick in order to bound higher order derivatives.  
On the other hand, since both $v^a$ and $v_\ep^a$ are divergence free, the commutator can be expressed in the form 
\[ Q^j(v,v) = \pr_a [ v_\ep^a v_\ep^j - \eta_{\ep + \ep} \ast ( v^a v^j ) ] \]
as in \cite{CET} and \cite{deLSzeC1iso} so in particular the estimates (\ref{eq:boundForQderivsHere}) we have obtained for spatial derivatives of $Q^j$ can be compared to those stated in Lemma 2.1 of \cite{deLSzeC1iso} for commutators of the above form, which are similar in nature although they are not strong enough for the bounds (\ref{eq:boundForQderivsHere}).

%% file: lowFreqTransport.tex
In this section, we use the estimates of Section (\ref{coarseScaleFlow}) to derive estimates for quantities which are transported by the coarse scale flow and for their derivatives.

\subsection{Stability of the Phase Functions} \label{subsec:phaseStability}

We start with the phase functions $\xi_I$ which satisfy the transport equation
\begin{align}
 (\pr_t+ v_\ep^j\pr_j) \xi_I &= 0 \\
 \xi_I(t(I), x) &= {\hat \xi}_I
\end{align}
with initial data ${\hat \xi}_I$ at time $t(I)$ as described in Sections (\ref{transport}) and (\ref{stress}).

Our first objective is to choose the lifespan parameter $\tau$ sufficiently small so that all the phase functions which appear in the analysis can be guaranteed to remain nonstationary in the time interval $|t - t(I)| \leq \tau$, and so that the stress equation studied in Section (\ref{stress}) can be solved.  In order for these requirements to be met, we will choose $\tau$ small enough so that the gradients of the phase functions do not depart dramatically from their initial configurations.  

The {\bf lifespan parameter} $\tau$ will be of the form
\begin{align} \label{eq:tauWillBe}
\tau &= b \Xi^{-1} e_v^{-1/2} \\
b &\leq 1
\end{align}
where $b$ is a dimensionless parameter that must be sufficiently small.  

The following proposition bounds the separation of the phase gradients from their initial values in terms of $b$.

\begin{prop}\label{prop:phaseStability}  Let $\Phi_s$ be the coarse scale flow defined by (\ref{def:csflow}) and let ${\hat \xi_I}(x)$ be the initial conditions defined in Section (\ref{subsec:indexPhase}).

There exists a constant $A$ such that for all $|s| < b \Xi^{-1} e_v^{-1/2}$, and all $I \in \II$
\begin{align}
|\nab \xi_I(\Phi_s(t(I), x)) - \nab {\hat \xi_I}(x)| &\leq A b
\end{align}
As a consequence, 
\begin{align}
\co{~|\nab \xi_I(t, x)) - \nab {\hat \xi_I}|~} &\leq A b 
\end{align}
for all $t$ in the interval $|t - t(I)| \leq b \Xi e_v^{-1/2}$.
\end{prop}

With this proposition in hand, we will require that
\begin{req}\label{req:reqsForB}
\begin{align}
b &\unlhd b_0
\end{align}
\end{req}
where $b_0 < 1$ is a constant ensuring that
\begin{align}
\co{~ |\nab \xi_I |^{-1}~} &\leq 3 |\nab {\hat \xi}_I|^{-1} \\
\co{~ |\nab(\xi_I + \xi_J)|^{-1}~} &\leq 3|\nab({\hat \xi}_I + {\hat \xi}_J)|^{-1}  \quad \quad J \neq {\bar I}
\end{align}
on $|t - t(I)| \leq b_0 \Xi^{-1} e_v^{-1/2}$, and also ensuring that
\begin{align}
\co{~|\nab \xi_I(t, x)) - \nab {\hat \xi_I}(x)|~} &\leq c 
\end{align}
where $c$ guarantees the conditions in Proposition (\ref{prop:needTauSmall}) are both satisfied.

Proposition (\ref{prop:phaseStability}) is a consequence of the inequality
\begin{prop} \label{prop:gronwallPhase0}  There exist absolute constants $C_1, C_2$ such that for all $|s| \leq b \Xi^{-1} e_v^{-1/2}$ we have
\begin{align} \label{ineq:gronwallPhase0}
|\nab \xi_I(\Phi_s(t_I, x)) - \nab {\hat \xi_I}(x)| &\leq C_1 e^{C_2 \Xi e_v^{1/2} |s|} b \\
&\leq C_1 e^{C_2 b} b
\end{align}
\end{prop}
\begin{proof}
We study the evolution of the quantity 
\begin{align}
u(s) &= |\nab \xi_I(\Phi_s(t(I), x)) - \nab {\hat \xi_I}(x)|^2
\end{align}
as the flow parameter $s$ varies.  To prepare to compute $\fr{d u}{ds}$ we first differentiate the transport equation
\begin{align}
(\pr_t + v_\ep^j \pr_j) \xi_I &= 0
\end{align}
to obtain the equation
\begin{align}
(\pr_t + v_\ep^j \pr_j) \pr^l \xi_I &= - \pr^l v_\ep^j \pr_j \xi_I
\end{align}
governing the evolution of $\nab \xi_I$.  Using this law, we compute
\begin{align}
\fr{d}{ds} |\nab \xi_I(\Phi_s(t_I, x)) - &\nab {\hat \xi_I}(x)|^2 = 2(\pr_t + v_\ep^j(\Phi_s) \pr_j) \pr^l \xi_I(\Phi_s)(\pr_l \xi_I(\Phi_s) - \pr_l {\hat \xi}_I) \\
&=  - 2\pr^l v_\ep^j \pr_j \xi_I(\Phi_s)(\pr_l \xi_I(\Phi_s) - \pr_l {\hat \xi}_I) \\
&= - 2\pr^l v_\ep^j (\pr_j \xi_I(\Phi_s) - \pr_j {\hat \xi}_I) (\pr_l \xi_I(\Phi_s) - \pr_l {\hat \xi}_I) - 2 \pr^l v_\ep^j \pr_j {\hat \xi}_I (\pr_l \xi_I(\Phi_s) - \pr_l {\hat \xi}_I)
\end{align}
which gives an estimate of the form
\begin{align}
|\fr{d u}{ds}| &\leq C_2 \Xi e_v^{1/2} u +  2 C_1 \Xi e_v^{1/2} u^{1/2}
\end{align}
In terms of the dimensionless variable ${\bar t} = \Xi e_v^{1/2} s$, the same inequality can be stated
\begin{align} \label{ineq:dudsbar}
|\fr{d u}{d{\bar t}}| &\leq C_2 u + 2 C_1 u^{1/2}
\end{align}
By the time reversal symmetry ${\bar t} \mapsto - {\bar t}$, it suffices to prove (\ref{ineq:gronwallPhase0}) for ${\bar t} > 0$.  In this case, we achieve the bound by first using (\ref{ineq:dudsbar}) to deduce
\begin{align}
\fr{du}{d{\bar t}} - C_2 u &\leq 2 C_1 u^{1/2} \\
e^{- C_2 {\bar t}}\fr{du}{d{\bar t}} - C_2 e^{- C_2 {\bar t}}u &\leq 2 C_1 e^{- C_2 {\bar t}} u^{1/2} \\
\fr{d}{d{\bar t}}(e^{- C_2 {\bar t}}u) &\leq 2 C_1 e^{- C_2{\bar t}} u^{1/2} \label{ineq:dvdspretwist}
\end{align}
In terms of the variable $v = e^{- C_2 {\bar t}}u$, inequality (\ref{ineq:dvdspretwist}) takes the form
\begin{align}
\fr{dv}{d{\bar t}} &\leq 2 C_1  e^{- \fr{C_2{\bar t}}{2}} v^{1/2}
\end{align}
To deduce a bound for $v$, we can integrate the above inequality by observing that for any $\ep > 0$
\begin{align}
\fr{d(v + \ep)}{d{\bar t}} &\leq 2 C_1  e^{- \fr{C_2 {\bar t}}{2}} (v + \ep)^{1/2} \\
\fr{1}{2 (v + \ep)^{1/2} } \fr{d(v + \ep)}{d{\bar t}} = \fr{d(v+\ep)^{1/2}}{d {\bar t}} &\leq  C_1  e^{- \fr{C_2 {\bar t}}{2}}
\end{align}
Integrating from ${\bar t} = 0$ to ${\bar t} = {\bar s}$ for ${\bar s} = s\Xi e_v^{1/2}$ in the range $0 \leq {\bar s} \leq b$ and taking $\ep \to 0$, we obtain
\begin{align}
v^{1/2}({\bar s}) - v^{1/2}(0) &\leq C_1 \int_0^{\bar s} e^{- \fr{C_2 {\bar t}}{2}} d{\bar t} \\
&\leq C_1 {\bar s} \\
&\leq C_1 b
\end{align}
In our case, $v(0) = 0$ because $\nab {\hat \xi}_I(x)$ is the initial value for $\nab {\hat \xi}_I(\Phi_s(t(I),x))$.  Thus, returning to the initial variable $u$,
\begin{align}
e^{-\fr{C_2}{2} {\bar s}} u^{1/2}({\bar s}) &\leq C_1 b \\
u^{1/2}({\bar s}) &\leq C_1 e^{\fr{C_2 {\bar s}}{2}} b
\end{align}
which is an inequality of the form (\ref{ineq:gronwallPhase0}) with a different value for $C_2$ in terms of the original parameter $s = \Xi^{-1} e_v^{-1/2} {\bar s}$.
\end{proof}

\subsection{Relative Velocity Estimates} \label{sec:relVelocity}

We now collect estimates for the derivatives of the phase functions.  In terms of the foliation generated by the level sets of the phase functions $\xi_I$, estimates for derivatives $\nab^k \xi_I$ encode quantitative information about the density and regularity of the foliation.  Each leaf $\xi_I = C$ moves with the coarse scale flow, so changes in the regularity of the foliation arise from variations of the velocity field $v_\ep(t,x)$ in space.  Hence we refer to these estimates collectively as relative velocity estimates.  The intuitive considerations above are captured by the structure of the transport equations in the analysis.  For example, the bounds we obtain never depend on the $C^0$ norm of $v_\ep$, which is consistent with the Galilean invariance of the equation.

We now generalize the argument of Section (\ref{subsec:phaseStability}) to obtain bounds for the higher derivatives $\nab^k \xi_I$.  From the estimates $\co{ \nab^a v_\ep } \leq \Xi^{|a|} N^{(a - L)_+/L} e_v^{1/2}$, we see that each derivative of the velocity field $v_\ep$ costs a factor of $\Xi$, with an extra cost of $N^{1/L}$ for each derivative beyond order $L$. We expect the same cost in the estimates for the derivatives of $\xi$, so our goal is to prove
\begin{prop} \label{bound:nabaxi}  There exist constants $C_m$ depending on $m$ such that for all multi-indices $\ga$ of order $|\ga| = m$ we have
\begin{align}
\co{\nab \xi_I} &\leq C \\
\co{ \nab^\ga \xi_I } &\leq C_{|\ga|} \Xi^{|\ga| - 1} N^{(|\ga| - L)_+/L} \quad \quad \mbox{ for all } |\ga| \geq 1
\end{align}
\end{prop}

Motivated by the statement of Proposition (\ref{bound:nabaxi}) and the structure of the transport equations to follow, we introduce a weighted energy to measure the pointwise values of the $M$-jet of $\xi_I$.
\begin{defn} \label{defn:transportNorm}
Let $\xi : \T^3 \times \R \to \R$ be a real valued function, and let $\Xi$, $e_v$, $N$ and $L$ be as in (\ref{def:subSolDef}) and (\ref{lem:iterateLem}).  For $1 \leq M$, we define the $M$'th order dimensionless energy of $\xi$ to be
\begin{align}
E_M[\xi](t,x) &= \sum_{m = 1}^M \sum_{|\ga| = m} \Xi^{-2(m-1)} N^{-2(m - L)_+/L} ( \pr^\ga \xi)^2
\end{align} 
where the summation runs over all multi-indices $\ga$ with $1 \leq |\ga| \leq M$.
\end{defn}

In terms of the weighted norm (\ref{defn:transportNorm}), the bound (\ref{bound:nabaxi}) can be deduced from the following more general theorem, which will be used again later on.
\begin{prop} \label{prop:transportBound1}
Suppose that $\xi$ satisfies the transport equation
\begin{align}
(\pr_t + v_\ep^j \pr_j) \xi &= 0
\end{align}
Then there exists a constant $C = C_M$ such that for all $s \in \R$ and all $(t,x) \in \R \times \T^3$
\begin{align} \label{bound:transportBound1}
E_M[\xi](\Phi_s(t, x)) &\leq e^{C \Xi e_v^{1/2}|s|} E_M[\xi](t,x)
\end{align}
In particular, for $|s| \leq \tau = b \Xi^{-1} e_v^{-1/2}$ with $b \leq 1$, we have a bound
\begin{align} \label{bound:transportBound2}
E_M[\xi](\Phi_s(t, x)) &\leq C E_M[\xi](t,x) \quad \quad \mbox{for all } |s| \leq \tau
\end{align}
\end{prop} 
\begin{proof}
For the case $M = 1$, the weighted norm coincides with the usual absolute value $|\nab \xi|^2$.  To study the evolution of $| \nab \xi |$ under the coarse scale flow, we use the equation
\begin{align} \label{eq:dtdgaPsi}
(\pr_t + v_\ep^j) \pr^\ga \xi &= - \pr^\ga v_\ep^j \pr_j \xi
\end{align}
to obtain a formula involving the deformation tensor of $v_\ep$
\begin{align}
\fr{d}{ds} |\nab \xi|^2(\Phi_s) &= 2 \sum_{|\ga| = 1} (\pr_t + v_\ep^j \pr_j) \pr^\ga \xi \pr_\ga \xi \\
&= - 2 \sum_{|\ga| = 1}  \pr^\ga v_\ep^j \pr_j \xi \pr_\ga \xi \\
&= - \sum_{|\ga| = 1} (\pr^\ga v_\ep^j + \pr^j v_\ep^\ga) \pr_j \xi \pr_\ga \xi
\end{align}
All we need to extract from this evolution formula is the estimate 
\begin{align}
| \fr{d}{ds} |\nab \xi|^2(\Phi_s) | &\leq C \Xi e_v^{1/2} |\nab \xi|^2(\Phi_s)
\end{align}
from which (\ref{bound:transportBound1}) follows by Gronwall's inequality.

For $M = 2$, we prepare to prove an estimate
\begin{align} \label{diffineq:preGron2}
| \fr{d}{ds} E_2[\xi](\Phi_s) | &\leq C_2 \Xi e_v^{1/2} E_2[\xi](\Phi_s)
\end{align}
by first differentiating (\ref{eq:dtdgaPsi}) to obtain
\begin{align} \label{eq:dtdga2Psi}
(\pr_t + v_\ep^j \pr_j) &\pr^{\ga_1} \pr^{\ga_2} \xi = - \pr^{\ga_1}[\pr^{\ga_2} v_\ep^j \pr_j \xi] - \pr^{\ga_1} v_\ep^j \pr_j \pr^{\ga_2} \xi \\
&= - ( \pr^{\ga_2} v_\ep^j \pr_j \pr^{\ga_1} \xi + \pr^{\ga_1} v_\ep^j \pr_j \pr^{\ga_2} \xi ) - \pr^{\ga_1}\pr^{\ga_2} v_\ep^j \pr_j \xi 
\end{align}
We can use this equation to obtain (\ref{diffineq:preGron2}) by first observing that
\begin{align}
E_2[\xi] = \sum_{|\ga| = 2} \Xi^{-2}(\sum_{|\ga| = 2} |\pr^\ga \xi|^2) + |\nab \xi|^2 
\end{align}
so that (\ref{diffineq:preGron2}) follows from an inequality 
\begin{align}
\Xi^{-2} |~ \fr{d}{ds}|\pr^\ga \xi|^2(\Phi_s) ~| &\leq C_\ga \Xi e_v^{1/2} E_2[\xi] \quad \quad \mbox{ for all } |\ga| = 2
\end{align}
The above bound on
\begin{align}
\fr{d}{ds}|\pr^\ga \xi|^2(\Phi_s) &= 2 [(\pr_t + v_\ep^j \pr_j)\pr^{\ga_1} \pr_{\ga_2} \xi ] \pr_{\ga_1} \pr_{\ga_2} \xi
\end{align}
follows from multiplying the equation (\ref{eq:dtdga2Psi}) by $\Xi^{-2} \pr_{\ga_1} \pr_{\ga_2} \xi$, and observing that
\begin{align}
\Xi^{-2} \pr_{\ga_1} \pr_{\ga_2} \xi ( \pr^{\ga_1}\pr^{\ga_2} v_\ep^j \pr_j \xi) &= (\Xi^{-1} \pr^{\ga_1}\pr^{\ga_2} v_\ep^j ) (\Xi^{-1} \pr_{\ga_1} \pr_{\ga_2} \xi) (\pr_j \xi ) \\
&\leq \Xi e_v^{1/2}~E_2^{1/2}[\xi](\Phi_s)\cdot E_2^{1/2}[\xi](\Phi_s) \\
&\leq \Xi e_v^{1/2}~E_2[\xi](\Phi_s)
\end{align}
Clearly, the terms arising from the remaining part in (\ref{eq:dtdga2Psi}) obey the same estimate.

In general, we can prove the differential inequality
\begin{align}
|~ \fr{d}{ds}E_M[\xi](\Phi_s) ~| &\leq C_M \Xi e_v^{1/2} E_M[\xi](\Phi_s(t,x))
\end{align}
inductively by proving a bound of the form 
\begin{align} \label{preGronGaM}
\Xi^{-2(M-1)} N^{-2(M-L)_+/L}  |~ \fr{d}{ds}|\pr^\ga \xi|^2(\Phi_s) ~| &\leq C_\ga \Xi e_v^{1/2} E_M[\xi](\Phi_s) \quad \quad \mbox{ for all } |\ga| = M
\end{align}

The transport equation obeyed by a higher order derivative $\pr^\ga \xi = \pr^{\ga_1} \cdots \pr^{\ga_M} \xi$ is of the form
\begin{align} \label{eq:dtdgapsideriv}
(\pr_t + v_\ep^j \pr_j) \pr^\ga \xi&= \sum_{\substack{\a + \b = \ga \\ |\a| \geq 1}} c_{\a,\b}^\ga \pr^\a v_\ep^j \pr_j(\pr^\b \xi) 
\end{align}
After multiplying this equation by $\Xi^{-2(M-1)} N^{-2(M-L)_+/L} \pr_\ga \xi$, each term has the form
\begin{align}
\Xi^{-2(M-1)} N^{-2(M-L)_+/L} \pr^\a v_\ep^j \pr_j(\pr^\b \xi) \pr_\ga \xi &= P_{(I)} \cdot P_{(II)} \label{eq:needTransport Bound} \\
P_{(I)} &= \Xi^{-(M-1)} N^{-(M-L)_+/L}  \pr^\a v_\ep^j \pr_j(\pr^\b \xi) \\
P_{(II)} &= \Xi^{-(M-1)} N^{-(M-L)_+/L} \pr_\ga \xi  \\ 
|\b| &\leq M - 1 \notag
\end{align}
In order to prepare to bound these terms in terms of the dimensionless energy, we use the simple observation that for all $\a \geq 1, L \geq 1$ and $\b \geq 0$, we have
\ali{
 (\a + \b - L)_+ &\geq ( (\a - 1) - (L-1) )_+ + (\b -(L - 1) )_+ = (\a - L)_+ + (\b + 1 - L)_+, \label{ineq:firstCountingN}
}
This inequality follows from the more general inequality (\ref{lem:countingIneq}) stated below.

Using (\ref{ineq:firstCountingN}), we can then distribute the powers of $N$ and $\Xi$
\begin{align}
|P_{(I)}| &= \Xi^{-(|\a| + |\b| - 1)} N^{-(M-L)_+/L} \left|\pr^\a v_\ep^j \pr_j(\pr^\b \xi) \right| \\
&\leq  \left| \Xi^{-|\a| + 1} N^{-(|\a|-L)_+/L} \pr^\a v_\ep^j \right| ~\left| \Xi^{-|\b|} N^{-(|\b| + 1-L)_+/L}\pr_j(\pr^\b \xi) \right| \label{bound:almostTransport}
\end{align}
Using the estimates
\begin{align}
\co{ \pr^\a v_\ep^j } &\leq \Xi^{|a|} N^{(|a|-L)_+/L} e_v^{1/2} \\
\Xi^{-(|\ga| - 1)} N^{-(|\ga|-L)_+/L} | \pr_\ga \xi| &\leq E_M^{1/2}[\xi]
\end{align}
and (\ref{bound:almostTransport}) we bound 
\begin{align}
|P_{(I)}| &\leq ( \Xi e_v^{1/2} ) \cdot  E_M^{1/2}[\xi](\Phi_s) \\
|P_{(II)}| &\leq E_M^{1/2}[\xi](\Phi_s)
\end{align}
giving an estimate for (\ref{eq:needTransport Bound}) of
\begin{align}
\Xi^{-2(M-1)} N^{-2(M-L)_+/L} | \pr^\a v_\ep^j \pr_j(\pr^\b \xi) \pr_\ga \xi | &\leq \Xi e_v^{1/2} E_M[\xi](\Phi_s)
\end{align}
which establishes (\ref{preGronGaM}), and by Gronwall, we conclude (\ref{bound:transportBound1}).
\end{proof}

To deduce the Proposition (\ref{bound:nabaxi}) from Proposition (\ref{prop:transportBound1}), we simply observe that the dimensionless norm of $\xi_I$ is initially bounded by
\begin{align}
E_M[\xi]^{1/2}(t(I), x) &= E_M[{\hat \xi}]^{1/2}(x) \\
&\leq C_M
\end{align}
and that every other point $(t,x)$ on the support of $v_I$ is obtained by flowing for a time less than $\tau$ from some other point $(t(I), x')$.

As a corollary to Proposition (\ref{bound:transportBound2}) and the equations (\ref{eq:dtdgapsideriv}) for $\xi = \xi_I$, we also have the bounds

\begin{prop}[Phase-Velocity Estimates]
For all multi-indices $\ga, k$ of orders $|\ga| \geq 0$, and $|k| \geq 0$
\begin{align} \label{bound:phaseVelocity}
\co{ \nab^k[ (\pr_t + v_\ep \cdot \nab) \nab^\ga (\nab \xi_I) ) ]} &\leq C \Xi^{|k| + |\ga| + 1} e_v^{1/2} N^{(|k| + |\ga| + 1 - L)_+/L}
\end{align}
\end{prop}
The above bounds are larger than the bounds in (\ref{bound:transportBound2}) by a factor $\Xi e_v^{1/2}$ coming from the material derivative.

The inequality used to count powers of $N^{1/L}$ in the proof of Proposition (\ref{bound:nabaxi}) follows from the following general inequality, which will be used similarly in further estimates below.
\begin{lem}[The Counting Inequality]\label{lem:countingIneq}
For any non-negative numbers $x_1, x_2, \ldots, x_M \geq 0$ and $Y \geq 0$, we have
\ali{
(x_1 + x_2 + \ldots + x_M - Y)_+ &\geq \sum_i (x_i - Y)_+ 
}
\end{lem}
\begin{proof}
The result for general $M$ follows by induction from the case where $M = 2$, which states that
\ali{
(x_1 - Y)_+ + (x_2 - Y)_+ &\leq (x_1 + x_2 - Y)_+ \label{eq:theMis2Case}
}
Inequality (\ref{eq:theMis2Case}) is obvious when the left hand side is zero.  If exactly one of the two terms, say $(x_1 - Y)_+$, is positive, then the inequality follows from $x_2 \geq 0$.  If both $(x_1 - Y)_+$ and $(x_2 - Y)_+$ are positive, then the inequality follows from $Y \geq 0$.
\end{proof}

\subsection{ Relative Acceleration Estimates } \label{sec:relAccel}

In this section, we gather estimates for the quantities
\[ (\pr_t + v_\ep^j\pr_j )^2 \nab^k \xi_I \]
and their spatial derivatives which will be useful in the analysis to follow.

These estimates along with the phase-velocity estimates of section (\ref{sec:relVelocity}) give quantitative information regarding how the regularity of the foliation $\{ \xi_I = C \}$ changes in time from the frame of reference of a particle moving along the coarse scale flow.  Since the acceleration in time of these geometric quantities is due to variations in the acceleration that such particles experience, we refer to the collection of estimates that follow as relative acceleration estimates.

In Section (\ref{coarseScaleFlow}) we derived the following estimates for the acceleration along the coarse scale flow:
\begin{align} \label{bounds:coarseAcceleration2}
\co{ \nab^\a[(\pr_t + v_\ep^j \pr_j) v_\ep^l] } &\leq C_\a N^{( 1 + |\a| - L)_+/L} \Xi^{1 + |\a|} e_v \quad \quad |\a| \geq 0
\end{align}
and as in that section, we will abbreviate $f_\ep^l = (\pr_t + v_\ep^j \pr_j) v_\ep^l$.

Starting with
\begin{align}
(\pr_t + v_\ep^b \pr_b) \pr^l \xi_I &= - \pr^l v_\ep^b \pr_b \xi_I 
\end{align}
we calculate
\begin{align}
(\pr_t + v_\ep^a \pr_a )[ (\pr_t + v_\ep^b \pr_b) \pr^l \xi_I ] &= - (\pr_t + v_\ep^a \pr_a)[ \pr^l v_\ep^b \pr_b \xi_I] \\
&= - \pr^l v_\ep^b (\pr_t + v_\ep^a \pr_a) \pr_b \xi_I - [ (\pr_t + v_\ep^a \pr_a)\pr^l v_\ep^b ] \pr_b \xi_I \\
&= \pr^l v_\ep^b \pr_b v_\ep^a \pr_a \xi_I - [ (\pr_t + v_\ep^a \pr_a)\pr^l v_\ep^b ] \pr_b \xi_I
\end{align}
and substituting
\begin{align}
(\pr_t + v_\ep^a \pr_a)\pr^l v_\ep^b &= \pr^l[ (\pr_t + v_\ep^a \pr_a) v_\ep^b ] - \pr^l v_\ep^a \pr_a v_\ep^b \\
&= \pr^l[ f_\ep^b ] - \pr^l v_\ep^a \pr_a v_\ep^b
\end{align}
leaves
\begin{align} \label{eq:2ndorderTransport}
(\pr_t + v_\ep^a \pr_a )[ (\pr_t + v_\ep^b \pr_b) \pr^l \xi_I ] &= - \pr^l f_\ep^b \pr_b \xi_I + 2 \pr^l v_\ep^b \pr_b v_\ep^a \pr_a \xi_I
\end{align}

Applying the estimates (\ref{bounds:coarseAcceleration2}), (\ref{bound:nabkvep}) and Proposition (\ref{bound:nabaxi}) to equation (\ref{eq:2ndorderTransport}) we obtain
\begin{prop}[Phase-Acceleration Estimates]  For any multi-index $\a$ of order $|\a| \geq 0$
\begin{align} \label{bound:phaseAccel1}
\co{\nab^\a[(\pr_t + v_\ep \cdot \nab)^2 \nab \xi_I] } &\leq C_{|\a|} N^{(2 + |\a| - L)_+/L} \Xi^{(2 + |\a|)} e_v
\end{align}
\end{prop}

We can obtain acceleration estimates for higher order derivatives $\nab^\ga\xi_I$ by repeatedly commuting spatial derivatives with $(\pr_t + v_\ep \cdot \nab)$ and combining the estimates (\ref{bound:phaseAccel1}) and (\ref{bound:nabaxi}).  The resulting bound is of the same form
\begin{prop}[Higher Order Phase-Acceleration Estimates]  For any multi-indices $\a, \ga$ of orders $|\a| \geq 0$ and $|\ga| \geq 0$
\begin{align} \label{bound:phaseAccel2}
\co{\nab^\a[(\pr_t + v_\ep \cdot \nab)^2 \nab^\ga (\nab \xi_I)] } &\leq C_{|\a| + |\ga|} N^{(1 + |\a| + |\ga| - L)_+/L} \Xi^{(2 + |\a| + |\ga|)} e_v
\end{align}
\end{prop}

One can summarize the estimates  (\ref{bound:nabaxi}), (\ref{bound:phaseVelocity}), (\ref{bound:phaseAccel2}) succinctly in the form
\begin{align} \label{bound:phaseTransportSummary}
\co{\nab^\a[(\pr_t + v_\ep \cdot \nab)^r \nab^\ga (\nab \xi_I)} &\leq C N^{( (r - 1)_+ + |\ga| + |\a| - L)_+/L} \Xi^{r + |\a| + |\ga|} e_v^{r/2}
\end{align}
and more generally we can commute the spatial and material derivatives to obtain
\begin{prop}[General Phase-Acceleration Estimates 1]
\begin{align}
\co{\nab^\a (\pr_t + v_\ep \cdot \nab) \nab^\b (\pr_t + v_\ep \cdot \nab) \nab^\ga (\nab \xi_I)} &\leq C N^{( 1 + |\a| + |\b| + |\ga| - L)_+/L} \Xi^{2 + |\a| + |\ga|} e_v^{r/2}
\end{align}
\end{prop}

We summarize some important points from the analysis.
\begin{itemize}
\item  The gradients $\nab \xi_I$ are bounded in size
\item  Each spatial derivative of $\nab \xi$ costs a factor of $\Xi$, which is consistent with dimensional analysis.
\item  Each material derivative $(\pr_t + v_\ep \cdot \nab)^r$, $|r| \leq 2$, costs a factor of $\Xi e_v^{1/2}$, which is also consistent with dimensional analysis.
\item  After taking $L$ derivatives of $\xi$, the estimates also lose a factor of $a_v^{-1} N^{1/L}$ for each additional derivative taken (although the first material derivative does not count in this regard).  
\item  {\bf In particular}, since $L \geq 2$, more than two derivatives of $\xi$ must be taken before this extra loss appears.
\end{itemize}

Let us introduce some notation to make the estimates more readable.  Denote by $D^{(k,r)}$ any differential operator of the form
\begin{align}
D^{(k,r)} &= \nab^{a_1} (\pr_t + v_\ep \cdot \nab)^{r_1} \nab^{a_2} (\pr_t + v_\ep \cdot \nab)^{r_2} \nab^{a_3} \\\
k &= a_1 + a_2 + a_3 \\
r &= r_1 + r_2 \\
a_i, r_i &\geq 0 \\
r &\leq 2
\end{align}

Then the general phase-acceleration estimate can be written concisely as
\begin{prop}[General Phase-Acceleration Estimates 2]
\begin{align}
\co{D^{(k,r)} \nab \xi_I)} &\leq C_{k+r} N^{((r - 1)_+ + k + 1 - L)_+/L} \Xi^{k + r} e_v^{r/2}
\end{align}
\end{prop}

%% file: fixMollifyAlongFlow.tex
With the transport estimates of Section (\ref{transportEstimates}) in hand, we are ready to discuss how to construct the appropriate mollification the Reynolds stress.

\subsection{The Problem of Mollifying the Stress in Time}

Unlike the velocity field, which was only mollified in the spatial variables and which earned its time-regularity through the Euler-Reynolds equation, the Reynolds stress must be mollified in both space and time.  To see that a mollification is necessary, first observe that according to the construction of $v_I \approx R^{1/2}$ in Section (\ref{stress}), the Transport term
\begin{align} \label{eq:transportRepeat}
\pr_j Q_T^{jl} &\approx \sum_I e^{i \la \xi_I} ( \pr_t + v_\ep^j \pr_j) v_I^l
\end{align}
involves a time derivative of $R$ and we cannot obtain $C^0$ bounds for the stress $Q_T$ arising from this term without having control over $(\pr_t + v^j \pr_j) R$.  On the other hand, in order to proceed to the next stage of the iteration, we must show that $(\pr_t + v_1^a \pr_a) Q_T^{jl}$ obeys good bounds as well, and verifying these bounds will require us to take a second material derivative of $R$ according to the equation (\ref{eq:transportRepeat}).

According to the above considerations, it is necessary to mollify $R$ in time.  It is also necessary to mollify along the flow lines of $v$, rather than in the $t$ direction, because $R$ fluctuates too rapidly in the $t$ direction, whereas the material derivative $(\pr_t + v^j \pr_j ) R$ obeys better bounds.  Unlike mollification along the $t$ direction, mollification along the flow is also consistent with the Galilean invariance of the equations.

\subsection{Mollifying the Stress in Space and Time}

After constructing the mollification $R_\ep$ of $R$, we will need to have estimates for the quantities
\begin{itemize}
\item $\co{R_\ep}$
\item $\co{\nab^k R_\ep}$, $|k| \geq 0$
\item $\co{(\pr_t + v_\ep \cdot \nab)^r R_\ep }$, $r = 1, 2$
\item $\co{\nab^k (\pr_t + v_\ep \cdot \nab)^r R_\ep}$, $r = 1, 2$ and $|k| \geq 0$
\end{itemize}

With these demands in mind, we construct the mollification in two steps.  First we average in space by defining.
\begin{align}
R_{\ep_x}(t,x) &= \eta_{\ep_x} \ast \eta_{\ep_x} \ast R(t,x) \\
&= \int \int R(t, x + y_1 + y_2) \eta_{\ep_x}(y_1) \eta_{\ep_x}(y_2) dy_1 dy_2
\end{align}
The double-mollification here plays the same role as the double mollification in the construction of $v_\ep$ in Section (\ref{coarseScaleVelocity}).

We then use the coarse scale flow $\Phi$ defined in (\ref{def:csflow}) to average in time by defining
\begin{align}
R_\ep(t,x) = R_{\ep_t\ep_x}(t,x) &= \eta_{\ep_t} \ast_{\Phi} R_{ \ep_x } (t,x) \\
&= \int R_{\ep_x}(\Phi_s(t,x)) \eta_{\ep_t}(s) ds
\end{align}

For the purpose of our analysis here, it is not important that we have chosen to mollify first in space and then in time rather than in the opposite order, because we will choose $\ep_t$ and $\ep_x$ small enough so that these operations commute up to acceptable errors.  In general, though, if $\ep_t = T$ is too large, the time $T$ flow of an $\ep_x$-ball can be deformed around $\T^3$ in a complicated way, and will look very different from the $\ep_x$ neighborhood of a single trajectory of time duration $T$.

\subsection{Choosing mollification parameters} \label{sec:chooseRParams}

As a goal, we will require that the error term generated by this mollification constitutes a small fraction of the allowable stress
\begin{align} \label{goal:mollifyStress}
\co{ R - R_\ep} &\unlhd \fr{e_v^{1/2} e_R^{1/2}}{100 N}
\end{align}
During this mollification $R \to R_\ep$, we are also not allowed to enlarge the support of $R$ by an amount larger than $\Xi^{-1} e_v^{-1/2}$, as we have only assumed a lower bound of
\begin{align}
e(t) &\geq K e_R \quad \quad t \in [a - \Xi^{-1} e_v^{-1/2}, b + \Xi e_v^{1/2} ] \\
\mbox{ supp } R &\subseteq [a, b] \times \T^3
\end{align}
and this lower bound is necessary in order to solve the Stress Equation in Section (\ref{sec:lowBounds}).

In pursuit of the goal (\ref{goal:mollifyStress}), we define the time-averaged stress
\begin{align}
R_{\ep_t}(t,x) &= \eta_{\ep_t} \ast_\Phi R(t,x) \\
&= \int R(\Phi_s(t,x)) \eta_{\ep_t}(s) ds
\end{align}
and decompose the error into two parts
\begin{align}
R - R_\ep &= (R - R_{\ep_t}) + (R_{\ep_t} - R_{\ep_t\ep_x}) \\
&= (R - R_{\ep_t}) + \eta_{\ep_t} \ast_\Phi (R - R_{\ep_x})
\end{align}
This particular decomposition of the error is sensible because the tensor $R_\ep(t,x)$ is obtained by averaging $R$ over an $\ep_x$-neighborhood of the time $\ep_t$-flow from the point $(t,x)$. 

We first discuss the second of these error terms.  It is clear that
\begin{align}
\co{ \eta_{\ep_t} \ast_\Phi (R - R_{\ep_x}) } &\leq \co{R - R_{\ep_x}}
\end{align}
so in pursuit of the goal (\ref{goal:mollifyStress}), we will set $\ep_x$ to achieve the goal
\begin{align} \label{goal:mollifyStressSpace}
\co{R - R_{\ep_x}} &\unlhd \fr{e_v^{1/2} e_R^{1/2}}{200 N}
\end{align}

From Section (\ref{prepareToMollify}) we have the estimate
\begin{align}
\co{R - R_{\ep_x}} &\leq \co{R - \eta_{\ep_x} \ast R} + \co{\eta_{\ep_x} \ast (R - \eta_{\ep_x} \ast R)} \\
&\leq C \ep_x^L \Xi^L e_R
\end{align}

Since $e_v \geq e_R$, we can achieve the goal (\ref{goal:mollifyStressSpace}) by choosing
\begin{align} \label{eqn:epxRdef}
\ep_x &= a_R \Xi^{-1} N^{-1/L}
\end{align}
for the appropriately chosen small constant $a_R > 0$.  Note that this choice coincides with the scale for the spatial mollification of $v$ in Section (\ref{coarseScaleVelocity}).

In order to achieve the goal (\ref{goal:mollifyStress}), it now suffices to pick $\ep_t$ small enough so that
\begin{align} \label{goal:mollifyStressTime}
\co{R - R_{\ep_t}} &\unlhd \fr{e_v^{1/2} e_R^{1/2}}{200 N}
\end{align}

In pursuit of (\ref{goal:mollifyStressTime}), we calculate
\begin{align}
R(t,x) - R_{\ep_t}(t,x) &= \int (R(t,x) - R(\Phi_s(t,x))) \eta_{\ep_t}(s) ds \\
&= - \int \left[ \int_0^1 \fr{d}{du} R(\Phi_{us}(t,x)) du \right] \eta_{\ep_t}(s) ds \\
&= - \int_0^1 \left[ \int [ (\pr_t + v_\ep^a \pr_a)R](\Phi_{us}(t,x)) s \eta_{\ep_t}(s) ds \right] ~du
\end{align}
which gives the bound
\begin{align}
\co{R(t,x) - R_{\ep_t}(t,x)} &\leq C \ep_t \co{(\pr_t + v_\ep^j \pr_j)R} \\
&\leq C \ep_t ( \co{(\pr_t + v^j \pr_j)R} + \co{ (v_\ep^j - v^j)\pr_j R } )
\end{align}
Applying the bounds (\ref{bound:dtnabkR}) and (\ref{ineq:mollifyVerror}) , we have an estimate of the form
\begin{align}
\co{R(t,x) - R_{\ep_t}(t,x)} &\leq C \ep_t \Xi e_v^{1/2} e_R
\end{align}
for some absolute constant $C$, so that we can finally achieve the goals (\ref{goal:mollifyStressTime}) and consequently (\ref{goal:mollifyStress}) by setting
\begin{align} \label{eq:eptdef}
\ep_t &\equiv c \Xi^{-1} N^{-1} e_R^{-1/2}
\end{align}
for an appropriately small constant $c > 0$.

At this point, we have to check that
\begin{align} \label{ineq:eptsmallenough}
\ep_t &\unlhd \Xi^{-1} e_v^{-1/2}
\end{align}
for two reasons:
\begin{itemize}
\item The flow $\Phi_s$ and its derivatives only remain under control for times up to $c \Xi^{-1} e_v^{-1/2}$.
\item In order to solve the stress equation (\ref{eq:approxGaScalar}), we must be sure that the tensor 
\[ \varepsilon^{jl} = - \fr{{\mathring R}_\ep^{jl}}{e(t)} \]
encountered in Section (\ref{sec:lowBounds}) is bounded by an appropriate constant specified in (\ref{eq:varepPerturb}).  Since $\co{ R_\ep } \leq e_R$, the bound (\ref{eq:varepPerturb}) is achieved as long as the lower bound
\[ e(t) \geq K e_R \]
is satisfied on the support of $R_\ep$.  In the hypotheses of the Main Lemma (\ref{lem:iterateLem}), the interval on which this lower bound is satisfied is assumed to be an amount $\th = \Xi^{-1} e_v^{-1/2}$ larger than the time interval supporting $R$ itself.
\end{itemize}

Thankfully, the bound (\ref{ineq:eptsmallenough}) follows from (\ref{eq:eptdef}) (choosing $c$ smaller if necessary), and the assumption (\ref{eq:conditionsOnN2}) in Lemma (\ref{lem:iterateLem}) which implies that 
\[ N \geq  \left( \fr{e_v}{e_R} \right)^{1/2} \]

\subsection{Estimates for the Coarse Scale Flow} \label{sec:estForCoarseScaleFlow}

In the following section, we will derive bounds for the mollified stress
\begin{align} \label{formula:timeMollifyStress}
R_\ep(t,x) &= \int R_{\ep_x}(\Phi_s(t,x)) \eta_{\ep_t}(s) ds \\
&= \int R_{\ep_x}(t + s, \Phi^i_s(t,x)) \eta_{\ep_t}(s) ds
\end{align}
and its derivatives.

As the formula (\ref{formula:timeMollifyStress}) suggests, it will be important to estimate the extent to which the flow $\Phi_s$ deforms the geometry of $\R \times \T^3$ during the time $|s| \leq \ep_t \leq c \Xi^{-1} e_v^{-1/2}$.  This deformation estimate can be summarized by the following proposition:

\begin{prop} \label{prop:coarseScaleFlowGeometry}
For every multi-index $a$ with $|a| \geq 1$, there exist constants $C_1$ and $C_2$ such that the spatial derivative \[ \pr_a \Phi_s^i(t,x) : \R \times \R \times \T^3 \to \R^3 \]
obeys the estimate
\begin{align} \label{ineq:expBoundForFlow}
|\pr_a \Phi_s^i(t,x)| &\leq C_1 e^{C_2 \Xi e_v^{1/2} s} N^{(|a| - L)_+/L} \Xi^{|a| - 1} 
\end{align}
In particular, for $|s| \leq \Xi^{-1} e_v^{-1/2}$, we have
\begin{align}\label{ineq:distortionBound}
|\pr_a \Phi_s^i(t,x)| &\leq C N^{(|a| - L)_+/L} \Xi^{|a| - 1} 
\end{align}
\end{prop}

Note that these estimates are very similar in character to the estimate (\ref{bound:nabaxi}) proven for the phase directions.  We will also choose a very similar method of proof, beginning by introducing a dimensionless energy analogous to Definition (\ref{defn:transportNorm}).

\begin{defn} \label{defn:dimensionlessQuantity2}
We define the $M$'th order dimensionless energy of $\Phi_s$ to be
\begin{align}
E_{M}[\Phi_s] &= \sum_{1 \leq |a| \leq M} \sum_{m = 1, 2, 3} \left| \fr{\pr_a \Phi_s^m}{\Xi^{|a| - 1} N^{(|a| - L)_+/L}} \right|^{\fr{2M}{|a|}}
\end{align} 
where the summation runs over all multi-indices $a$ with $1 \leq |a| \leq M$.
\end{defn}

In terms of the dimensionless energy, the inequality (\ref{ineq:expBoundForFlow}) follows from the Gronwall Lemma once we have established the following differential inequality.
\begin{lem}  There exists a constant $C = C_M$ such that
\begin{align} \label{diffIneq:coarseScaleFlow}
|~\fr{d}{ds} E_{M}[\Phi_s]~| &\leq C \Xi e_v^{1/2} E_{M}[\Phi_s]
\end{align}
\end{lem}
\begin{proof}
Starting from the definition
\begin{align} \label{eq:coarseScaleFlowLem}
\fr{d\Phi_s}{ds} &= v_\ep(\Phi_s) 
\end{align}
of the coarse scale flow, we derive evolution equations for the spatial derivatives of $\Phi_s^i$ by differentiating equation (\ref{eq:coarseScaleFlowLem}).   For example, by applying the chain rule, any first derivative $\pr_a\Phi_s^i$, $|a| = 1$ of a component $\Phi_s^i$ of $\Phi_s$ is coupled to the first derivatives of the other components by the equation
\begin{align} \label{eq:coarseScaleFlowFirstDeriv} 
\fr{d}{ds}\pr_a\Phi_s^i &= \pr_m v_\ep^i(\Phi_s) \pr_a \Phi_s^m
\end{align}
with an implied summation over the components $m$. By multiplying this equation by $2\pr_a\Phi_s^i$, summing over $i$, and applying the bound $\co{\nab v} \leq \Xi e_v^{1/2}$ we obtain
\begin{align}
|~\fr{d}{ds}\sum_{i=1}^3 (\pr_a\Phi_s^i)^2 ~| &\leq C \Xi e_v^{1/2}\left(\sum_{i=1}^3 (\pr_a\Phi_s^i)^2\right)
\end{align}
which implies (\ref{diffIneq:coarseScaleFlow}) and (\ref{ineq:expBoundForFlow}).  For higher order derivatives, it becomes convenient to use the dimensionless energy.

To continue, any second order spatial derivative $\pr_{a} \Phi_s^i$, $a = (a_1, a_2)$ satisfies
\begin{align} \label{eq:coarseScaleFlow2ndDeriv} 
\fr{d}{ds}(\pr_{a_1} \pr_{a_2} \Phi_s) &= \pr_{m_1} \pr_{m_2} v_\ep(\Phi_s) \pr_{a_2} \Phi_s^{m_2} \pr_{a_1} \Phi_s^{m_1} + \pr_{m_2} v_\ep(\Phi_s) \pr_{a_1} \pr_{a_2} \Phi_s^{m_2}
\end{align}
with an implied summation over each index $m_1, m_2$.  From this formula we see that even though $\pr_{a} \Phi_s$ is initially $0$ when $s = 0$ and $\Phi_0(t,x) = (t,x)$, the first of these terms is already of size $\Xi^2 e_v^{1/2}$, which is responsible for the immediate growth of $\pr_{a} \Phi_s$.

A general higher order derivative $\pr_{a} \Phi_s$ of order $|a| = |(a_1, \ldots, a_A)| = A$ evolves according to the equation
\begin{align} \label{eq:highDerivCoarseScaleFlow}
\fr{d}{ds}(\pr_{a} \Phi_s) &= \sum_{K = 1}^A  \sum_{(a^1, \ldots, a^K) \in P_K(a)} \pr_{m_1} \ldots \pr_{m_K} v_\ep \cdot \prod_{i = 1}^K \pr_{a^i} \Phi_s^{m_i} 
\end{align}
where $P_K(a)$ denotes the set of ordered $K$-partitions of $a$, as explained in Definition (\ref{defn:orderedPartition}).

The important feature of the above formula is that the number of $\Phi$ terms in each product (which is denoted by $K$) is equal to the number of times that $v_\ep$ has been differentiated.

Using equation (\ref{eq:highDerivCoarseScaleFlow}) we can now prove the differential inequality (\ref{diffIneq:coarseScaleFlow}).  We first differentiate the individual term in the dimensionless energy
\begin{align}
\fr{d}{ds} \left| \fr{\pr_a \Phi_s^m}{\Xi^{(|a| - 1)} N^{(|a| - L)_+/L}} \right|^{\fr{2M}{|a|}} &= \left( \fr{\pr_a \Phi_s^m}{\Xi^{(|a| - 1)} N^{(|a| - L)_+/L}} \right)^{(\fr{2M}{a} - 1)} \cdot \left( \fr{1}{\Xi^{(|a| - 1)} N^{(|a| - L)_+/L}} \fr{d}{ds}(\pr_a \Phi_s^m) \right)
\end{align}
which already gives a termwise bound of
\begin{align}
|~ \fr{d}{ds} \left| \fr{\pr_a \Phi_s^m}{\Xi^{(|a| - 1)} N^{(|a| - L)_+/L}} \right|^{\fr{2M}{a}} ~| &\leq E_{M}[\Phi_s]^{1 - \fr{|a|}{2M}} \cdot\fr{|\fr{d}{ds}(\pr_{a} \Phi_s)|}{\Xi^{(|a| - 1)} N^{(|a| - L)_+/L}}
\end{align}
Then using the equation (\ref{eq:highDerivCoarseScaleFlow}), we can bound this last factor if we can bound each term in the summation
\begin{align} \label{ineq:eachTermForDimensionlessEnergy}
\fr{|\fr{d}{ds}(\pr_{a} \Phi_s)|}{\Xi^{(|a| - 1)} N^{(|a| - L)_+/L}} &\leq \sum_{K = 1}^A  \sum_{(a^1, \ldots, a^K) \in P_K(a) }  \fr{|\pr_{m_1} \ldots \pr_{m_K} v_\ep \cdot \prod_{i = 1}^K \pr_{a^i} \Phi_s^{m_i}|}{\Xi^{(|a| - 1)} N^{(|a| - L)_+/L}}
\end{align}
by a constant times $E_{M}[\Phi_s]^{\fr{|a|}{2M}}$.

As in the proof of the estimate (\ref{preGronGaM}) we first apply the counting inequality in Lemma (\ref{lem:countingIneq}) with $K - 1$, $|a_i| - 1$ and $L - 1 \geq 0$ to conclude that
\begin{align}
(K - L)_+ + \sum_{i = 1}^K (|a_i| - L)_+ &= \left((K - 1) - (L - 1)\right)_+ + \sum_{i = 1}^K \left((|a_i| - 1) - (L - 1)\right)_+ \\
&\leq (|a| - L)_+ = \left( (K - 1) + \sum_{i = 1}^K (|a_i| - 1) - (L - 1) \right)_+ \label{ineq:countingAgain}
\end{align}
Using inequality (\ref{ineq:countingAgain}), we distribute the powers of $N$ and $\Xi$
\begin{align}
\fr{|\pr_{m_1} \ldots \pr_{m_K} v_\ep \cdot \prod_{i = 1}^K \pr_{a^i} \Phi_s^{m_i}|}{\Xi^{(|a| - 1)} N^{(|a| - L)_+/L}} &\leq \fr{|\pr_{m_1} \ldots \pr_{m_K} v_\ep|}{\Xi^{K - 1}N^{(|K| - L)_+/L}} \prod_{i = 1}^k \fr{|\pr_{a^i} \Phi_s^{m_i}|}{\Xi^{(|a_i| - 1)} N^{(|a_i| - L)_+/L}}
\end{align}
and applying the bound $\co{\pr_m v} \leq \Xi^{|m|} e_v^{1/2} = \Xi^K e_v^{1/2}$, the right hand side can be bounded by
\begin{align}
\fr{|\pr_{m_1} \ldots \pr_{m_K} v_\ep|}{\Xi^{K - 1}N^{(|K| - L)_+/L}} \prod_{i = 1}^k \fr{|\pr_{a^i} \Phi_s^{m_i}|}{\Xi^{(|a_i| - 1)} N^{(|a_i| - L)_+/L}} &\leq \Xi e_v^{1/2}\cdot \prod_{i = 1}^k E_{M}[\Phi_s]^{\fr{|a_i|}{2M}} \\
&= \Xi e_v^{1/2} E_{M}[\Phi_s]^{\fr{|a|}{2M}}
\end{align}
which is the estimate we desired to deduce (\ref{diffIneq:coarseScaleFlow}) from (\ref{ineq:eachTermForDimensionlessEnergy}).
\end{proof}

\subsection{Spatial Variations of the Mollified Stress} \label{sec:spatVarStress}

Here we collect estimates for the mollified stress and its derivatives, starting with the $C^0$ bound
\begin{prop}
\begin{align}
\co{R_\ep} &\leq e_R
\end{align}
\end{prop}
\noindent which is clear from the definition of $R_\ep$ as an average value of $R$.

Using the distortion estimate (\ref{ineq:distortionBound}), we can also prove bounds on the derivatives of $R_\ep$
\begin{prop}\label{bound:daRep}
\begin{align}
\co{ \pr_a R_\ep } &\leq C N^{(|a| - L)_+/L} \Xi^{|a|} e_R
\end{align}
\end{prop}
\begin{proof}
The bound
\begin{align} \label{bound:spaceMollR}
\co{ \pr_a R_{\ep_x} } &\leq C_{|a|} N^{(|a| - L)_+/L} \Xi^{|a|} e_R
\end{align}
is clear for
\[ R_{\ep_x} = \eta_{\ep_x} \ast \eta_{\ep_x} \ast R \]
since $\co{\nab^{|a|} R} \leq \Xi^{|\a|} e_R$ for all $|a| \leq L$ by (\ref{bound:nabkR}), and from the choice of $\ep_x = a \Xi^{-1} N^{-1/L}$ in (\ref{eqn:epxRdef}).

Taking a worse constant, we can prove the same estimate for the fully mollified version
\begin{align}
R_\ep(t,x) &= \int R_{\ep_x}(\Phi_s(t,x)) \eta_{\ep_t}(s) ds \\
&= \int R_{\ep_x}(t+s, \Phi_s^i(t,x)) \eta_{\ep_t}(s) ds
\end{align}
by applying the chain rule and the bounds (\ref{ineq:distortionBound}) as follows.  We first calculate for any fixed $t$ and $s$ the derivative
\begin{align} \label{eq:chainRuleRep}
\pr_a R_{\ep_x}(t + s, \Phi_s^i(t,x)) &= \sum_{K = 1}^A  \sum_{(a^1, \ldots, a^K) \in P_K(a) } \pr_{m_1} \ldots \pr_{m_K} R_{\ep_x} \cdot \prod_{i = 1}^K \pr_{a^i} \Phi_s^{m_i}
\end{align}
The terms involving derivatives of $\Phi_s^i$ can be estimated using (\ref{ineq:distortionBound}) because the time $\ep_t$ is sufficiently small by (\ref{ineq:eptsmallenough}).  By also applying the estimate (\ref{bound:spaceMollR}), we can bound each term which appears in the sum by
\begin{align}
| \pr_{m_1} \ldots \pr_{m_K} R_{\ep_x} \cdot \prod_{i = 1}^K \pr_{a^i} \Phi_s^{m_i} | &\leq C (N^{(K - L)_+/L}\Xi^K e_R) \cdot \prod_{i=1}^K (N^{(|a_i| - L)_+/L} \Xi^{|a_i| - 1}) \\
&= C (N^{(K - L)_+/L + \sum_i(|a_i| - L)_+/L }\Xi^{K + \sum_{i=1}^K (|a_i| - 1)} e_R \\
&= C (N^{(K - L)_+/L + \sum_i(|a_i| - L)_+/L }\Xi^{|a|} e_R \label{eq:highDerivRofPhi}
\end{align}
for any $K$-partition $(a^1, \ldots, a^K)$ of $a$.

Again we apply the counting inequality
\begin{align}\label{ineq:countDerivs}
 (K - L)_+ + \sum_{i = 1}^K (|a_i| - L)_+ &\leq (|a| - L)_+  \quad \quad \quad |a_i|, L \geq 1
\end{align}
proven in (\ref{ineq:countingAgain}) to see that 
\begin{align}
| \pr_{m_1} \ldots \pr_{m_K} R_{\ep_x} \cdot \prod_{i = 1}^K \pr_{a^i} \Phi_s^{m_i} | &\leq C N^{(|a| - L)_+/L}\Xi^{|a|} e_R
\end{align}
for any ordered $K$-partition $(a^1, \ldots, a^K)$ of $a$, which bounds every term in (\ref{eq:chainRuleRep}) and therefore completes the proof of (\ref{bound:daRep}).
\end{proof}

\subsection{Transport Estimates for the Mollified Stress} \label{sec:DdtofRep}

Here we collect estimates for the material derivative of the mollified stress
\[ \Ddtof{R_\ep} = ( \pr_t + v_\ep^a \pr_a) R_\ep^{jl} \]
as well as its spatial derivatives.

The results of this section are summarized by the following Proposition
\begin{prop}[First material derivative of the mollified stress] \label{prop:firstDdtOfRep}  For $k \geq 0$, there exist constants $C_k$ such that
\begin{align}
\co{ \nab^k \Ddtof{R_\ep} } &\leq C_k N^{(k + 1 - L)_+/L} \Xi^{k + 1} e_v^{1/2} e_R
\end{align}
\end{prop}

We begin by calculating
\begin{align}
(\pr_t + v_\ep^a(t,x) \pr_a) R_\ep^{jl}(t,x) &= \int (\pr_t + v_\ep^a(t,x) \pr_a) R_{\ep_x}^{jl}(\Phi_s(t,x)) \eta_{\ep_t}(s) ds \\
&= \int DR_{\ep_x}^{jl}(\Phi_s(t,x)) D\Phi_s(t,x) (\pr_t + v_\ep^a(t,x) \pr_a) \eta_{\ep_t}(s) ds \label{eq:needToEvaluateDdtRep}
\end{align}
using the chain rule.

Here and in the remainder of the section, $DF(t,x)$ denotes the derivative of $F$ in the $(t,x)$ variables.  Thus
\begin{align}
DF X^\mu \pr_\mu &= DF(t,x) (X^\mu(t,x) \pr_\mu) \\
&= X^\mu(t,x) \pr_\mu F(t,x)
\end{align}
for any $4$-vector $X^\mu(t,x) \pr_\mu$.

Because $\Phi_s$ has been defined to be the flow map of the vector field $\Ddt = (\pr_t + v_\ep^a(t,x) \pr_a)$, we can show that applying the flow $\Phi_s$ to $\Ddt$
\begin{align}
D\Phi_s(t,x)[ \pr_t + v_\ep^a(t,x) \pr_a ] &= \pr_t + v_\ep^a(\Phi_s(t,x)) \pr_a
\end{align}
simply moves the vector field $(\pr_t + v_\ep^a(t,x) \pr_a)$ forward along the flow.

As a consequence, we can show that 
\begin{lem}[Derivatives and Averages Along the Flow Commute]\label{eq:flowDerivAndAverage}  For all $C^1$ functions $F$ on $\R \times \T^3$ we have an equality
\begin{align}
(\pr_t + v_\ep^a(t,x) \pr_a)[F(\Phi_s(t,x))] &= \Ddtof{F}(\Phi_s(t,x))
\end{align}

In particular,
\begin{align}
(\pr_t + v_\ep^a(t,x) \pr_a) R_\ep^{jl}(t,x) &= \int \Ddtof{R_{\ep_x}}(\Phi_s(t,x)) \eta_{\ep_t}(s) ds
\end{align}
\end{lem}
where
\begin{align}
\int \Ddtof{R_{\ep_x}}(\Phi_s(t,x)) \eta_{\ep_t}(s) ds &= \int [ \pr_t R_{\ep_x}^{jl}(\Phi_s(t,x)) + v_\ep^a(\Phi_s(t,x)) [\pr_a R_{\ep_x}^{jl}](\Phi_s(t,x)) ] \eta_{\ep_t}(s) ds \\
&= \int DR_{\ep_x}(\Phi_s(t,x))[(\pr_t + v_\ep^a(\Phi_s(t,x)) \pr_a)] \eta_{\ep_t}(s) ds
\end{align}

Since this calculation is crucial, we give a proof in the following Section.  We then apply the Lemma in Section (\ref{sec:firstMatDvofRep}) to establish the estimates in Proposition (\ref{prop:firstDdtOfRep}). 

\subsubsection{Derivatives and Averages along the Flow Commute} 

\input{commuteDvAvgAlongFlow}

\subsubsection{First material derivative estimates for the mollified stress} \label{sec:firstMatDvofRep}

\input{commuteMollifyFlow}

\subsubsection{Second time derivative of the mollified stress along the coarse scale flow}

In this section, we establish bounds for the second material derivative
\begin{align}
\DDdtof{R_\ep} &= (\pr_t + v_\ep \cdot \nab)^2 R_\ep
\end{align}
of the mollified stress $R_\ep$, along with its spatial derivatives.  These estimates will be necessary to control the material derivative $\Ddtof{Q_T^{jl}}$ where $Q_T$ is the part of the new stress which arises from the transport term
\begin{align}
\pr_j Q_T^{jl} &= (\pr_t + v_\ep^j \pr_j) V^l \\
&= \sum_I e^{i \la \xi_I}[(\pr_t + v_\ep^j \pr_j){\tilde v}_I^l] 
\end{align}

If we recall that $R_\ep$ was constructed by mollifying a time $\ep_t$ along the coarse scale flow
\begin{align} \label{eq:formulaForRepDdt2}
R_\ep(t,x) &= \int R_{\ep_x}(\Phi_s(t,x)) \eta_{\ep_t}(s) ds 
\end{align}
then by analogy with standard mollification we expect the estimates for the second material derivative $\DDdt{R_\ep}$ to be a factor $\ep_t^{-1}$ worse than the bounds on the first material derivative $\Ddtof{R_\ep}$.  

To see that this expectation is correct, we apply Lemma (\ref{eq:flowDerivAndAverage}) to see that
\begin{align}
(\pr_t + v_\ep^a(t,x) \pr_a) R_\ep^{jl}(t,x) &= \int \Ddtof{R_{\ep_x}}(\Phi_s(t,x)) \eta_{\ep_t}(s) ds \\
&= \int \fr{d}{ds}[ R_{\ep_x}(\Phi_s(t,x)) ] \eta_{\ep_t}(s) ds \\
&= - \int R_{\ep_x}(\Phi_s(t,x)) \eta_{\ep_t}'(s) ds
\end{align}

The proof of Lemma (\ref{eq:flowDerivAndAverage}) has nothing to do with the choice of the function $\eta_{\ep_t}$, and we can express the second material derivative of $R_\ep$ as
\begin{align}
\DDdtof{R_\ep} &= - \int \Ddtof{R_{\ep_x}}(\Phi_s(t,x)) \eta_{\ep_t}'(s) ds
\end{align}
which is almost exactly the kind of average estimated in Section (\ref{sec:firstMatDvofRep}), except that the bound on
$\| \eta_{\ep_t}'(s) \|_{L^1_s} $ is worse by a factor \[ \ep_t^{-1} = N \Xi e_R^{1/2} \]  
The same methods from Sections (\ref{sec:firstMatDvofRep}) and (\ref{sec:spatVarStress}) then imply the bounds
\begin{prop}[Second material derivative of the mollified stress]\label{prop:secondDdtOfRep}
For $k \geq 0$, there exist constants $C_k$ such that
\begin{align}
\co{ \nab^k \DDdtof{R_\ep} } &\leq C_k N^{1 + (k + 1 - L)_+/L} \Xi^{k + 2} e_v^{1/2} e_R^{3/2}
\end{align}
\end{prop}

Now that we have established the bounds of Proposition (\ref{prop:secondDdtOfRep}), we should check that this estimate is acceptable for the purpose of proving the Main Lemma (\ref{lem:iterateLem}).

\subsubsection{An acceptability check} \label{sec:acceptableCheck}

The purpose of mollifying $R^{jl}$ along the coarse scale flow is to ensure that the material derivative of the transport term $Q_T^{jl}$ obeys good bounds.  Let us now calculate the bounds we expect for $\Ddtof{Q_T^{jl}}$ and check how they compare to the requirements of Conjecture (\ref{conj:idealCase}).

Recall that the first term in the parametrix expansion for the solution to 
\begin{align}
\pr_j Q_{T,I}^{jl} &= e^{i \la \xi_I} u_{T,I,(1)}^l \\
u_{T,I,(1)}^l &= (\pr_t + v_\ep^j \pr_j) {\tilde v}_I^l
\end{align}
is of the form
\begin{align}
Q_{T,I,(1)}^{jl} &= \fr{1}{\la} e^{i \la \xi_I} q_{T,I, (1)}^{jl} \\
 q_{T,I, (1)}^{jl} &= i^{-1} q^{jl}(\nab \xi_I)[u_{T,I,(1)}]
\end{align}
where $q^{jl}(\nab \xi_I)[u]$ is the solution to $\pr_j q^{jl} = u^l$ described in Section (\ref{sec:oscEstimate}) which depends linearly on $u$.

To leading order in $\la$, 
\begin{align}
u_{T,I,(1)}^l &\approx \Ddt v_I^l \\
v_I^l &= a_I^l + i b_I^l \\
b_I^l &= \eta_{k_4}(t) \psi_k(t,x) e^{1/2}(t) \ga_I(\nab \xi_k, \varepsilon) P_I^\perp(\nab \xi_{\si I})^l \\
\varepsilon^{jl} &= \fr{- {\mathring R}^{jl}_\ep}{e(t)} \\
a_I &= - \fr{(\nab \xi_I)}{|\nab \xi_I|}\times b_I 
\end{align}
which means that one of the terms appearing in $u_{T,I,(1)}^l$ involves $\Ddt{R_\ep}$ when the material derivative hits the $\varep$ inside of $\ga_I$.  This term has size $\Xi e_v^{1/2} e_R^{1/2}$ because it costs $\Xi e_v^{1/2}$ to take the first material derivative of $R_\ep$, and therefore gives a contribution to the transport term of size
\begin{align}
|Q_{T,I,(1)}^{jl}| &= \fr{e_v^{1/2} e_R^{1/2}}{N} + |\mbox{ other terms }|
\end{align}
It will ultimately turn out that some of other terms will be larger, but for now we will focus on the term involving $\Ddtof{R_\ep}$.

To establish Lemma (\ref{lem:iterateLem}), we will have to verify a bound on 
\[ (\pr_t + v_1 \cdot \nab) Q_{T,I,(1)}^{jl} = (\pr_t + v\cdot \nab) Q_{T,I,(1)}^{jl} + V\cdot \nab Q_{T,I,(1)}^{jl} \]
and its first $L - 1$ spatial derivatives.  

During this verification, we approximate $(\pr_t + v\cdot \nab) = \Ddt + (v - v_\ep) \cdot \nab$ because $\Ddt$ will not hit the phase functions.  Then, in order to bound $\Ddtof{Q_{T,I,(1)}^{jl}}$ we will take a second material derivative of $R_\ep$.  This second derivative will cost a factor $\ep_t^{-1}$, leading to the bounds
\begin{align}
|\Ddt Q_{T,I,(1)}^{jl}| &\leq \ep_t^{-1} \fr{e_v^{1/2} e_R^{1/2}}{N} + |\mbox{ other terms }| \\
&\leq C \Xi e_v^{1/2} e_R + |\mbox{ other terms }|
\end{align}
In the ideal case where $(\Xi', e_v', e_R') = ( C N \Xi, e_R, \fr{e_v^{1/2} e_R^{1/2}}{N})$, the right hand side is exactly the benchmark necessary to be deemed acceptable.  Namely,
\[ C \Xi e_v^{1/2} e_R = \Xi' (e_v')^{1/2}e_R' \quad \quad \mbox{for Conjecture (\ref{conj:idealCase})} \]
To prove Lemma (\ref{lem:iterateLem}), we actually only need to verify that
\begin{align}
|\Ddt Q_{T,I,(1)}^{jl}| &\unlhd C N \Xi e_R^{1/2} \left( \fr{e_v^{1/2}}{e_R^{1/2} N} \right)^{1/2} e_R \\
&= C \left( \fr{e_v^{1/2}}{e_R^{1/2} N} \right)^{-1/2} (\Xi e_v^{1/2} e_R)
\end{align} 
which is an even more forgiving benchmark.

For the purpose of proving the Main Lemma (\ref{lem:iterateLem}), we can therefore afford at other points in the argument to be slightly wasteful in the estimates.

%% file: commuteDvAvgAlongFlow.tex
The purpose of this Section is to establish Lemma (\ref{eq:flowDerivAndAverage}), which is a consequence of the following more general fact, applied to the $4$-vector fields 
\begin{align*}
 X^\mu(t,x) \pr_\mu &= \pr_t + v_\ep^a(t,x) \pr_a \\
 Y^\mu(t,x) \pr_\mu &=  \pr_t + v_\ep^a(t,x) \pr_a 
\end{align*}
which, in our case, are equal to each other.

\begin{lem}
Let $Y^\mu(t,x) \pr_\mu : \R \times \T^3 \to \R \times \R^3$ be any smooth vector field on $\R \times \T^3$, let $\Phi_s(t,x)$ denote the time $s$ flow of $Y$.  That is, $\Phi_s(t,x)$ solves the ODE
\begin{align}
\fr{d\Phi_s^\mu}{ds} &= Y^\mu(\Phi_s(t,x)) \label{eq:theYFlow}\\
\Phi_0(t,x) &= (t,x)
\end{align}
Suppose that $X$ commutes with $Y$ in the sense that the commutator 
\begin{align}
 [X,Y]^\nu \pr_\nu = (X^\mu \pr_\mu Y^\nu - Y^\mu \pr_\mu X^\nu) \pr_\nu 
\end{align}
vanishes.

Then $X$ also commutes with pullback along the flow of $Y$, 
\[ D\Phi_s(t,x)(X^\mu(t,x) \pr_\mu) = X^\mu(\Phi_s(t,x) ) \pr_\mu \]
That is, for any smooth function $F$ on $\R \times \T^3$, we have
\begin{align}
X^\mu(\Phi_s(t,x) ) (\pr_\mu F)(\Phi_s(t,x)) &= X^\mu(t,x) \pr_\mu\left[ F(\Phi_s(t,x) \right] 
\end{align}
\end{lem}
\begin{proof}
In order to compare
\begin{align}
X^\mu(\Phi_s(t,x) ) (\pr_\mu F)(\Phi_s(t,x)) &= DF(\Phi_s(t,x))[ X^\mu(\Phi_s(t,x) )\pr_\mu] \\
&= DF(\Phi_s(t,x)) D\Phi_s(t,x) [ D\Phi_s(t,x)]^{-1} X^\mu(\Phi_s(t,x) )\pr_\mu
\end{align} 
and
\begin{align}
X^\mu(t,x) \pr_\mu[ F(\Phi_s(t,x) ] &= DF(\Phi_s(t,x)) D\Phi_s(t,x)[X^\mu(t,x) \pr_\mu] \\
&=  DF(\Phi_s(t,x)) D\Phi_s(t,x) [ D\Phi_0(t,x)]^{-1} X^\mu(\Phi_0(t,x) )\pr_\mu
\end{align}
it suffices to show that the vector field
\begin{align} 
Z_s^\mu(t,x) \pr_\mu &=  [ D\Phi_s(t,x)]^{-1} X^\mu(\Phi_s(t,x) )\pr_\mu,
\end{align}
which is characterized by the equation
\begin{align}
\pr_\mu \Phi_s^\nu(t,x) Z_s^\mu(t,x) \pr_\nu &= X^\nu(\Phi_s(t,x)) \pr_\nu, \label{eq:defineEqnZs}
\end{align}
satisfies
\begin{align} 
{\dot Z}_s^\mu \pr_\mu = \fr{\pr}{\pr s} Z_s^\mu(t,x) \pr_\mu &= 0. \label{eq:Zdotis0}
\end{align}
To show (\ref{eq:Zdotis0}), we differentiate (\ref{eq:defineEqnZs}) and apply (\ref{eq:theYFlow}) to see that
\begin{align}
\pr_\mu \Phi_s^\nu(t,x) {\dot Z}_s^\mu(t,x) \pr_\nu &= - \pd{}{s}[\pr_\mu \Phi_s^\nu(t,x)] Z_s^\mu(t,x) + \pd{}{s}[X^\nu(\Phi_s(t,x))] \pr_\nu \\
&= - \pd{}{s}[\pr_\mu \Phi_s^\nu(t,x)] Z_s^\mu(t,x) + Y^\mu(\Phi_s(t,x)) [\pr_\mu X^\nu](\Phi_s(t,x)) \pr_\nu \\
\pd{}{s}[\pr_\mu \Phi_s^\nu(t,x)] &= \pr_\a Y^\nu(\Phi_s(t,x)) \pr_\mu \Phi_s^\a(t,x) \label{eq:derivPhidot}
\end{align}
Then comparing (\ref{eq:derivPhidot}), and (\ref{eq:defineEqnZs}) we obtain
\begin{align}
\pr_\mu \Phi_s^\nu(t,x) {\dot Z}_s^\mu(t,x) \pr_\nu &= ( - X^\a(\Phi_s) [\pr_\a Y^\nu](\Phi_s(t,x)) + Y^\mu(\Phi_s(t,x)) [\pr_\mu X^\nu](\Phi_s(t,x)) ) \pr_\nu \\
&= - [X, Y]^\nu(\Phi_s(t,x)) \pr_\nu = 0
\end{align}
\end{proof}

The reader may recognize that the above argument is essentially an ODE proof that $e^{sA}B = B e^{sA}$ for any pair of commuting matrices $A, B$ and any $s \in \R$, except that $e^{sA}$ has been replaced by the pull-back operator $\circ \Phi_s$, and $B$ has been replaced by the differential operator $X^\mu \pr_\mu$.

%% file: commuteMollifyFlow.tex
In this Section, we prove Proposition (\ref{prop:firstDdtOfRep}) by estimating the spatial derivatives of
\begin{align}
\Ddtof{R_\ep} &= \int \Ddtof{R_{\ep_x}}(\Phi_s(t,x)) \eta_{\ep_t}(s) ds
\end{align}
as well as its spatial derivatives, namely
\begin{prop}[Material derivative bounds for the spatially mollified stress] \label{prop:oneMatDvOfRFlow}  There exist constants depending on $k$ such that
\begin{align}
\co{ \nab^k \Ddtof{R_\ep}} &\leq C_k N^{(k + 1 - L)_+/L} \Xi^{k + 1} e_v^{1/2} e_R
\end{align}
\end{prop}

As a first step, we estimate the tensor field
\begin{align}
\Ddtof{R_{\ep_x}^{jl}} &= (\pr_t + v_\ep^a(t,x)\pr_a ) R_{\ep_x}^{jl}(t,x)
\end{align}
and its spatial derivatives, which are summarized by the following lemma.

\begin{prop}[Bounds for $\Ddtof{R_{\ep_x}}$]\label{prop:oneMatDvOfR} For every $k \geq 0$ there is a constant $C_k$ such that
\begin{align}
\co{ \nab^k \Ddtof{R_{\ep_x}} } &\leq C_k N^{(k + 1 - L)_+/L} \Xi^{k + 1} e_v^{1/2} e_R 
\end{align}
\end{prop}

Given Proposition (\ref{prop:oneMatDvOfR}), the estimate in Proposition (\ref{prop:oneMatDvOfRFlow}) follows just as in Section (\ref{sec:spatVarStress}).  To summarize the argument of that section, each spatial derivative derivative up to order $L - 1$ costs at most $\Xi$, even if it is taken on the flow $\Phi_s$, and beyond $L - 1$ derivatives, the cost of each spatial derivative increases to $\Xi N^{1/L}$.  No exponential factors appear in the estimates for derivatives of $\Phi_s$ because we have checked that $s \leq \ep_t = N^{-1} \Xi^{-1} e_R^{-1/2} < \Xi^{-1} e_v^{-1/2}$.

Estimating $\Ddtof{R_{\ep_x}}$ requires us to commute the operator \[\Ddt = \pr_t + v_\ep^a(t,x)\pr_a\] with the spatial mollifier \[ \eta_{\ep + \ep}\ast = \eta_{\ep_x} \ast \eta_{\ep_x} \ast \]
This commutation leads to an expansion of $\Ddtof{R_{\ep_x}}$ into three terms \footnote{Note that these indices are Roman numerals, and should not be confused with the indices $I$ for the individual waves $V_I$. }
\begin{align}
\Ddtof{R_{\ep_x}} &= \eta_{\ep + \ep} \ast \Ddtof{R} + [v_\ep^a(t,x)\pr_a, \eta_{\ep_x} \ast](\eta_{\ep_x} \ast R ) + \eta_{\ep_x} \ast( [v_\ep^a(t,x)\pr_a, \eta_{\ep_x} \ast] R ) \\
&= A_{(I)}(t,x) + A_{(II)}(t,x) + A_{(III)}(t,x)
\end{align}

We first express the commutator terms as
\begin{align}
A_{(II)}(t,x) &= \int [ v_\ep^a(t,x) - v_\ep^a(t,x+h) ] \pr_a (\eta_{\ep_x} \ast R )(t,x) \eta_{\ep_x}(h) dh \\
&= \int_0^1 \int \pr_i v_\ep^a(t, x + s h) \pr_a (\eta_{\ep_x} \ast R )(t,x) h^i \eta_{\ep_x}(h) dh \\
&= \ep_x \int_0^1 \int \pr_i v_\ep^a(t, x + s h) \pr_a (\eta_{\ep_x} \ast R )(t,x) {\tilde \eta}^i_{\ep_x}(h) dh
\end{align}
with $\ep_x = \Xi^{-1} N^{-1/L}$.  From this expression, we can conclude that
\begin{align}
\co{ \nab^k A_{(II)} } &\leq C_k \Xi^{-1} N^{-1/L} \sum_{|a_1| + |a_2| = k} \co{ \nab^{|a_1| + 1} v_\ep } \co{ \nab^{|a_2| + 1} (\eta_{\ep_x} \ast R )}  \\
&\leq C \Xi^{-1} N^{-1/L} \sum_{|a_1| + |a_2| = k} \left( N^{(|a_1| + 1 - L)_+/L} \Xi^{|a_1| + 1} e_v^{1/2} \right) \left( N^{(|a_2| + 1 - L)_+/L} \Xi^{|a_2| + 1} e_R \right) \\
&\leq C \Xi^{k + 1}e_v^{1/2} e_R \sum_{|a_1| + |a_2| = k} (N^{[ (|a_1| - (L - 1))_+ + (|a_2| - (L-1))_+ - 1 ]/L }) \\
&\leq C N^{(k + 1 - L)_+/L} \Xi^{k + 1}e_v^{1/2} e_R.
\end{align}
where in the last line we have used the counting inequality in Lemma (\ref{lem:countingIneq}).

We can similarly treat
\begin{align}
A_{(III)}(t,x) &= \eta_{\ep_x} \ast \left( [v_\ep^a(t,x)\pr_a, \eta_{\ep_x} \ast] R  \right) \\
[v_\ep^a(t,x)\pr_a, \eta_{\ep_x} \ast] R &= \int_0^1 \int \pr_i v_\ep^a(t, x + s h) \pr_a R (t,x) h^i \eta_{\ep_x}(h) dh 
\end{align}
The commutator can be differentiated $L - 1$ times,
\begin{align}
\co{ \nab^k [v_\ep^a(t,x)\pr_a, \eta_{\ep_x} \ast] R} &\leq C_k \ep_x \Xi^{2 + k} e_v^{1/2} e_R \\
&\leq C_k N^{-1/L} \Xi^{1 + k} e_v^{1/2} e_R
\end{align}
Taking more than $L - 1$ derivatives of $A_{(III)}$ incurs a loss of $\ep_x^{-1} \leqc \Xi N^{1/L}$, giving us a bound of
\begin{align}
\co{ \nab^k A_{(III)} } &\leq N^{(k - L)_+/L} \Xi^{k + 1}e_v^{1/2} e_R  \\
&\leq N^{(k + 1 - L)_+/L} \Xi^{k + 1}e_v^{1/2} e_R
\end{align}

In order to estimate the term
\[ \eta_{\ep + \ep} \ast \Ddtof{R} = \eta_{\ep + \ep} \ast[ (\pr_t + v_\ep \cdot \nabla) R ] \]
we use the assumed bounds on $(\pr_t + v \cdot \nab) R$ by writing
\begin{align}
\eta_{\ep + \ep} \ast \Ddtof{R} &= \eta_{\ep + \ep} \ast[ (\pr_t + v \cdot \nab) R ] + \eta_{\ep + \ep} \ast[ (v_\ep - v) \cdot \nab R ] \\
&= A_{(I, I)} + A_{(I, II)}
\end{align}

Then
\begin{align}
\co{ \nab^k A_{(I,I)} } &\leq C_k N^{(k + 1 - L)_+/L} \Xi^{k + 1} e_v^{1/2} e_R
\end{align}
since the first $L - 1$ derivatives do not need to fall on the mollifier, and beyond $L - 1$, each derivative taken costs $\ep_x^{-1} \leqc \Xi N^{1/L}$ 

For $A_{(I, II)}$ we have
\begin{align}
(v_\ep - v) \cdot \nab R &= \int [ v^a(t,x + h) - v^a(t,x)] \pr_a R(t,x) \eta_{\ep+\ep}(h) dh \\
&= \int_0^1 \int \pr_i v^a(t, x + s h) \pr_a R(t,x) h^i \eta_{\ep + \ep}(h) dh \\
\end{align}
which can be differentiated up to $L - 1$ times at a cost of $\Xi$ per derivative
\begin{align}
\co{ \nab^k[(v_\ep - v) \cdot \nab R] } &\leq C_k N^{-1/L} \Xi^{k + 1} e_v^{1/2} e_R 
\end{align}
so that each derivative taken beyond $L - 1$ costs $\ep_x^{-1} \leqc \Xi N^{1/L}$ as it falls on the mollifier, giving
\begin{align}
\co{ \nab^k A_{(I,II)} } &\leq C_k N^{(k - L)_+/L} \Xi^{k + 1} e_v^{1/2} e_R \\
&\leq N^{(k + 1 - L)_+/L} \Xi^{k + 1} e_v^{1/2} e_R
\end{align}

This finishes the proof of Proposition (\ref{prop:oneMatDvOfR}).

%% file: accountForParams.tex
Although many estimates have been proven so far in the argument, only a few of the parameters have been chosen.

Here we give a list of all the parameters, and summarize their current status
\begin{itemize}
\item The parameter 
\begin{align}
 \ep_x &= a_R \Xi^{-1} N^{-1/L} 
\end{align} 
used to mollify $R$ in space has been chosen in line (\ref{eqn:epxRdef}) of Section (\ref{sec:chooseRParams}).
\item The parameter 
\begin{align}
\ep_t &= c \Xi^{-1} N^{-1} e_R^{-1/2}  
\end{align}
used to mollify $R_{\ep_x}$ along the flow has been chosen in line (\ref{eq:eptdef}) of Section (\ref{sec:chooseRParams}).
\item The parameter $\ep_v$ used to mollify the velocity field $v_\ep$ in space has the form
\begin{align}
\ep_v &= a_v \Xi^{-1} N^{-1/L}
\end{align}
where the constant $a_v \leq 1$ was chosen in Section (\ref{coarseScaleVelocity}) line (\ref{eq:whereWeChooseav}).
\item The lifespan parameter $\tau$ is of the form
\begin{align}
\tau &= b \Xi^{-1} e_v^{-1/2}
\end{align}
for a dimensionless parameter $b$ which has not been chosen.  

Although $b$ has not been chosen, $b$ is required to satisfy an upper bound of the form 
\begin{align}
b &\unlhd b_0 \label{goal:weNeedbb0}
\end{align}
The number $b_0 < 1$ above is an absolute constant which will ensure the conditions specified in Proposition (\ref{req:reqsForB}) are satisfied.  The parameter $\tau$ does not depend in any way on the choice of the mollifying parameter $\ep_v$ for the velocity, which will be apparent shortly as $b$ is chosen.
\item The parameter $\la$ used to make sure the phase functions are high frequency is of the form
\begin{align}
\la &= B_\la \Xi N
\end{align}
where $B_\la \geq 1$ is the largest constant, which will be chosen at the very end of the proof.
\end{itemize}

For the sake of the logical sequence of the argument, we will use this opportunity to choose the parameter $\tau$, and then provide a motivation for the choice.  Namely, we set
\begin{align}
b &= b_0 \left( \fr{e_v^{1/2}}{B_\la e_R^{1/2} N} \right)^{1/2} \\
&= b_0 B_\la^{-1/2} \left( \fr{e_v^{1/2}}{e_R^{1/2} N} \right)^{1/2} \label{eq:iChoseb}
\end{align}
so that
\begin{align}
\tau &= b_0  B_\la^{-1/2} \left( \fr{e_v^{1/2}}{e_R^{1/2} N} \right)^{1/2}  \Xi^{-1} e_v^{-1/2}
\end{align}
Note that $N \geq \left( \fr{e_v}{e_R} \right)^{1/2}$, so that $b \leq b_0$ as required in (\ref{goal:weNeedbb0}).

The motivation for the above choice is the following.  Consider the High-High term and the Transport term, which solve the equations
\begin{align}
\pr_j Q_H^{jl} &= \la \sum_{J \neq {\bar I}} e^{i \la (\xi_I + \xi_J)}[ v_I \times ( |\nab \xi_J| - 1 ) v_J + v_J \times (|\nab \xi_I| - 1) v_I ] + \mbox{ lower order terms } \label{eq:QHisLike}\\
\pr_j Q_T^{jl} &= e^{i \la \xi_I} (\pr_t + v_\ep^j(t,x) \pr_j) v_I^l + \mbox{ lower order terms } \label{eq:QTisLike}
\end{align}
and recall that $v_I$ is given by $a_I + i b_I$, where
\begin{align}
b_I^l &= \eta\left(\fr{t - t(I)}{\tau}\right) \psi_k(t,x) e^{1/2}(t) \ga_I(\nab \xi_k, \fr{R_\ep}{e}) P_I^\perp(\nab \xi_{\si I})^l \label{eq:bIlooksLike1}
\end{align}
and $a_I = - \fr{(\nab \xi_I)\times}{|\nab \xi_I|} b_I $ is obtained by a $\pi/2$ rotation of $b_I$ in $\langle \nab \xi_I \rangle^\perp$.  

From the bounds we have established, it is already clear that
\begin{align}
\co{ b_I } &\leq C e_R^{1/2}
\end{align}
uniformly in $I$, where the smallness comes from the function $e^{1/2}(t)$, and all the other factors are bounded.

The solution $Q_H$ to (\ref{eq:QHisLike}) will not be any smaller than the first term in the parametrix expansion for any of its terms, which gains a factor of $1 / \la$ compared to the right hand side.  At each time, there are only a bounded number (fewer than $(2 \times 2^3 \times 12)^2$) interaction terms which are nonzero.  Therefore, we expect a bound
\begin{align}
|Q_H| &\leq C e_R \max_I \sup_{|t - t(I)| \leq \tau,x \in \T^3}(|\nab \xi_I|(t,x) - 1)
\end{align}
which is only smaller than $e_R$ if the phase gradients are very close to their initial size $1$ in absolute value.

According to Proposition (\ref{prop:phaseStability}), we have a bound
\begin{align}
| |\nab \xi_I| - 1 | &\leq A b
\end{align}
(which we remark is dimensionless), giving  
\begin{align}
|Q_H| &\leq C b e_R \label{ineq:projectedQR}
\end{align}
Therefore $b$ must be chosen small so that the estimate for $\co{ Q_H }$ will guarantee that $Q_H$ is smaller than the stress from the previous stage.

Unfortunately, choosing a short lifespan forces the transport term to be large, since the material derivative hits the time cutoff function $\eta(\fr{t - t(I)}{\tau})$ in (\ref{eq:bIlooksLike1}).  Therefore we expect an estimate
\begin{align}
\co{ (\pr_t + v_\ep^j \pr_j) v_I } &\leq C \tau^{-1} e_R^{1/2}
\end{align}
which leads to an expected estimate for $Q_T$ of
\begin{align}
|Q_T| &\leq C \fr{\tau^{-1} e_R^{1/2}}{B_\la \Xi N} \\
&\leq C b^{-1} \fr{e_v^{1/2} e_R^{1/2}}{B_\la N} \label{ineq:projectedQT}
\end{align}

If we had been able to choose $b$ to be a constant, then the transport term $Q_T$ could be made to have size $\fr{e_v^{1/2} e_R^{1/2}}{N}$ which is what is required for regularity up to $1/3$.  Unfortunately, the best we can do is optimize $b$ in order to balance the terms (\ref{ineq:projectedQR}) and (\ref{ineq:projectedQT}), which leads to the choice of $b$ in (\ref{eq:iChoseb}).

Now the only parameter that remains to be chosen is $B_\la$, which will be chosen at the very end of the argument.

%% file: correctionAmplitudes.tex
In order to estimate the corrections $P$ and $V$ and their spatial and material derivatives for the proof of Lemma (\ref{lem:iterateLem}), and to prepare to estimate each term that composes the new stress, we begin by proving estimates for the vector amplitudes $v_I$.

Let us recall again that $v_I$ is given by $a_I + i b_I$, where
\begin{align}
b_I^l &= \eta\left(\fr{t - t(I)}{\tau}\right) \psi_k(t,x) e^{1/2}(t) \ga_I(\nab \xi_k, \varepsilon) P_I^\perp(\nab \xi_{\si I})^l \label{eq:bIlooksLike2} \\
\varepsilon^{jl} &= \fr{- {\mathring R}^{jl}_\ep}{e(t)}
\end{align}
and the real part
\begin{align}
a_I &= - \fr{(\nab \xi_I)}{|\nab \xi_I|} \times b_I 
\end{align} 
is obtained by a $\pi/2$ rotation of $b_I$ in the plane $\langle \nab \xi_I \rangle^\perp$.

We need to estimate spatial derivatives of $v_I$, and spatial derivatives of its first two material derivatives.  To begin this process, we start with the coefficients $\ga_I$.

\section{ Bounds for coefficients from the stress equation } \label{sec:coeffBounds}

Recall that the coefficients $\ga_I$ are defined implicitly by the equation
\begin{align}
 \sum_{I \in \vec{k} \times \F} A(\nab \xi_k)_J^I \ga_I^2 &= (\fr{\de^{jl}}{2n} + \varepsilon(R_\ep)^{jl}) \pr_j \xi_J \pr_l \xi_J \label{eq:eqnForGaI}
\end{align}
of Section (\ref{sec:lowBounds}).  The fact that this equation can be solved for $\ga_I$ was ensured in the discussion of Section (\ref{sec:chooseRParams}).  In order to estimate the derivatives of $\ga_I$, we can differentiate the Equation (\ref{eq:eqnForGaI}).

As a preliminary measure, we estimate the derivatives for 
\begin{align}
\varepsilon^{jl} &= \fr{- {\mathring R}^{jl}_\ep}{e(t)}
\end{align}
These are summarized by
\begin{prop}[Bounds for $\varepsilon$] \label{prop:boundsForVarep}
\begin{align}
\co{ \nab^k \varepsilon } &\leq C_k N^{(k - L)_+/L} \Xi^k \label{ineq:nabkvarep} \\
\co{ \nab^k \Ddtof{\varepsilon} } &\leq C_k N^{(k + 1 - L)_+/L} \Xi^{k + 1} e_v^{1/2} \label{ineq:nabkDtvarep} \\
\co{\nab^k \DDdtof{\varepsilon} } &\leq C_k N^{1 + (k + 1 - L)_+/L} \Xi^{k + 2} e_v^{1/2} e_R^{1/2} \label{ineq:nabkDDtvarep}
\end{align}
\end{prop}
Note that these only differ from the bounds (\ref{bound:daRep}), (\ref{prop:firstDdtOfRep}) and (\ref{prop:secondDdtOfRep}) by a factor of $e_R^{-1}$.  
\begin{proof}
The proof of Proposition (\ref{prop:boundsForVarep}) proceeds by studying the equations
\begin{align}
e(t) \varepsilon^{jl} &= - {\mathring R}^{jl}_\ep \label{eq:varep0}\\
e(t) \Ddtof{\varepsilon^{jl}} &= - \Ddtof{ {\mathring R}^{jl}_\ep } - e'(t) \varepsilon^{jl} \label{eq:varep1}\\
e(t) \DDdtof{\varepsilon^{jl}} &= - \DDdtof{{\mathring R}^{jl}_\ep} - 2 e'(t) \Ddtof{\varepsilon^{jl}} - e''(t) \varepsilon^{jl} \label{eq:varep2}
\end{align}
The bound (\ref{ineq:nabkvarep}) follows from differentiating (\ref{eq:varep0}), and applying the bounds (\ref{bound:daRep}) and the lower bound
\[ e(t) \geq K e_R \]
The bound (\ref{ineq:nabkDtvarep}) follows from differentiating (\ref{eq:varep1}) in space, and applying the bounds (\ref{ineq:nabkvarep}), (\ref{ineq:goodEnergy}) and (\ref{prop:firstDdtOfRep}).  The point is that the first transport derivative always costs a factor $\Xi e_v^{1/2}$ in the estimates, and each spatial derivative up to order costs $\Xi$ until the total order of differentiation exceeds $L$, at which point there is an additional cost of $N^{1/L}$ per derivative.

The proof of (\ref{ineq:nabkDDtvarep}) proceeds similarly by differentiating (\ref{eq:varep2}) and applying the bounds from (\ref{ineq:goodEnergy}), (\ref{ineq:nabkvarep}), (\ref{ineq:nabkDtvarep}), and (\ref{prop:secondDdtOfRep}).  The equation for the second material derivative has the form
\begin{align}
\nab^k \DDdtof{\varepsilon^{jl}} &= e^{-1}(t) \left\{ - \nab^k \DDdtof{{\mathring R}^{jl}_\ep} - 2 e'(t) \nab^k \Ddtof{\varepsilon^{jl}} - e''(t) \nab^k \varepsilon^{jl} \right \} \\
\co{ \nab^k \DDdtof{\varepsilon^{jl}} } &\leq [N \Xi e_R^{1/2} ] \cdot [ N^{(k + 1 - L)_+/L} \Xi^{k + 1} e_v^{1/2} ] + [ \Xi e_v^{1/2} ] [ N^{(k + 1 - L)_+/L} \Xi^{k + 1} e_v^{1/2} ] \\
&+ [ \Xi e_v^{1/2} ]^2 [N^{(k - L)_+} \Xi^k] \\
&= A_{(I)} + A_{(II)} + A_{(III)}
\end{align}
Counting powers of $N$, we see that $A_{(III)} < A_{(II)}$, and the fact that $A_{(II)} < A_{(I)}$ follows from $N \geq \left( \fr{e_v}{e_R} \right)^{1/2}$.
\end{proof}

We can now differentiate equation (\ref{eq:eqnForGaI}) in order to prove that
\begin{prop}[Bounds for $\ga_I$] \label{prop:theBoundsForGaI} The coefficients $\ga_I = \ga_I(\nab \xi_k, \varep)$ satisfy the bounds 
\begin{align}
\co{ \nab^k \ga_I } &\leq C_k N^{(k + 1 - L)_+/L} \Xi^k \label{ineq:nabkgaI} \\
\co{ \nab^k \Ddtof{\ga_I} } &\leq C_k N^{(k + 1 - L)_+/L} \Xi^{k + 1} e_v^{1/2} \label{ineq:nabkDtgaI} \\
\co{\nab^k \DDdtof{\ga_I} } &\leq C_k N^{1 + (k + 1 - L)_+/L} \Xi^{k + 2} e_v^{1/2} e_R^{1/2} \label{ineq:nabkDDtgaI}
\end{align}
\end{prop}

Since the proof is a routine application of the chain rule, we will only give a schematic outline of which terms appear in order to avoid clutter.

\begin{proof}
The proof proceeds either by differentiating (\ref{eq:eqnForGaI}), or equivalently applying the chain rule to the implicit function
\begin{align} 
\ga_I(p, \varep) &= \ga_I(\nab \xi_k, \varep) \label{eq:thisIsGaIpvarep}
\end{align}
Recall that the implicit function $\ga_I$ appearing in this estimate is one of finitely many different functions $\ga_I = \ga_{f(I)}$ from Section (\ref{sec:lowBounds}) which depend only on the direction coordinate $f(I) \in F$.  For the present proof, we will omit the subscript $k = k(I) \in (\Z / 2\Z)^3 \times \Z$ in the location coordinate for $\nab \xi = \nab \xi_k$, so that $k$ may be used to denote the order of differentiation in the bounds of Proposition (\ref{prop:theBoundsForGaI}) and the calculations below.  It is only important to remember that for each index $I$, only phase gradients $\xi_J$ on the same region $k(J) = k(I)$ enter into (\ref{eq:thisIsGaIpvarep}).

According to the estimates (\ref{bound:nabaxi}), (\ref{bound:phaseVelocity}) and (\ref{bound:phaseAccel2}) in Section (\ref{transportEstimates}), each material derivative on $\nab \xi$ costs $\Xi e_v^{1/2}$ and each spatial derivative on $\nab \xi$, $\Ddtof{\nab \xi}$ or $\DDdtof{\nab \xi}$ costs either $\Xi$ or $N^{1/L} \Xi$ depending on whether or not the total order of differentiation on $\xi$ exceeds $L$.

For spatial derivatives of $\ga_I$, the terms with the highest order derivatives are schematically given by
\begin{align}
\nab^k[\ga_I(\nab \xi, \varep)] &\subseteq [\pr_p \ga ] (\nab^{k + 1} \xi) + [\pr_{\varep}\ga ] (\nab^k \varep) + \mbox{ cross terms} \label{formeq:nabKgaI}
\end{align}
As we have seen already in several cases such as Section (\ref{sec:spatVarStress}), the cross terms are lower order in terms of powers of $N^{1/L}$.

The two main terms in (\ref{formeq:nabKgaI}) obey the same estimates as each other, except that the $k + 1$ derivative of $\xi$ costs an extra factor of $N^{1/L}$
\begin{align}
\co{ \nab^{k +1} \xi } &\leq C_k N^{(k + 1 - L)_+/L} \Xi^k \\
\co{ \nab^k \varep } &\leq C_k N^{(k - L)_+/L} \Xi^k
\end{align}
All the derivatives of the functions $\ga_I$ which appear in the expansion are bounded by universal constants, giving (\ref{ineq:nabkgaI}).

Taking a material derivative of $\ga_I$ and differentiating in space gives terms which are schematically of the form
\begin{align}
\Ddt[\ga_I(\nab \xi, \varep)] &\subseteq [\pr_p \ga ] \Ddtof{\nab \xi} + [\pr_{\varep}\ga] \Ddtof{\varep} \\
\nab^k \Ddt[\ga_I(\nab \xi, \varep)] &\subseteq [\pr_p \ga ] \nab^k \Ddtof{\nab \xi} + [\pr_p^2 \ga ] \nab^{k + 1} \xi  \Ddtof{\nab \xi} \\
&+ [\pr_{\varep}\ga] \nab^k \Ddtof{\varep} +[\pr_p \pr_{\varep}\ga] \nab^{k+1} \xi \Ddtof{\varep}  \\
&+ \mbox{ smaller cross terms}
\end{align}
By Proposition (\ref{bound:nabaxi}), (\ref{bound:phaseVelocity}), and (\ref{ineq:nabkDtvarep}) all of these terms are bounded by 
\[ C_k N^{(k + 1 - L)_+/L} \Xi^{k + 1} e_v^{1/2} \]
In particular, no extra factor of $N^{1/L}$ appears for the first material derivative of $\nab \xi$; see Section (\ref{transportEstimates}) for a review of how this fact follows from the transport equation for $\nab \xi$.

Finally, we consider
\begin{align}
\DDdt[\ga_I(\nab \xi, \varep)] &\subseteq [\pr_p \ga ] \DDdtof{\nab \xi} + [\pr_{\varep}\ga] \DDdtof{\varep} \label{eq:theMainTermsDDdtgaI}\\
&+ [\pr_p^2 \ga ] (\Ddtof{\nab \xi})^2 + [\pr_p \pr_\varep \ga](\Ddtof{\nab \xi})(\Ddtof{\varep})  + [\pr_{\varep}^2 \ga](\Ddtof{\varep})^2 \label{eq:lowerOrderTermsDdtgaI}
\end{align}
The terms in line (\ref{eq:lowerOrderTermsDdtgaI}) are lower order, so schematically the main terms from the estimates come from derivatives of the terms in (\ref{eq:theMainTermsDDdtgaI}), which include
\begin{align}
\nab^k \DDdt[\ga_I(\nab \xi, \varep)] &\subseteq [\pr_p \ga ] \nab^k \DDdtof{\nab \xi} +  [\pr_{\varep}\ga] \nab^k \DDdtof{\varep} \\
&+ [\pr_p^2 \ga ] \nab^{k + 1} \xi \DDdtof{\nab \xi} + [\pr_p \pr_\varep \ga] \nab^{k + 1} \xi \DDdtof{\varep}
\end{align}
At this stage, the second material derivative contributes towards the power of $N^{1/L}$ appearing in the estimates for both the phase gradients and for $\varep$.  The only difference is that while the second material derivative costs $\Xi e_v^{1/2}$ when applied to the phase gradients, the second material derivative of $\varep$ costs a larger factor of $N \Xi e_R^{1/2}$.  Thus, the final bound for (\ref{ineq:nabkDDtgaI}) is exactly the same quality as the corresponding bound for $\varep$.
\end{proof}

\section{ Bounds for the vector amplitudes }\label{sec:vecAmpBounds}

Having estimated the coefficients $\ga_I$, we are now ready to begin estimating the vector amplitudes $v_I$ and $w_I$ of the correction.

Let us recall once more that $v_I$ is given by $a_I + i b_I$, where
\begin{align}
b_I^l &= \eta\left(\fr{t - t(I)}{\tau}\right) \psi_k(t,x) e^{1/2}(t) \ga_I P_I^\perp(\nab \xi_{\si I})^l \label{eq:bIlooksLike3}
\end{align}
and the real part 
\[ a_I = - \fr{(\nab \xi_I)}{|\nab \xi_I|}\times b_I \] is obtained by a $\pi/2$ rotation of $b_I$ in $\langle \nab \xi_I \rangle^\perp$.

Let us compress the notation slightly by writing
\begin{align}
v_I^l &= \eta\left(\fr{t - t(I)}{\tau}\right) e^{1/2}(t) \ga_I \a_I^l \label{eq:compressvIform}\\
\a_I^l &= [ i - \fr{(\nab \xi_I)}{|\nab \xi_I|}\times ] \psi_k P_I^\perp(\nab \xi_{\si I})^l
\end{align}

The vector field $\a_I^l$ takes values in the positive eigenspace of the operator $(i \nab \xi_I) \times$ within the plane $\langle \nab \xi_I \rangle^\perp$, and is composed entirely from solutions to the transport equation ($\psi$ itself being no worse than $\xi$), so it obeys the estimates
\begin{prop}[Estimates for the Transport Part of $v_I$]
\begin{align}
\co{ \nab^k \a_I^l } &\leq C_k N^{(k + 1 - L)_+/L} \Xi^k \label{eq:spaceDvsOfAlpha}\\
\co{ \nab^k \Ddt{\a_I^l} } &\leq C_k N^{(k + 1 - L)_+/L} \Xi^{k+1} e_v^{1/2} \label{eq:matDvOfAlpha} \\
\co{ \nab^k \DDdt{\a_I^l} } &\leq C_k N^{(k + 2 - L)_+/L}\Xi^{k+2} e_v \label{eq:matDv2ofAlpha}
\end{align}
or more concisely, using the notation of Section (\ref{sec:relAccel})
\begin{align}
\co{ D^{(k,r)} \a_I^l } &\leq C_{k + r} N^{( (r - 1)_+ + k + 1 - L)_+/L} \Xi^{k + r} e_v^{r/2}
\end{align}
\end{prop}

The correction itself is composed of terms
\[ V_I = \nab \times W_I \]
with
\begin{align}
W_I &= \fr{1}{\la} e^{i \la \xi_I} w_I \\
w_I &= \fr{v_I}{|\nab \xi_I|}
\end{align}
so $w_I$ also takes on the form (\ref{eq:compressvIform}), with a slightly different $\a_I^l(w)$ that obeys the same bounds.

Let us now compute estimates for $v_I$, starting with the spatial derivatives.
\begin{prop}[Spatial derivatives of vector amplitudes]
\begin{align} \label{eq:spatDvsVecAmps}
\co{ \nab^k v_I } + \co{\nab^k w_I} &\leq C_k N^{(k + 1 - L)_+/L} \Xi^k e_R^{1/2}
\end{align}
\end{prop}
\begin{proof}

Let $\nab^k$ be any $k$'th order spatial derivative.  Then, the main terms from differentiating (\ref{eq:compressvIform}) are 
\begin{align}
\nab^k v_I^l &\subseteq \eta_I(t) e^{1/2}(t)[ \nab^k \ga_I \a_I^l + \ga_I \nab^k \a_I^l ]
\end{align}
The estimate (\ref{eq:spatDvsVecAmps}) follows from (\ref{eq:spaceDvsOfAlpha}) and (\ref{ineq:nabkgaI}), where the factor $e_R^{1/2}$ comes from the bound (\ref{ineq:goodEnergy}) on $e^{1/2}(t)$.
\end{proof}

Let us now study the first material derivative $\Ddt{v_I^l}$.  It takes on the form
\begin{align}
\Ddtof{v_I^l} &= \tau^{-1} \eta'\left(\fr{t - t(I)}{\tau} \right) e^{1/2}(t) \ga_I \a_I^l + \eta_I(t) \fr{d e^{1/2}}{dt} \ga_I \a_I^l \\
&+ \eta_I(t) e^{1/2}(t)(\Ddtof{\ga_I} \a_I^l + \ga_I \Ddtof{\a_I^l} )
\end{align}

For every term besides the time cutoff, the first material derivative costs $\Xi e_v^{1/2}$, whereas for the time cutoff, we lose a factor
\[ \tau^{-1} = b^{-1} \Xi e_v^{1/2} = b_0^{-1} B_\la^{1/2} \left( \fr{e_R^{1/2} N}{e_v^{1/2}} \right)^{1/2}  \]
This leads to a bound of 
\begin{prop}[First material derivative of vector amplitudes]
\begin{align}
\co{ \nab^k \Ddtof{v_I^l} } + \co{ \nab^k \Ddtof{w_I^l} } &\leq C_k B_\la^{1/2} \left( \fr{e_R^{1/2} N}{e_v^{1/2}} \right)^{1/2} N^{(k + 1 - L)_+} \Xi^{k+1} e_v^{1/2} e_R^{1/2}
\end{align}
\end{prop}

For the second material derivative, there are two factors which lose more than a factor of $\Xi e_v^{1/2}$: differentiating the time cutoff again gives another factor of $\tau^{-1}$, whereas the second material derivative of $\ga_I$ gives rise to a factor of $\ep_t^{-1}$.  It turns out that these extra factors are in balance with each other; namely
\ali{
\tau^{-2} &\leq C [ B_\la^{1/2} \left( \fr{e_R^{1/2} N}{e_v^{1/2}} \right)^{1/2} ]^2 (\Xi e_v^{1/2})^2\\
&\leq C B_\la N \Xi e_v^{1/2} e_R^{1/2} \\
\ep_t^{-1} (\Xi e_v^{1/2}) &\leq C (N e_R^{1/2} \Xi )(\Xi e_v^{1/2} ) \\
&\leq C N \Xi e_v^{1/2} e_R^{1/2}
}
This observation leads to the following bounds

\begin{prop}[The second material derivative of the vector amplitude] \label{prop:matDv2ofvI}
\begin{align}
\co{ \nab^k \DDdtof{ v_I} } + \co{ \nab^k \DDdtof{ w_I} } &\leq C_k B_\la N^{1 + ( k + 1 - L)_+/L } \Xi^{k + 2} e_v^{1/2} e_R 
\end{align}

\end{prop}
\begin{proof}
Let us write the second material derivative of $v_I$ schematically as
\begin{align}
\DDdtof{v_I} &= \fr{d^2}{dt^2}[ \eta\left(\fr{t - t(I)}{\tau} \right) e^{1/2}(t) ] \ga_I \a_I^l + \eta_I(t) e^{1/2}(t) \DDdt[  \ga_I \a_I^l ] + \mbox{ smaller cross terms} \\
&= A_{(I)} + A_{(II)} 
\end{align}

Now, by (\ref{eq:spaceDvsOfAlpha}) and (\ref{ineq:nabkgaI}), the first of these terms is bounded by 
\begin{align}
\co{ \nab^k A_{(I)} } &\leq C \tau^{-2} N^{( k + 1 - L)_+/L} \Xi^k e_R^{1/2} \\
&\leq C b^{-2} N^{( k + 1 - L)_+/L} \Xi^{k + 2} e_v e_R^{1/2} \\
&\leq C B_\la \left( \fr{ e_R^{1/2} N }{e_v^{1/2}} \right) N^{( k + 1 - L)_+/L} \Xi^{k + 2} e_v e_R^{1/2} \\
&\leq C B_\la N^{1 + ( k + 1 - L)_+/L } \Xi^{k + 2} e_v^{1/2} e_R 
\end{align}

For the second of these terms
\begin{align}
A_{(II)} &= \DDdt[  \ga_I \a_I^l ]
\end{align}
we compare
\begin{align}
\nab^k A_{(II)} \subseteq \nab^k \DDdtof{\ga_I} \a_I^l + \ga_I \nab^k \DDdtof{\a_I^l}
\end{align}
All the bounds for $\ga_I$ and $\a_I$ are identical until the second material derivative.  For the second material derivative, we use (\ref{ineq:nabkDDtgaI}) and (\ref{eq:matDv2ofAlpha}) to compare
\begin{align}
\co{ \nab^k \DDdtof{\ga_I} } &\leq C_k N^{1 + (k + 1 - L)_+/L} \Xi^{k + 2} e_v^{1/2} e_R^{1/2} \\
\co{ \nab^k \DDdtof{\a_I} } &\leq C_k N^{(k + 2 - L)_+/L} \Xi^{k + 2} e_v
\end{align}
In fact, the bound for $\ga_I$ is the larger of the two, as can be seen by proving
\begin{align}
N^{(k + 2 - L)_+/L} e_v^{1/2} &\leq N^{1 + ( k + 1 - L)_+/L }e_R^{1/2} \\
\Leftrightarrow \left( \fr{e_v}{e_R} \right)^{1/2} &\leq N^{ 1 + [( k + 1 - L)_+ - (k + 2 - L)_+/L] } \label{ineq:gaVsalpha}
\end{align}
but
\[ N^{ 1 + [( k + 1 - L)_+ - (k + 2 - L)_+/L] } \geq N^{1 - 1/L} \geq N^{1/2} \] 
so (\ref{ineq:gaVsalpha}) follows from the fact that
\[ N \geq \left( \fr{e_v}{e_R} \right)^{3/2} \geq \left( \fr{e_v}{e_R} \right) \]

Also observe that the bound for $\co{ \nab^k \DDdtof{\ga_I} }$ is the same bound as for $\co{ \nab^k A_{(I)} }$ without the large constant $B_\la$, which concludes the proof of (\ref{prop:matDv2ofvI}).

\end{proof}

The results of this section also extend to the corrected amplitude
\[ {\tilde v_I} = v_I + \fr{\nab \times w_I}{B_\la N \Xi} \] 
appearing in the representation
\[ V_I = e^{i \la \xi_I} {\tilde v}_I^l \]

\begin{cor}[Estimates for corrected amplitudes] 
The vector field 
\[ {\tilde v_I} = v_I + \fr{\nab \times w_I}{B_\la N \Xi} \] 
also satisfies all the estimates stated in Section (\ref{sec:vecAmpBounds}) for $v_I$.
\end{cor}
\begin{proof}
The bounds for $\co{ \nab^k v_I}$ are clear since 
\begin{align}
\nab^k {\tilde v}_I &= \nab^k v_I + \fr{1}{ B_\la N \Xi } \nab^k \nab \times w_I \\
\co{ \nab^k {\tilde v}_I} &\leq C_k N^{(k + 1 - L)_+/L} \Xi^k e_R^{1/2} + (B_\la N \Xi)^{-1} \cdot ( N^{(k + 2 - L)_+/L} \Xi^{k+1} e_R^{1/2} ) \\
&\leq  C_k N^{(k + 1 - L)_+/L} \Xi^k e_R^{1/2} 
\end{align}
To estimate $\Ddtof{{\tilde v}_I}$ we write
\begin{align}
\Ddtof{ {\tilde v}_I } &= \Ddtof{v_I} + \fr{1}{ B_\la N \Xi } [ \nab \times ( \Ddtof{ w_I } ) + D[v_\ep]( w ) ] 
\end{align}
where the commutator term is an operator which is schematically of the form
\begin{align}
D[v_\ep]( w ) &=  \nab v_\ep \cdot \nab w
\end{align}
From this expression, it is clear that the commutator gives a cost of $\Xi e_v^{1/2}$, whereas the material derivative itself carries a larger cost of $\tau^{-1} = b^{-1} \Xi e_v^{1/2}$ in the estimates.
\end{proof}

\section{Bounds for the Velocity and Pressure Corrections}

\input{boundsForVP}

\section{Energy Approximation}

\input{energyApprox1}

\section{Checking Frequency and Energy Levels for the Velocity and Pressure}

\input{compareMainLem1}

%% file: boundsForVP.tex
We now give bounds for the correction terms using the estimates of the preceding sections.

\subsection{Bounds for the Velocity Correction }

Since $V$ is of the form
\[ V = \nab \times W \]
estimating $V$ and its derivatives will follow from estimating the derivatives of
\begin{align}
W &= \sum_I W_I \\
W_I &=(B_\la N \Xi)^{-1} e^{i \la \xi_I} w_I
\end{align}

Let us first estimate the spatial derivatives.
\begin{prop}[Spatial Derivatives of $W$]\label{prop:nabkW}
\begin{align}
\co{ \nab^k W } &\leq C_k (B_\la N \Xi)^{k - 1} e_R^{1/2}
\end{align}
\end{prop}

\begin{proof}
Since there are only finitely many $W_I$ supported at any given region of $\R \times \T^3$, it suffices to estimate $W_I$ uniformly in $I$.  For an individual wave, it is easy to see that the estimate will hold.  At the level of $\co{ W }$ we have
\begin{align}
\co{ (B_\la N \Xi)^{-1} e^{i \la \xi_I} w_I } &\leq C (B_\la N \Xi)^{-1} e_R^{1/2}
\end{align}
For the derivatives, we apply the product rule to
\begin{align}
\nab^k e^{i \la \xi_I} w_I &= \sum_{|a| + |b| = k} \nab^a[ e^{i \la \xi_I}] \nab^b[w_I]
\end{align}
During repeated differentiation, the derivative hits either
\begin{itemize}
\item The oscillatory factor $e^{i \la \xi_I}$, which costs $C \la = C B_\la N \Xi$ 
\item The phase direction $\nab \xi_I$ or one of its derivatives, which costs at most $C N^{1/L} \Xi$
\item The amplitude $w_I$ or one of its derivatives, which costs at most $C N^{1/L} \Xi$
\end{itemize} 
In any case, the largest cost happens when differentiating the phase function, leading to the estimate in Proposition (\ref{prop:nabkW}).
\end{proof}

We can similarly give estimates for derivatives of the coarse scale material derivative of $W$
\begin{prop}[Coarse Scale Material Derivative of $W$]\label{prop:matDvofW}
\begin{align}
\co{ \nab^k \Ddtof{W} } &\leq C_k B_\la^{1/2} \left( \fr{e_R^{1/2} N}{e_v^{1/2}} \right)^{1/2}  (B_\la N \Xi)^{k - 1} \Xi e_v^{1/2} e_R^{1/2}
\end{align}
\end{prop}
\begin{proof}
The proof is by the same method as in Proposition (\ref{prop:nabkW}).  Here we simply observe that because the phase is transported by the coarse scale flow
\begin{align}
\Ddtof{W_I} &= \fr{1}{\la} e^{i \la \xi_I} \Ddtof{w_I} 
\end{align}
where the amplitude of $\Ddtof{w_I}$ is bounded by
\[ \co{ \Ddtof{w_I} } \leq C \tau^{-1} e_R^{1/2} \leq  C B_\la^{1/2} \left( \fr{e_R^{1/2} N}{e_v^{1/2}} \right)^{1/2} \Xi e_v^{1/2} e_R^{1/2} \]
Upon taking spatial derivatives, the largest cost of $B_\la N \Xi$ always occurs when the derivative hits the oscillatory factor.
\end{proof}

The Main Lemma (\ref{lem:iterateLem}) asks for bounds on
\[ (\pr_t + v \cdot \nab) W \]
rather than 
\[ \Ddtof{W} = (\pr_t + v_\ep \cdot \nab) W \]

As the following lemma illustrates, it is always possible to obtain estimates for $(\pr_t + v \cdot \nab)$ of a quantity, once one has appropriate estimates for coarse scale material derivatives and for spatial derivatives

\begin{cor}[Estimates for $(\pr_t + v \cdot \nab) W$]\label{cor:realMatDvOfW} 

For all $k = 0, \ldots, L$, the bounds for $\nab^k \Ddtof{W}$ and $\nab^k (\pr_t + v \cdot \nab) W$ are of the same order.  More precisely,
\begin{align}
\co{ \nab^k (\pr_t + v \cdot \nab) W } &\leq C_k B_\la^{1/2} \left( \fr{e_R^{1/2} N}{e_v^{1/2}} \right)^{1/2}  (B_\la N \Xi)^{k - 1} \Xi e_v^{1/2} e_R^{1/2}
\end{align}
\end{cor}
\begin{proof}
For $k = 0$, we simply write
\[ (\pr_t + v \cdot \nab) = (\pr_t + v_\ep \cdot \nab) + (v - v_\ep) \cdot \nab \]
It then suffices to estimate
\begin{align}
(v - v_\ep) \cdot \nab W &= (v^a - v_\ep^a) \pr_a W
\end{align}
and its derivatives.  At the level of $C^0$, we have
\begin{align}
 \co{v - v_\ep} &\leq C \fr{e_v^{1/2}}{N} \label{ineq:mollifyVerrorAgain}
\end{align}
by (\ref{ineq:mollifyVerror}).

Since a spatial derivative of $W$ costs $B_\la N \Xi$, we see that the cost of $(v - v_\ep) \cdot \nab$ is at most $\Xi e_v^{1/2}$.
\begin{align}
\co{(v - v_\ep) \cdot \nab W} &\leq C [\Xi e_v^{1/2}] ( B_\la N \Xi)^{-1} e_R^{1/2}
\end{align} 
which is better than the cost of $\tau^{-1} = b^{-1} \Xi e_v^{1/2}$ of $\Ddt$.

For the spatial derivatives, we are allowed to lose a factor of $N \Xi$ per spatial derivative.  Comparing the bounds
\begin{align}
 \co{v - v_\ep} &\leq C \fr{e_v^{1/2}}{N} \\
 \co{\nab v} + \co{\nab v_\ep} &\leq \Xi e_v^{1/2}
\end{align}
we see that our estimate worsens by a factor of $N \Xi$ when we treat the terms $\nab v$ and $\nab v_\ep$ separately.  We can now proceed to take up to $L$ derivatives of $v$ altogether, giving the Corollary.
\end{proof}

\begin{cor}[Spatial derivatives of $V$] \label{prop:spaceDvOfV}
\begin{align}
\co{ \nab^k V } &\leq C_k (B_\la N \Xi)^k e_R^{1/2}
\end{align}
\end{cor}
\begin{proof}
This follows from (\ref{prop:nabkW}) and the fact that $V = \nab \times W$.
\end{proof}

\begin{cor}[Coarse scale material derivative of $V$]\label{cor:matDvOfV}
\begin{align}
\co{ \nab^k \Ddtof{V} } &\leq C_k B_\la^{1/2} \left( \fr{e_R^{1/2} N}{e_v^{1/2}} \right)^{1/2}  (B_\la N \Xi)^{k} \Xi e_v^{1/2} e_R^{1/2}
\end{align}
\end{cor}
\begin{proof}
This estimate follows from (\ref{prop:matDvofW}) and (\ref{prop:nabkW}) after we express
\begin{align}
\Ddt{V} &= \Ddt \nab \times W \\
&= \nab \times\Ddtof{ W } - D[v_\ep][W]
\end{align}
where the commutator $D[v_\ep][W]$ is a spatial derivative operator of the form
\begin{align}
D[v_\ep][W] &\approx \nab v_\ep \nab W
\end{align}
The material derivative costs $\tau^{-1} = b^{-1} \Xi e_v^{1/2}$
\begin{align}
\co{ \nab \times\Ddtof{ W } } &\leq C (B_\la N \Xi) [ (B_\la N \Xi)^{-1} \tau^{-1} e_R^{1/2} ] 
\end{align}
whereas the commutator only costs $(\Xi e_v^{1/2})$
\begin{align}
\co{ \nab v_\ep \nab W } &\leq C (B_\la N \Xi) [ (B_\la N \Xi)^{-1} (\Xi e_v^{1/2}) e_R^{1/2} ]
\end{align} 
Each additional spatial derivative costs a factor $(B_\la N \Xi)$ as it falls on the oscillatory factor in $W$.
\end{proof}

\begin{cor}[Estimates for $\nab^k (\pr_t + v \cdot \nab )V$]\label{cor:realMatDvOfV}

For all $k = 0, \ldots, L$, the bounds for $\nab^k \Ddtof{V}$ and $\nab^k (\pr_t + v \cdot \nab) V$ are of the same order.  More precisely,
\begin{align}
\co{ \nab^k (\pr_t + v \cdot \nab) V } &\leq C_k B_\la^{1/2} \left( \fr{e_R^{1/2} N}{e_v^{1/2}} \right)^{1/2}  (B_\la N \Xi)^{k} \Xi e_v^{1/2} e_R^{1/2}
\end{align}
\end{cor}
\begin{proof}
The proof is identical to the proof of Corollary (\ref{cor:realMatDvOfW}).
\end{proof}

\subsection{Bounds for the Pressure Correction}

The correction $P$ to the pressure consists of two parts
\begin{align}
P &= P_0 + \sum_{J \neq {\bar I} } P_{I, J} \\
P_0 &= - \fr{e(t)}{3} - \fr{R_\ep^{jl}\de_{jl}}{3} \\
P_{I,J} &= - \fr{V_I \cdot V_J}{2}\label{eq:formForPIJ} \\
&= - \fr{1}{2} e^{i \la (\xi_I + \xi_J) } {\tilde v}_I \cdot {\tilde v}_J 
\end{align}

From the preceding section, we can see that the estimates for the high frequency part $V_I \cdot V_J$ are the dominant ones, since we have seen that the primary cost of spatial derivatives comes from the oscillatory factor, and the primary cost of material derivatives comes from the time cutoff.

We also see from the quadratic nature of formula (\ref{eq:formForPIJ}) that the estimates for $P_{I,J}$ will gain a factor $e_R^{1/2}$ compared to those of $V_I$.  To summarize:

\begin{prop}[Estimates for Pressure Correction] \label{prop:pressureCorrectEstimates}
\begin{align}
\co{ \nab^k P } &\leq C_k (B_\la N \Xi)^k e_R \\
\co{ \nab^k \Ddtof{P} } &\leq C_k B_\la^{1/2} \left( \fr{e_R^{1/2} N}{e_v^{1/2}} \right)^{1/2}  (B_\la N \Xi)^{k} \Xi e_v^{1/2} e_R 
\end{align}
Also, for $k = 0, \ldots, L$, we have
\begin{align}
\co{ \nab^k (\pr_t + v \cdot \nab) P } &\leq C_k B_\la^{1/2} \left( \fr{e_R^{1/2} N}{e_v^{1/2}} \right)^{1/2}  (B_\la N \Xi)^{k} \Xi e_v^{1/2} e_R \label{bounds:nabKDdtofPcheck}
\end{align}
\end{prop}

For example, by Propositions (\ref{prop:spaceDvOfV}) and (\ref{bound:daRep})
\ali{
\co{ P_0 } &\leq C( \| e(t) \|_{C^0(\R)} + \co{ R_\ep} ) \leq C e_R \\
\co{ P_{I,J} } &\leq \fr{1}{2} ( \co{ V_I } \co{V_J} ) \leq C e_R \\
\co{ \nab P_0 } &\leq C \co{ \nab R_\ep } \leq C \Xi e_R \\
\co{ \nab P_{I,J} } &\leq \fr{1}{2} (  \co{ \nab V_I } \co{V_J} + \co{ V_I } \co{ \nab V_J} ) \\
&\leq C (N \Xi e_R^{1/2} )e_R^{1/2} = C N \Xi e_R
}
For the first material derivative of the low frequency term, we can apply (\ref{ineq:mollifyVerrorAgain}) and Proposition (\ref{prop:firstDdtOfRep})
\ali{
\co{ (\pr_t + v \cdot \nab) P_0 } &= \co{ \Ddtof{P_0} + ( v - v_\ep ) \cdot \nab P_0 } \\
&\leq \co{ \Ddtof{P_0} } + \co{ v - v_\ep } \co{ \nab P_0 } \\
\co{ v - v_\ep } \co{ \nab P_0 } &\leq ( \fr{e_v^{1/2} }{N} ) (\Xi e_R ) \\
 \co{ \Ddtof{P_0} } &\leq C( \| e'(t) \|_{C^0(\R)} + \co{ \Ddtof{R_\ep} } ) \\
 &\leq C \Xi e_v^{1/2} e_R \\
\co{ (\pr_t + v \cdot \nab) P_0 } &\leq C \Xi e_v^{1/2} e_R
}
For the first material derivative of the high frequency terms, we have
\ali{
\co{ (\pr_t + v \cdot \nab) P_{I,J} } &\leq \co{ \Ddtof{P_{I,J}} } + \co{ v - v_\ep } \co{ \nab P_{I,J} } \\
\co{ v - v_\ep } \co{ \nab P_{I,J} } &\leq C (\fr{e_v^{1/2}}{N})(N \Xi e_R) \\
&\leq C \Xi e_v^{1/2} e_R \\
\co{ \Ddtof{P_{I,J}} } &\leq ( \co{ \Ddtof{V_I} } \co{ V_J } + \co{ V_I } \co{ \Ddtof{V_J} } ) \\
&\leq C B_\la^{1/2} \left( \fr{e_R^{1/2} N}{e_v^{1/2}} \right)^{1/2}  \Xi e_v^{1/2} e_R^{1/2}
}
The bounds stated for $\co{\Ddtof{V_I}}$ in the preceding inequality do not follow directly from Corollary (\ref{cor:realMatDvOfV}), but follow from the same proof.

%% file: energyApprox1.tex
In this section, we will establish the estimates (\ref{eq:energyPrescribed}) and (\ref{eq:DtenergyPrescribed}) of the Main Lemma (\ref{lem:iterateLem}), which state that we are able to accurately prescribe the energy
\begin{align}
\int |V|^2(t,x) dx &\approx \int_{\T^3} e(t) dx  
\end{align}
that the correction adds to the solution, and also bound the difference between the time derivatives of these two quantities.

Quantitatively, we have a bound
\begin{prop}[Prescribing energy; proof of (\ref{eq:energyPrescribed})] 
\begin{align}
\| \int |V|^2(t,x) dx - \int_{\T^3} e(t) dx  \|_{C_t^0} &\leq  C \fr{e_R}{N}
\end{align}
\end{prop}
\begin{proof}
We calculate
\begin{align}
\int |V|^2(t,x) dx &= \int V^j V^l \de_{jl} dx \notag \\
&= \int \sum_{I, J} V_I^j V_J^l \de_{jl} dx \notag \\
&= \sum_I \int V_I^j {\bar V}_I^l \de_{jl} dx + \sum_{J \neq {\bar I} } \int V_I \cdot V_J dx \notag \\
&=  \sum_I \int {\tilde v}_I^j {\overline {\tilde v}}_I^l \de_{jl} dx + \sum_{J \neq {\bar I} } \int e^{i \la (\xi_I + \xi_J)} {\tilde v}_I \cdot {\tilde v}_J dx \notag \\
&= \int \left( \sum_I v_I^j {\bar v}_I^l \right) \de_{jl} dx + \sum_I \int \fr{(\nab \times w_I)}{\la} \cdot{\bar v}_I dx  \notag \\
&+ \sum_I \int \fr{(\nab \times {\bar w}_I)}{\la} \cdot v_I dx + \sum_I \int \fr{(\nab \times w_I)}{\la} \cdot \fr{(\nab \times w_I)}{\la} dx \notag  \\
&+ \sum_{J \neq {\bar I} } \int e^{i \la (\xi_I + \xi_J)} {\tilde v}_I \cdot {\tilde v}_J dx \notag \\
&= \int e(t) dx + E_{(A)}(t) + \sum_{J \neq {\bar I}}E_{(B, IJ)}(t)
\end{align}
where it is clear that we have a bound for
\begin{align}
|E_{(A)}(t)| &\leq 2 \sum_I \int \left|\fr{(\nab \times w_I)}{\la} \cdot{\bar v}_I \right| dx + \sum_I \int \left|\fr{(\nab \times w_I)}{\la} \cdot \fr{(\nab \times w_I)}{\la} \right| dx \\
&\leq C ( \fr{ \Xi e_R^{1/2} \cdot e_R^{1/2} }{B_\la \Xi N} +  (\fr{ \Xi e_R^{1/2} }{B_\la \Xi N})^2 ) \\
&\leq C \fr{e_R}{N}
\end{align}
and each integral contributing to the other term can be estimated by integrating by parts
\begin{align}
E_{(B, IJ)}(t) &= \int e^{i \la (\xi_I + \xi_J)} {\tilde v}_I \cdot {\tilde v}_J dx \\
&= \int \left[ \fr{\pr^a(\xi_I + \xi_J)}{\la |\nab(\xi_I + \xi_J)|^2} \pr_a[e^{i \la (\xi_I + \xi_J)}] \right] {\tilde v}_I \cdot {\tilde v}_J dx \\
&= \fr{-1}{\la} \int e^{i \la (\xi_I + \xi_J)} \pr_a \left[ \fr{\pr^a(\xi_I + \xi_J)}{|\nab(\xi_I + \xi_J)|^2} {\tilde v}_I \cdot {\tilde v}_J \right] dx \label{eq:intByParts} \\
| E_{(B, IJ)}(t) | &\leq C \fr{\Xi e_R}{B_\la N \Xi}
\end{align}
\end{proof}

Similarly, we can estimate our control over the rate of energy variation
\begin{prop}[Rate of energy variation estimate; proof of (\ref{eq:DtenergyPrescribed})]
\begin{align}
\| \fr{d}{dt} [ \int |V|^2(t,x) dx - \int_{\T^3} e(t) dx ] \|_{C_t^0} &\leq C \badFactor ( \Xi e_v^{1/2} ) \fr{e_R}{N}
\end{align}
\end{prop}
\begin{proof}
We again use the expression
\begin{align}
\int |V|^2(t,x) dx - &\int e(t) dx = E_{(A)}(t) + E_{(B)}(t) \\
E_{(A)}(t) &= \sum_I \int \fr{(\nab \times w_I)}{\la} \cdot{\bar v}_I dx + \sum_I \int \fr{(\nab \times {\bar w}_I)}{\la} \cdot v_I dx \\
&+ \sum_I \int \fr{(\nab \times w_I)}{\la} \cdot \fr{(\nab \times w_I)}{\la} dx \\
E_{(B)}(t) &= \sum_{J \neq {\bar I} } \int e^{i \la (\xi_I + \xi_J)} {\tilde v}_I \cdot {\tilde v}_J dx
\end{align}
And differentiate in time, starting with
\begin{align}
\fr{d}{dt} E_{(A)}(t) &= \sum_I \int \Ddt \left[ \fr{(\nab \times w_I)}{\la}\cdot {\bar v}_I + \fr{(\nab \times {\bar w}_I)}{\la} \cdot v_I + \fr{(\nab \times w_I)}{\la} \cdot \fr{(\nab \times w_I)}{\la} \right] dx
\end{align}
As we know, the coarse scale material derivative $\Ddt$ costs a factor 
\[ | \Ddt | \leq C \badFactor \Xi e_v^{1/2} \]
giving 
\begin{align}
| \fr{d}{dt} E_{(A)}(t) | &\leq C \badFactor \Xi e_v^{1/2} \fr{e_R}{N}
\end{align}
and similarly, we can write
\begin{align}
\fr{d}{dt} E_{(B)}(t) &= \sum_{J \neq {\bar I}} \int \Ddt \left[ e^{i \la (\xi_I + \xi_J)} {\tilde v}_I \cdot {\tilde v}_J \right] dx \\
&= \sum_{J \neq {\bar I}} \int e^{i \la (\xi_I + \xi_J)} \Ddt \left[  {\tilde v}_I \cdot {\tilde v}_J \right] dx 
\end{align}
which can be estimated by
\begin{align}
| \fr{d}{dt} E_{(B)}(t) | &\leq C \badFactor \Xi e_v^{1/2} \fr{e_R}{N}
\end{align}
using another integration by parts, just as in (\ref{eq:intByParts}).
\end{proof}

%% file: compareMainLem1.tex
Since we have established many estimates for the corrections to the velocity and the pressure, we are in a position to compare with the Main Lemma.

We first remark that the estimates (\ref{eq:Wco})-(\ref{eq:matWco}) for $W$ were established in Proposition (\ref{prop:nabkW}) and Corollary (\ref{cor:realMatDvOfW}).  Similarly, the estimates (\ref{eq:Vco}) - (\ref{eq:matVco}) were special cases of the Proposition (\ref{prop:spaceDvOfV}) and Corollary (\ref{cor:realMatDvOfV}).  Also, the estimates (\ref{eq:Pco}) - (\ref{eq:matPco}) were established by Proposition (\ref{prop:pressureCorrectEstimates}).  At least, all these bounds will be established for a particular constant $C$ once the constant $B_\la$ has been chosen.

We are also in a position to check that the frequency and energy levels of the new velocity and pressure
\begin{align}
v_1 &= v + V \\
p_1 &= p + P
\end{align}
are consistent with the claims of the Main Lemma (\ref{lem:iterateLem}).

For the velocity, at the level of the first derivative, we have
\begin{align}
\co{ \nab v_1 } &\leq \co{\nab v} + \co{\nab V} \\
&\leq C( \Xi e_v^{1/2} + N B_\la \Xi e_R^{1/2} ) \\
&\leq C N \Xi e_R^{1/2}
\end{align}
since $N \geq \left( \fr{e_v}{e_R} \right)^{1/2}$ and since $B_\la$ will be chosen to be some constant which has been absorbed into the $C$.  

We also have
\begin{align}
\co{ \nab^k v_1 } &\leq C_k (B_\la N \Xi)^k e_R^{1/2} 
\end{align}
for $k = 1, \ldots, L$, which is exactly consistent with the Definition (\ref{bound:nabkv}) for the new frequency energy levels $\Xi' = C N \Xi$ and $e_v' = e_R$.

For the pressure, we have
\begin{align}
\co{ \nab p_1 } &\leq \co{\nab p} + \co{\nab P} \\
&\leq C ( \Xi e_v + N B_\la \Xi e_R ) \\
&\leq C B_\la N \Xi e_R
\end{align}
since $N \geq \left( \fr{e_v}{e_R} \right)$.  This bound is also exactly consistent with the Definition (\ref{bound:nabkp}) for the new frequency energy levels $\Xi' = C N \Xi$ and $e_v' = e_R$.

To complete the proof of the Main Lemma (\ref{lem:iterateLem}), it now only remains to choose a constant $B_\la$ so that (\ref{bound:nabkR}) and (\ref{bound:dtnabkR}) can be verified for the new energy levels.  This choice of $B_\la$ is the last step of the proof.

%% file: introToStressEstimates.tex
To conclude the argument, we must calculate the new stress $R_1$ and verify that this stress obeys all the estimates required in Lemma (\ref{lem:iterateLem}).

A complete list of the terms contributing to the new stress $R_1$ includes the following:
\begin{enumerate}
\item {\bf The Mollification Terms} 
\begin{align}
Q_{M,v}^{jl} &= (v^j - v_\ep^j)V^l + V^j (v^l - v_\ep^l) \\ 
Q_{M,R}^{jl} &= R^{jl} - R_\ep^{jl}
\end{align}
\item {\bf The Stress Term}
\begin{align}
\begin{split}
Q_S^{jl} &= \sum_I \fr{(\nab \times w_I)^j{\bar v}_I^l }{\la} + \sum_I \fr{v_I^j(\nab \times {\bar w}_I)^l }{\la} \\
&+  \sum_I \fr{(\nab \times w_I)^j(\nab \times {\bar w}_I)^l}{\la^2}
\end{split}
\end{align}
\item {\bf The High-Low Interaction Term }
\begin{align}
\pr_j Q_L^{jl} &= \sum_I e^{i \la \xi_I}( {\tilde v}_I^j \pr_j v_\ep^l )
\end{align}
\item {\bf The High-High Interference Terms }
\begin{align}
\pr_j Q_H^{jl} &= - \sum_{J \neq {\bar I}}  \la e^{i \la(\xi_I + \xi_J) }[ v_I \times ( |\nab \xi_J| - 1 ) v_J + v_J \times (|\nab \xi_I| - 1) v_I ] + \mbox{ lower order terms }
\end{align}
\item {\bf The Transport Term}
\begin{align}
\pr_j Q_T^{jl} &= \sum_I e^{i \la \xi_I}[ (\pr_t + v_\ep^j \pr_j){\tilde v}_I^l ]
\end{align}
\end{enumerate}

For all of these terms, we must verify\footnote{Recall that the notation $\unlhd$ refers inequalities which are goals but which generally have not yet been proven.} the following estimates
\begin{align}
\label{ineq:theC0Goal5th}
\begin{split}
\co{ R_1 } &\unlhd e_R' \\
\Leftrightarrow \quad \quad \quad \co{ R_1 } &\unlhd \badFactor \fr{e_v^{1/2} e_R^{1/2}}{N} 
\end{split} \\
\label{eq:theSpatDvBdsWeNeed} 
\begin{split}
\co{\nab^k R_1 } &\unlhd (\Xi')^k e_R'  \quad \quad \quad \quad k = 1, \ldots, L \\
\Leftrightarrow  \quad \quad \quad  \co{\nab^k R_1 } &\unlhd C (N \Xi)^k \left[ \badFactor \fr{e_v^{1/2} e_R^{1/2}}{N} \right] \quad \quad \quad  k = 1, \ldots, L 
\end{split} \\
\label{ineq:matDvNeed5th} 
\begin{split}
\co{(\pr_t + v_1\cdot \nab )R_1 } &\unlhd (\Xi' (e_v')^{1/2} ) e_R'  \\
\Leftrightarrow \quad \quad \co{(\pr_t + v_1\cdot \nab )R_1 } &\unlhd C \badFactor \Xi e_v^{1/2} e_R 
\end{split} \\
\co{\nab^k (\pr_t + v_1\cdot \nab) R_1 } &\unlhd C ( N \Xi )^k\left[ \badFactor \Xi e_v^{1/2} e_R \right] \quad \quad k = 0, \ldots, L - 1 \label{ineq:matDvNeed5th2} \\
v_1 &= v + V \notag
\end{align}

Note that, in contrast to the goals (\ref{eq:theSpatDvBdsWeNeed}) - (\ref{ineq:matDvNeed5th2}), the goal (\ref{ineq:theC0Goal5th}) does not allow for an implied constant $C$.

The goals (\ref{ineq:matDvNeed5th}) and (\ref{ineq:matDvNeed5th2}) have a slightly inconvenient feature in that while most of our bounds are stated for the coarse scale material derivative
\[ \Ddt = (\pr_t + v_\ep \cdot \nab) \]
the goals (\ref{ineq:matDvNeed5th}) and (\ref{ineq:matDvNeed5th2}) require estimates for the derivative
\ali{
(\pr_t + v_1 \cdot \nab) &= (\pr_t + v \cdot \nab) +  V \cdot \nab  \\
&= (\pr_t + v_\ep \cdot \nab) + [ (v - v_\ep) + V ] \cdot \nab 
}

Therefore, before we proceed, it is useful to remark that once the spatial derivative bounds (\ref{eq:theSpatDvBdsWeNeed}) have been verified, then the bounds (\ref{ineq:matDvNeed5th}) and (\ref{ineq:matDvNeed5th2}) need only be checked for the derivative $\Ddt$ in place of $(\pr_t + v_1 \cdot \nab)$.  We summarize this observation through the following Proposition.
\begin{prop}[Material Derivative Criterion 1] \label{prop:canJustCheckCoarseScale}
There exists a constant $C$ such that for all symmetric $(2,0)$ tensor fields $Q^{jl}$ satisfying the bounds
\ali{
\co{\nab^k Q } &\leq A (N \Xi)^k \left[ \badFactor \fr{e_v^{1/2} e_R^{1/2}}{N} \right] \quad \quad k = 1, \ldots, L \label{bound:alreadyHaveNabkQ}
}
and
\ali{
\co{\nab^k (\pr_t + v_\ep \cdot \nab) Q } &\leq A ( N \Xi )^k \left[ \badFactor \Xi e_v^{1/2} e_R \right] \quad \quad k = 0, \ldots, L - 1 \label{bound:alreadyHaveDdtQ}
}
then we also have
\ali{
\co{\nab^k (\pr_t + v_1 \cdot \nab) Q } &\leq C A ( N \Xi )^k \left[ \badFactor \Xi e_v^{1/2} e_R \right] \quad \quad k = 0, \ldots, L - 1  \label{bound:wishWeHadDdtv1Q}
}
\end{prop}
\begin{proof}
By writing
\ali{
(\pr_t + v_1 \cdot \nab)Q &= (\pr_t + v_\ep \cdot \nab)Q + [ (v - v_\ep) + V ] \cdot \nab Q
}
and using (\ref{bound:alreadyHaveDdtQ}) it suffices to prove the bound
\ali{
\co{ \nab^k [ [ (v - v_\ep) + V ] \cdot \nab Q ] } &\leq C A ( N \Xi )^k\left[ \badFactor \Xi e_v^{1/2} e_R \right] \quad \quad k = 0, \ldots, L - 1 
}
For $k = 0$, using (\ref{ineq:mollifyVerror}), Proposition (\ref{prop:spaceDvOfV}), and the assumption (\ref{bound:alreadyHaveNabkQ}) we have
\ali{
\co{ [ (v - v_\ep) + V ] \cdot \nab Q } &\leq ( \co{ v - v_\ep} + \co{ V } ) \co{\nab Q } \\
&\leq C \left( \fr{e_v^{1/2}}{N} + e_R^{1/2} \right) \left( A (N \Xi) \badFactor \fr{e_v^{1/2} e_R^{1/2}}{N} \right) \\
&\leq C A \badFactor ( \Xi e_v^{1/2} e_R )
}
since $N \geq \left( \fr{e_v}{e_R} \right)^{1/2}$.  Taking a spatial derivative will costs a factor $C N \Xi$ in the estimate.  For example, if we estimate $\co{ \nab ( v - v_\ep ) }$ slightly suboptimally and we compare
\ali{
\co{ ( v - v_\ep) } &\leq \fr{e_v^{1/2}}{N} \\
\co{ \nab( v - v_\ep) } \leq \co{ \nab v} + \co{ \nab v_\ep } &\leq C \Xi e_v^{1/2}
}
we see that we lose a factor $C N \Xi$ by taking the derivative.  For the whole product we have the bound
\ali{
\co{ \nab[ [ (v - v_\ep) + V ] \cdot \nab Q ] } &\leq ( \co{ \nab v} + \co{\nab v_\ep} + \co{ \nab V } ) \co{\nab Q } \\
&+ \co{[ (v - v_\ep) + V ] } \co{ \nab^2 Q } \\
&\leq C \left( \Xi e_v^{1/2} + N \Xi e_R^{1/2} \right) \left( A \badFactor \Xi e_v^{1/2} e_R^{1/2} \right) \\
&+ C e_R^{1/2} \left( A (N \Xi) \badFactor \Xi e_v^{1/2} e_R^{1/2} \right)  \\
&\leq C A \badFactor (N \Xi) ( \Xi e_v^{1/2} e_R )
}
using (\ref{bound:alreadyHaveNabkQ}), (\ref{prop:spaceDvOfV}) and the fact that $N \geq \left( \fr{e_v}{e_R} \right)^{1/2}$.

The assumption (\ref{bound:alreadyHaveNabkQ}) and the bounds for $v$, $v_\ep$ and $V$ also ensure that each higher spatial derivative costs at most $C N \Xi$ per derivative up to order $L - 1$, which is what is required for the bound (\ref{bound:wishWeHadDdtv1Q}).
\end{proof}
The same proof also shows that it also suffices to verify estimates for the derivative $(\pr_t + v \cdot \nab) Q$.
\begin{prop}[Material Derivative Criterion 2] \label{prop:canJustCheckOldMat}
Proposition (\ref{prop:canJustCheckCoarseScale}) also holds with the condition (\ref{bound:alreadyHaveDdtQ}) replaced by
\ali{
\co{\nab^k (\pr_t + v \cdot \nab) Q } &\leq A ( N \Xi )^k\left[ \badFactor \Xi e_v^{1/2} e_R \right] \quad \quad k = 0, \ldots, L - 1 \label{bound:alreadyHaveOldDdtQ}
}
\end{prop}

We begin verifying the bounds (\ref{ineq:theC0Goal5th}) - (\ref{ineq:matDvNeed5th2}) by estimating the terms which do not require solving the divergence equation.

%% file: mollTermEstimates.tex
In this section, we estimate the terms in the new stress which do not involve oscillations.  These terms include

\begin{enumerate}
\item {\bf The Mollification Terms} 
\begin{align}
Q_{M,v}^{jl} &= (v^j - v_\ep^j)V^l + V^j (v^l - v_\ep^l) \\ 
Q_{M,R}^{jl} &= R^{jl} - R_\ep^{jl}
\end{align}
\item {\bf The Stress Term}
\begin{align}
\begin{split}
Q_S^{jl} &= \sum_I \fr{(\nab \times w_I)^j{\bar v}_I^l }{\la} + \sum_I \fr{v_I^j(\nab \times {\bar w}_I)^l }{\la} \\
&+  \sum_I \fr{(\nab \times w_I)^j(\nab \times {\bar w}_I)^l}{\la^2}
\end{split}
\end{align}\end{enumerate}

Throughout the estimates, we will assume that $B_\la$ has been chosen to be some constant.  We start with the Mollification Terms.

\subsection{The mollification term from the velocity} \label{sec:mollVOK}

The part of the new stress $R_1$ which arises from mollifying the velocity is given by
\begin{align}
Q_{M,v}^{jl} &= (v^j - v_\ep^j)V^l + V^j (v^l - v_\ep^l)
\end{align}

The $C^0$ estimate for $Q_{M,v}$ has already been verified by the choice of $\ep_v$ in Section (\ref{sec:accountForParams}).  In fact, by verifying (\ref{target:R1v}) we have that
for the main term
\begin{align}
Q_{v,(i)}^{jl} &= \sum_I 2[(v - v_\ep)(e^{i \la \xi_I} v_I )]^{jl}
\end{align}
we have
\begin{align}
\co{ Q_{v, (i)}^{jl}} &\leq \fr{1}{50} \fr{e_v^{1/2} e_R^{1/2}}{N} 
\end{align}
since
\begin{align}
\co{ (v^j - v_\ep^j) } &\leq c \fr{e_v^{1/2}}{N}
\end{align}

The other term for $Q_{M,v}$ is of the form
\begin{align}
Q_{v, (ii)} &= \sum_I 2 e^{i \la \xi_I} \fr{[(v - v_\ep)\nab\times w_I]^{jl}}{\la} \label{eq:waitingForLaAgain} \\
\co{ Q_{v, (ii)} } &\leq C \left(\fr{e_v^{1/2}}{N} \right) \left( \fr{\Xi e_R^{1/2}}{B_\la N \Xi} \right) 
\end{align}
which satisfies the goal 
\begin{align}
\co{ Q_{v, (ii)}^{jl}} &\leq \fr{1}{50} \fr{e_v^{1/2} e_R^{1/2}}{N} 
\end{align}
once the constant $B_\la$ is chosen large enough.

We need to check the spatial derivatives of $Q_{M,v}$.  By Definition (\ref{def:subSolDef}), the estimates for spatial derivatives are allowed to lose a factor 
\[ |\nab| \leq \Xi' = C N \Xi \]
for each spatial derivative $\nab^k$, $k \leq L$.  This freedom allows us to separate the factors $v$ and $v_\ep$ to bound
\begin{align}
\nab Q_{M,v}^{jl} &= (\nab v V)^{jl} - (\nab v_\ep V)^{jl} + [( v - v_\ep) \nab V]^{jl} \label{eq:QMvnab} \\
\co{\nab Q_{M,v}} &\leq C ( \co{ \nab v} + \co{\nab v_\ep} ) e_R^{1/2} + \co{ v - v_\ep} \co{\nab V} \\
&\leq C ( \Xi e_v^{1/2} e_R^{1/2} + \fr{e_v^{1/2}}{N} (B_\la N \Xi e_v^{1/2}) \\
&\leq C B_\la (\Xi e_v^{1/2}) e_R^{1/2}
\end{align}

By comparing 
\begin{align}
\co{ (v^j - v_\ep^j) } &\leq c \fr{e_v^{1/2}}{N} \\
\co{ \nab v} + \co{\nab v_\ep} &\leq 2 \Xi e_v^{1/2}
\end{align}
we can see that the cost of separating $v$ and $v_\ep$ in the estimate is exactly the factor $\Xi' = N \Xi$ that we can afford.  This principle will be useful many times in the checking.

By taking up to $L - 1$ derivatives of (\ref{eq:QMvnab}), we see that for all $k = 1, \ldots, L$ we have
\begin{prop}  For $k = 0, \ldots, L$, there exists constants $C_k$ depending on $B_\la$ such that
\begin{align}\label{eq:weHaveTheNabDvQMv}
\co{ \nab^k Q_{M,v} } &\leq C_k (B_\la N \Xi)^k \fr{ e_v^{1/2} e_R^{1/2} }{N}
\end{align}
\end{prop}

This estimate is actually sufficient for the ideal case discussed in Section (\ref{sec:onOnsag}).

We now turn to estimating the material derivative
\begin{prop}
Let $v_1 = v+ V$.  Then, for $k = 0, \ldots, L - 1$, we have
\begin{align}\label{eq:weHaveTheMatDvQMv}
\co{ \nab^k (\pr_t + v_1 \cdot \nab) Q_{M,v} } &\leq C \badFact (N \Xi)^k \fr{e_v^{1/2} e_R^{1/2}}{N}
\end{align}
\end{prop}

To prove this proposition, it actually suffices to estimate $\Ddtof{Q_{M,v}}$
\begin{prop}  For $k = 0, \ldots, L - 1$, we have
\begin{align} \label{eq:weCanCheckDdt}
\co{ \nab^k (\pr_t + v_\ep \cdot \nab) Q_{M,v} } &\leq C \badFact (N \Xi)^k \fr{e_v^{1/2} e_R^{1/2}}{N}
\end{align}
\end{prop}
\noindent by the criterion in Proposition (\ref{prop:canJustCheckCoarseScale}).  Now let us check the bounds (\ref{eq:weCanCheckDdt}).  
\begin{proof}
Here we take
\begin{align}
\Ddt [ (v - v_\ep) V ]^{jl} &= (\Ddtof{v}V)^{jl} - (\Ddtof{v_\ep}V)^{jl} + [(v - v_\ep) \Ddtof{V} ]^{jl} \\
&= A_{(I)} + A_{(II)} + A_{(III)}
\end{align}

First, we compare
\begin{align}
\co{ A_{(III)} } &\leq C \co{ v - v_\ep } \co{ \Ddtof{V} } \\
&\leq C \fr{e_v^{1/2}}{N}\left(\badFact \Xi e_v^{1/2} e_R^{1/2} \right) \\
&\leq C \badFact \fr{ \Xi e_v e_R^{1/2} }{N}
\end{align}
to our goal of
\begin{align}
\co{ A_{(III)} } &\unlhd \Xi' (e_v')^{1/2} e_R' \\ 
&=C ( N \Xi) e_R^{1/2} \left(\badFact \fr{e_v^{1/2} e_R^{1/2}}{N}\right) \\
&= C \badFact \Xi e_v^{1/2} e_R \label{eq:theGoalForDdtQ}
\end{align}
to see that $\co{ A_{(III)} }$ has a correct bound since $N \geq \left( \fr{e_v}{e_R} \right)^{1/2}$.  As we discussed in verifying, (\ref{eq:weHaveTheMatDvQMv}), the estimates for
\[ \co{ \nab^k A_{(III)}} \leq C_k (N\Xi)^k \badFact \Xi e_v^{1/2} e_R  \] 
also hold for $k = 1, \ldots, L - 1$, since each spatial derivative costs at most $N \Xi$, and because at most $L$ derivatives fall on the given $v$.

Next we check how the bound
\begin{align}
\co{A_{(II)}} &\leq \co{ \Ddtof{v_\ep} } \co{V} \\
&\leq (\Xi e_v ) e_R^{1/2} 
\end{align}
compares to the goal
\begin{align}
\co{A_{(II)}} &\unlhd C \badFact \Xi e_v^{1/2} e_R \label{ineq:theGoalWeBarelyHit}
\end{align}
from (\ref{eq:theGoalForDdtQ}).

Checking this bound reduces to verifying that
\begin{align}
\left( \fr{e_v}{e_R} \right)^{3/4} &\leq N^{1/2} \label{ineq:whatWeNeedForcoAII}
\end{align}
which follows from the hypothesis
\begin{align}
\left( \fr{e_v}{e_R} \right)^{3/2} &\leq N \label{eq:theNReqIsBig}
\end{align}
in line (\ref{eq:conditionsOnN2}) of the Main Lemma.

Now that the estimate for $A_{(II)}$ has been satisfied at the level of $C^0$, it worsens by a factor of $N \Xi$ for each spatial derivative, the largest cost arising when the derivative hits the oscillatory factor in the correction $V$.  Therefore we have the bound
\begin{align}
\co{\nab^k A_{(II)}} &\leq C (N \Xi)^k (\Xi e_v ) e_R^{1/2}  \label{eq:boundOnNabKAIImatdvMollv}
\end{align}
The goal 
\begin{align}
\co{\nab^k A_{(II)}} &\unlhd C (N \Xi)^k  \badFact \Xi e_v^{1/2} e_R \quad \quad k = 1, \ldots, L
\end{align}
permits a loss of $C N \Xi$ per derivative, so that the bound (\ref{eq:boundOnNabKAIImatdvMollv}) is acceptable.

We proceed similarly to estimate
\begin{align}
A_{(I)} &= \Ddtof{v}V^l 
\end{align}
by writing
\begin{align}
\Ddtof{v^l} &= (\pr_t + v\cdot \nab) v^l + (v_\ep - v) \cdot \nab v^l \\
&= \pr^l p + \pr_j R^{jl} + (v_\ep - v) \cdot \nab v^l
\end{align}
giving
\begin{align}
\co{ A_{(I)} } &\leq C \co{\Ddtof{v}} e_R^{1/2} \\
&\leq C (\Xi e_v + \left( \fr{e_v^{1/2}}{N} \right)(\Xi e_v^{1/2} ) ) e_R^{1/2} \\
&\leq C \Xi e_v e_R^{1/2}
\end{align}
which we have just verified is acceptable for the term $A_{(II)}$.  We can now take spatial derivatives of this term up to order $L - 1$, and each spatial derivative costs at most $C N \Xi$ per derivative; for example, comparing
\ali{
\co{v - v_\ep} &\leq C \fr{e_v^{1/2}}{N} \\
\co{ \nab( v - v_\ep) } \leq C \co{ \nab v} + \co{\nab v_\ep } &\leq 2 \Xi e_v^{1/2}
}
costs a factor $|\nab| \leq C N \Xi$.  
\end{proof}

\paragraph{Wastefulness in the estimate}  

To prove the ideal case Conjecture (\ref{conj:idealCase}), the bound (\ref{eq:weCanCheckDdt}) must be replaced by the more stringent goal
\begin{align}
\co{\Ddt [ (v - v_\ep) V ]^{jl}} &\unlhd C \Xi e_v^{1/2} e_R \label{ineq:theGoalWeShouldBeMeeting}
\end{align}
In order to verify inequality (\ref{ineq:theGoalWeShouldBeMeeting}), it is necessary to establish a bound
\ali{
 \co{\Ddt[ v - v_\ep] } &\leq C \Xi e_v^{1/2} e_R^{1/2} \label{ineq:complicatedCommutator}
}
which is stronger than the bound of $\Xi e_v$ that can be proved for $\co{ \Ddtof{v}}$ and $\co{ \Ddtof{v_\ep} }$ individually.

Inequality (\ref{ineq:complicatedCommutator}) can be proven through a straightforward but long commutator argument similar to those in Sections (\ref{coarseScaleFlow}) and (\ref{sec:firstMatDvofRep}).  Here we only note that this commutator estimate requires control over $\nab[ (\pr_t + v \cdot \nab) v ]$, and suggests that it may be important to work with Frequency Energy levels of order $L \geq 2$.

Here we have taken the slightly easier route of to establishing the bound (\ref{ineq:theGoalWeBarelyHit}), at the cost of imposing the lower bound (\ref{eq:theNReqIsBig}) on $N$.  The previous bounds are actually the only place in the argument in the requirement (\ref{eq:theNReqIsBig}) is used, and the Lemma (\ref{lem:iterateLem}) can be proven with (\ref{eq:theNReqIsBig}) replaced by 
\[ N \geq \left( \fr{e_v}{e_R} \right) \]
We also remark that in the ideal case construction, inequality (\ref{eq:theNReqIsBig}) is satisfied during the iteration of Section (\ref{sec:onOnsag}), but in that case the power $3/2$ cannot be replaced by any larger value.  During the iteration of Section (\ref{sec:lemImpThm}), the value of $N$ is essentially $\left(\fr{e_v}{e_R} \right)^{5/2}$.

\subsection{The mollification term from the Stress}

We now verify the necessary estimates for the term
\begin{align}
Q_{M,R}^{jl} &= R^{jl} - R_\ep^{jl}
\end{align}  

We have already verified in Section (\ref{sec:chooseRParams}) that
\begin{align}
\co{ Q_{M,R}^{jl} } &\leq \fr{e_v^{1/2} e_R^{1/2}}{100 N}
\end{align}

And regarding the spatial derivatives, we have for $k = 1, \ldots, L$
\begin{align}
\co{ \nab^k Q_{M,R} } &\leq \co{ \nab^k R^{jl} } + \co{\nab^k R_\ep^{jl} } \\
&\leq C ( \Xi^k e_R + \Xi^k e_R ) \\
&\leq C \Xi^k e_R
\end{align}
from Proposition (\ref{bound:daRep}).  Here, no powers of $N^{1/L}$ arise since we have taken no more than $L$ derivatives.  Comparing to our goal of
\begin{align}
\co{ \nab^k Q_{M,R} } &\unlhd (\Xi')^k e_R' \\
&= \badFact (N \Xi)^k \fr{e_v^{1/2} e_R^{1/2}}{N}
\end{align}
we see that this goal has been achieved by a large margin.

We now check the bounds for the material derivative, which are of the form
\begin{prop} For $v_1 = v + V$, $k = 0, \ldots, L - 1$ we have
\begin{align} \label{eq:nextDtvQmR}
\co{ \nab^k( \pr_t + v_1 \cdot \nab) Q_{M,R} } &\leq C (N \Xi)^k \Xi e_v^{1/2} e_R
\end{align}
\end{prop}

For our purposes, it is enough to recall from (\ref{bound:dtnabkR}) and Proposition (\ref{prop:firstDdtOfRep}) that
\begin{prop} For $0 \leq k \leq L - 1$,
\begin{align}
\co{ \nab^k ( \pr_t + v \cdot \nab) R } &\leq C \Xi^{k + 1} e_v^{1/2} e_R \\
\co{ \nab^k ( \pr_t + v_\ep \cdot \nab) R_\ep }&\leq  C \Xi^{k + 1} e_v^{1/2} e_R
\end{align}
\end{prop}
according to the criteria stated in Propositions (\ref{prop:canJustCheckCoarseScale}) and (\ref{prop:canJustCheckOldMat}).

We need to check, at least for $k = 0$, that we have achieved the goal
\begin{align}
\co{ ( \pr_t + v_1 \cdot \nab) Q_{M,R} } &\unlhd C \left( \Xi' (e_v')^{1/2} e_R' \right) \\
&\unlhd C \badFact \left( N \Xi e_R^{1/2} \fr{e_v^{1/2} e_R^{1/2}}{N} \right) \\
&= C \badFact \Xi e_v^{1/2} e_R
\end{align}

In fact, this goal has been achieved even for the ideal case without the factor $\badFact$.  For spatial derivatives the estimates (\ref{eq:nextDtvQmR}) show that each spatial derivative costs at most $N \Xi$ as desired, the largest losses coming from $V \cdot \nab (R - R_\ep)$.

\subsection{Estimates for the Stress Term} \label{sec:estimatesForStressTerm}
We now verify the estimates (\ref{ineq:theC0Goal5th})-(\ref{ineq:matDvNeed5th2}) for the Stress Term, which turns out to be the most straightforward term to estimate
\begin{align}
\label{eq:recallStressTerm}
\begin{split}
Q_S^{jl} &= \sum_I \fr{(\nab \times w_I)^j{\bar v}_I^l }{\la} + \sum_I \fr{v_I^j(\nab \times {\bar w}_I)^l }{\la} \\
&+  \sum_I \fr{(\nab \times w_I)^j(\nab \times {\bar w}_I)^l}{\la^2}
\end{split}
\end{align}
First, we have the bound
\begin{align}
\co{ Q_S } &\leq C \fr{\Xi e_R^{1/2} \cdot e_R^{1/2}}{B_\la N \Xi} = C \fr{e_R}{B_\la N}  
\end{align}
which implies that the goal for the ideal case
\begin{align}
\co{ Q_S } &\leq \fr{e_v^{1/2} e_R^{1/2}}{500 N} \label{eq:goalForStressTerm}
\end{align}
is satisfied for $B_\la$ sufficiently large.  From  (\ref{eq:recallStressTerm}), we also have the estimates
\begin{align}
\co{ \nab^k Q_S } &\leq C_k (N \Xi)^k \fr{e_R}{N}
\end{align}
for $k = 1, \ldots, L$ (in fact, the bounds from Section (\ref{sec:vecAmpBounds}) are much better), which implies the desired cost of $|\nab| \leq C N \Xi$ demanded by the goal (\ref{eq:theSpatDvBdsWeNeed}).

Next, we have an estimate
\begin{align}
\co{ (\pr_t + v_\ep \cdot \nab) Q_S } &\leq C \badFact \Xi e_v^{1/2} \fr{e_R}{N} \label{eq:matDvBoundStressTerm}
\end{align}
coming from the estimates in Section (\ref{sec:vecAmpBounds}), which already implies the bound 
\begin{align}
\co{ (\pr_t + v_\ep \cdot \nab) Q_S } &\leq C \Xi e_v^{1/2} e_R
\end{align}
which is desired for the ideal case as in (\ref{ineq:theGoalWeShouldBeMeeting}).  In particular, the bound (\ref{eq:matDvBoundStressTerm}) is enough to verify the required bound (\ref{ineq:matDvNeed5th}).  It also follows from the estimates of Section (\ref{sec:vecAmpBounds}) that the coarse scale material derivative $\Ddt Q_S$ can be differentiated in space arbitrarily many times at a cost no larger than $|\nab| \leq C N \Xi$ per derivative.  By Proposition (\ref{prop:canJustCheckCoarseScale}), it follows that all of the required estimates (\ref{ineq:theC0Goal5th})-(\ref{ineq:matDvNeed5th2}) are satisfied by $Q_S$ with room to spare.

%% file: oscillatoryStress.tex
In this section, we will estimate the terms in the stress which involve solving a divergence equation of the form
\begin{align}
\pr_j Q^{jl} &= U^l = e^{i \la \xi} u^l
\end{align}
These terms include
\begin{enumerate}
\item {\bf The High-Low Term }
\begin{align}
\pr_j Q_L^{jl} &= \sum_I e^{i \la \xi_I}( {\tilde v}_I^j \pr_j v_\ep^l ) \\
&= \sum_I e^{i \la \xi_I} u_L^l
\end{align}
\item {\bf The Main High-High Terms }
\begin{align}
\pr_j Q_H^{jl} &= -\sum_{J \neq {\bar I}}  \la e^{i \la(\xi_I + \xi_J) } [ v_I \times ( |\nab \xi_J| - 1 ) v_J^l + v_J \times (|\nab \xi_I| - 1) v_I^l ] \\
&=  \sum_{J \neq {\bar I}} u_{H, (IJ)}^l
\end{align}
\item {\bf The Remainder of the High-High terms}
\begin{align}
\begin{split}
\pr_j Q_{H'}^{jl} &= - e^{i \la (\xi_I + \xi_J) }\left( (\nab \times w_I) \times [ (i \nab \xi_J) \times {\tilde v}_J ] + (\nab \times w_J) \times [ (i \nab \xi_I) \times {\tilde v}_I ] \right) \\
&- e^{i \la (\xi_I + \xi_J) } \left( v_I \times [ (i \nab \xi_J) \times  (\nab \times w_J) ] + v_J \times [ (i \nab \xi_I) \times (\nab \times w_I) ]\right) \\
&- e^{i \la(\xi_I + \xi_J)} \left( \tilde{v}_I \times(\nab \times \tilde{v}_J) + \tilde{v}_J \times(\nab \times \tilde{v}_I) \right)
\end{split}
\end{align}
\item {\bf The Transport Term}
\begin{align}
\pr_j Q_T^{jl} &= \sum_I e^{i \la \xi_I}[ (\pr_t + v_\ep^j \pr_j){\tilde v}_I^l ] \\
&= \sum_I e^{i \la \xi_I} u_{T,(I)}^l
\end{align}
\end{enumerate}

For each of these factors, we will use the parametrix expansion for the divergence equation, whose first term has the form
\begin{align}
\pr_j Q^{jl} &= e^{i \la \xi_I} u_1^l \\
Q^{jl} &= {\tilde Q}_{(1)}^{jl} + Q_{(1)}^{jl} \\
{\tilde Q}_{(1)}^{jl} &= e^{i \la \xi_I} \fr{q^{jl}(\nab \xi)[u^l] }{\la} \\
\pr_j Q_{(1)}^{jl} &= - e^{i \la \xi_I} \fr{1}{\la} \pr_j [ q^{jl}(\nab \xi)[u^l] ] \label{eq:theNextEllipticTermAgain}
\end{align}
and whose $D$'th term, in general, has the form
\begin{align}
Q^{jl} &= {\tilde Q}_{(D)}^{jl} + Q_{(D)}^{jl} \\
{\tilde Q}_{(D)}^{jl} &= \sum_{(k) = 1}^D e^{i \la \xi_I} \fr{q_{(k)}^{jl}}{\la^{k}}
\end{align}
where each $q_{(k)}$ solves a linear equation
\begin{align}
i \pr_j \xi q_{(1)}^{jl} &= u^l \\
i \pr_j \xi q_{(k)}^{jl} &= u_{(k)}^l \\
u_{(k+1)}^l &= - \pr_j q_{(k)}^{jl} \quad \quad 1 \leq k \leq D
\end{align}
using the linear map 
\begin{align}
q^{jl} &= q^{jl}(\nab \xi)[u]
\end{align}
defined in line (\ref{eq:theLinMapq}).

The error of the expansion is then eliminated by solving the divergence equation
\begin{align} \label{eq:getRidOfErr}
\pr_j Q_{(D)}^{jl} &= e^{i \la \xi} \fr{u_{(D+1)}^l}{\la^D}
\end{align}

To eliminate the error, we need to solve the divergence equation (\ref{eq:getRidOfErr}) in such a way that we can have bounds for both $\co{ \nab^k Q}$ and $\co{ \nab^k \Ddtof{Q}}$.  This task is accomplished in Section (\ref{sec:transEllipt}).

Let us first study the bounds which are obeyed for the parametrices of these oscillatory terms.

\subsection{Expanding the Parametrix }

For concreteness, consider the High-Low term
\begin{align}
\pr_j Q_L^{jl} &= \sum_I e^{i \la \xi_I} u_{L,I, (1)}^l   \\
u_{L,I, (1)}^l &= {\tilde v}_I^j \pr_j v_\ep^l
\end{align}
Observe that 
\begin{align}
\co{ u_{L,I, (1)}} &\leq \Xi e_v^{1/2} e_R^{1/2} \label{eq:boundForuLI} \\
\co{ \nab u_{L,I, (1)} } &\leq ( \co{ \nab {\tilde v}_I } \co{\nab v_\ep} + \co{{\tilde v}_I} \co{ \nab^2 v_\ep } ) \\
&\leq ( (\Xi e_R^{1/2}) (\Xi e_v^{1/2}) + (e_R^{1/2})(\Xi^2 e_v^{1/2} ) ) \\
&\leq C \Xi^2 e_v^{1/2} e_R^{1/2} \\
\co{ \nab^k u_{L,I,(1)} } &\leq ( \co{ \nab^k {\tilde v}_I } \co{ \nab v_\ep} + \ldots + \co{ {\tilde v}_I} \co{ \nab^{k + 1} v_\ep} \\
&\leq C (N^{1/L} \Xi)^k (\Xi e_v^{1/2} e_R^{1/2} )
\end{align}
from Section (\ref{sec:vecAmpBounds}) and (\ref{bound:nabkvep}).

The first term in the parametrix has the form
\begin{align}
{\tilde Q}_{L, I,(1)}^{jl} &= e^{i \la \xi_I} \fr{q_{L, I,(1)}^{jl}}{\la} \\
q_{L, I,(1)}^{jl} &= q^{jl}(\nab \xi_I)[u_{L,I, (1)}^l] \label{eq:theHighLowq}
\end{align}
In size, the absolute value of the first term is 
\begin{align}
\co{ {\tilde Q}_{L, I,(1)}^{jl} } &\leq C \fr{ \co{u_{L,I, (1)}} }{\la} \\
&\leq C \fr{ \Xi e_v^{1/2} e_R^{1/2} }{ B_\la N \Xi }  \\
\co{ {\tilde Q}_{L, I,(1)}^{jl} } &\leq C B_\la^{-1} \fr{e_v^{1/2} e_R^{1/2} }{N}
\end{align}
which is actually good enough for the ideal theorem.  Note that this is a gain of $\fr{1}{\la}$ compared to (\ref{eq:boundForuLI}).

Having bounded the first term in the expansion, we now show that the later terms in the parametrix obey stronger bounds.  Observe that the error term for the first expansion of the parametrix has the form
\begin{align}
e^{i \la \xi_I } u_{L,I,(2)}^l &= -e^{i \la \xi_I } \fr{\pr_j q_{L, I,(1)}^{jl}}{\la} \label{eq:theSecondLowTerm}
\end{align}
Applying the chain rule to the function (\ref{eq:theHighLowq}) defined in Line (\ref{eq:theLinMapq}), we see that (\ref{eq:theSecondLowTerm}) can be bounded by
\begin{align}
\fr{1}{\la} \co{ \nab q_{L, I,(1)}^{jl} } &\leq \fr{1}{\la} ( \co{\nab^2 \xi_I} \co{u_{L,I, (1)}} + \co{ \nab u_{L,I, (1)} } \\
&\leq \fr{C}{B_\la N \Xi} ( \Xi (\Xi e_v^{1/2} e_R^{1/2}) + \Xi^2 e_v^{1/2} e_R^{1/2} ) \\
&\leq \fr{C}{B_\la N \Xi} (N^{1/L} \Xi^2) e_v^{1/2} e_R^{1/2}
\end{align}

Compared to $u_{L,I,(1)}$, the derivative $\pr_j$ costs at most
\[ |\pr_j| \leq C \Xi N^{1/L} \]
whereas we gain a factor of
\[ \la^{-1} = B_\la^{-1} \Xi^{-1} N^{-1} \]
Giving us a gain of a factor
\begin{align} 
\fr{|\pr_j|}{\la^{-1}} &\leq \fr{C}{B_\la N^{(1 - 1/L)}}
\end{align}
compared to the estimate for $u_{L,I, (1)}$.  More precisely,
\begin{align}
\la^{-1} \co{ u_{L,I,(2)} } &\leq C B_\la^{-1} \Xi^{-1} N^{-1}  (\co{ \nab^2 \xi_I} \co{u_{L,I}} + \co{ \nab u_{L,I, (1)} } ) \\ 
&\leq C B_\la^{-1} N^{-(1 - 1/L)} \left( \Xi e_v^{1/2} e_R^{1/2} \right)
\end{align}

This bound on the cost of a derivative continues to hold for higher derivatives, so that each iteration of the parametrix gains a smallness factor
\ali{
\fr{|\nab|}{\la} &\leq C \fr{N^{1/L} \Xi}{B_\la N \Xi} = C B_\la^{-1} N^{-(1-1/L)}
}

For example,
\begin{align}
\fr{1}{\la} \co{ \nab^k q_{L, I,(1)}^{jl} } &\leq \fr{C}{B_\la N \Xi} (\co{ \nab^{k +1} \xi_I} \co{u_{L,I, (1)} } + \ldots + \co{ \nab^k \nab u_{L,I, (1)} }) \\
&\leq \fr{C}{B_\la N \Xi} (N^{1/L} \Xi )^k (\Xi e_v^{1/2} e_R^{1/2} )
\end{align}

After iterating the parametrix $D$ times, we have
\begin{align}
\fr{1}{\la^D}\co{ u_{L,I,(D+1)} } &\leq C_D B_\la^{- D} N^{-D(1 - 1/L)} \left( \Xi e_v^{1/2} e_R^{1/2} \right) \label{eq:afterDIterations}
\end{align}
Thus, the estimate on the error improves by a factor $N^{-D(1 - 1/L)}$ after $D$ iterations of the parametrix.

Because we have $N \geq \Xi^\eta$ for some $\eta > 0$, the large power $B_\la^{-D} N^{-D(1 - 1/L)}$ will eventually overtake the factor of $\Xi$, and we will be able to gain the factor $\fr{1}{B_\la N \Xi} $ that solving the elliptic equation is intended to gain.  That is,
\begin{align}
\fr{1}{\la^D}\co{ u_{L,I,(D+1)} } &\leq C_D  B_\la^{- D} \fr{e_v^{1/2} e_R^{1/2} }{N}
\end{align}
once $D$ is large enough depending on the given $\eta$ and $L$.

If we were only concerned about the $C^0$ norms and spatial derivatives, we could apply the operator $\RR$ from Section (\ref{sec:oscEstimate}) to obtain a solution to (\ref{eq:getRidOfErr}) of size
\begin{align}
Q_{L,I,(D)} &= - \fr{1}{\la^D} \RR[ e^{i \la \xi_I} u_{L,I,(D+1)}] \\
\co{ Q_{L,I,(D)} } &\leq C \fr{\co{ u_{L,I,(D+1)} }}{\la^D} \\
&\leq C_D  B_\la^{- D} \fr{e_v^{1/2} e_R^{1/2} }{N}
\end{align}
which can be bounded by
\begin{align}
\co{ Q_{L,I,(D)} } &\leq \fr{1}{1000} \fr{e_v^{1/2} e_R^{1/2} }{N}
\end{align}
once $B_\la$ is chosen large enough.  Furthermore its derivatives would grow by a factor $(N \Xi)^k$ from
\begin{align}
\co{\nab^k Q_{L,I,(D)}} &\leq C_k \co{ \nab^k \left[ e^{i \la \xi_I} u_{L,I,(D) } \right] }  \label{eq:C0forRR} \\
&\leq C_k \fr{1}{1000} (N \Xi)^k \fr{e_v^{1/2} e_R^{1/2} }{N}
\end{align}

However, we must also establish estimates for $\Ddtof{Q_{L,I,(D)}}$, and for this purpose we have constructed a different operator to solve equation (\ref{eq:getRidOfErr}).

Rather than the estimates for (\ref{eq:C0forRR}), the solution which we construct in Section (\ref{sec:transEllipt}) to solve the equation
\begin{align}
\pr_j Q_{L,I,(D)}^{jl} &= U^l = \fr{-1}{\la^D} e^{i \la \xi_I} u_{L,I,(D+1)}
\end{align}
is bounded by
\begin{align}
\co{ Q_{L,I,(D)}}  &\leq C \co{ U } + \tau \co{ \Ddtof{U} } \\
&\leq \la^{-D} (\co{ u_{L,I,(D+1)} } + \tau \co{ \Ddtof{u_{L,I,(D+1)} }})
\end{align}
and its support in time may grow to 
\begin{align}
 \mbox{ supp } Q_{L,I,(D)} &\subseteq \{ |t - t(I)| \leq \fr{3 \tau}{2} \} \times \T^3 \label{subset:suppGrowth}
\end{align}
when the support of $U$ is contained in $| t -t(I) | \leq \tau$.  Even if the support does grow as in (\ref{subset:suppGrowth}), at each time $t$ there will still be a bounded number of stress terms $Q_{L,I,(D)}$ which are nonzero at that time.

Furthermore, this solution has the property that if the bounds on 
\[ \co{ \nab^k U} \leq A (B_\la N \Xi)^k  \]
\[ \tau \co{ \nab^k \Ddtof{U} } \leq A (B_\la N \Xi)^k \]
are satisfied, then $Q_{L,I,(D)}$ obeys similar bounds
\begin{align}
\co{ \nab^k Q_{L,I,(D)}} &\leq C A (B_\la N \Xi)^k 
\co{ \nab^k \Ddtof{Q_{L,I,(D)}}}  &\leq C A\tau^{-1}(B_\la N \Xi)^k 
\end{align}
Namely, a spatial derivative of $Q_{L,I,(D)}$ costs $(B_\la N \Xi)$ and a material derivative of the stress costs 
$\tau^{-1}$ as long as these costs hold true for the data.  

For all of the examples 
\[ U^l = e^{i \la \xi} u^l \]
we will encounter, they share the feature that material derivatives cost $\tau^{-1}$.  Thus, even at the level of the material derivative, we will be able to conclude that each iteration of the parametrix gains a factor of 
\[ \fr{ |\pr_j| }{\la} \leq N^{-(1- 1/L)} \]

Let us show now that this operator will be suitable for the conclusion of the proof.  

\subsection{Applying the Parametrix}

First consider the transport term, since it is the largest
\begin{align}
\pr_j Q_{T,I}^{jl} &= \sum_I e^{i \la \xi_I} u_{T,(I), 1}^l \\
u_{T,I, (1)}^l &= (\pr_t + v_\ep^j \pr_j){\tilde v}_I^l
\end{align}
and its parametrix expansion up to order $D$
\begin{align}
Q^{jl}_{T,I} &= {\tilde Q}_{T, I, (D)}^{jl} + Q_{T, I, (D)}^{jl} \\
{\tilde Q}_{T, I, (D)}^{jl} &= \sum_{k = 1}^D e^{i \la \xi_I} \fr{q_{(k)}^{jl}}{\la^{k}} \\
\pr_j Q_{T, I, (D)}^{jl} &= \fr{-1}{\la^D} e^{i \la \xi_I} u_{T, I, (D+1)}^l
\end{align}
From Section (\ref{sec:vecAmpBounds}), we have seen that each material derivative costs essentially $\tau^{-1}$, so we know that 
\begin{align}
\co{ u_{T,I, (1)}^l } + \tau \co{ \Ddtof{u_{T,I, (1)}^l} } &\leq C B_\la^{1/2} \badFact \Xi e_v^{1/2} e_R^{1/2} \\
\co{ \nab^k u_{T,I, (1)}^l } + \tau \co{ \nab^k \Ddtof{u_{T,I, (1)}^l} } &\leq C B_\la^{1/2} (B_\la N \Xi)^k \badFact \Xi e_v^{1/2} e_R^{1/2} 
\end{align}
The $C^0$ estimate was already checked in Section (\ref{sec:accountForParams}) where $\tau$ was chosen.

Using the estimates for spatial derivatives in (\ref{sec:vecAmpBounds}) to commute transport and spatial derivatives, we also have 
\begin{align}
\la^{-D}\Big(\co{ u_{T,I, (D+1)}^l } + &\tau \co{ \Ddtof{u_{T,I, (D+1)}^l} }\Big) \leq  \notag \\
&\leq C_D B_\la^{- D} N^{-D(1- 1/L)} B_\la^{1/2} \badFact \Xi e_v^{1/2} e_R^{1/2}  \label{eq:DorderPmtrx} \\
\la^{-D}\Big(\co{ \nab^k u_{T,I, (D+1)}^l } + &\tau \co{ \nab^k \Ddtof{u_{T,I, (D+1)}^l} }\Big) \leq \notag \\
&\leq C_D B_\la^{- D} N^{-D(1- 1/L)} B_\la^{1/2} (B_\la N \Xi)^k \badFact \Xi e_v^{1/2} e_R^{1/2}  \label{eq:DorderPmtrx2}
\end{align}
Using (\ref{eq:DorderPmtrx}), choose $D$ large enough depending on the $\eta$ in the lower bound $N \geq \Xi^\eta$ such that 
\begin{align}
N^{-D(1- 1/L)} &\leq N^{-1} \Xi^{-1}
\end{align}
Namely, we expand the parametrix until we have gained the factor $N^{-1} \Xi^{-1}$ which we expect to gain from solving the divergence equation.

With this choice, the bounds (\ref{eq:DorderPmtrx}), (\ref{eq:DorderPmtrx2}) read
\begin{align}
\la^{-D}\Big(\co{ u_{T,I, (D+1)}^l } + \tau \co{ \Ddtof{u_{T,I, (D+1)}^l} }\Big) &\leq C_D B_\la^{- D} B_\la^{1/2} \badFact \fr{ e_v^{1/2} e_R^{1/2} }{B_\la N} \label{eq:DorderPmtrxGood} \\
\la^{-D}\Big(\co{ \nab^k u_{T,I, (D+1)}^l } + \tau \co{ \nab^k \Ddtof{u_{T,I, (D+1)}^l} }\Big) &\leq C_D B_\la^{- D} B_\la^{1/2} (B_\la N \Xi)^k \badFact \fr{ e_v^{1/2} e_R^{1/2} }{B_\la N} \label{eq:DorderPmtrx2Good}
\end{align}

Applying Theorem (\ref{thm:transportElliptic}), we have a solution $Q_{D,T}^{jl}$ to
\begin{align}
\pr_j Q_{T,I, (D)}^{jl} &= \sum_I \fr{-1}{\la^D} e^{i \la \xi_I} u_{T, I, (D+1)}^l
\end{align}
supported in $| t- t(I)|\leq 3\tau/2$ and obeying the bounds
\begin{align}
\co{ \nab^k Q_{T,I, (D)} } + \tau \co{\nab^k \Ddtof{Q_{T,I, (D)}} } &\leq C_D B_\la^{1/2} (B_\la N \Xi)^k \badFact \fr{ e_v^{1/2} e_R^{1/2} }{B_\la N}
\end{align}

The parametrix itself 
\begin{align}
{\tilde Q}_{T,I, (D)}^{jl} &= \sum_{(k) = 1}^D e^{i \la \xi_I} \fr{q_{(k)}^{jl}}{\la^{k}}
\end{align}
gains a factor of $\la^{-1}$ compared to the first $u_{T,I, (1)}^l$, so it satisfies the bounds
\begin{align}
\co{{\tilde Q}_{T,I, (D)}^{jl}} &= C_D B_\la^{1/2} (B_\la N \Xi)^k \badFact \fr{ e_v^{1/2} e_R^{1/2} }{B_\la N }
\end{align}
where the powers of $(B_\la N \Xi)$ come from differentiating the oscillatory factor.

Let us check that these are the estimates which are required in Lemma (\ref{lem:iterateLem}).

We essentially checked the $C^0$ bound for 
\begin{align}
\co{ {\tilde Q}_{T,I, (1)}^{jl} } &\leq C B_\la^{-1/2} \badFact \fr{e_v^{1/2} e_R^{1/2}}{N}
\end{align}
in Section (\ref{sec:accountForParams}).  This estimate clearly worsens by a factor $N \Xi$ upon taking spatial derivatives, and worsens by a factor
\begin{align}
|\Ddt| &\leq \tau^{-1} = b^{-1} \Xi e_v^{1/2} \label{eq:costOfDdt}\\
&= B_\la^{1/2}  N^{1/2} \Xi e_R^{1/4} e_v^{1/4}
\end{align}
upon taking a material derivative.  Our desired cost for a material derivative is
\begin{align}
|\Ddt| &\unlhd \Xi' (e_v')^{1/2} \\
&= C N \Xi e_R^{1/2}
\end{align} 
and checking this inequality is equivalent to verifying that
\begin{align}
\left( \fr{e_v}{e_R} \right)^{1/4} &\leq N^{1/2} \\
\Leftrightarrow \left( \fr{e_v}{e_R} \right)^{1/2} &\leq N
\end{align}
which has been guaranteed in Lemma (\ref{lem:iterateLem}).

Every other stress term can be treated similarly with even better bounds, except for the main term in the High-High term, namely

\begin{align}
\pr_j Q_H^{jl} &= \sum_{J \neq {\bar I}}  \la e^{i \la(\xi_I + \xi_J) } u_{H, IJ}^l \\
u_{H, IJ}^l &= -[ v_I \times ( |\nab \xi_J| - 1 ) v_J^l + v_J \times (|\nab \xi_I| - 1) v_I^l ]
\end{align}

The first term in the parametrix expansion for this term takes the form
\begin{align}
{\tilde Q}_{H, IJ, (1)}^{jl} &= e^{i \la (\xi_I + \xi_J) } q^{jl}(\nab(\xi_I + \xi_J))[u_{H, IJ}] \label{eq:mainHighTerm}
\end{align}

For this term, we observed in Section (\ref{sec:accountForParams}) that
\begin{align}
\co{ {\tilde Q}_{H, IJ, (1)}^{jl} } &\leq C e_R b \\
&\leq C B_\la^{-1/2} \left( \fr{e_v^{1/2}}{e_R^{1/2} N} \right)^{1/2} e_R \\
&\leq C B_\la^{-1/2} \badFact \fr{e_v^{1/2} e_R^{1/2}}{N}
\end{align}
thanks to our choice of $\tau$.

For this term one must also check that the cost of a spatial derivative is bounded below $N \Xi$.  To check this bound, we compare
\ali{
\co{ |\nab \xi_J| - 1 } &\leq C B_\la^{-1/2} \left( \fr{e_v^{1/2}}{e_R^{1/2} N} \right)^{1/2} \\
\co{ \nab (|\nab \xi_J| - 1 ) } \leq C \co{\nab^2 \xi_J } &\leq C \Xi
}
to reveal that differentiating worsens the bounds by at most a factor
\ali{
|\nab| &\leq C  B_\la^{1/2} \left(\fr{e_R^{1/2} N}{e_v^{1/2}} \right)^{1/2} \Xi \\
|\nab| &\leq C B_\la^{1/2} N^{1/2} \Xi
}
so that each iteration of the parametrix gains at least
\ali{
\fr{|\nab|}{\la} &\leq C \fr{ B_\la^{1/2} N^{1/2} \Xi }{B_\la N \Xi} \\
&\leq C B_\la^{-1/2} N^{-1/2}
}
which will also gain the factor of $ N^{-1} \Xi^{-1}$ which we require once sufficiently many terms in the parametrix expansion have been taken.  Likewise, we must check that each material derivative of this term costs at most $\tau^{-1}$, which follows from comparing
\ali{
\co{ |\nab \xi_J| - 1 } &\leq  C b =  C B_\la^{-1/2} \left( \fr{e_v^{1/2}}{e_R^{1/2} N} \right)^{1/2} \\
\co{ \Ddt (|\nab \xi_J| - 1 ) } \leq C \co{\Ddt \nab \xi_J } &\leq C \Xi e_v^{1/2} \\
\Rightarrow \left|\Ddt\right| &\leq C B_\la^{1/2} \left( \fr{e_v^{1/2}}{e_R^{1/2} N} \right)^{-1/2} (\Xi e_v^{1/2} ) \\
&\leq C b^{-1} \Xi e_v^{1/2} \leq C \tau^{-1}
}
Thus, the High-High term can be treated in the same way as the Transport term, taking more terms in the parametrix expansion if necessary.  In each case, the terms $Q_T$, $Q_H$ and $Q_L$ enjoy a bound of the same type as the bound on the first term of their parametrix expansion but with a worse constant.  Among all the stress terms, the terms $Q_T$ and $Q_H$ obey the worst bound of
\ali{
\co{ Q_T } + \co{Q_H } &\leq C B_\la^{-1/2} \badFact \fr{e_v^{1/2} e_R^{1/2}}{N}
}

In particular, we have a bound
\ali{
\co{ R_1 } &\leq  C B_\la^{-1/2} \badFact \fr{e_v^{1/2} e_R^{1/2}}{N}
}
for the entirety of the new stress.

At this point, we can finally choose the constant $B_\la$ so that the goal
\begin{align}
\co{ R_1 } &\leq  C B_\la^{-1/2} \badFact \fr{e_v^{1/2} e_R^{1/2}}{N} \\
&\unlhd \fr{1}{20} \badFact \fr{e_v^{1/2} e_R^{1/2}}{N}
\end{align}
is satisfied.

Once $B_\la$ has been chosen for this purpose, we know that each spatial derivative of $R_1$ costs at most
\[ |\nab| \leq C N \Xi \]
which is exactly the target cost for a spatial derivative.  Each material derivative costs
\ali{
 \left|\Ddt \right| &\leq C \tau^{-1} = C \badFact  \Xi e_v^{1/2}
}
which must be below the target cost of
\ali{
\left| \Ddt \right| &\unlhd \Xi' (e_v')^{1/2} \\
&\unlhd C N \Xi e_R^{1/2}
}
and this target cost has been satisfied as long as
\ali{
\left( \fr{e_v}{e_R} \right)^{1/4} &\leq N^{1/2}
}
This inequality follows from the condition
\[ N \geq \left( \fr{e_v}{e_R} \right) \]
and is enough to conclude the proof of Lemma (\ref{lem:iterateLem}).

%% file: transportElliptic.tex
In order to eliminate the error term in the parametrix it is necessary to solve the underdetermined, elliptic equation
\begin{align} \label{eq:theElliptic}
\begin{split}
\pr_j Q^{jl} &= U^l \\
Q^{jl} &= Q^{lj}
\end{split}
\end{align}

For the proof of the Main Lemma, we need to have estimates for $Q$, spatial derivatives $\nab^k Q$ of $Q$, and also the material derivative $\Ddtof{Q}$ and its spatial derivatives.

First we recall from Lemma (\ref{eq:smoothingOp}) that the divergence operator in equation (\ref{eq:theElliptic}) can be inverted by an operator $\RR$ which has order $-1$.

Rather than commuting $\Ddt$ with the nonlocal operator $\RR$, we find a solution to (\ref{eq:theElliptic}) with good transport properties by solving (\ref{eq:theElliptic}) via a transport equation obtained by commuting the divergence operator with the material derivative.  In this way, we directly obtain estimates for $\Ddtof{Q}$ and its derivatives.

\begin{thm}\label{thm:transportElliptic}  Let $U^l$ be a vector field such that
\begin{align}
\int U^l(t,x) dx &= 0 
\end{align}
for all $t$, and such that 
\begin{align}
\mbox{ supp } U &\subseteq [t(I) - \tau, t(I) + \tau] \times \T^3
\end{align}
for some time $t(I) \in \R$ and some $\tau \leq \Xi^{-1} e_v^{-1/2}$.

Assume also that, for 
\[ \La = B_\la N \Xi \]
the velocity field $U$ and its material derivative $\Ddtof{U}$ obey the estimates
\begin{align}
\label{eq:boundsForU}
\begin{aligned}
\| \nab^k U \|_{C^0_tL_x^4} &\leq A \La^k \quad \quad |k| = 0, \ldots, L \\
\| \nab^k \Ddtof{U} \|_{C^0_t L_x^4} &\leq A \tau^{-1} \La^k \quad \quad k = 0, \ldots, L - 1
\end{aligned}
\end{align}

Then there exists a solution $Q$ to the equation (\ref{eq:theElliptic}) depending linearly on $U$ such that
\begin{enumerate}
\item For all $t$, 
\begin{align}
 \int Q(t,x) dx &= 0
 \end{align}
\item 
\begin{align}
 \mbox{ supp } Q \subseteq [t(I) - 3\tau/2, t(I) + 3\tau/2] \times \T^3 
\end{align}
\item For $k = 0, \ldots, L$ 
\begin{align}
\| \nab^{k + 1} Q \|_{C^0_tL_x^4} &\leq C A \La^k
\end{align}
\item For $k = 0, \ldots, L - 1$ 
\begin{align}
\| \nab^{k + 1} \Ddtof{Q} \|_{C^0_tL_x^4} &\leq C A \tau^{-1} \La^k
\end{align}
\end{enumerate}

\end{thm}

The proof of Theorem (\ref{thm:transportElliptic}) relies on a transport equation which we now derive.  To set up the initial data, we set
\begin{align}
Q^{jl}(t(I), x) &= \RR^{jl}[U(t(I), \cdot)](x) 
\end{align}
With this choice of data, the equation (\ref{eq:theElliptic}) is satisfied at the initial time $t(I)$.  The divergence equation will remain satisfied at future times if we ensure that
\begin{align}
\Ddt \pr_j Q^{jl} = (\pr_t + v_\ep^i \pr_i)\pr_jQ^{jl} &= \Ddtof{U^l}
\end{align}
Commuting $\Ddt$ with the divergence operator leads to an underdetermined elliptic equation
\begin{align} \label{eq:unDetEllipTrans}
\Ddt \pr_j Q^{jl} = \pr_j[ \Ddt Q^{jl}] - \pr_j v_\ep^i \pr_iQ^{jl}  &= \Ddtof{U^l}
\end{align}
We solve the above equation by inverting the divergence operator, which leads to an equation which we call the Transport-Elliptic equation.

\begin{lem}  Suppose that $U^l$ is a smooth vector field of integral $0$, then if $Q^{jl}$ solves the transport equation
\begin{align}
\Ddtof{Q^{jl}} &= \RR^{jl}[ \pr_i v_\ep^b \pr_b Q^{ik} + \Ddtof{U^k} ] \label{eq:transEllipt1}
\end{align}
with initial data
\begin{align}
Q^{jl}(t(I), x) &= \RR^{jl}[U(t(I), \cdot)](x) 
\end{align}
then $Q$ also solves the equation
\begin{align} \label{eq:theElliptic2}
\pr_j Q^{jl} &= U^l
\end{align}
with 
\begin{align}
\int Q dx &= 0
\end{align}
for all $t$.
\end{lem}
\begin{proof}
At the initial time $t = t(I)$, we have
\begin{align}
\pr_j Q^{jl}(t(I), x) &= \pr_j \RR^{jl}[ U(t(I), \cdot) ] \\
&= \PP U^l(t(I), x)
\end{align}
where $\PP$ is the operator which projects to integral $0$ vector fields.  Because $U$ has integral $0$, we have
\[ \PP U^l(t, x) = U^l(t,x) \]
so it is clear that equation (\ref{eq:theElliptic2}) is satisfied at time $t(I)$. 

It now suffices to verify that (\ref{eq:unDetEllipTrans}) is satisfied by the uniqueness of solutions to the initial value problem for transport equations.  This calculation will rely crucially on the fact that $v_\ep$ is divergence free, which implies, for example, that the term
\begin{align}
\pr_j v_\ep^b \pr_b Q^{jl} &= \pr_j \pr_b [ v_\ep^b Q^{jl} ]
\end{align}
has integral $0$.

By taking the divergence of (\ref{eq:transEllipt1}), we compute
\begin{align}
 \pr_j[ \Ddt Q^{jl}] &= \PP[ \pr_j v_\ep^b \pr_b Q^{jl} ] + \PP \Ddtof{U^l} \\
 &= \PP\left[ \pr_j \pr_b [ v_\ep^b Q^{jl} ] \right] + \PP \Ddtof{U^l} \\
 &= \pr_j \pr_b [ v_\ep^b Q^{jl} ] + \PP \Ddtof{U^l} \label{eq:almostOKelTr}
\end{align}

Since $U^l$ has integral $0$ and $v_\ep$ is divergence free, we have that 
\begin{align}
\int \Ddtof{U} dx &= \fr{d}{dt} \int U dx = 0
\end{align}
implying that $\PP \Ddtof{U^l} = \Ddtof{U^l}$.

From (\ref{eq:almostOKelTr}), we have 
\begin{align}
 \pr_j[ \Ddt Q^{jl}] &= \pr_j[ v_\ep^b \pr_b Q^{jl} ] + \Ddtof{U^l} 
\end{align}
which is the equation (\ref{eq:unDetEllipTrans}) that we had to verify.

To see that $\int Q dx = 0$ for all $t$, first observe that 
\begin{align}
\int Q(t(I), x) dx &= 0 
\end{align}
by the property that the operator $\RR$ maps to integral $0$ tensors, and for the same reason,
\begin{align}
\fr{d}{dt} \int Q(t(I), x) dx &= \int \Ddtof{Q} dx = 0
\end{align}
as well.
\end{proof}

Our next goal for the section is to establish existence and a-priori estimates for the solution to the PDE (\ref{eq:transEllipt1}).

\subsection{Existence of Solutions for the Transport-Elliptic Equation} \label{sec:ellipTransExistence}

Global solutions to the linear equation
\begin{align}
\label{eq:transEllipt2}
\begin{aligned}
\Ddtof{Q^{jl}} &= \RR^{jl}[ \pr_i v_\ep^b \pr_b Q^{ik} + \Ddtof{U^k} ] \\
Q^{jl}(t(I), x) &= \RR^{jl}[ U(t(I), \cdot)](x)
\end{aligned}
\end{align}
can be easily constructed by the method of Picard iteration, once the appropriate spaces have been identified.  

To begin the Picard iteration we define an operator $T$ acting on symmetric $(2,0)$-tensor fields which solves the transport equation
\begin{align}
\label{eq:defineIteration}
\begin{aligned}
\Ddt[TQ]^{jl} &= \RR^{jl}[ \pr_i v_\ep^b \pr_b Q^{ik} + \Ddtof{U^k} ] \\
TQ^{jl}(t(I), x) &= \RR^{jl}[ U(t(I), \cdot)](x)
\end{aligned}
\end{align}
so that a fixed point of $T$ is a solution to the initial value problem (\ref{eq:transEllipt2}).

The operator $T$ must be defined on an appropriately defined complete metric space, and the key to identifying this metric space is the following, a-priori estimate
\begin{prop}[A priori estimate] \label{prop:transEllAPriori}
If $Q$ solves (\ref{eq:transEllipt2}) and $U$ satisfies the bounds (\ref{eq:boundsForU}), then there are constants $C_1, C_2$ such that
\begin{enumerate}
\item
\begin{align}
\int Q(t,x) dx &= 0 
\end{align}
\item
\begin{align}
\sum_{|a| = 1} \int |\pr_a Q|^4(t,x) dx &\leq C_1 A^4 e^{C_2 \tau^{-1} |t - t(I)|} \label{ineq:firstOrderEllTr}
\end{align}
\end{enumerate}

\end{prop}
\begin{proof}

At time $t = t(I)$, the bound (\ref{ineq:firstOrderEllTr}) is a consequence of the bound
\[ \| \nab \RR[U] \|_{L^4} \leq C \| U \|_{L^4} \]
so by Gronwall it suffices to prove that
\begin{align}
| \sum_{|a| = 1} \fr{d}{dt} \int |\pr_a Q|^4(t,x) dx | &\leq C_2 \tau^{-1} ( \sum_{|b| = 1} \sum_{j,l = 1}^3 \int |\pr_b Q^{jl}|^4(t,x) dx + A^4 )  \label{ineq:wantGron}
\end{align}
To establish this inequality, we let $\pr_a$ be any derivative of order $|a| = 1$, and for any component $\pr_a Q^{jl}$ of $\pr_a Q$ we compute the time derivative
\begin{align}
\fr{d}{dt} \int |\pr_a Q^{jl}|^4(t,x) dx  &= \int \Ddt  |\pr_a Q^{jl}|^4 dx \\
&= \int 4 (\pr_a Q^{jl})^3[ \Ddt \pr_a Q^{jl} ] dx \\
&= \int 4 (\pr_a Q^{jl})^3[\pr_a(\Ddt Q^{jl}) - \pr_a v_\ep^i \pr_i Q^{jl} ] dx \label{eq:theTwoEllTrTerms}
\end{align}
Since $\co{ \nab v_\ep } \leq \Xi e_v^{1/2} < \tau^{-1}$, the second of these terms can be bounded by 
\begin{align}
| 4 \int (\pr_a Q^{jl})^3 (\pr_a v_\ep^i \pr_i Q^{jl} ) dx | \leq C \tau^{-1} \int |\pr_a Q^{jl}|^4(t,x) dx
\end{align}

For the first term in (\ref{eq:theTwoEllTrTerms}), we use the equation (\ref{eq:transEllipt2}) to estimate
\begin{align}
| \int (\pr_a Q^{jl})^3[ \pr_a(\Ddt Q^{jl}) ] dx | &= | \int (\pr_a Q^{jl})^3\left(\pr_a \RR^{jl}[ \pr_i v_\ep^b \pr_b Q^{ik} + \Ddtof{U^k} ] \right) dx | \\
&\leq \int \left( |\pr_a Q^{jl}|^3 \left|\pr_a \RR^{jl}[ \pr_i v_\ep^b \pr_b Q^{ik}]\right| + |\pr_a Q|^3 |\pr_a \RR^{jl}[\Ddtof{U^k}]| \right) dx \label{ineq:twoTermUsedEqn}
\end{align}

The first of these terms is similar to what we have already encountered, and can be estimated using H{\" o}lder's inequality, 
\begin{align}
\int |\pr_a Q^{jl}|^3 \left|\pr_a \RR^{jl}[ \pr_i v_\ep^b \pr_b Q^{ik}] \right| dx &\leq \| \nab Q \|_{L^4}^{3} \| \nab \RR^{jl}[ \pr_i v_\ep^b \pr_b Q^{ik}] \|_{L^4} \\
&\leq  \| \nab Q^{jl} \|_{L^4}^{3} \| \pr_i v_\ep^b \pr_b Q^{ik} \|_{L^4} \\
&\leq \Xi e_v^{1/2} \| \nab Q \|_{L^4}^4 \\
&\leq C \tau^{-1}  \sum_{|b| = 1} \sum_{j,l = 1}^3 \int |\pr_b Q^{jl}|^4(t,x) dx
\end{align}

The latter term in (\ref{ineq:twoTermUsedEqn}) can be estimated equivalently by applying Young's inequality with exponents $\fr{3}{4} + \fr{1}{4} = 1$ to the pointwise product
\begin{align}
|\pr_a Q^{jl}|^3 |\RR^{jl}[\Ddtof{U}]| &= ( \tau^{-4/3} |\pr_a Q^{jl}|^3 )(\tau^{3/4} |\pr_a\RR^{jl}[\Ddtof{U}]| ) \\
&\leq \fr{3}{4} \tau^{-1} |\pr_a Q^{jl}|^4 + \fr{1}{4} \tau^3 \left|\pr_a\RR^{jl}[\Ddtof{U}] \right|^4
\end{align}
This calculation gives the bounds
\begin{align}
\int |\pr_a Q^{jl}|^3 |\RR^{jl}[\Ddtof{U}]| dx &\leq C \int (\tau^{-1} |\pr_a Q^{jl}|^4 + \tau^3 \left|\pr_a \RR^{jl}[\Ddtof{U}] \right|^4 ) dx \\
&\leq C \tau^{-1} \left( \int |\pr_a Q^{jl}|^4 dx +  \int \tau^4 \left| \Ddtof{U} \right|^4 dx \right) \\
&\leq C \tau^{-1} \left(\int |\pr_a Q^{jl}|^4 dx +  A^4 \right)
\end{align}
which is the estimate required.
\end{proof}

The proof of the above estimate can also be used to establish that on the space 
\[ X \subseteq C_t^0 W_x^{1,4}(\R \times \T^3 ; \SS ) \subseteq C^0_{t,x}(\R \times \T^3 ; \SS )
\] 
defined by the conditions
\begin{align*} 
X = \left\{ Q^{jl} : \R \times \T^3 \to \SS \right.&~|~ \int Q(t,x) dx = 0 ~\forall~ t \in \R, Q^{jl}(t(I), x) = \RR^{jl}[ U(t(I), \cdot)](x) \\
&\left. \sum_{|a| = 1} \sum_{i, j = 1}^3 \int |\pr_a Q^{ij}|^4(t,x) dx \leq C_1 A^4 e^{C_2 \tau^{-1} |t - t(I)|} \right\}
\end{align*}
the map $T$ defined in (\ref{eq:defineIteration}) maps $X$ to itself when $C_1$ and $C_2$ are chosen appropriately.  Furthermore, when the space $X$ is endowed with the metric deriving from the norm 
\begin{align}
|| Q||_X &= \sup_{t \in \R} e^{-B \tau^{-1} |t - t(I)| } \left( \sum_{|a| = 1} \sum_{j,l = 1}^3 \int |\pr_a Q^{jl}|^4(t,x) dx \right)^{1/4}
\end{align}
then $X$ is a complete metric space and essentially the same proof as that of Proposition (\ref{prop:transEllAPriori}) also shows that $T$ is a contraction on $X$ whenever $B$ is a sufficiently large constant.  Therefore a unique, global solution to the initial value problem (\ref{eq:transEllipt2}) exists.

\subsection{Spatial derivative estimates for the solution to the Transport-Elliptic equation}

Following the methods of Section (\ref{sec:ellipTransExistence}), we can also give bounds on the higher derivatives of solutions to the equation (\ref{eq:transEllipt2}).

\begin{prop}  The solution $Q$ to the equation (\ref{eq:transEllipt2}) obeys the bounds
\begin{align} \label{bd:spatialDerivsEllTr}
\| \nab^{k + 1} Q \|_{L_x^4} &\leq C_1 \La^k e^{C_2 \tau^{-1} | t - t(I)| } A 
\end{align}
for $k = 0, \ldots, L$
\end{prop}
\begin{proof}
At the time $t(I)$, the above inequality follows from the bounds (\ref{bound:stdElliptic}) for $\RR$ and the assumed bounds (\ref{eq:boundsForU}) on $U$ and its spatial derivatives.  

To establish the estimates (\ref{bd:spatialDerivsEllTr}), we follow the strategy of Section (\ref{sec:ellipTransExistence}) and define a weighted energy
\begin{defn}
\begin{align}
E_M[Q](t) &\equiv \sum_{K = 1}^M \sum_{j,l = 1}^3 \sum_{|a| = K} \int \left| \fr{\pr_a Q^{jl}}{\La^{K - 1}} \right|^4 dx
\end{align}
\end{defn}

Since \[ E_M[Q](t(I)) \leq C_1 A^4, \]
the exponential bound (\ref{bd:spatialDerivsEllTr}) follows from Gronwall after we establish a differential inequality
\begin{align} \label{eq:preGronEllTr}
|\fr{d E_M[Q]}{dt}| &\leq C_2 \tau^{-1} ( E_M[Q] + A^4 )
\end{align}

To establish the differential inequality, we fix a spatial derivative $\pr_a$ of order $K$ and components $jl$ of $\pr_a Q^{jl}$ and calculate
\begin{align}
\fr{d}{dt}  \int \left| \fr{\pr_a Q^{jl}}{\La^{K - 1}} \right|^4 dx &=  \int \Ddt \left| \fr{\pr_a Q^{jl}}{\La^{K - 1}} \right|^4 dx \\
&= 4 \int \left( \fr{\pr_a Q^{jl}}{\La^{K - 1}} \right)^3 \cdot \La^{-(K - 1) } [ \Ddt \pr_a Q^{jl} ] \label{eq:beforeCommuteEllTr}
\end{align}

\paragraph{Commutator Terms}

As in Section (\ref{sec:ellipTransExistence}), we now commute $\Ddt = (\pr_t + v_\ep^b \pr_b)$ with $\pr_a = \pr_{a_1} \cdots \pr_{a_K}$ to see that
\begin{align}
\Ddt \pr_a Q^{jl} &= \pr_a[ \Ddtof{Q^{jl}} ] - \sum_{|a_1| + |a_2| = K, |a_1| \geq 1} C_{a_1, a_2}^a \pr_{a_1} v_\ep^b \pr_{a_2} \pr_b Q^{jl}
\end{align}

The commutator terms are acceptable in (\ref{eq:beforeCommuteEllTr}), because they give rise to terms bounded by
\begin{align}
 \int \left| \fr{\pr_a Q^{jl}}{\La^{K - 1}} \right|^3 \cdot( \La^{-(K - 1) } | \pr_{a_1} v_\ep^b \pr_{a_2} \pr_b Q^{jl} |) dx
 &\leq \int \left| \fr{\pr_a Q^{jl}}{\La^{K - 1}} \right|^3 \left| \fr{\pr_{a_1} v_\ep^b}{\La^{|a_1| - 1}} \right| \left| \fr{ \pr_{a_2} \pr_b Q^{jl} }{\La^{|a_2|}} \right|
\end{align}
Because $|a_1| \geq 1$, we have $|a_2| \leq K - 1$, and we can estimate the above by  
\begin{align}
 \int \left| \fr{\pr_a Q^{jl}}{\La^{K - 1}} \right|^3 \cdot( \La^{-(K - 1) } | \pr_{a_1} v_\ep^b \pr_{a_2} \pr_b Q^{jl} |) dx &\leq \co{\fr{\pr_{a_1} v_\ep^b}{\La^{|a_1| - 1}}} E_M[Q](t) \\
 &\leq C \fr{\Xi^{|a_1|}}{\La^{|a_1| - 1} } e_v^{1/2} E_M[Q](t) \\
 &\leq C \Xi e_v^{1/2} E_M[Q](t) \\
 &\leq C \tau^{-1} E_M[Q](t)
\end{align}

\paragraph{Concluding the Estimate for (\ref{eq:beforeCommuteEllTr}) }

To conclude the proof of (\ref{eq:preGronEllTr}), it remains to show that for all multi-indices of order $|a| = K - 1$, we have
\begin{align}
| \int \left( \fr{\pr_a Q^{jl}}{\La^{K - 1}} \right)^3 \cdot \La^{-(K - 1) } [ \pr_a \Ddtof{Q^{jl}} ] dx | &\leq C \tau^{-1} ( E_M[Q](t) + A^4 )
\end{align}

Using the equation (\ref{eq:transEllipt2}), the integral inside the absolute values can be written as a sum
\begin{align}
\int \left( \fr{\pr_a Q^{jl}}{\La^{K - 1}} \right)^3 \cdot \La^{-(K - 1) } [ \pr_a \Ddtof{Q^{jl}} ] dx &= \int \left( \fr{\pr_a Q^{jl}}{\La^{K - 1}} \right)^3 \cdot \La^{-(K - 1) } \left( \pr_a \RR^{jl}[ \pr_i v_\ep^b \pr_b Q^{ik} ] \right) dx \\
&+  \int \left( \fr{\pr_a Q^{jl}}{\La^{K - 1}} \right)^3 \cdot \La^{-(K - 1) } \pr_a \RR^{jl}[\Ddtof{U}] dx \label{eq:termIIellTr}\\
&= I + II
\end{align}

Let us first estimate the term $I$.  By H{\" o}lder's inequality, $I$ is bounded by
\begin{align} \label{ineq:ellDerivsI}
|I| &\leq \| \fr{\pr_a Q^{jl}}{\La^{K - 1}} \|_{L_x^4}^{3} \cdot \La^{-(K - 1)} \| \pr_a \RR^{jl}[ \pr_i v_\ep^b \pr_b Q^{ik}  \|_{L_x^4}
\end{align}
By the boundedness properties (\ref{bound:stdElliptic}) of $\RR$, we have that 
\begin{align}
\La^{-(K - 1)} \| \pr_a \RR^{jl}[ \pr_i v_\ep^b \pr_b Q^{ik} ] \|_{L_x^4} &\leq C \La^{-(K - 1)}  \| \nab^{K - 1} [ \pr_i v_\ep^b \pr_b Q^{ik} ] \|_{L_x^4} \\
&\leq C \La^{-(K - 1)}  \sum_{|a_1| + |a_2| = K - 1} \| \pr_{a_1} \nab v_\ep^b \|_{C^0} \| \pr_{a_2} \pr_b Q \|_{L^4} \\
&\leq C \sum_{|a_1| + |a_2| = K - 1} \left\| \fr{\pr_{a_1} \nab v_\ep^b}{\La^{|a_1|}} \right\|_{C^0} \left\| \fr{\pr_{a_2} \pr_b Q}{\La^{|a_2|}} \right \|_{L^4} \\
&\leq C \Xi e_v^{1/2} E_M[Q]^{1/4}(t) \\
&\leq C \tau^{-1} E_M[Q]^{1/4}(t)
\end{align}

Therefore, from (\ref{ineq:ellDerivsI}) $I$ is bounded by
\begin{align}
I &\leq C \tau^{-1} E_M[Q](t)
\end{align}

Now it remains to bound $II$ in line (\ref{eq:termIIellTr}).  We estimate this term by
\begin{align}
\int \left( \fr{\pr_a Q^{jl}}{\La^{K - 1}} \right)^3 \cdot \left( \La^{-(K - 1) } \right.&\left.\pr_a \RR^{jl}[\Ddtof{U}] \right) dx \leq C \tau^{-1} \int \left| \fr{\pr_a Q^{jl}}{\La^{K - 1}} \right|^4 dx \\
&+ \int \tau^3 \La^{-4(K - 1)} \left| \pr_a \RR^{jl}[\Ddtof{U}] \right|^4 dx \\
&\leq C \tau^{-1}\left( E_M[Q](t) + \tau^4 \La^{-4(K - 1)} \int \left| \nab^{K - 1} [\Ddtof{U}] \right|^4 dx \right) \\
&\leq C \tau^{-1}( E_M[Q](t) + A^4)
\end{align}
which establishes (\ref{eq:preGronEllTr}) and concludes the proof of (\ref{bd:spatialDerivsEllTr}).

\subsection{Material derivatives of the solution to the Transport-Elliptic equation}
With the estimates on spatial derivatives of the solution to (\ref{eq:transEllipt2}) in hand, we can now give bounds for the coarse-scale material derivative $\Ddtof{Q}$ of the solution to (\ref{eq:transEllipt2}) as well as its spatial derivatives.

We start by recalling the Transport-Elliptic equation
\begin{align} 
\Ddtof{Q^{jl}} &= \RR^{jl}[ \pr_i v_\ep^b \pr_b Q^{ik} + \Ddtof{U^k} ] 
\end{align}
and estimating each term by
\begin{align}
\| \nab^k \RR^{jl}[ \pr_i v_\ep^b \pr_b Q^{ik} ] \|_{L^4} &\leq C \| \nab^{|k| - 1} [ \pr_i v_\ep^b \pr_b Q^{ik} ] \|_{L^4} \\
&\leq C \sum_{|a_1| + |a_2| = |k| - 1 } \| \nab^{|a_1| + 1} v_\ep^b \|_{C^0} \| \nab^{|a_2| + 1} Q^{ik} \|_{L^4} \\
&\leq C\sum_{|a_1| + |a_2| = |k| - 1 } \Xi e_v^{1/2} \La^{|a_1|} ( \La^{|a_2|} A e^{C \tau^{-1} |t - t(I)|} ) \\
&\leq C \tau^{-1} \La^{|k| - 1} A e^{C \tau^{-1} |t - t(I)|}
\end{align}
and
\begin{align}
\| \nab^k \RR^{jl}[ \Ddtof{U} ] \|_{L^4} &\leq C \| \nab^{|k| - 1} \Ddtof{U} \|_{L^4} \\
&\leq C \tau^{-1} \La^{|k| - 1} A
\end{align}

\subsection{Cutting off the solution to the Transport-Elliptic equation}

We now use the preceding estimates on the solution to the Transport-Elliptic equation in order to establish (\ref{thm:transportElliptic}).

Let $Q_*$ be the solution to the transport-elliptic equation (\ref{eq:transEllipt2}).  Then $Q_*$ satisfies the elliptic equation
\begin{align*}
\pr_j Q_*^{jl} &= U^l,
\end{align*} 
but $Q_*$ may will not satisfy the properties desired by Theorem (\ref{thm:transportElliptic}) because $Q_*$ will not be compactly supported in time, and the estimates on its derivatives grow exponentially in time.

In order to fix this problem, we can simply cutoff $Q_*$.  Namely, let ${\bar \eta}(t)$ be a cutoff function which is equal to ${\bar \eta}(t) = 1$ on the time interval $[t(I) - \tau, t(I) + \tau]$ containing the support of $U^l$, and such that ${\bar \eta}(t)$ is supported in $[t(I) - 3\tau/2, t(I) + 3\tau/2]$.  Now define 
\begin{align}
Q^{jl} &= {\bar \eta}(t) Q_*
\end{align}

Then $Q$ solves $\pr_j Q^{jl} = U^l$, $Q$ has integral $0$, and $Q$ satisfies all of the estimates stated in the Theorem (\ref{thm:transportElliptic}) because the cutoff ensures that the exponential factors of $e^{C \tau^{-1} | t - t(I)|} $ in the estimates for $Q_*$ are all bounded.  This concludes the proof of Theorem (\ref{thm:transportElliptic}).
\end{proof}

%% file: uniqueIco.tex
Here we prove some facts about regular dodecahedra which have been used in the proof.

To prove the fact stated in Lemma (\ref{lem:someRotations}), we establish the following
\begin{lem}  Let $u \in \R^3$ be such that
\ali{
u \times f &\neq 0  \label{eq:sureWeRotate}
}
for all $f \in F$.

For every integer $n \geq 1$ there exists a constant $c > 0$ and a set of rotations $O_m$ of the form
\ali{
O_m &= e^{\th_m u \times} \quad \quad m = 1, \ldots, n
}
with the property that
\ali{
\label{eq:nonColliding}
\begin{aligned}
 |f \circ O_m  + f' \circ O_{m'}| &\geq c  \\
 f, f' &\in F \\
  m, m' &= 1, \ldots, n 
  \end{aligned}
}
holds unless $f' = -f$ and $m' = m$.
\end{lem}
\begin{proof}
We proceed by induction on $n$.

For $n = 1$, we take $\th_1 = 0$, so that 
\[ O_1 = e^{0 u \times} = \mbox{ Id } \]
is the identity.  Then the property (\ref{eq:nonColliding}) holds with $c = \min \{ |f + f'| ~|~ f, f' \in F, f \neq - f' \}$.

The rest of our rotations will all be part of a one parameter group of rotations
\ali{
 O v &= e^{\th u \times} v.
}  
Such maps are rotations, since for any $v, w \in \R^3$ we can see that for the auxiliary function $\phi(\th)$ defined by
\[\phi(\th) = {<} e^{\th u \times} v, e^{\th u \times} w {>} \]
we have
\ali{
\phi(0) &= {<} v, w {>} \\
\phi'(\th) &= {<} u \times (e^{\th u \times} v), e^{\th u \times} w {>} + {<} e^{\th u \times} v, u \times (e^{\th u \times} w) {>} \\
&= 0
}
by a basic property of the cross product.

Now assume by induction that we have rotations $e^{\th_m u \times}$, $m = 1, \ldots, n$ for which the condition (\ref{eq:nonColliding}) holds.  To find another rotation with the property (\ref{eq:nonColliding}), we observe that for $O_\th = e^{\th u \times}$, we have
\ali{
|f \circ O_\th + f' \circ O_{m}|^2 &= |f \circ O_\th|^2 + |f' \circ O_m|^2 + 2 {<} e^{\th u \times} f, e^{\th_m u \times} f' {>} \\
&= 2( 1 + {<} e^{\th u \times} f, e^{\th_m u \times} f' {>} ) \\
&= 2 ( 1 + {<} e^{(\th - \th_m) u \times} f, f' {>} )
}
By Cauchy Schwartz, this quantity can be zero only when the function
\ali{
\psi_{m, f, f'}(\th) = \psi(\th) &= {<} e^{(\th - \th_m) u \times} f, f' {>}
}
reaches a minimum value equal to $-1$.

Now observe that $\psi(\th)$ obeys the differential equation
\ali{
\psi'''(\th) + |u|^2 \psi'(\th) &= {<} (u \times)^3 e^{(\th - \th_m) u \times} f, f' {>} + |u|^2 {<} (u \times) e^{(\th - \th_m) u \times} f, f' {>} \\
\psi'''(\th) + |u|^2 \psi'(\th) &= 0 \label{eq:theDiffEqRotate}
}
As a consequence, $\psi$ has the form
\ali{
\psi(\th) &= A + B \cos( |u| \th) + C \sin( |u| \th) 
}
although at this point it is possible that $B = 0$ and $C = 0$.

Using the differential equation (\ref{eq:theDiffEqRotate}), we can give a bound on the measure of the set
\ali{
 | \{ \th \in \R / (\fr{2 \pi}{|u|} \Z) ~|~ | \psi(\th) + 1 | \leq c \} | &\leq A c^{1/2} \label{ineq:measureBound}
}
for some constant $A$.  This bound will follow from the conservation law
\ali{
 (\psi''(\th))^2 + |u|^2 (\psi'(\th))^2 &=  (\psi''(\th_m))^2 + |u|^2 (\psi'(\th_m))^2 \label{eq:conservationLaw}
}
once the conserved quantity has been shown to be nonzero at the initial $\th = \th_m$.

Let us assume by contradiction that both terms in (\ref{eq:conservationLaw}) are $0$ at $\th = \th_m$.  Namely,
\ali{
\psi'(\th_m) = {<} u \times f, f' {>} &= 0 \label{eq:lin1uf} \\
\psi''(\th_m) = {<} (u \times)^2 f, f' {>} &= 0 \\
- {<} u \times f, u \times f' {>} &= 0 \label{eq:lin2uf}
}
Then by the basic properties of the cross product we also know that
\ali{
{<} u \times f, u {>} &= 0 \label{eq:lin3uf}
}
The condition (\ref{eq:sureWeRotate}) implies that the set
\[ \{ u, f', u \times f' \} \]
form a basis of $\R^3$, so from (\ref{eq:lin1uf}), (\ref{eq:lin2uf}), and (\ref{eq:lin3uf}) we conclude that
\ali{
 u \times f &= 0
}
which contradicts (\ref{eq:sureWeRotate}).  This argument confirms that the conserved quantity
\ali{
 (\psi''(\th))^2 + |u|^2 (\psi'(\th))^2 &=  E^2 > 0 \label{eq:theConservLaw}
}
is strictly positive.  This fact can now be used to conclude the bound (\ref{ineq:measureBound}), since it implies, for example, that one of
\begin{itemize}
\item $|\psi'(\th)| \geq \fr{E}{\sqrt{2} |u|}$
\item or $|\psi''(\th)| \geq \fr{E}{\sqrt{2}}$
\end{itemize}
holds at every point $\th \in \R / ( \fr{2 \pi}{|u|} \Z )$.  Using (\ref{ineq:measureBound}), we can find a new constant $c'$ and a point $\th \in \R / ( \fr{2 \pi}{|u|} \Z )$ such that
\[  |\psi_{m, f, f'}(\th) + 1| \geq c' \quad \quad \mbox{ for all } f, f' \in F, m = 1, \ldots, n \]
which is enough to verify (\ref{eq:nonColliding}).
\end{proof}

We now give a proof of the identity (\ref{eq:icoMetric}) which we restate here in the form
\begin{align}
\de^{jl} &= \fr{1}{2} \sum_{f \in \F} f^j f^l \label{eq:newIcoMetric}  \\
&= \fr{1}{4} \sum_{f \in F} f^j f^l
\end{align}
\begin{proof}[Proof of identity (\ref{eq:icoMetric})]
To begin the proof, first observe that the bilinear form $G^{jl} = \sum_{f \in \F} f^j f^l = \fr{1}{2} \sum_{f \in F} f^j f^l$ is invariant under the action of the icosahedral group in the sense that for every symmetry $g : \R^3 \to \R^3$ of the dodecahedron, we have
\[ \tx{Sym}^2 g ( G ) (u, w) = G( u \circ g, w \circ g ) = G(u, w) ~\quad ~\tx{ for all } u, w \in (\R^3)^* \]
In particular, $G$ is invariant under the subgroup $A_5$ of orientation-preserving symmetries of the dodecahedron.  We claim that every $G^{jl} \in \SS$ that is invariant under $A_5$ has the form
\begin{align} \label{eq:icoInvt}
G^{jl} &= C \de^{jl}. 
\end{align}
Taking the trace we must have $C = \fr{G^{jl}\de_{jl}}{3}$, thus establishing (\ref{eq:icoMetric}) for $G^{jl} = \fr{1}{2} \sum_{f \in \F} f^j f^l$, which we have restated here in identity (\ref{eq:newIcoMetric}).  In the language of representation theory, our claim is that the dimension of the space 
\[ \SS^{A_5} = \{ G^{jl} ~|~ \tx{Sym}^2 g ( G^{jl} ) = G^{jl} ~\forall~ g \in A_5 \} \]
of $A_5$-invariant elements of $\SS$ is $1$, or that the trivial representation of $A_5$ has multiplicity $1$ in the representation $\SS = \tx{Sym}^2(\R^3)$.  This is a nice exercise in character theory.  To complete it, one should recall the conjugacy classes in $A_5$, namely:
\begin{itemize}
 \item The identity class $c = \langle \tx{Id} \rangle$ has $1$ element
 \item The class $c = \langle (12)(34) \rangle$ has $15$ elements of order $2$, which rotate the dodecahedron by an angle $\pi$ about an axis drawn between the midpoints of opposite edges
 \item The class $c = \langle (123) \rangle$ has $20$ three-cycles, which rotate the dodecahedron by an angle $2 \pi / 3$ about an axis drawn between two opposite vertices
 \item The class $c = \langle (12345) \rangle$ has $12$ five-cycles, which rotate the dodecahedron by an angle $2\pi/5$ about an axis drawn between the centers of opposite pentagonal faces
 \item The class $c = \langle (13524) \rangle = \langle (12345)^2 \rangle$ has $12$ five-cycles, which rotate the dodecahedron by an angle $4\pi/5$ about an axis drawn between the centers of opposite pentagonal faces
\end{itemize}

Having a geometric description of how each element in $A_5$ acts on $\R^3$ allows us to compute the eigenvalues, and in particular the traces, of the corresponding matrices.  These are useful quantities because the dimension of $\SS^{A_5}$ is equal to the trace of the operator
\[ \fr{1}{|A_5|} \sum_{g \in A_5} \tx{Sym}^2 g : \SS \to \SS \]
which projects to the invariant subspace $\SS^{A_5}$ of $\SS$.  That is,
\begin{align} \label{eq:dimensionFormula}
 \tx{dim } \SS^{A_5} &= \fr{1}{|A_5|} \left( \sum_{g \in A_5} \tx{tr Sym}^2 g \right) 
\end{align}

We can compute this trace using the plethysm formula
\[ \tx{tr Sym}^2(g) = \fr{1}{2} \left( \tx{ tr }(g^2) + \tx{ tr}^2(g) \right) \]
All these operators satisfy some polynomial $g^k - 1 = 0$, so they can be diagonalized on the complex vector space $\C^3 = \C \otimes \R^3$ and the pairwise symmetric products of their eigenvectors form a basis for $\C \otimes \SS$; then, in terms of the eigenvalues of the operators, the above identity simply expresses the equality
\begin{align} 
\sum_{1 \leq i \leq j \leq 3} \la_i \la_j = \fr{1}{2} \sum_{i = 1}^3 \la_i^2 + \fr{1}{2} \left( \sum_{i = 1}^3 \la_i \right)^2 \quad \la_i \in \C 
\end{align}
For example, when $g$ is the identity, all the eigenvalues are $1$ and we confirm that the dimension of $\SS$ is $\fr{3 + 9}{2} = 6$, and the above numerical identity is a very familiar formula for triangular numbers.


All of our operators $g$ are rotations in $\R^3$ by some angle $\th$; the trace of such an operator is given by $1 + e^{i \th} + e^{-i\th}$.  For example, letting $\om_k = e^{2 \pi i / k}$, the three cycle acts with trace $\tx{tr }(123) = 1 + \om_3 + \om_3^{-1} = 1 + \om_3 + \om_3^2 = (1 - \om_3^3)/(1-\om_3)= 0$.  
Using these observations, we can build a table
\begin{table}[h]
\begin{tabular}{| l | l | l | l | l | l |}
    \hline
      & Id & $\langle (12)(34) \rangle$ & $\langle (123) \rangle$ & $\langle (12345) \rangle$ & $\langle (13524) \rangle$ \\ \hline
    $\#$ & $1$ & $15$ & $20$ & $12$ & $12$ \\ \hline
    $\tx{tr }$ & $3$ & $-1$ & $0$ & $1 + \om_5 + \om_5^4$ & $1 + \om_5^2 + \om_5^3$ \\ \hline
 $\tx{tr Sym}^2$ & $6$ & $2$ & $0$ & $\fr{(1 + \om_5 + \om_5^4)^2 + (1 + \om_5^2 + \om_5^3)}{2}$ & $\fr{(1 + \om_5^2 + \om_5^3)^2 + (1 + \om_5 + \om_5^4)}{2}$ \\
\hline
  \end{tabular}
\end{table}

When we sum the terms in the formula (\ref{eq:dimensionFormula}) which come from the $5$-cycles, there is a cancellation from $1 + \om_5 + \om_5^2 + \om_5^3 + \om_5^4 = 0$.  Using this observation and some observations about the geometry of a regular pentagon, we can now see that
\begin{align}
 \tx{dim } \SS^{A_5} &= \fr{1}{60} \left( 6 + 2 \cdot 15 + 12 \cdot \fr{1 + (1 + \om_5^2 + \om_5^3)^2 + (1 + \om_5 + \om_5^4)^2}{2} \right) \\
&< \fr{1}{10} \left(1 + 5 + ( 1 + 2^2 + 3^2 ) \right) = 2
\end{align}
Of course, the left hand side is an integer, and it is at least $1$ because $\de^{jl} \in \SS^{A_5}$ is nonzero, so this bound is enough to conclude the proof.
\end{proof}

As a first application of the identity (\ref{eq:newIcoMetric}), we calculate the angle between projectively distinct faces of the dodecahedron.
\begin{lem}
Suppose that $x, y \in \F$ and $y \in  \F \setminus \{ x \}$, then
\begin{align} \label{eq:theInnerProd}
 ( x \cdot y )^2 &= \fr{1}{5}
\end{align}
Hence,
\begin{align}
 | x \wed y |^2 = |x|^2 |y|^2 - (x \cdot y)^2 &= \fr{4}{5} 
\end{align}
\end{lem}
\begin{proof}
Let $x$ be a fixed face of the projective dodecahedron.  Let $\si_x$ be a nontrivial rotation which fixes the pentagonal face $x$ and permutes the $5$ faces adjacent to $x$.  Then $\si_x$ has order $5$, and cannot fix any of the five adjacent projective faces because it does not obtain the eigenvalue $-1$ while acting on $\R^3$, and the eigenvalue $1$ is obtained only in the $x$ direction itself since the rotation is nontrivial.  

Hence, $\si_x$ acts transitively on the other five faces in $\F \setminus \{ x \}$, and therefore the number
\begin{align}
( x \cdot y )^2 &= ( \si_x^k x \cdot \si_x^k y )^2 \\
&= (x \cdot \si_x^k y)^2
\end{align}
does not depend on the other face $y \in \F \setminus \{ x \}$.

To calculate this number, we apply the identity (\ref{eq:newIcoMetric}) to give
\begin{align}
\de^{jl} x_j x_l &= \fr{1}{2} \sum_{f \in \F} f^j f^l x_j x_l \\
|x|^2 &= \fr{1}{2}( |x|^2 + 5(x \cdot y)^2 ) \\
(x \cdot y)^2 &= \fr{|x|^2}{5} = \fr{1}{5}
\end{align}
\end{proof}

As a final application of the identity (\ref{eq:newIcoMetric}), we establish the linear independence Claim (\ref{en:basis1}) of Lemma (\ref{lem:basisLem}).  Claim (\ref{en:basis2}) of Lemma (\ref{lem:basisLem}) can be proven using the same argument below.
\begin{proof}[Proof of Lemma (\ref{lem:canInvert})]
Suppose we have a linear relation
\begin{align} \label{eq:fakeCombo}
\sum_{f \in \F} \a_f f^jf^l &= 0 
\end{align}
As we discussed in the proof of (\ref{eq:theInnerProd}), every face $f_\ast \in \F$, has a set of five adjacent faces $\NN(f_\ast)$.  They can be obtained by starting with a single adjacent side $x$, and then rotating around the center of the face $f_\ast$ using the five-cycle $\si_{f_\ast}$.  These faces represent all the other projective faces in $\F$ 
\[ \NN(f_\ast) = \{ \si_{f_\ast}^k x ~|~ k = 0, \ldots, 4 \} = \F \setminus \{ f_\ast \} \] 
since $\si_{f_\ast}$ only obtains an eigenvalue of $\pm 1$ on $f_\ast$ itself, and $\si_{f_\ast}$ therefore must act transitively on the other $5$ faces since $\si_{f_\ast}$ has order $5$ and acts faithfully.

We can now act on the linear combination (\ref{eq:fakeCombo}) with the group $\langle \si_{f_\ast} \rangle$ to produce potentially new relations, and by averaging these relations we see that
\begin{align}
\a_{f_\ast} f_\ast^jf_\ast^l + \hat{\a}(f_\ast) \sum_{f \in \NN(f_\ast)} f^j f^l = 0,
\end{align}
where ${\hat \a}(f_\ast) = \fr{1}{5} \sum_{f \in \NN(f_\ast)} \a_f$.  We first subtract both the summation over $\NN(f_\ast)$ and the term ${\hat \a}(f_\ast) f_\ast^j f_\ast^l$ from the equation, then use the identity (\ref{eq:newIcoMetric}) to see that
 \begin{align}
( \a_{f_\ast} - {\hat \a} ) f_\ast^jf_\ast^l &= - \hat{\a} \sum_{f \in \F} f^j f^l \\
&= - 2 {\hat \a} \de^{jl}.
\end{align}
Both sides must be zero; otherwise they would differ in rank.  In particular, ${\hat \a}(f_\ast) = 0 = \a_{f_\ast}$ for all $f_\ast \in \F$.
\end{proof}

